\documentclass[reqno,12pt]{amsart}
\usepackage{amsmath,amsfonts,amsthm,amssymb,indentfirst}
\usepackage[all,poly]{xy}
\usepackage{color}
\usepackage[pagebackref,colorlinks]{hyperref}
\usepackage{kantlipsum}
\usepackage[all,poly]{xy}
\usepackage{mathrsfs}
\usepackage{enumerate}

\theoremstyle{plain}
\newtheorem{lemma}{Lemma}[section]
\newtheorem{theorem}[lemma]{Theorem}
\newtheorem{corollary}[lemma]{Corollary}
\newtheorem{proposition}[lemma]{Proposition}

\theoremstyle{definition}
\newtheorem{remark}[lemma]{Remark}
\newtheorem{example}[lemma]{Example}
\newtheorem{definition}[lemma]{Definition}

\newcommand{\Zset}{\mathbb Z}

\newcommand{\M}{\operatorname{\mathbb M}}
\newcommand{\so}{\mathbf{s}}
\newcommand{\ra}{\mathbf{r}}
\newcommand{\V}{\mathcal V}
\newcommand{\F}{\mathcal F}
\newcommand{\Pa}{\mathcal P}
\newcommand{\ol}{\overline}
\newcommand{\POG}{\mathbf {POG}}
\newcommand{\POM}{\mathbf {POM}}

\newcommand{\Ter}{\operatorname{Ter}}
\newcommand{\tot}{\operatorname{tot}}
\newcommand{\can}{\operatorname{can}}
\newcommand{\red}{\operatorname{red}}
\newcommand{\gr}{\operatorname{gr}}
\newcommand{\mymod}{\operatorname{mod}}
\newcommand{\tails}{\operatorname{tails}}
\newcommand{\quot}{\operatorname{quot}}\newcommand{\direct}{\operatorname{dir}}
\newcommand{\cut}{\operatorname{cut}}

\title[GCC for graphs with disjoint cycles]{The Graded Classification Conjecture holds for graphs with disjoint cycles}

\author{Lia Va\v s}
\address{Department of Mathematics, Saint Joseph's University, Philadelphia, PA 19131, USA}
\email{lvas@sju.edu}



\thanks{The author is very grateful to S\o{}ren Eilers and Roozbeh Hazrat for their comments which contributed to improvement of the arXiv versions v1 and v2.}

\begin{document}

\begin{abstract}
The Graded Classification Conjecture (GCC) states that the pointed $K_0^{\operatorname{gr}}$-group is a complete invariant of the Leavitt path algebras of finite graphs when these algebras are considered with their natural grading by $\mathbb Z.$ The conjecture has previously been shown to hold in some rather special cases. The main result of the paper shows that the GCC holds for a significantly more general class of graphs -- countable graphs with disjoint cycles, with only finitely many infinite emitters, finitely many sinks and cycles and such that every infinite path ends in a cycle. In particular, our result holds for finite graphs with disjoint cycles (the Toeplitz graph is such, for example).
We formulate and show the main result also for graph $C^*$-algebras. As a consequence, the graded version of the Isomorphism Conjecture holds for the class of graphs we consider.

Besides showing the conjecture for the class of graphs we consider, we realize the Grothendieck $\mathbb Z$-group isomorphism by a specific graded $*$-isomorphism. In particular, we introduce a series of graph operations which preserve the graded $*$-isomorphism class of their algebras. After performing these operations on a graph, we obtain well-behaved ``representative'' graphs, which we call canonical forms. We consider an equivalence relation $\approx$ on graphs such that $E\approx F$ holds when there is a graph isomorphism between some of the canonical forms of $E$ and $F$. In the main result, we show that the relation $E\approx F$ is equivalent to the existence of an isomorphism $f$ of the Grothendieck $\mathbb Z$-groups of the algebras of $E$ and $F$ in the appropriate category. As $E\approx F$ can be realized by a finite series of specific graph operations, any such isomorphism $f$ can be realized by {\em an explicit graded $*$-algebra isomorphism}. Because of this, our main result describes the graded ($*$-)isomorphism classes of the algebras of graphs we consider. Besides the ties to symbolic dynamics and Williams' Problem, such a description is relevant for the active program of classification of graph $C^*$-algebras.
\end{abstract}

\maketitle

\section{Introduction}

If $E$ is a directed graph and $K$ a field, the Leavitt path algebra $L_K(E)$ and its operator theory counterpart, the graph $C^*$-algebra $C^*(E),$ are naturally graded by the group of integers $\Zset$. It is often advantageous to consider the algebras with this grading. For example, it is easy to come up with examples of graphs with isomorphic pointed $K_0$-groups of their Leavitt path algebras while the algebras themselves are not isomorphic. No such examples are known if the algebras are considered with the grading and their $K_0$-groups are adjusted accordingly. The Graded Classification Conjecture, formulated for finite graphs by Roozbeh Hazrat in \cite{Roozbeh_Annalen}, states that such examples do not exist.

\subsection{The conjecture and its  state} \label{subsection_the_GCC_intro}
In the unital case (when $E$ has finitely many vertices), the $\Zset$-grading of $L_K(E)$ induces an action of the infinite cyclic group $\Gamma=\langle t\rangle\cong \Zset$ on the set of the graded isomorphism classes of finitely generated graded projective $L_K(E)$-modules. This action makes the Grothendieck group, formed using the finitely generated {\em graded} projective modules and their {\em graded} isomorphism classes, into a pre-ordered $\Gamma$-group. Although the notation $K_0^{\operatorname{gr}}(L_K(E))$ has often been used for this group, we use $K_0^\Gamma(L_K(E))$ in order to emphasize that the Grothendieck group itself is not graded by $\Gamma$ but that $\Gamma$ acts on it. If $L_K(E)$ is unital, the graded isomorphism class $[L_K(E)]$ is an order-unit of the group $K_0^\Gamma(L_K(E)).$
If $L_K(E)$ is not unital, $K_0^\Gamma(L_K(E))$
can be defined via a unitization
of $L_K(E)$ and a certain generating interval can be considered instead of $[L_K(E)]$ (see section \ref{subsection_graded_Grothendieck} or \cite{Roozbeh_Lia_Ultramatricial} for more details).
A widely accepted formulation (e.g. in \cite{LPA_book} and in  \cite{Ara_Pardo_graded_K_classification}) of The Graded Classification Conjecture (GCC) states the equivalence of the two conditions below for finite graphs $E$ and $F$ and a field $K.$
\begin{enumerate}
\item[(1$^u$)] There is an order-preserving $\Gamma$-group isomorphism of $K_0^\Gamma(L_K(E))$ and $K_0^\Gamma(L_K(F))$ which maps the order-unit $[L_K(E)]$ to the order-unit $[L_K(F)].$
\item[(2$^u$)] The algebras $L_K(E)$ and $L_K(F)$ are graded isomorphic.
\end{enumerate}

By allowing infinite emitters to be present and by considering the generating intervals instead of the order-units, one can also formulate the GCC for the non-row-finite and the non-unital cases and thus extend its scope to all graphs. This generalized version, which we also refer to as the GCC, is stating the equivalence of conditions (1) and (2) below. Any involution on $K$ equips $L_K(E)$ with an involution and we also consider the equivalence of condition (3) with (1) and (2).
\begin{enumerate}
\item There is an order-preserving $\Gamma$-group isomorphism of $K_0^\Gamma(L_K(E))$ and $K_0^\Gamma(L_K(F))$ which maps the generating interval $D_{L_K(E)}$ onto the generating interval $D_{L_K(F)}.$
\item The algebras $L_K(E)$ and $L_K(F)$ are graded isomorphic.

\item The algebras $L_K(E)$ and $L_K(F)$ are graded $*$-isomorphic.
\end{enumerate}

In \cite{Roozbeh_Annalen}, the GCC is shown to hold for polycephalic graphs. These are finite graphs in which every vertex connects to a sink, a cycle without exits, or to a vertex emitting no other edges but finitely many loops and the graph is such that when these loops and an edge of each cycle with no exits are removed, the result is a finite acyclic graph. In \cite{Ara_Pardo_graded_K_classification}, a weaker version of the GCC is shown to hold for finite graphs without sources or sinks. In \cite{Roozbeh_Lia_Ultramatricial}, the GCC, generalized to include the non-unital case, is shown to hold for row-finite, countable graphs in which no cycle has an exit and in which every infinite path ends in a (finite or infinite) sink or in a cycle if the involution on the underlying field is reasonably well-behaved (as the complex-conjugation on $\mathbb C$ is). By \cite{Eilers_Ruiz_Sims_reconstruction}, the GCC holds for countable graphs for which whenever there is an edge from a vertex $v$ to a vertex $w,$ there are infinitely many edges from $v$ to $w$.
In \cite{Eilers_Ruiz_3-bit}, the GCC is shown to hold for two  subclasses of the class of graphs we consider: acyclic graphs with finitely many vertices and 2-S-NE graphs with finitely many vertices. We also note that \cite{Roozbeh_Cortinas} contains a survey of the present state of the GCC.

In recent work \cite{Do_et_al_Williams}, the term ``Graded Classification Conjecture''
is defined differently, as a statement that condition (1) {\em without the condition on the generating intervals} is equivalent to the condition that the Leavitt path algebras are graded Morita equivalent. In our present work, we use the term GCC to refer only to its {\em original} form.

The strong version of the GCC states that the functor $K_0^{\Gamma}$ is full and faithful when considered on the category of Leavitt path algebras of finite graphs and their graded homomorphisms modulo conjugations by invertible elements of the zero components. Recently and almost simultaneously, in \cite{Arnone} and in \cite{Lia_classification}, it was shown that the functor $K_0^{\Gamma}$ is full, in \cite{Lia_classification} for countable graphs with finitely many vertices and, in \cite{Arnone} for finite graphs.  In \cite{Ara_Pardo_graded_K_classification}, it is shown that $K_0^{\Gamma}$ is not faithful.

\subsection{The composition S-NE graphs} Composition series of graphs were introduced in \cite{Lia_porcupine_quotient}. Since this concept is pivotal for our approach to the proof, we review it in section \ref{section_prerequisites}, together with some material on  porcupine-quotients, the graph $\Gamma$-monoid and some other results of \cite{Lia_porcupine_quotient}. We recall that each graded ideal of $L_K(E)$ uniquely corresponds to a pair $(H,S)$ of two subsets of the set of vertices $E^0$ of a graph $E$ called an {\em admissible pair}. The admissible pairs can be ordered in such a way that the lattice of graded ideals is isomorphic to the lattice of admissible pairs. The {\em porcupine-quotient} graph $(G,T)/(H,S)$ of two such pairs $(H,S)\leq (G,T),$ also introduced in \cite{Lia_porcupine_quotient}, has its Leavitt path algebra graded $*$-isomorphic to the quotient $I(G,T)/I(H,S)$ of the two corresponding ideals. This correspondence enables us to transfer the consideration of a graded composition series of $L_K(E)$ to the consideration of a composition series of $E.$ In particular, $E$ has a {\em composition series of length $n$} if there is a finite chain of admissible pairs \[(\emptyset, \emptyset)=(H_0, S_0) \leq (H_1, S_1)\leq \ldots\leq (H_n, S_n)=(E^0, \emptyset)\] such that the porcupine-quotient graph $(H_{i+1}, S_{i+1})/(H_i, S_i)$ is cofinal for all $i=0, \ldots, n-1$ (cofinality of a graph corresponds to graded simplicity of its Leavitt path algebra, see section \ref{subsection_graphs} for a review of this concept).

By \cite[Theorem 5.7]{Lia_porcupine_quotient} (stated here as Theorem \ref{theorem_graded_simple}), a cofinal graph has exactly one of  four types of certain {\em terminal} vertices. Only two of the four types are relevant in this paper: a sink or the set of vertices on a cycle without exits. If only those two types appear for every cofinal porcupine-quotient of $E$, we say that $E$ is an {\em S-NE} graph. Here S is used for ``sinks'' and NE for ``no-exits''. If $E$ is an S-NE graph and it has a composition series,  we say that $E$ is a {\em composition S-NE graph}. If the length of such a composition series for $E$ is $n$, $E$ is an {\em $n$-S-NE graph} (Definition \ref{definition_S_NE_graphs}).

If $E$ is a composition S-NE graph, each of its (finitely many) infinite emitters is a sink of exactly one composition factor. Each of the (disjoint and finite in number) cycles of $E$ is a cycle without exits of exactly one composition factor. Our choice to work with composition S-NE graphs with possibly infinitely many vertices is not a vacuous exercise in generalization, but an absolute necessity: a porcupine graph of a graph with finitely many vertices can have infinitely many vertices as the example of the Toeplitz graph in section \ref{subsection_2SNE_intro} illustrates.

\subsection{1-S-NE graphs}
To prove our main result, Theorem \ref{theorem_GCC_disjoint_cycles}, we start by considering 1-S-NE graphs (i.e., cofinal graphs with either a sink or a cycle without exits)
in section \ref{section_n=1}. It is known that the GCC holds for such finite graphs (\cite[Theorem 4.8]{Roozbeh_Annalen}) and such general graphs under some assumption on $K$ ( \cite[Theorem 5.5]{Roozbeh_Lia_Ultramatricial}). We show that GCC holds for all 1-S-NE graphs in Proposition \ref{proposition_cofinal_graphs}. We introduce a canonical form of a 1-S-NE graph (unique up to a graph isomorphism) and show that the conditions from the GCC are equivalent to the canonical forms of the graphs being isomorphic (as graphs). In graphical representations of canonical forms of 1-S-NE graphs below, we abbreviate the graphs as follows: if $v$ receives $k$ edges originating at sources, no matter whether $k$ is a finite or an infinite cardinal, we depict this as $\xymatrix{\bullet\ar[r]^{(k)}&\bullet^v}.$

The first graph below is a canonical form of a 1-S-NE graph with a sink. Its algebra is graded isomorphic to
$\M_{\kappa}(1,\mu_1, \mu_2,\ldots, \mu_k),$ where $k$ can be finite or infinite cardinal and $\kappa$ is the (cardinal) sum $1+\sum_{i=1}^k \mu_i.$
The second graph below is a canonical form of a 1-S-NE graph with a cycle of length $m$. Its algebra is graded isomorphic to $\M_{\kappa}[x^m, x^{-m}](\mu_0, \mu_1,\ldots, \mu_{m-1}),$ where $\kappa=\sum_{i=0}^{m-1} \mu_i.$
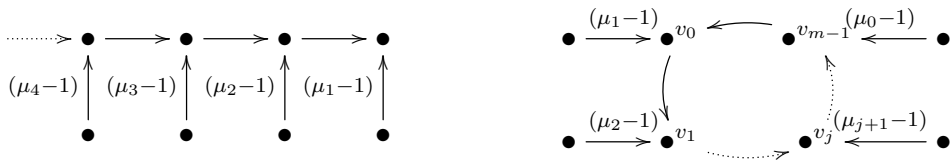
\[\xymatrix{\ar@{.>}[r]&\bullet\ar[r]&\bullet\ar[r]&\bullet\ar[r]&\bullet\\&\bullet\ar[u]^{(\mu_4-1)}&\bullet\ar[u]^{(\mu_3-1)}&\bullet\ar[u]^{(\mu_2-1)}&\bullet\ar[u]^{(\mu_1-1)}}
\hskip2cm
\xymatrix{  \bullet\ar[r]^{(\mu_1-1)}&
{\bullet}^{v_0}\ar@/_/[d]&
{\bullet}^{v_{m-1}}\ar@/_/[l]&  \bullet\ar[l]_{(\mu_0-1)}&\\
\bullet\ar[r]^{(\mu_2-1)}&
{\bullet}^{v_1}\ar@{.>}@/_/[r]&
{\bullet}^{v_j}\ar@{.>}@/_/[u]&  \bullet\ar[l]_{(\mu_{j+1}-1)}&} \]

We refer to the horizontal line of length $k$ in the first graph as the {\em spine} of the graph and to the edges not on the spine which end in the spine as the {\em tails}. In the second graph, we also refer to the edges ending in the cycle but not on the cycle as the {\em tails.}

\subsection{2-S-NE graphs}
\label{subsection_2SNE_intro}
The work on the 2-S-NE graphs in section \ref{section_n=2} contains crucial arguments for the general case. We introduce {\em canonical} graphs (Definition \ref{definition_canonical_forms_n=2}) of 2-S-NE graphs
and we develop graph operations needed to transform a given 2-S-NE graph into its canonical form. For example, if $E$ is the Toeplitz graph
$\;\;\;\xymatrix{{\bullet}^{v}\ar@(lu,ld)  \ar[r] & {\bullet}^{w}},$ then it is
a 2-S-NE graph since its composition series is $(\emptyset, \emptyset)\leq (\{w\}, \emptyset)\leq (E^0, \emptyset).$ The composition factors are the porcupine graph $P_{\{w\}}$, listed first, and the quotient $E/\{w\}$ listed second below.
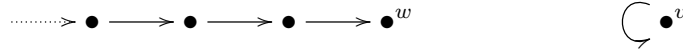
\[\xymatrix{\ar@{.>}[r] &
\bullet\ar[r] &  \bullet\ar[r]& \bullet\ar[r] & \bullet^w }\hskip3cm \xymatrix{{\bullet}^{v}\ar@(lu,ld)}\]
It turns out that the Toeplitz graph $E$ is its own canonical form and that it is a canonical form of any of the graphs below.
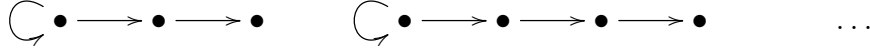
\[\xymatrix{\bullet\ar@(lu,ld)  \ar[r] & \bullet \ar[r] & \bullet}
\hskip1.5cm \xymatrix{\bullet\ar@(lu,ld) \ar[r] & \bullet \ar[r] & \bullet \ar[r] & \bullet}\hskip1.5cm\ldots\]

The order-unit is not available to us for graphs with infinitely many vertices. So, for such graphs, we use cardinality arguments instead of considering the order units  -- {\em we count the tails}.
The next example illustrates how this is achieved in a simple case.

Let $E_1, E_2, E_3,$ and $E_4$ be the four graphs below.
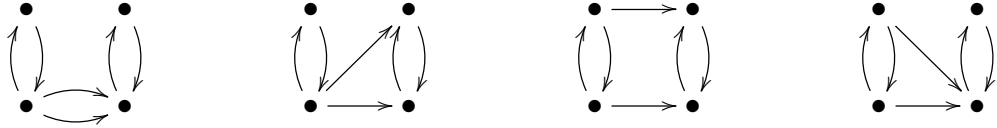
\[\xymatrix{\bullet\ar@/^/[d]&\bullet\ar@/^/[d]\\
\bullet\ar@/^/[r]\ar@/_/[r]\ar@/^/[u]&\bullet\ar@/^/[u]}
\hskip2cm\xymatrix{\bullet\ar@/^/[d]&\bullet\ar@/^/[d]\\
\bullet\ar[ur]\ar[r]\ar@/^/[u]&\bullet\ar@/^/[u]}\hskip2cm\xymatrix{\bullet\ar[r]\ar@/^/[d]&\bullet\ar@/^/[d]\\
\bullet\ar[r]\ar@/^/[u]&\bullet\ar@/^/[u]}\hskip2cm
\xymatrix{\bullet\ar[dr]\ar@/^/[d]&\bullet\ar@/^/[d]\\
\bullet\ar[r]\ar@/^/[u]&\bullet\ar@/^/[u]}\]
By Definition \ref{definition_canonical_forms_n=2}, $E_1$ and $E_2$ are canonical and the two graphs below, $E_3'$ and $E_4',$ are canonical forms of $E_3$ and $E_4.$
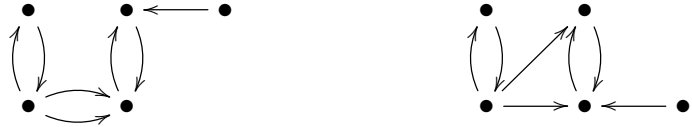
\[\xymatrix{\bullet\ar@/^/[d]&\bullet\ar@/^/[d]&\bullet\ar[l]\\
\bullet\ar@/^/[r]\ar@/_/[r]\ar@/^/[u]&\bullet\ar@/^/[u]&}\hskip3cm
\xymatrix{\bullet\ar@/^/[d]&\bullet\ar@/^/[d]&\\
\bullet\ar[ur]\ar[r]\ar@/^/[u]&\bullet\ar@/^/[u]&\bullet\ar[l]}\]

If the types of exits from cycles differ (as, for example, they do for $E_1$ and $E_2$), we show that there is no isomorphism (in the appropriate category) of the $\Gamma$-groups of the graphs. If the types of exits are the same but the number of tails is different (as, for example, it is the case for $E_1$ and $E_3'$), we also we eliminate the possibility that the $\Gamma$-groups are isomorphic. Thus, the four graphs have $\Gamma$-groups in different isomorphism classes and, consequently, their algebras are in different graded $*$-isomorphism classes.

A canonical form is obtained by a series of graph operations we introduce. Some of these operations are out-split and in-split plus moves and their inverses  and, as such, are  the ``graph moves'' of symbolic dynamics. However, some of the operations we consider are not any of the previously considered graph moves in particular if they entail infinitely many changes of some graph elements. Each of the operations $\phi: E\to F$ we consider is defined on the set of vertices and edges of $E$ and its values are elements of $L_K(F)$ which satisfy the axioms (V) and (E1) of the Leavitt path algebra axioms (we review them in section \ref{subsection_LPA}). The map $\phi$ is such that $\phi(v)\neq 0$ for $v\in E^0,$ such that $\phi(v)$ is a homogeneous element of degree zero, and such that $\phi(e)$ is a homogeneous element of degree one for $e\in E^1$. The map $\phi$ is also such that if the values on the ghost edges are defined by $\phi(e^*)=\phi(e)^*$ for $e\in E^1$, then $\phi$ is such that (E2), (CK1), and (CK2) hold. By the Universal Property of Leavitt path algebras, $\phi$ extends to a homomorphism which we continue to call $\phi.$ The property $\phi(e^*)=\phi(e)^*$ ensures that the extension is a $*$-homomorphism. The extension is a graded homomorphism by the requirement on the degrees of $\phi(v)$ and $\phi(e)$ for $v\in E^0$ and $e\in E^1$ and it is injective by the Graded Uniqueness Theorem. For every such map $\phi$, we check that there is an operation $\psi: F\to E$ having analogous properties as $\phi$ and such that $\phi$ and $\psi$ compose to the identity maps on the sets of vertices and edges. This requirement implies that the extension of $\phi$ is a {\em graded $*$-algebra isomorphism.} In particular, the process of obtaining a canonical form of a graph $E$ yields a specific graded $*$-isomorphism of the algebras of $E$ and its canonical form.

The basic operations we consider are of the following three types. The first two types are
the existing out-splits and in-split plus moves. The operation we refer to as a {\em path rearrangement} is related to in-split plus moves except that it can be conducted without any consideration of the in-split minus helper. The third type is an operation we refer to as a {\em cut map}. Cut maps have not been considered before and they are applicable only to graphs with infinitely many vertices. As opposed to the out-splits and in-split plus moves, both the path rearrangement and the cut maps are {\em global} operations in the sense that they are not related to any specific vertex.

We also introduce an equivalence relation $\approx$ on $n$-S-NE graphs so that $E\approx F$ holds if there are canonical forms $E_{\can}$ and $F_{\can}$ of $E$ and $F$ which are isomorphic.
This makes the graphs $E$ and $F$ relatable on the graph level by $\phi_F^{-1}\iota\phi_E$ for $\phi_E: E\to E_{\can},$ $\phi_F: F\to F_{\can},$ and $\iota: E_{\can}\cong F_{\can}.$ Theorem \ref{theorem_n=2}, In the main result on countable 2-S-NE graphs,  we show that the relation
\begin{enumerate}
\item[(4)] $E\approx F$.
\end{enumerate}
holds exactly when conditions (1) to (3) hold. In addition, if
$f: K_0^\Gamma(E)\to K_0^\Gamma(F)$ is an isomorphism from condition (1), we {\em realize} $f$ by a specific graded $*$-algebra isomorphism obtained from a graph operation on two isomorphic canonical forms of $E$ and $F.$ By ``realize $f$'', we mean that we exhibit a graph operation which extends to a graded $*$-isomorphism $\phi$ of the two algebras such that $K_0^\Gamma(\phi)=f.$
Since every canonical form is obtained by a finite series of specific graph operations, any isomorphism of two $\Gamma$-groups of any two graphs, not necessarily canonical, is realizable. We generalize this result to countable $n$-S-NE graphs. Thus, canonical forms can be used to explicitly describe the graded ($*$-)isomorphism class of Leavitt path and graph $C^*$-algebras.

The proof of Theorem \ref{theorem_n=2} consists of  three steps. First, we show Proposition \ref{proposition_E_without_H} which we refer to as {\em the Quotient Proposition} and which focuses on the in-split plus related moves. It establishes an isomorphism of the quotients of graphs under certain conditions. Second, we prove Lemma \ref{lemma_f_identity_implies_cut_graphs_isomorphic}, which we refer to as {\em the Cut Lemma}. It implies that the tail graphs have equal cardinality of tails under certain conditions. This lemma focuses on the cut maps. Having the proposition and the lemma, the proof of Theorem \ref{theorem_n=2} focuses on the format of the $\POM^D$-isomorphism and the operations related to out-splits.

\subsection{The main result and its corollaries}
In section \ref{section_n>2}, we introduce the graph operations on $n$-S-NE graphs for $n>2$ and define canonical forms of such graphs. Suitably selected out-splits transform an $n$-S-NE graph to a graph without breaking vertices and with hereditary and saturated sets $H_j, j=1,\ldots, n$ such that $\emptyset \subseteq H_1\subseteq H_2\subseteq\ldots \subseteq H_n=E$ is a composition series of $E.$

Let $c_j$ be a terminal cycle of the 1-S-NE graph $H_{j+1}/H_{j}$ (we treat a sink as an improper cycle of length zero) for $j=0,\ldots, n-1$.
Our particular focus is on the most involved case when $c_{j+1}$ emits paths to $c_j$ for every $j.$ We start consideration with the 1-S-NE graph $E/H_{n-1}$ which we refer to as the 1-quotient. Using out-splits, we make $c_n$ emit exits to $H_{n-1}$ from a single vertex, say $v_{n0}$ and we look for a form which either prevents us from making certain in-split plus moves or allows us to make arbitrarily many of them. Out-split reducing the resulting graph as much as possible does not impact the 1-quotient any more and we continue the process by considering the 2-quotient $E/H_{n-2}$ and by repeating the previous steps: we make $c_{n-1}$ emit exits to $H_{n-2}$ from a single vertex, then we consider the in-split plus moves of this vertex and make $E/H_{n-2}$ have the 2-quotient canonical form in the 3-S-NE graph $E/H_{n-3}.$ This process involves the consideration of $H_{n-1}/H_{n-2},$ so each porcupine-quotient gets to have its role in the process.

We continue the process of considering $E/H_{j}$ as the quotient of $E/H_{j-1}$ for $j$ deceasing from $n$ to $2.$ The process terminates when the $(n-1)$-S-NE quotient $E/H_1$ is in its quotient canonical form. After that, we consider the $H_j$-to-$H_1$ part for $j$ taking values $2,3,\ldots, n$ in that order. Any changes to this part do not impact the quotient $E/H_1$ and they do not impact the $H_k$-to-$H_1$ parts for $k<j.$  This  ``bottom up'' consideration of the quotients first followed by the ``top down'' consideration of the paths ending in $c_1,$ produces a graph which is a canonical form of $E$. The main result, Theorem \ref{theorem_GCC_disjoint_cycles}, states that the conditions (1) to (4) are equivalent for countable composition S-NE graphs.

The proof of Theorem \ref{theorem_GCC_disjoint_cycles} consists of three steps which match those used for Theorem \ref{theorem_n=2}: Proposition \ref{proposition_on_canonical_quotients_n}, {\em the General Quotient Proposition} focuses on the in-split plus related operations, Lemma \ref{lemma_glavna_in_general_case}, {\em the General Cut Lemma}, focuses on the cut maps, and, finally, the proof of Theorem \ref{theorem_GCC_disjoint_cycles} focuses on the out-split related operations.

A direct corollary of Theorem \ref{theorem_GCC_disjoint_cycles} is that the GCC holds for countable graphs with finitely many vertices and disjoint cycles (Corollary \ref{corollary_unital}). In Corollary \ref{corollary_C_star}, we establish the graph $C^*$-algebra version of our main result: there is a gauge-invariant isomorphism of the graph $C^*$-algebras of two countable composition S-NE graphs if and only if the $K_0^\Gamma$-groups of the two $C^*$-algebras, considered with their generating intervals, are isomorphic.
This implies that the graded version of the {\em Strong Isomorphism Conjecture} holds for the class of graphs we consider. The Isomorphism Conjecture, posed in \cite{Gene_Mark_Iso_Conjecture}, states that two Leavitt path algebras over $\mathbb C$ are isomorphic as algebras precisely when they are isomorphic as $*$-algebras, which is known as
the Isomorphism Conjecture, and that any such algebras are isomorphic as rings precisely when they are isomorphic as $*$-algebras, which is known as the Strong Isomorphism Conjecture.
By \cite[Propositions 7.4 and 8.5]{Gene_Mark_Iso_Conjecture}, the latter conjecture holds for Leavitt path algebras of countable acyclic graphs as
well as row-finite cofinal graphs with a cluster of extreme cycles. By \cite[Theorem 14.7]{Eilers_et_al_classification}, the Strong Isomorphism Conjecture holds for all graphs with finitely many vertices.
The graded versions of these conjectures have every instance of ``isomorphism'' replaced by ``graded isomorphism''. Corollary \ref{corollary_iso_conjecture} states that the following conditions are equivalent with conditions (1) to (4) above.
\begin{enumerate}
\item[(5)] The algebras $L_K(E)$ and $L_K(F)$ are graded isomorphic as rings.

\item[(6)] The algebras $L_K(E)$ and $L_K(F)$ are graded isomorphic as $*$-rings.

\item[(7)] The algebras $C^*(E)$ and $C^*(F)$ are graded isomorphic.
\end{enumerate}

The equivalence of conditions (2) to (4) indicates that Theorem \ref{theorem_GCC_disjoint_cycles} does more than prove the GCC for the class of graphs we consider -- it describes the graded isomorphism class (and graded $*$-isomorphism class) of the algebra (Leavitt path as well as graph $C^*$). This is relevant for an ongoing initiative to classify all graph $C^*$-algebra (the introduction to \cite{Eilers_et_al_classification} contains a comprehensive overview of such initiative and some important results towards obtaining a complete classification).

The operations realizing relation $\approx$ on composition S-NE graphs are diagonal preserving (Lemma \ref{lemma_approx_preserves_the_diagonal}). Thus, the conditions (1) to  (7) are equivalent with the requirements that the isomorphisms in conditions (2), (3) and (5) to (7) are also diagonal-preserving  (Proposition \ref{proposition_on_diagonal}). Given \cite[Theorem 7.3]{Carlsen_et_al}, this relates the GCC to {\em Williams' Problem} and section \ref{subsection_diagonal} contains more details.

We briefly reflect on possible future directions. The methods of our proof indicate that the GCC should generally be considered along with a condition on graphs and that using the length of a composition series for induction is promising.

Every graph with finitely many vertices has a composition series where each composition factor has either a sink, a cycle without exits, or a cluster of extreme cycles. When considering composition graphs whose factors may have any of these three types of terminal elements, some of the methods of our proof may still be applicable. The remaining roadblock to proving the GCC for all graphs with finitely many vertices seems to be proving it for cofinal graphs with extreme cycles. If this turns out to be possible, some of our arguments on the length of a composition series could be applicable to establishing the equivalence of  conditions (1) to (7) for  graphs with finitely many vertices.

\section{Prerequisites}\label{section_prerequisites}

In the rest of the paper,
we use $\Zset^+$ to denote the set of nonnegative integers and $\omega$ for the same set but when considered as the countably infinite cardinal (the smallest infinite ordinal). When considering $n\in \omega$ as a finite cardinal, it is equal to the set containing all of the finite cardinals smaller than it, so $n=\{0, 1, \ldots, n-1\}.$ Writing $i\in n$ means that $i$ is an element of the set $\{0, 1, \ldots, n-1\}.$ When working with cardinals, we  assume the usual cardinal arithmetic laws. We also let $+_n$ denote the addition modulo $n$ (i.e. the addition in $\Zset/n\Zset$).

We let $\Gamma=\langle t\rangle$ be the infinite cyclic group generated by $t$ and $K$ be a field trivially graded by $\Zset$.

While the sets of vertices and edges of graphs in section \ref{section_n=1} have arbitrary cardinalities, all graphs in subsequent sections have {\em countably} many vertices and edges.

\subsection{Graded rings}
\label{subsection_graded_rings}
A ring $R$ (not necessarily unital) is {\em graded} by a group $\Gamma$ if $R=\bigoplus_{\gamma\in\Gamma} R_\gamma$ for additive subgroups $R_\gamma$ and if $R_\gamma R_\delta\subseteq R_{\gamma\delta}$ for all $\gamma,\delta\in\Gamma.$ The elements of the set $\bigcup_{\gamma\in\Gamma} R_\gamma$ are said to be {\em homogeneous}. A left ideal $I$ of a graded ring $R$ is {\em graded} if $I=\bigoplus_{\gamma\in \Gamma} I\cap R_\gamma.$ Graded right ideals and graded ideals are defined similarly. A graded ring is {\em graded simple} if there are no nontrivial and proper two-sided graded ideals (note that we do not require it to be graded Artinian).

A ring $R$ is an involutive ring, or a $*$-ring, if there is an anti-automorphism $*:R\to R$ of order two. If $R$ is also a $K$-algebra for some commutative $*$-ring $K$, then $R$ is a $*$-algebra if $(kx)^*=k^*x^*$ for all $k\in K$ and $x\in R.$
If $R$ is a $\Gamma$-graded ring with involution, it is a {\em graded $*$-ring} if $R_\gamma^*\subseteq R_{\gamma^{-1}}.$

We use the standard definitions of graded right and left $R$-modules, graded module homomorphisms and we use $\cong_{\gr}$ for a graded module or a graded ring isomorphism. If $M$ is a graded right $R$-module and $\gamma\in\Gamma,$ the $\gamma$-\emph{shifted} graded right $R$-module $(\gamma)M$ is defined as the module $M$ with the $\Gamma$-grading given by $(\gamma)M_\delta = M_{\gamma\delta}$ for all $\delta\in \Gamma.$ If $N$ is a graded left module, the $\gamma$-shift of $N$ is the graded module $N$ with the $\Gamma$-grading given by $M(\gamma)_\delta = M_{\delta\gamma}$ for all $\delta\in \Gamma.$

Any finitely generated graded free right $R$-module has the form $(\gamma_1)R\oplus\ldots\oplus (\gamma_n)R$ and any finitely generated graded free left $R$-module has the form
$R(\gamma_1)\oplus\ldots\oplus R(\gamma_n)$ for $\gamma_1, \ldots,\gamma_n\in\Gamma$ (\cite[Section 1.2.4]{Roozbeh_book} contains more details). A finitely generated graded projective module is a direct summand of a finitely generated graded free module.

The presence of the shifts $\gamma_1, \ldots, \gamma_n$ in the above form of a finitely generated graded free module explains the presence of $\gamma_1, \ldots, \gamma_n\in \Gamma$ in the graded matrix ring over a graded ring $R$.  In \cite{Roozbeh_book}, $\M_n(R)(\gamma_1,\dots,\gamma_n)$ denotes the ring of matrices $\M_n(R)$ with the $\Gamma$-grading given by
\begin{center}
$(r_{ij})\in\M_n(R)(\gamma_1,\dots,\gamma_n)_\delta\;\;$ if $\;\;r_{ij}\in R_{\gamma_i^{-1}\delta\gamma_j}$ for $i,j=1,\ldots, n.$
\end{center}
The definition of $\M_n(R)(\gamma_1,\dots,\gamma_n)$ in \cite{NvO_book} is different: $\M_n(R)(\gamma_1,\dots,\gamma_n)$ in \cite{NvO_book} corresponds to $\M_n(R)(\gamma_1^{-1},\dots,\gamma_n^{-1})$ in \cite{Roozbeh_book}. More details on the relations between the two definitions
can be found in \cite[Section 1]{Lia_realization}. Although the definition
from \cite{NvO_book} has been in circulation longer, some matricial representations of Leavitt path algebras involve positive integers instead of negative integers making the definition from \cite{Roozbeh_book} more convenient for us. Since we deal almost extensively with Leavitt path algebras, we opt to use the definition from \cite{Roozbeh_book}. With this definition, if $F$ is the graded free right module $(\gamma_1^{-1})R\oplus \dots \oplus (\gamma_n^{-1})R,$ then Hom$_R(F,F)\cong_{\gr} \;\M_n(R)(\gamma_1,\dots,\gamma_n)$ as graded rings.

We also recall \cite[Remark 2.10.6]{NvO_book} stating the first two parts in Lemma \ref{lemma_on_shifts} and \cite[Theorem 1.3.3]{Roozbeh_book}  stating part (3) for $\Gamma$ abelian. Although we use the lemma below only in the case $\Gamma\cong \Zset,$ we note that it generalizes to arbitrary $\Gamma.$ The part on the isomorphism being $*$-isomorphisms follows by \cite[Proposition 1.3]{Roozbeh_Lia_Ultramatricial}.

\begin{lemma}\cite[Remark 2.10.6]{NvO_book}, \cite[Proposition 1.4.4, Theorems 1.3.3 and 1.4.5]{Roozbeh_book}, \cite[Proposition 1.3]{Roozbeh_Lia_Ultramatricial}.
Let $R$ be a $\Gamma$-graded $*$-ring, $\gamma_1,\ldots,\gamma_n\in \Gamma,$ and let $e_{ij}$ denote the standard matrix unit for $i,j\in\{1,\ldots, n\}.$
\begin{enumerate}
\item If $\pi$ a permutation of the set $\{1,\ldots, n\},$ the $K$-linear extension of the map $e_{ij}\mapsto e_{\pi^{-1}(i)\pi^{-1}(j)}$ is a graded $*$-isomorphism
\begin{center}
$\M_n (R)(\gamma_1, \gamma_2,\ldots, \gamma_n)\;\cong_{\gr}\;\M_n (R)(\gamma_{\pi(1)}, \gamma_{\pi(2)} \ldots, \gamma_{\pi(n)}).$
\end{center}

\item If $\delta$ in the center of $\Gamma,$ then $e_{ij},$ which is in the $\gamma_i\gamma_j^{-1}$-component, can be considered as an element of the $\gamma_i\delta(\gamma_j\delta)^{-1}$-component, so the identity becomes a graded $*$-isomorphism $\;\M_n (R)(\gamma_1, \gamma_2, \ldots, \gamma_n)\;=\;\M_n (R)(\gamma_1\delta, \gamma_2\delta,\ldots, \gamma_n\delta).$

\item If $\delta_i\in\Gamma$ is such that there is an invertible element $a_{\delta_i}$ in $R_{\delta_i}$ for $i=1,\ldots, n,$ then the $K$-linear extension $\phi$ of the map $e_{ij}\mapsto a_{\delta_i}^{-1}a_{\delta_j}e_{ij}$ is a graded isomorphism
\begin{center}
$\M_n (R)(\gamma_1, \gamma_2, \ldots, \gamma_n)\;\cong_{\gr}\;\M_n (R)(\gamma_1\delta_1, \gamma_2\delta_2\ldots, \gamma_n\delta_n).$
\end{center}
If the elements $a_{\delta_i}$ can be found so that $a_{\delta_i}^{-1}=a_{\delta_i}^*,$ then $\phi$ is a graded $*$-isomorphism.
\end{enumerate}
If $\Gamma$ is abelian and $R$ and $S$ are $\Gamma$-graded division rings, then $$\hskip1.3cm\M_n (R)(\gamma_1, \gamma_2, \ldots, \gamma_n)\;\cong_{\gr}\; \M_m (S)(\delta_1, \delta_2, \ldots, \delta_m)$$ implies that $R\cong_{\gr}S,$ that $m=n,$ and that the list $(\delta_1, \delta_2, \ldots, \delta_m)$ is obtained from the list $(\gamma_1, \gamma_2, \ldots, \gamma_n)$ by applying finitely many operations of the lists as in parts (1) to (3).
\label{lemma_on_shifts}
\end{lemma}

We are going to be interested exclusively in the case when $\Gamma=\Zset$ and the matrices are formed over either a field $K$ graded trivially by $\Zset$ or the ring $K[x^m, x^{-m}]$ of Laurent polynomials $\Zset$-graded by $K[x^m, x^{-m}]_{mk}=Kx^{mk}$ and $K[x^m, x^{-m}]_{n}=0$ if $m$ does not divide $n.$ Note that $K[x^m, x^{-m}],$ graded as above, is a graded field.

It turns out that dealing with graphs having infinitely many paths ending in a sink or a cycle requires us to generalize previously noted results to matrices of arbitrary size. Let $\Gamma$ be any group, $K$ be a $\Gamma$-graded division ring, and let $\kappa$ be an infinite cardinal. We let $\M_\kappa (K)$ denote the ring of infinite matrices over $K$, having rows and columns indexed by $\kappa$, with only finitely many nonzero entries. If $\ol\gamma$ is any function $\kappa\to \Gamma,$ we let $\M_\kappa(K)(\ol\gamma)$ denote the $\Gamma$-graded ring $\M_\kappa(K)$ with the $\delta$-component consisting of the matrices $(a_{\alpha\beta}),$ $\alpha, \beta\in \kappa,$ such that $a_{\alpha\beta}\in K_{\ol\gamma(\alpha)^{-1}\delta\ol\gamma(\beta)}.$ If $K$ is a graded $*$-ring, $\M_\kappa(K)(\ol\gamma)$ is a graded $*$-ring with the $*$-transpose involution.

The first part of Lemma \ref{lemma_on_shifts} has been generalized in \cite[Proposition 4.12]{Roozbeh_Lia_Ultramatricial} as the lemma below.

\begin{lemma} \cite[Proposition 4.12]{Roozbeh_Lia_Ultramatricial}
Let $\Gamma$ be an abelian group, $R$ be a $\Gamma$-graded $*$-ring, $\kappa$ a cardinal, $\ol\gamma\in \Gamma^\kappa$, and $e_{\alpha\beta}$ denotes the standard matrix units for $\alpha,\beta\in\kappa.$
\begin{enumerate}[\upshape(1)]
\item If $\pi$ is a bijection $\kappa\to \kappa$ and $\ol\gamma\pi$ denotes the composition  of $\ol\gamma$ and $\pi,$ then the $R$-linear extension of the map $e_{\alpha\beta}\mapsto e_{\pi^{-1}(\alpha)\pi^{-1}(\beta)}$ is a graded $*$-isomorphism
$\M_\kappa (R)(\ol \gamma)\cong_{\gr}M_\kappa (R)(\ol\gamma\pi).$

\item If $\delta$ is in the center of $\Gamma$ and $\ol\gamma\delta$ denotes the map $(\ol\gamma\delta)(\alpha)=\ol\gamma(\alpha)\delta,$ then $e_{\alpha\beta},$ which is in the $\ol\gamma(\alpha)\ol\gamma(\beta)^{-1}$-component, can be considered as an element of the $\ol\gamma(\alpha)\delta(\ol\gamma(\beta)\delta)^{-1}$-component, so the identity becomes a graded isomorphism
$\M_\kappa (R)(\ol \gamma)\cong_{\gr}M_\kappa (R)(\ol\gamma\delta).$

\item If $\ol\delta: \kappa\to \Gamma$ is such that the component $R_{\ol\delta(\alpha)}$ contains an invertible element $a_\alpha$ for every $\alpha\in\kappa,$ then the $R$-linear extension of the map $e_{\alpha\beta}\mapsto a_{\alpha}^{-1}a_{\beta}e_{\alpha\beta}$ is a graded isomorphism  $\phi:\M_\kappa (R)(\ol \gamma)\cong_{\gr}
\M_\kappa (R)(\ol \gamma\ol \delta).$
If the elements $a_\alpha$ are unitary (i.e. $a_\alpha^{-1}=a_\alpha^*$), then $\phi$ is a graded $*$-isomorphism.
\end{enumerate}
\label{lemma_shifts_general}
\end{lemma}

\subsection{Graphs and properties of vertex sets}\label{subsection_graphs}
If $E$ is a directed graph, we let $E^0$ denote the set of vertices, $E^1$ denote the set of edges, and $\so$ and $\ra$ denote the source and the range maps of $E.$ If $\kappa$ is any cardinal, we use $\xymatrix{\bullet^v\ar[r]^\kappa&\bullet^w}$ to depict that $v$ emits $\kappa$ edges to $w.$
If $V\subseteq E^0,$ a subgraph {\em generated by} $V$ is the graph with $V$ as its vertex set and with its edge set consisting of the edges of $E$ which have both their source and their range in $V.$ The graphs $E$ and $F$ are {\em isomorphic}, written as $E\cong F$, if there are bijective maps $f^0$ and $f^1$ of their vertices and edges such that  $\so(f^1(e))=f^0(\so(e))$ and $\ra(f^1(e))=f^0(\ra(e))$ for all $e\in E^1.$ It is direct to show that this definition corresponds to $f=(f^0, f^1)$ being an invertible graph homomorphism in the sense of \cite[Definition 1.6.2]{LPA_book}. In this case, we write $f: E\cong F$ and refer to both $f^0$ and $f^1$ as $f$.

A {\em sink} of $E$ is a vertex which emits no edges and an {\em infinite emitter} is a vertex which emits infinitely many edges. A vertex of $E$ is {\em regular} if it is neither a sink nor an infinite emitter. A graph $E$ is {\em row-finite} if it has no infinite emitters and $E$ is {\em finite} if it has finitely many vertices and edges.

A {\em path} is a single vertex or a sequence of edges $e_1e_2\ldots e_n$ for some positive integer $n$ such that $\ra(e_i)=\so(e_{i+1})$ for $i=1,\ldots, n-1.$  The length $|p|$ of a path $p$ is zero if $p$ is a vertex and it is $n$ if $p$ is a sequence of $n$ edges. The set of vertices on a path $p$ is denoted by $p^0.$ A {\em line} of length $n$ is a path of length $n$ such that all its vertices, except possibly the source and the range, do not emit edges except the one in $p$ and do not receive edges except the one in $p$.

The functions $\so$ and $\ra$ extend to paths naturally. A path $p$ is {\em closed}  if $\so(p)=\ra(p).$ A {\em cycle} is a closed path such that different edges in the path have different sources. A cycle has {\em an exit} if a vertex on the cycle emits an edge which is not an edge of the cycle. A cycle $c$ is {\em extreme} if $c$ has exits and for every path $p$ with $\so(p)\in c^0,$ there is a path $q$ such that $\ra(p)=\so(q)$ and $\ra(q)\in c^0.$  Two cycles $c$ and $d$ of any graph are {\em disjoint} if $c^0\cap d^0=\emptyset.$

If $E$ is $\xymatrix{\bullet\ar@/^/[r]^e&\bullet\ar@/^/[l]_f},$ it  is not uncommon to hear a reference to the ``single cycle'' of $E$. However, $E$ has two cycles: $ef$ and $fe.$ In our work, if $c_0=e_0\ldots e_{n-1}$ is a cycle of length $n$, we use $[c_0]$ for the set of $n$ elements $c_i=e_ie_{i+_n1}\ldots e_{i-_n1}$ for $i=0,\ldots, n-1.$ The relation $c\sim d$ if $[c]=[d]$ is an equivalence relation of the set of all cycles of a graph. Thus, while a cycle in the above graph with two vertices is not unique, there is a unique equivalence class of the relation $\sim$ on this graph.

If $u,v\in E^0$ are such that there is a path $p$ with $\so(p)=u$ and $\ra(p)=v$, we write $u\geq v.$ For $V\subseteq E^0,$
the set $T(V)=\{u\in E^0\mid v\geq u$ for some $v\in V\}$ is called the {\em tree} of $V,$ and, following \cite{Lia_irreducible}, we use $R(V)$ to denote the set $\{u\in E^0\mid u\geq v$ for some $v\in V\}$ called the {\em root} of $V.$ If $V=\{v\},$ we use $T(v)$ for $T(\{v\})$ and $R(v)$ for $R(\{v\}).$

An {\em infinite path} is a sequence of edges $e_1e_2\ldots$ such that $\ra(e_n)=\so(e_{n+1})$ for $n=1,2\ldots.$ Just as for finite paths, we use $p^0$ for the set of vertices of an infinite path $p.$ In \cite[Definition 5.3]{Lia_porcupine_quotient}, an infinite path $p$ of a graph $E$ is said to be {\em terminal} if no element of $T(p^0)$ is an infinite emitter or on a cycle and if every infinite path $q$ originating at a vertex of $p$ is such that $T(q^0)\subseteq R(q^0).$ Note that the property
$T(p^0)\subseteq R(p^0)$ is shared for any path (finite or infinite) which contains a vertex which is on an extreme cycle, or on a cycle without exits or which is a sink and the relevance of these types of vertices is evident from \cite[Theorem 5.7]{Lia_porcupine_quotient} which we review in Theorem \ref{theorem_graded_simple}.
A vertex is {\em terminal} if it is a sink, on a cycle without exits, on an extreme cycle or on a terminal path.

Let $E^{\leq\infty}$ be the set of infinite paths or finite paths ending in a sink or an infinite emitter. A vertex $v$ is {\em cofinal} if for each $p\in E^{\leq\infty}$ there is $w\in p^0$ such that $v\geq w$ and  $E$ is {\em cofinal} if each vertex is cofinal. This property matches the existence of a unique equivalence class of terminal vertices such that $w\approx v$ if and only if $v$ and $w$ are on the elements of $E^{\leq\infty}$ which have the same root (see  \cite{Lia_porcupine_quotient} for more details). In \cite{Lia_porcupine_quotient}, this equivalence class of terminal vertices is called a {\em terminal cluster}, or a {\em cluster} for short. For example, the cluster of a sink $v$ is $\{v\},$ the cluster of a cycle $c$ without exists is $c^0,$ and the cluster of a vertex on an extreme cycle $c$ is $T(c^0).$

A subset $H$ of $E^0$ is said to be {\em hereditary} if $T(H)\subseteq H.$ The set $H$ is {\em saturated} if $v\in H$ for any regular vertex $v$ such that $\ra(\so^{-1}(v))\subseteq H.$ For every $V\subseteq E^0,$ the intersection of all saturated sets of vertices which contain $V$ is the smallest saturated set which contains $V.$ This set is the {\em saturated closure of $V$}. The saturated closure $\ol{V}$ of $T(V)$ is both hereditary and saturated and it is the smallest hereditary and saturated set which contains $V.$

\subsection{Leavitt path algebras}\label{subsection_LPA}
If $K$ is a field, the {\em Leavitt path algebra} $L_K(E)$ of $E$ over $K$ is a free $K$-algebra generated by the set  $E^0\cup E^1\cup\{e^\ast\mid e\in E^1\}$ such that

\begin{tabular}{ll}
(V)  $vw =0$ if $v\neq w$ and $vv=v,$ & (E1)  $\so(e)e=e\ra(e)=e,$\\
(E2) $\ra(e)e^\ast=e^\ast\so(e)=e^\ast,$ & (CK1) $e^\ast f=0$ if $e\neq f$ and $e^\ast e=\ra(e),$\\
(CK2) $v=\sum_{e\in \so^{-1}(v)} ee^*$ for each regular vertex $v$ \\
\end{tabular}

\noindent hold for for $v,w\in E^0$ and $e,f\in E^1.$ Let $v^*=v$ for $v\in E^0$ and $p^*=e_n^*\ldots e_1^*$ for a path $p=e_1\ldots e_n.$ The  set of elements $pq^*$ where $p$ and $q$ are paths of $E$ with $\ra(p)=\ra(q)$ generates $L_K(E)$ as a $K$-algebra and
$\left(\sum_{i=1}^n k_ip_iq_i^\ast\right)^*=\sum_{i=1}^n k_i^*q_ip_i^\ast,$ where $k_i\mapsto k_i^*$ is any involution on $K,$ is an involution of $L_K(E).$ In addition, $L_K(E)$ is graded locally unital (with the finite sums of vertices as the local units, the elements $u$ such that for every finite set of elements $F$, $xu=ux=x$ for every $x\in F$). The algebra $L_K(E)$ is unital if and only if $E^0$ is finite in which case $\sum_{v\in E^0}v$ is the identity.

If we consider $K$ to be trivially graded by $\Zset,$ $L_K(E)$ is naturally graded by $\Zset$ so that the $n$-component $L_K(E)_n$ is the $K$-linear span of the elements $pq^\ast$ for paths $p, q$ with $|p|-|q|=n.$ This grading and the involutive structure make $L_K(E)$ into a graded $*$-algebra.

If $R$ is a $K$-algebra which contains  elements $p_v$ for $v\in E^0,$ and $x_e$ and $y_e$ for $e\in E^1$ such that the five axioms hold for these elements, the Universal Property of $L_K(E)$ states that there is a unique algebra homomorphism $\phi:L_K(E)\to R$ such that $\phi(v)=p_v, \phi(e)=x_e,$ and $\phi(e^*)=y_e$ (see \cite[Remark 1.2.5]{LPA_book}). If $R$ is $\Zset$-graded and $p_v\in R_0$ for $v\in E^0,$ $x_e\in R_1$ and $y_e\in R_{-1}$ for $e\in E^1,$ then $\phi$ is graded.
By the Graded Uniqueness Theorem (\cite[Theorem 2.2.15]{LPA_book}), such graded map $\phi$ is injective if $p_v\neq 0$ for $v\in E^0.$ If $R$ is involutive and $\phi$ is such that $y_e=x_e^*,$ then $\phi$ is a $*$-homomorphism (i.e., $\phi(x^*)=\phi(x)^*$ for every $x\in L_K(E)$).

We recall \cite[Theorem 5.7]{Lia_porcupine_quotient} characterizing graded simplicity of $L_K(E)$ in terms of the existence of exactly one type of four types of terminal vertices.

\begin{theorem} \cite[Theorem 5.7]{Lia_porcupine_quotient}
Let $E$ be a graph and $K$ be a field. The following conditions are equivalent.
\begin{enumerate}[\upshape(1)]
\item $L_K(E)$ is graded simple (equivalently, $E$ is cofinal).

\item The set of terminal vertices is nonempty and it consists of a single cluster $C$ such that $E^0$ is the (hereditary and) saturated closure of $C.$

\item Exactly one of the following holds.
\begin{enumerate}[\upshape(a)]
\item The set $E^0$ is the (hereditary and) saturated closure of a sink. In this case, $E$ is row-finite and acyclic and $E^0=R(v)$ for a sink $v$.

\item The set $E^0$ is the (hereditary and) saturated closure of $c^0$ for a cycle $c$ without exits. In this case, $E$ is row-finite, $E^0=R(c^0),$ and $c$ is the only cycle in $E$.

\item  The set $E^0$ is the hereditary and saturated closure of $c^0$ for an extreme cycle $c.$ In this case, every cycle of $E$ is extreme, every infinite emitter is on a cycle, and $E^0=R(c^0)$.

\item The set $E^0$ is the hereditary and saturated closure of $\alpha^0$ for a terminal path $\alpha.$ In this case, $E$ is acyclic and row-finite and $E^0=R(\alpha^0)$.
\end{enumerate}
\end{enumerate}
\label{theorem_graded_simple}
\end{theorem}

\subsection{Porcupine-quotient graphs}\label{subsection_quotient_and_porcupine}

If $H$ is hereditary and saturated, a {\em breaking vertex} of $H$ is an element of the set
\[B_H=\{v\in E^0-H\,|\, v\mbox{ is an infinite emitter and }0<|\so^{-1}(v)\cap \ra^{-1}(E^0-H)|<\omega\}.\]
For each $v\in B_H,$ let $v^H$ stands for $v-\sum ee^*$ where the sum is taken over $e\in \so^{-1}(v)\cap \ra^{-1}(E^0-H).$

An {\em admissible pair} is a pair $(H, S)$ where $H\subseteq E^0$ is hereditary and saturated and $S\subseteq B_H.$
For an admissible pair $(H,S)$,
the ideal $I(H,S)$ generated by $H\cup \{v^H \,|\, v\in S \}$ is graded since it is generated by homogeneous elements. It is the $K$-linear span of the elements $pq^*$ for paths $p,q$ with $\ra(p)=\ra(q)\in H$ and the elements $pv^Hq^*$ for paths $p,q$ with $\ra(p)=\ra(q)=v\in S$ (see \cite[Lemma 5.6]{Tomforde}). Conversely, for a graded ideal $I$, $H=I\cap E^0$ is hereditary and saturated and for $S=\{v\in B_H\mid v^H\in I\},$ $I=I(H,S)$ (\cite[Theorem 5.7]{Tomforde}, also \cite[Theorem 2.5.8]{LPA_book}). If $S=\emptyset,$ we shorten $(H,\emptyset)$ to $H$ and $I(H, \emptyset)$ to $I(H).$

The set of admissible pairs is a lattice with respect to the relation
\[(H,S)\leq (G,T)\;\; \mbox{ if }H\subseteq K\mbox{ and }S\subseteq G\cup T\]
(see \cite[Proposition 2.5.6]{LPA_book} for the meet and the join of this lattice). The correspondence $(H,S)\mapsto I(H,S)$ is a lattice isomorphism of this lattice and the lattice of graded ideals.

An admissible pair $(H,S)$ gives rise to the {\em quotient graph} $E/(H,S)$ and to the {\em porcupine graph}  $P_{(H,S)}$
so that the algebras $L_K(E)/I(H,S)$ and $L_K(E/(H,S))$ are graded isomorphic and that the algebras $L_K(P_{(H,S)})$ and $I(H,S)$ are also graded isomorphic (by \cite[Theorem 5.7]{Tomforde} and \cite[Theorem 3.3]{Lia_porcupine}).
Recently, the two constructions have been generalized by a single construction and the {\em porcupine-quotient graph} corresponding to the quotient of one admissible pair with respect to another admissible pair was introduced in \cite{Lia_porcupine_quotient}. Below are the relevant definitions.

If $H\subseteq G$ are two sets of vertices of $E$, let
\[B_H^G=\{v\in E^0-H\mid v\mbox{ is an infinite emitter and } 0<|\so^{-1}(v)\cap \ra^{-1}(G-H)|<\omega\}.\]

If $(H,S)$ and $(G,T)$ are two admissible pairs of a graph $E$ such that $(H,S)\leq (G,T),$ we let
\[
\begin{array}{ll}
F_1(G-H, T-S)=\{e_1e_2\ldots e_n\mbox{ is a path of }E\mid \ra(e_n)\in G-H, \so(e_n)\notin (G-H)\cup (T-S)
\}\mbox{ and}\\
F_2(G-H, T-S)=\{p\mbox{ is a path of }E\mid \ra(p)\in T-S,  |p|>0\}.
\end{array}
\]

The {\em porcupine-quotient graph} $(G,T)/(H,S)$ is defined by letting the set of its vertices be
\[(G-H)\cup (T-S)\cup \{w^p\mid p\in F_1(G-H,T-S)\cup F_2(G-H,T-S)\}\cup \{v'\mid v\in ((G\cup T)-S)\cap B^G_H
\}\]
and the set of its edges be
\[\{e\in E^1\mid \ra(e)\in G-H\mbox{ and either }\so(e)\in G-H\mbox{ or }\so(e)\in T-S\}\cup\]\[\{f^p\mid p\in F_1(G-H,T-S)\cup F_2(G-H,T-S)\}\cup \{e'\mid \ra(e)\in ((G\cup T)-S)\cap B^G_H\}.\]

The source and range of an edge of $(G,T)/(H,S)$ which is also in $E^1$ are the same as in $E.$

If $e\in E^1\cap(F_1(G-H, T-S)\cup F_2(G-H, T-S)),$ we let $\so(f^e)=w^e$ and $\ra(f^e)=\ra(e).$
If $p=eq$ where $e\in E^1,$ $q\in F_1(G-H, T-S)\cup F_2(G-H, T-S),$ and $|q|>0,$ let $\so(f^p)=w^p$ and $\ra(f^p)=w^q.$

If $\ra(e)\in ((G\cup T)-S)\cap B^G_H,$ we let $\ra(e')= \ra(e)'.$ If $\ra(e)\in (G-S)\cap B_H^G$ and if $\so(e)\in (G-H)\cup (T-S),$ we let $\so(e')=\so(e).$ If $\so(e)\notin (G-H)\cup (T-S),$ then either $\so(e)\notin G\cup T$ or $\so(e)\in S\cap T.$ In either case,  $e\in F_1(G-H, T-S)$ and we let $\so(e')=w^e.$ If $\ra(e)\in (T-S)\cap B_H^G,$ then $e\in F_2(G-H, T-S)$ and we let $\so(e')=w^e.$

If $S=T=\emptyset,$ we write $(G, \emptyset)/(H,\emptyset)$ shorter as $G/H.$

If $(G,T)=(E^0, \emptyset),$ the porcupine-quotient graph is exactly the quotient graph $E/(H,S).$ If $(H,S)=(\emptyset, \emptyset),$ the porcupine-quotient graph is exactly the porcupine graph $P_{(G,T)}.$

By \cite[Theorem 3.6]{Lia_porcupine_quotient}, the Leavitt path algebra of a porcupine-quotient $(G,T)/(H,S)$ is graded isomorphic to the quotient $I(G,T)/I(H,S)$ of two corresponding graded ideals.

\subsection{Out-split and in-split plus moves}
\label{subsection_outsplit}
If $E$ is a graph and $v$ a vertex which emits edges, let $\mathcal E_1,\ldots \mathcal E_n$ be a partition $\mathcal P$ of $\so^{-1}(v).$ We review the definition of the {\em out-split graph} $E_{v,\mathcal P}$ below (\cite[Definition 6.3.23]{LPA_book}) and the out-split operation $E\to E_{v, \mathcal P}$ which we call the {\em out-split of $v$} with respect to $\mathcal P$. If $\so^{-1}(v)=\mathcal E_1\cup \ldots\cup \mathcal E_n,$ the sets of vertices and edges of $E_{v,\mathcal P}$ are
\[E_v^0=E^0-\{v\}\cup\{v_1,\ldots, v_n\},\hskip.3cm
E_v^1=\{f_1, \ldots, f_n\mid f\in E^1, \ra(f)=v\}\cup\{f\in E^1\mid \ra(f)\neq v\},\]
the range of $f_i$ in $E_{v,\mathcal P}$ is $v_i$ and the range of $f$ in $E_{v,\mathcal P}$ is $\ra(f),$
the source of $g\in E_{v,\mathcal P}^1$ is
\begin{itemize}
\item $v_i\;\;\;$ if $g=f_j\in \mathcal E_i$ (so $\so(f)=\ra(f)=v$)
\item $v_i\;\;\;\,$ if $g=f\in \mathcal E_i$ (so $\so(f)=v$ and $\ra(f)\neq v$)
\item $\so(f)$ if $g=f_j$ and $\so(f)\neq v$ (so $\ra(f)=v$)
\item $\so(f)$ if $g=f\;$ and $\so(f)\neq v$ (so $\ra(f)\neq v$).
\end{itemize}

The map on the vertices and edges of $E$ given by $\phi(v)=\sum_{i=1}^n v_i,$ $\phi(w)=w$ for $w\in E^0-\{v\},$ $\phi(f)=\sum_{i=1}^n f_i $ if $\ra(f)=v$ and $\phi(f)=f$ otherwise, extends to a graded $*$-monomorphism $L_K(E)\to L_K(E_{v, \mathcal P}).$ This map is
a graded $*$-isomorphism if  $v$ is regular or if it is an infinite emitter with all but one set $\mathcal E_1,\ldots \mathcal E_n$ finite.
In this last case, and if $\mathcal E_n$ is the infinite set, the inverse of $\phi$ can be defined on the vertices and edges by
$\psi(v_i)=\sum_{e\in \mathcal E_i} ee^*$ for $i=1,\ldots, n-1$ $\psi(v_n)=v-\sum_{i=1}^{n-1}\sum_{e\in \mathcal E_i} ee^*,$ $\psi(w)=w$ for $w\in E^0-\{v_1,\ldots, v_n\},$ and $\psi(f)=f\psi(\ra(f)).$ If $v$ is an infinite emitter and more than one partition set is infinite, $\phi$ may not be onto. For example, the first graph below has a composition series length 3 (see section \ref{subsection_composition}). The second graph is the out-split of $v$ with respect to the two-element partition consisting of the sets of all edges with equal range. It has a composition series of length 4, so the algebras of the two graphs are not graded isomorphic.
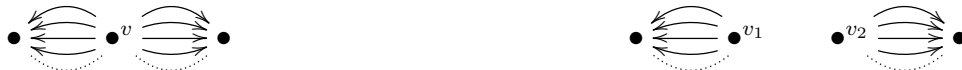
\[
\xymatrix{\bullet&\bullet^v
\ar @/_1pc/ [l] _{\mbox{ } } \ar@/_/ [l] \ar [l] \ar@/^/ [l] \ar@{.}@/^1pc/ [l]
\ar@{.} @/_1pc/ [r] _{\mbox{ } } \ar@/_/ [r] \ar [r] \ar@/^/ [r] \ar@/^1pc/ [r]&\bullet}\hskip5cm
\xymatrix{\bullet&\bullet^{v_1}\ar @/_1pc/ [l] _{\mbox{ } } \ar@/_/ [l] \ar [l] \ar@/^/ [l] \ar@{.}@/^1pc/ [l]}\hskip0.6cm
\xymatrix{\bullet^{v_2}\ar@{.} @/_1pc/ [r] _{\mbox{ } } \ar@/_/ [r] \ar [r] \ar@/^/ [r] \ar@/^1pc/ [r]&\bullet}\]

All out-splits we consider are with respect to either a partition of the set of edges emitted from a regular vertex or the set of edges emitted from an infinite emitter such that only one partition set is infinite. The operation of taking the {\em out-amalgamation} of a vertex $v$ with respect to some $\mathcal P$ is the inverse of taking the out-split. So, the graph $E$ is the out-amalgamation of $E_{v, \mathcal P}.$

We often deal with the case when $v$ is regular and $\mathcal E_i$ is a single element set for $i=1,\ldots, |\so^{-1}(v)|.$ We call the out-split graph with respect to this partition the {\em maximal out-split of $v$}. If $V\subseteq E^0$
is a set of regular vertices,
the graph obtained by performing the maximal out-splits of every $v\in V$ is said to be the {\em maximal out-split of $V$} and the term {\em the maximal out-split} is used in the case when $V$ is the set of all regular vertices.
We are interested in the case when $V$ is the set of all regular vertices not on any cycle. We say that the maximal out-split with respect to this set is the {\em total out-split} of $E$ and denote it by $E_{\tot}$. So, $E_{\tot}$ is a graph in which every regular vertex which is not in a cycle emits exactly one edge.

If $v$ is a regular vertex of a graph $E$ and if $\mathcal E_i, i=1,\ldots, n$ is a partition $\mathcal P$ of $\ra^{-1}(v),$ an {\em in-split of $v$ with respect to $\mathcal P$} is defined analogously as the out-split except that the roles of $\ra$ and $\so$ are reversed. If some of the partition sets are allowed to be empty (in which case the collection is still called a partition despite the deviation from the standard definition), the resulting operation is called an {\em in-split minus} move. If $v$ is a regular vertex of a graph $M$ and $E$ and $F$ are the outcomes of the in-split minus moves applied to $v$ and the partitions of $\ra^{-1}(v)$ with the same number of elements, then we say that $E$ is an {\em in-split plus} of $F$ (and, vice versa, $F$ is an in-split plus of $E$). We
refer to the graph $M$ as  the {\em mother} graph which relates $E$ and $F$ by the in-split minus moves.
For example, if $M$ is the graph $\xymatrix{\bullet \ar[r]^e& \bullet^v \ar@/^1pc/[r] & \bullet \ar@/^1pc/ [l]_f}$ and $E$ and $F$ are two graphs from Example \ref{example_canonical_representation}, then the in-split minus of $M$ with respect to  $\{\{e\}, \{f\}\}$ produces $E$ and the in-split minus of $M$ with respect to $\{\{e,f\}, \emptyset\}$ produces $F$.

An in-split plus move extends to a graded $*$-isomorphism of the algebras by \cite[Theorem 2.2.3]{Eilers_Ruiz_3-bit} (and this may not be the case for in-split and in-split minus moves). The reference \cite{Eilers_Ruiz_3-bit} also has more information on the origin and the use of the in-split plus move. For composition S-NE graphs, the in-split plus move of a vertex which does not emit exits is related to the relative path rearrangements of the cofinal subgraph from Proposition \ref{proposition_canonical_tails}.

\subsection{Pre-order monoids and their order-ideals}\label{subsection_order-ideals}
If $G$ is an abelian group with a left action of a group $\Gamma$ compatible with the group operation, then $G$ is a {\em $\Gamma$-group}. If $M$ is an abelian monoid with a left action of a group $\Gamma$ compatible with the monoid operation, then $M$ is a {\em $\Gamma$-monoid}. Let $\geq$ be a reflexive and transitive relation (a pre-order) on a $\Gamma$-monoid $M$ ($\Gamma$-group $G$) such that $g_1\geq g_2$ implies $g_1 + h\geq g_2 + h$ and $\gamma g_1 \geq \gamma g_2$ for all $g_1, g_2, h$ in $M$ (in $G$) and $\gamma\in \Gamma.$ We say that such monoid $M$ is a {\em pre-ordered $\Gamma$-monoid} and that such a group $G$ is a {\em pre-ordered $\Gamma$-group}. Let $\POM$ denote the category whose objects are pre-ordered $\Gamma$-monoids and whose morphisms are order-preserving $\Gamma$-monoid homomorphisms and $\POG$ denote the category whose objects are pre-ordered $\Gamma$-groups and whose morphisms are order-preserving $\Gamma$-group homomorphisms ($\Zset[\Gamma]$-homomorphisms).

If $G$ is a pre-ordered $\Gamma$-group, $G^+=\{x\in G\mid x\geq 0\}$ is a pre-ordered $\Gamma$-monoid. If a $\Gamma$-group homomorphism $f:G\to H$ is order-preserving, then $f(G^+)\subseteq H^+.$ Conversely, if
$M$ is a cancellative pre-ordered $\Gamma$-monoid, then its Grothendieck group is a pre-ordered $\Gamma$-group such that $G^+=M$ and if $f:M\to N$ is an order-preserving $\Gamma$-monoid homomorphism of cancellative pre-ordered $\Gamma$-monoids, then $f$ induces an order-preserving homomorphism of the Grothendieck groups of $M$ and $N.$
The objects of $\POM$ we consider are cancellative, so the formulation of our main result in terms of the elements of $\POM$ is equivalent to the formulation in terms of the elements of $\POG.$

If $\Zset^+$ denotes set of nonnegative integers, the set of finite sums of the $\Zset^+$-multiples of the elements of $\Gamma$ is a pre-ordered $\Gamma$-monoid which we denote by $\Zset^+[\Gamma]$. An element $u$ of a pre-ordered $\Gamma$-monoid $M$ is an {\em order-unit} if for any
$x\in M,$ there is a nonzero $a\in\Zset^+[\Gamma]$ such that $x\leq au$. If $M$ and $N$ are pre-ordered $\Gamma$-monoids with order-units $u$ and $v$ respectively,
we refer to the pairs $(M, u)$ and $(N, v)$ as the {\em pointed monoids}. An order-preserving $\Zset[\Gamma]$-module
homomorphism $f:M\to N$ is a {\em homomorphism of pointed monoids} if $f(u)=v$. Let
$\POM^u$ denote the category whose objects are pointed monoids and whose morphisms are homomorphisms of pointed monoids.

An element $u$ of a pre-ordered $\Gamma$-group $G$ is an \emph{order-unit} if $u\in G^+$ and for any $x\in G$, there is a nonzero $a\in \Zset^+[\Gamma]$ such that $x\leq au.$
Let $\POG^u$ denote the category whose objects are pairs $(G, u)$ where $G$ is an object of $\POG$ and $u$ is an order-unit of $G$ and whose morphisms are morphisms of $\POG$ which are order-unit-preserving. If $G$ is an upwards  directed pre-ordered $\Gamma$-group, then $u$ is an order-unit of $G$ if and only if $u$ is an order-unit of $G^+.$
Since all the objects of $\POG$ we consider are going to be upwards directed, the formulation of our main result in terms of objects and morphisms of $\POM^u$ is equivalent to the formulation in terms of their counterparts in $\POG^u.$

A submonoid $I$ of a pre-ordered monoid $M$ is an {\em order-ideal} of $M$ if  $x\geq y$ and $x\in I$ implies $y\in I$. A $\Gamma$-submonoid $I$ of a pre-ordered $\Gamma$-monoid $M$ which is an order-ideal is a {\em $\Gamma$-order-ideal}. A subset $D$ of such $M$ is {\em convex} if
$x\leq z\leq y$ for $x, y\in D$ and $z\in M$ implies $z\in D.$ A subset $D$ of $M$ is {\em upwards directed} if for $x, y\in D$, there is $z\in D$ such
that $x\leq z$ and $y\leq  z.$ If $M$ is a pre-ordered $\Gamma$-monoid, a subset $D$ of $M$ is a {\em generating
interval} if $D$ it is  upwards directed and convex such that $\Zset^+[\Gamma]D=M.$ Let $\POM^D$ denote the category whose objects are pairs $(M, D)$ where $M$ is an object
of $\POM$ and $D$ is a generating interval of $M$, and whose morphisms
are morphisms of $\POM$ which map the generating interval into the generating interval. Thus, a $\POM$-homomorphism $f:(M, D)\to (N,F)$ is a $\POM^D$ homomorphism if and only if $f(D)\subseteq F.$ If $f$ is an $\POM^D$-isomorphism, then $F=f(f^{-1}(F))\subseteq f(D)\subseteq F,$ so $f(D)=F.$

If $G$ is a pre-ordered $\Gamma$-group,
a subset $D$ of $G^+$ is a {\em generating interval} of $G$ if $D$ is a generating interval of $G^+.$ Let $\POG^D$ denote the category of $\POG$ objects considered with their generating intervals and $\POG$-morphisms mapping the generating interval into the generating interval. For the $\Gamma$-monoids and $\Gamma$-groups we are considering, the categories $\POM^D$ and $\POG^D$ are equivalent.

\subsection{The \texorpdfstring{$\Gamma$}{}-monoid and the Grothendieck \texorpdfstring{$\Gamma$}{}-group of a graded ring}
\label{subsection_graded_Grothendieck}
If $R$ is a unital $\Gamma$-graded ring, let $\V^{\Gamma}(R)$ denote the monoid of the graded isomorphism classes $[P]$ of finitely generated graded projective right $R$-modules $P$ with the addition given by $[P]+[Q]=[P\oplus Q]$ and the left $\Gamma$-action given by $(\gamma, [P])\mapsto [(\gamma^{-1})P].$ The monoid $\V^\Gamma(R)$ can be represented using the classes of left modules in which case the $\Gamma$-action is $(\gamma, [P])\mapsto [P(\gamma)].$ The two representations are equivalent (see \cite[Section 1.2.3]{Roozbeh_book} or \cite[Section 2.3]{Lia_realization}). Note that the action in \cite{Roozbeh_book} is given by $(\gamma, [P])\mapsto [(\gamma)P],$ not as above. For abelian groups, this is equivalent, but for non-abelian groups, this map would not be an action, so we consider the action with $(\gamma, [P])\mapsto [(\gamma^{-1})P].$ This causes some formulas to be different than in \cite{Roozbeh_book}.
The monoid $\V^{\Gamma}(R)$ can also be defined via homogeneous matrices (see \cite[section 3.2]{Roozbeh_book}). Although we focus on the case when $\Gamma$ is $\Zset,$ we note that the definitions and results of \cite[Section 3.2]{Roozbeh_book} carry to the case when $\Gamma$ is not necessarily abelian by \cite[Section 1.3]{Lia_realization}.

The \emph{Grothendieck $\Gamma$-group}  $K_0^{\Gamma}(R)$ is the group completion of the $\Gamma$-monoid $\V^{\Gamma}(R)$ with
the action of $\Gamma$ inherited from $\V^{\Gamma}(R)$. If $\Gamma$ is the trivial group, $K_0^{\Gamma}(R)$ is the usual $K_0$-group.
The pointed monoid $(\V^{\Gamma}(R), [R])$ is an object of $\POM^u$ and the image of $\V^{\Gamma}(R)$ under the natural map $\V^{\Gamma}(R)\to K_0^{\Gamma}(R)$ makes $(K_0^{\Gamma}(R), [R])$ into an object of $\POG^u.$  If $\phi$ is a graded ring homomorphism, then $\V^\Gamma(\phi)$ is a morphism of $\POM$ and $K_0^\Gamma(\phi)$ is a morphism of $\POG.$ If $\phi$ is unital (i.e. $\phi$ maps the identity onto identity), $\V^\Gamma(\phi)$ is a morphism of $\POM^u$ and $K_0^\Gamma(\phi)$ is a morphism of $\POG^u.$

If $K$ is a $\Gamma$-graded division ring, recall that the support $\Gamma_K=\{\gamma\in\Gamma\mid K_\gamma\neq \{0\}\}$ is a subgroup of $\Gamma$. If $\Gamma$ is abelian,
the map
$\V^{\Gamma}(K)\to \Zset^+[\Gamma/\Gamma_K]$ given by $\left[(\gamma_1^{-1})K^{p_1}\oplus\ldots\oplus (\gamma_n^{-1})K^{p_n}\right]\mapsto \sum_{i=1}^n p_i(\gamma_i\Gamma_K)$
is a canonical isomorphism of $\Gamma$-monoids by \cite[Proposition 3.7.1]{Roozbeh_book} (note that the difference of signs in the formula in \cite[Proposition 3.7.1]{Roozbeh_book} is present because the action of $\Gamma$ on $\V^\Gamma(R)$ is different than the action we consider here).

If the grade group is the group of integers and $K$ is a trivially graded field, $K_0^\Gamma(K)$ is a $\Zset[\Zset]$-module. To avoid confusion of working with $\Zset$-modules which also have an additional $\Zset$-action, we consider the infinite group $\Gamma=\langle t \rangle$ on a single generator $t$ to act on the monoids and their Grothendieck groups.
Hence, if $n$ is a positive integer, $\gamma_1, \gamma_2,\ldots, \gamma_n\in\Zset,$ and $R=\M_n(K)(\gamma_1, \ldots, \gamma_n),$ then $R$
is graded by $\Zset$, but its $\Gamma$-monoid $\V^\Gamma(R)$ and the $\Gamma$-group $K_0^\Gamma(R)$ are considered with the action of $\Gamma=\langle t\rangle$ not the action of $\Zset.$
By Lemma \ref{lemma_on_shifts}, we can assume that $\gamma_1, \gamma_2,\ldots, \gamma_n\in\Zset^+$ when considering $R.$ Let $k$ be the maximum of $\gamma_1, \gamma_2,\ldots, \gamma_n$ and $l_0, \ldots, l_k$ be such that $l_i$ is the number of times $i$ appears on the list $\gamma_1, \gamma_2,\ldots, \gamma_n$ for $i=0, \ldots, k.$ There is a canonical $\POM^u$-isomorphism
\[f_{n,\ol\gamma}:(\V^\Gamma(\M_n(K)(\ol\gamma)), [\M_n(K)(\ol\gamma)])\cong (\Zset^+[t,t^{-1}], \sum_{i=0}^kl_it^i).\]
(both \cite[Section 3.1]{Lia_realization} and \cite[Section 1.5]{Roozbeh_Lia_Ultramatricial} have more details). The choice to consider $\POM^u$ instead of $\POM$ is relevant: if $\Gamma=\langle t\rangle\cong \Zset,$ $R=K=K(0),$ and $S=\M_2(K)(0,1),$ then both $\V^\Gamma(R)$ and $\V^\Gamma(S)$ are isomorphic to $\Zset^+[t,t^{-1}],$ so only their order-units, corresponding to 1 and $1+t$ respectively, carry the information on the size of $R$ and $S$ and their shifts.

Let $m$ and $n$ be positive integers, $\gamma_1, \gamma_2,\ldots, \gamma_n\in \Zset$ and $R=\M_n(K[x^m, x^{-m}])(\ol\gamma)$. By using Lemma \ref{lemma_on_shifts} when considering $R$, we can consider the shift $\gamma_i$ modulo $m$ so we can assume $\gamma_1, \gamma_2,\ldots, \gamma_n\in\{0, 1, \ldots, m-1\}.$ Let
$l_0, \ldots, l_{m-1}$ be such that $l_i$ is the number of times $i$ appears on the list $\gamma_1, \gamma_2,\ldots, \gamma_n$ for $i=0,\ldots, m-1.$
There is a a canonical $\POM^u$-isomorphism
\[f_{n,m,\ol\gamma}:(\V^\Gamma(\M_n(K[x^m, x^{-m}])(\ol\gamma)), [\M_n(K[x^m, x^{-m}])(\ol\gamma)])\cong (\Zset^+[t,t^{-1}]/(t^m=1), \sum_{i=0}^{m-1}l_it^i).\]

Let $\kappa$ be any cardinal now. Let $\ol\gamma: \kappa\to \Zset^+$ be a map and $\mu_n$ be the cardinality of $\ol\gamma^{-1}(n)$ for any $n\in\Zset^+.$
We use $e_{\alpha\beta}$ for a standard matrix unit and we denote the standard matrix unit with one in the first row and the first column by $e_{00}$ not $e_{11}$ so that this agrees with the fact that 0, not 1, is the smallest element of a finite ordinal $n.$ With this convention, we note that the standard generating interval $D_{\ol\gamma}$ of  $\V^\Gamma(\M_\kappa(K)(\ol\gamma))$
consists of the finite sums of the elements of the form $l_nt^{n}[e_{00}\M_\kappa(K)(\ol\gamma)]$ where $l_n$ is the cardinality of a finite subset of $\mu_n.$

If $\ol\gamma(\alpha)$ is considered modulo $m$ for every $\alpha\in\kappa,$ and $\mu_n$ is the cardinality of $\ol\gamma^{-1}(n)$ for any $n\in\{0, \ldots, m-1\},$ the standard generating interval $D_{\ol\gamma}$ of $\V^\Gamma(\M_\kappa(K[x^m, x^{-m}])(\ol\gamma))$ consists of the finite sums of the elements of the form $l_nt^{n}[e_{00}\M_\kappa(K[x^m, x^{-m}])(\ol\gamma)]$ where $l_n$ is the cardinality of a finite subset of $\mu_n.$

\subsection{The talented monoid}\label{subsection_graph_monoid}
For any infinite emitter $v$ of a graph $E$ and any finite and nonempty $Z\subseteq \so^{-1}(v),$ let $q^v_Z=v-\sum_{e\in Z}ee^*.$ If it is clear which infinite emitter we consider, we use $q_Z$ for $q_Z^v.$
If $\Gamma=\langle t\rangle$ is the infinite cyclic group on $t,$ the {\em talented monoid} or the {\em graph $\Gamma$-monoid $M_E^\Gamma$}
is the free abelian $\Gamma$-monoid on generators $[v]$ for $v\in E^0$ and $[q^v_Z]$ for infinite emitters $v$ and nonempty and finite sets $Z\subseteq \so^{-1}(v)$ subject to the relations
\[[v]=\sum_{e\in \so^{-1}(v)}t[\ra(e)],\hskip.4cm [v]=[q^v_Z]+\sum_{e\in Z}t[\ra(e)],\,\mbox{ and }\;[q^v_Z]=[q^v_W]+\sum_{e\in W-Z}t[\ra(e)]\]
where $v$ is a vertex which is regular for the first relation and an infinite emitter for the second two relations in which $Z\subsetneq W$ are finite and nonempty subsets of $\so^{-1}(v).$
The map $[v]\mapsto [vL_K(E)]$ and $[q^v_Z]\mapsto [q^v_ZL_K(E)]$ extends to an isomorphism $\gamma^\Gamma_E$ of $M^\Gamma_E$ and $\V^\Gamma(L_K(E))$ (see \cite[Proposition 5.7]{Ara_et_al_Steinberg}).

If $M_E$ is the monoid obtained analogously but by considering the trivial group instead of the group $\Gamma$, the monoid $M_E$ registers only whether two vertices are connected, while the ``talent'' of $M_E^\Gamma$ is to register the lengths of paths between vertices: if $p$ is a path of length $n$, the relation $[\so(p)]=t^n[\ra(p)]+x$ holds in $M_E^\Gamma$ for some $x\in M_E^\Gamma.$

If $E$ has finitely many vertices, the finite sum $u_E=\sum_{v\in E^0}[v]$ is an order-unit of $M_E^\Gamma.$ If $E^0$ is infinite, $D_E=\{x\geq 0\mid x\leq \sum_{v\in F}[v]\mbox{ for some finite }F\subseteq E^0\}$ is a generating interval of $M_E^\Gamma$ (one can indeed check that $D_E$ is convex, upwards directed, and that $\Zset^+[\Gamma]D$ generates $M_E^\Gamma$).
The interval $D_E$ coincides with the ``standard'' generating interval of $\V^\Gamma(L_K(E))$ obtained using unitization of $L_K(E)$ (see \cite[Section 3.5]{Roozbeh_book}, \cite[Section 4.1]{Roozbeh_Lia_Ultramatricial}, and \cite[Section 4.3]{Lia_realization}).
If $F$ is another graph and $\phi: L_K(E)\to L_K(F)$ is a graded  homomorphism, then $\V^{\Gamma}(\phi)$ is a morphism of $\POM^D$ and $K_0^{\Gamma}(\phi)$ is a morphism of $\POG^D.$ We use $\ol\phi$ to shorten the notation of both morphisms. If $\iota: E\to F$ is a graph isomorphism, it extends to a unique graded $*$-algebra isomorphism which we continue to denote by $\iota$ and a unique $\POM^D$-isomorphism which we denote by $\ol\iota.$

If $(H,S)$ is an admissible pair of a graph $E$ and $I(H,S)$ the graded ideal of $L_K(E)$ generated by $\{v\mid v\in H\}\cup\{v^H\mid v\in S\},$ then  the $\Gamma$-order-ideal $J^\Gamma(H,S)$ of $M_E^\Gamma$ generated by $\{[v]\mid v\in H\}\cup\{[v^H]\mid v\in S\}$ is such that $(H,S),$ $I(H,S)$ and $J^\Gamma(H,S)$ correspond to each other in the isomorphisms of the three lattices: the lattice of admissible pairs of $E,$ the lattice of graded ideals of $L_K(E),$ and the lattice of $\Gamma$-order-ideals of $M_E^\Gamma$ (by \cite[Theorem 2.5.8]{LPA_book} and \cite[Theorem 5.11]{Ara_et_al_Steinberg}). By \cite[Proposition 3.7]{Lia_porcupine_quotient}, if $(H,S)$ and $(G,T)$ admissible pairs of a graph $E$ such that $(H,S)\leq (G,T),$ then there is a pre-ordered $\Gamma$-monoid isomorphism $M^\Gamma_{(G,T)/(H,S)}\cong J^\Gamma(G,T)/J^\Gamma(H,S).$

We briefly review an alternative construction of $M_E^\Gamma$ which we also use in one of the proofs. Let $\F_E^\Gamma$ be a free commutative $\Gamma$-monoid on generators $v\in E^0$
and $q_Z^v$ for $v$ infinite emitter and $Z$ finite and nonempty subset of $\so^{-1}(v).$ The monoid $M^\Gamma_E$ can be obtained as a quotient of $\F^\Gamma_E$ subject to the congruence closure $\sim$ of the relation $\to_1$ on $\F_E^\Gamma-\{0\}$ defined by (A1), (A2) and (A3) below
for any $n\in \Zset$ and $a\in \F_E^\Gamma.$  In (A1), $v$ is a regular vertex and in (A2) and (A3), $v$ is an infinite emitter and $Z$ and $W$ are finite and nonempty subsets of $\so^{-1}(v)$ such that  $Z\subsetneq W.$
\begin{enumerate}
\item[(A1)] $a+t^n v\to_1 a+\sum_{e\in \so^{-1}(v)}t^{n+1}\ra(e).$
\item[(A2)] $a+t^n v\to_1 a+t^n q_Z+\sum_{e\in Z}t^{n+1}\ra(e).$
\item[(A3)] $a+t^n q_Z\to_1 a+t^n q_W+\sum_{e\in W-Z}t^{n+1}\ra(e).$
\end{enumerate}
If $\to$ is the reflexive and transitive closure of $\to_1$
on $\F_E^\Gamma,$ then $\sim$ is the congruence on $\F^\Gamma_E$ generated by the relation $\to$.
We review the Confluence Lemma.

\begin{lemma} {\bf The Confluence Lemma.} \cite[Lemma 5.9]{Ara_et_al_Steinberg}  Let $E$ be a graph and $a,b\in  \F^\Gamma_E-\{0\}$. The relation $a\sim b$ holds if and only if $a \to c$ and $b\to c$ for some $c \in \F_E-\{0\}.$
\label{lemma_confluence}
\end{lemma}

The monoid $M_E^\Gamma$ is cancellative (by \cite[Corollary 5.8]{Ara_et_al_Steinberg}) so the natural pre-order is, in fact, an order.
By \cite[Proposition 3.4]{Roozbeh_Lia_Comparable}, the relation $x< t^nx$ is impossible for any $x\in M_E^\Gamma$ and any positive integer $n.$ The remaining possibilities give rise to the following types.
\begin{enumerate}
\item If $x=t^nx$ for some positive integer $n,$ we say that $x$ is {\em periodic}.

\item If $x>t^nx$ for some positive integer $n,$ we say that $x$ is {\em aperiodic}.

\item If $x$ and $t^nx$ are incomparable for any positive integer $n,$ we say that $x$ is {\em incomparable}.
\end{enumerate}

If $x$ is periodic or aperiodic, $x$ is {\em comparable}. This terminology matches the one used in  \cite{Roozbeh_Lia_Comparable}.
We note that \cite{Roozbeh_Alfilgen_Jocelyn} uses ``cyclic'' for ``periodic'' and ``non-comparable'' for ``incomparable''. In our terminology, the authors of \cite{Roozbeh_Alfilgen_Jocelyn} define a
$\Gamma$-order-ideal $I$ of $M_E^\Gamma$ to be {\em periodic} (respectively, {\em comparable, incomparable}) if its every nonzero element is periodic (respectively, comparable, incomparable). We also say that $I$ is {\em aperiodic} if its every nonzero element is aperiodic.

We summarize some results of \cite{Lia_porcupine_quotient} which we use.

\begin{theorem} \cite[Theorems 5.7, 7.4 and 7.5]{Lia_porcupine_quotient}
Let $E$ be any graph.
\begin{enumerate}[\upshape(1)]
\item If $E$ is cofinal, exactly one of the following holds.

\begin{enumerate}[\upshape(a)]
\item $E$ has a sink, has no infinite paths, and every  $x\in M_E^\Gamma, x\neq 0,$ is incomparable.

\item $E$ has a cycle $c$ without exits, every infinite path ends in $c$, and every $x\in M_E^\Gamma$ is periodic.

\item $E$ has an extreme cycle and two non-disjoint cycles and every $x\in M_E^\Gamma, x\neq 0,$ is aperiodic.

\item $E$ has a terminal path and an infinite path which does not end in a cycle and every  $x\in M_E^\Gamma, x\neq 0,$ is incomparable.
\end{enumerate}

\item The following two conditions are equivalent for any graph $E.$
\begin{enumerate}[\upshape(a)]
\item If $(H,S)$ and $(G,T)$ are admissible pairs of $E$ such that $(G,T)/(H,S)$ is cofinal, then  $M_{(G,T)/(H,S)}^\Gamma$ is either periodic or incomparable.

\item The cycles of $E$ are mutually disjoint.
\end{enumerate}
\end{enumerate}
\label{theorem_from_porcupine_quot_on_monoid}
\end{theorem}

\subsection{Composition series}
\label{subsection_composition}
By \cite[Corollary 4.3]{Lia_porcupine_quotient}, the following are equivalent for any graph $E$.
\begin{enumerate}
\item There is a chain of admissible pairs \[(\emptyset, \emptyset)=(H_0, S_0) \leq (H_1, S_1)\leq \ldots\leq (H_n, S_n)=(E^0, \emptyset)\] such that the porcupine-quotient graph $(H_{i+1}, S_{i+1})/(H_i, S_i)$ is cofinal for all $i=0, \ldots, n-1.$ If $S_i=\emptyset$ for all $i,$ we write the above chain shorter as $\emptyset=H_0\leq H_1\leq \ldots\leq H_n=E^0.$

\item There is a chain of graded ideals $\{0\}=I_0 \leq I_1\leq \ldots\leq I_n=L_K(E)$ such that the graded algebra $I_{i+1}/I_i$ is graded simple for all $i=0, \ldots, n-1.$

\item There is a chain of $\Gamma$-order-ideals $\{0\}=J_0 \leq J_1\leq  \ldots\leq J_n=M_E^\Gamma$ such that the $\Gamma$-monoid $J_{i+1}/J_i$ is simple  (i.e., without any nontrivial and proper $\Gamma$-order-ideals) for all $i=0, \ldots, n-1.$
\end{enumerate}

If the above conditions hold, we say that $E$ (respectively $L_K(E),$ $M_E^\Gamma$) has  {\em composition length $n$}.
The graph $E$ (respectively $L_K(E),$ $M_E^\Gamma$) has a {\em composition series} if $E$ (respectively $L_K(E),$ $M_E^\Gamma$) has composition  length $n$ for some positive integer $n.$

Let $\Ter(E)$ denote the saturated closure of the (hereditary) set of terminal vertices. In \cite{Lia_porcupine_quotient}, the {\em composition quotients} of a graph $E$ are graphs defined by
\begin{center}
$E_0=E\;$ and $\;E_{n+1}=E_n/(\Ter(E_n, B_{\Ter(E_n)}))$ if $\Ter(E_n)\subsetneq E_n^0\;$ and $\;E_{n+1}=\emptyset$ otherwise.
\end{center}
These quotients are used to characterize graphs with composition series in the result below.

\begin{theorem} \cite[Theorem 6.5]{Lia_porcupine_quotient}
The following conditions are equivalent for a graph $E$.
\begin{enumerate}[\upshape(1)]
\item The graph $E$ has a composition series.
\item There is a nonnegative integer $n$ such that $E_{n+1}=\emptyset$ and $E_n\neq\emptyset$ and for each $n$ for which the composition quotient $E_n$ is nonempty, $\Ter(E_n)$ is nonempty, the set of terminal vertices of $E_n$ contains finitely many clusters, and the set of breaking vertices of $\Ter(E_n)$ is finite.
\end{enumerate}
\label{theorem_comp_series}
\end{theorem}
By \cite[Corollary 6.6]{Lia_porcupine_quotient}), every graph with finitely many vertices has a  composition series.

One can use Theorem \ref{theorem_comp_series} to obtain a specific composition series. We illustrate the construction in the example below.

\begin{example}
Let $E$ be the graph $
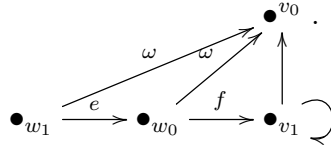
\xymatrix{&& \bullet^{v_0}\\\bullet_{w_1}\ar[r]^e\ar[urr]^{\omega}&\bullet_{w_0}\ar[r]^f\ar[ur]^{\omega} & \bullet_{v_1}\ar[u]\ar@(ur, dr)}.$
This graph has composition length 4. We have that $\Ter(E)=\{v_0\},$ $B_{\Ter(E)}=\{w_0, w_1\},$ so that $E_1$ is the graph $\xymatrix{\bullet_{w_1}\ar[r]&\bullet_{w_0}\ar[r]& \bullet_{v_1}\ar@(ur, dr)}\;\;\;$ with a unique cluster containing $v_1$, so $\Ter(E_1)=E_1^0$ and  \[(\emptyset,\emptyset)\leq (\{v_0\}, \emptyset)\leq (\{v_0\}, \{w_0\})\leq (\{v_0\}, \{w_0, w_1\}) \leq (E^0, \emptyset)\] is a composition series of $E$.
The composition factors are the porcupine graph of the pair $(\{v_0\}, \emptyset),$ and the graphs $\xymatrix{\bullet_{w^{ee^*}}\ar[r]&\bullet_{w_0}},$ $\bullet_{w_1},$ and $\xymatrix{\bullet_{w_1}\ar[r]&\bullet_{w_0}\ar[r]& \bullet_{v_1}\ar@(ur, dr)}\;\;\;.$
\label{example_cs}
\end{example}

\section{1-S-NE graphs}
\label{section_n=1}

We recall that
we use $\Zset^+$ to denote the set of nonnegative integers,
$\Gamma=\langle t\rangle$ for  the infinite cyclic group generated by $t,$ and $K$ for a field trivially graded by $\Zset$. For any positive integer $n$, we consider $K[x^n, x^{-n}]$ to be graded as in section \ref{subsection_graded_rings}, we let $+_n$ denote the addition modulo $n$ (i.e. the addition in $\Zset/n\Zset$), and $i\in n$ means that $i$ is an element of the set $\{0, 1, \ldots, n-1\}.$  If $\phi$ is a $\Zset$-graded algebra homomorphism, then $\ol\phi$ denotes the map induced by $\phi$ on the $\Gamma$-monoids and on the $\Gamma$-groups.

\subsection{S-NE graphs} We recall that S shortens ``sinks'', that NE stands for ``no-exits'' and indicates cycles without exits, and we recall the definitions from the introduction.

\begin{definition}
A graph $E$ is an {\em S-NE graph} if $(G,T)/(H,S)$ has either a sink or a cycle without exits for every two admissible pairs $(H,S)\leq (G,T)$ such that $(G,T)/(H,S)$ is cofinal.
An S-NE graph $E$ is an {\em $n$-S-NE graph} for a positive integer $n$ if $E$ has  composition length $n$.
An S-NE graph $E$ is a {\em composition S-NE graph} if it is an $n$-S-NE graph for some positive integer $n.$
\label{definition_S_NE_graphs}
\end{definition}

By Theorem \ref{theorem_graded_simple}, $E$ is a 1-S-NE graph if and only if $E$ is a cofinal graph with a sink or a cycle without exits. A finite acyclic graph with $n$ sinks is an $n$-S-NE graph. The Toeplitz graph (see the introduction) is a 2-S-NE graph and the graph from Example \ref{example_cs} is a 4-S-NE graph.
By Theorem \ref{theorem_from_porcupine_quot_on_monoid}, an S-NE graph has disjoint cycles.
If $E$ has finitely many vertices, the converse holds by Theorem \ref{theorem_comp_series} as well as by \cite[Corollary 6.6]{Lia_porcupine_quotient}. By Theorem \ref{theorem_comp_series} and \cite[Lemma 6.3 and Corollary 6.6]{Lia_porcupine_quotient}, $E$ is a composition S-NE graphs if and only if the cycles of $E$ are disjoint, $E$ has finitely many cycles, sinks and infinite emitters, and every infinite path ends in a cycle.
The first graph below is an S-NE graph which is not a composition S-NE graph because it has infinitely many cycles. The second graph below is an acyclic graph of composition length two, but it is not an S-NE graph because there is an infinite path which does not end in a cycle.
\[\xymatrix{\\\ar@{.}[r]&\bullet\ar@(ur,ul)\ar[r]& \bullet\ar@(ur,ul)\ar[r]&\bullet\ar@(ur,ul)}\hskip3cm\xymatrix{\bullet &&\\
\bullet\ar[r]\ar[u] &\bullet \ar[r]
\ar[ul] & \bullet\ar@{.}[r] \ar[ull] & \ar@{.}[ulll]
}\]

By \cite[Proposition 4.2]{Lia_porcupine_quotient}, $E$ has a composition series if and only if $P_{(H,S)}$ and $E/(H,S)$ have  composition series for each admissible pair $(H,S)$ of $E$. By this result and by Definition \ref{definition_S_NE_graphs}, $E$ is a composition S-NE graph if and only if $P_{(H,S)}$ and $E/(H,S)$ are composition S-NE graphs.

\subsection{The graded isomorphism classes of matrix algebras.}  In Proposition \ref{proposition_matrix_rings_general}, we generalize the second part of Lemma \ref{lemma_on_shifts}, show that $\kappa$ and the cardinals $\mu_k$ for $k\in \Zset^+,$ considered in section \ref{subsection_graded_Grothendieck},
are unique for the graded isomorphism class of $\M_\kappa(K)(\ol\gamma)$
and for the graded isomorphism class of $\M_\kappa(K[x^m, x^{-m}])(\ol\gamma)$ for positive integer $m.$ If $\kappa$ is infinite, the proof of Lemma \ref{lemma_on_shifts} as in \cite{Roozbeh_book} does not generalize to the matrix algebras of infinite size because the endomorphism ring of a $\kappa$-dimensional $K$-vector space, when represented via $\kappa\times\kappa$ matrices, contains matrices with infinitely many nonzero entries (most prominently, the identity has infinitely many nonzero entries on the diagonal).
Hence, such an endomorphism ring is not isomorphic to the ring $\M_\kappa(K)$ if $\kappa$ is infinite. So, to prove the proposition, we use a different strategy which involves the Grothendieck $\Gamma$-monoids.

\begin{proposition}
Let $\kappa$ and $\kappa'$ be cardinals, $\ol\gamma: \kappa\to \Zset^+$ and $\ol\gamma': \kappa'\to \Zset^+$ be any functions, and let $e_{\alpha\beta}$ to denote the standard matrix units of a matrix algebra of $\kappa\times\kappa$ size.

\begin{enumerate}[\upshape(1)]
\item Let $\mu_k$ be the cardinality of $\ol\gamma^{-1}(k)$ and $\mu'_k$ be the cardinality of $(\ol\gamma')^{-1}(k)$ for any $k\in \Zset^+.$ If $f: (\V^\Gamma(\M_\kappa (K)(\ol\gamma)), D_{\ol\gamma})\cong (\V^\Gamma(\M_{\kappa'} (K)(\ol\gamma')), D_{\ol\gamma'})$ is a $\POM^D$-isomorphism, then
\begin{center} $\kappa=\kappa'$ and $\;\;\mu_k=\mu'_k$ for each $k\in \Zset^+$
\end{center} and there is a bijection $\rho:\kappa\to\kappa$ and a graded algebra $*$-isomorphism $\phi:\M_\kappa (K)(\ol\gamma)\to \M_{\kappa'} (K)(\ol\gamma')$ such that
$\phi(e_{\alpha\beta})=e_{\rho(\alpha)\rho(\beta)}$ for every $\alpha, \beta\in \kappa$ and such that $\V^\Gamma(\phi)=f.$

\item For positive integers $m$ and $m',$ $\alpha\in\kappa,$ $k\in m,$ $\alpha'\in\kappa',$ and $k'\in m',$
let $\mu_k$ be the cardinality of $\ol\gamma^{-1}(k)$ when $\ol\gamma(\alpha)$ is considered modulo $m$ and let $\mu'_k$ be the cardinality of $(\ol\gamma')^{-1}(k)$ when $\ol\gamma'(\alpha)$ is considered modulo $m'.$ If $f: (\V^\Gamma(\M_\kappa (K[x^m, x^{-m}])(\ol\gamma)), D_{\ol\gamma})\to (\V^\Gamma(\M_{\kappa'} (K[x^{m'}, x^{-m'}])(\ol\gamma')), D_{\ol\gamma'})$ is a $\POM^D$-isomorphism, then
\begin{center}
$\kappa=\kappa',$ $m=m',$ there is $i\in m$ such that
$\mu_k=\mu'_{k+_mi}$ for each $k\in m,$
\end{center} and there is a bijection $\rho:\kappa\to \kappa$ and a graded algebra $*$-isomorphism $\phi:\M_\kappa (K[x^m, x^{-m}])(\ol\gamma)$ $\to \M_{\kappa} (K[x^m, x^{-m}])(\ol\gamma')$ such that $\phi$ maps a standard matrix unit $e_{\alpha\beta}$ to the element of the form $x^{lm}e_{\rho(\alpha)\rho(\beta)}$ for some $l\in \Zset$ determined by $\alpha$ and $\beta$ using $\ol\gamma$ and $\ol\gamma'$ and such that $\V^\Gamma(\phi)=f.$
\end{enumerate}
\label{proposition_matrix_rings_general}
\end{proposition}
\begin{proof}
To show part (1), let $f_{\kappa,\ol\gamma}$ be the $\POM^D$-isomorphism $\V^\Gamma(\M_\kappa(K)(\ol\gamma))\to \Zset^+[t,t^{-1}]$ induced by the map $[e_{00}]\mapsto 1$ and let $f_{\kappa',\ol\gamma'}$ be analogous such isomorphism for $\kappa'$ and $\ol\gamma'.$ The assumption of part (1) implies that $g=f_{\kappa',\ol\gamma'}ff_{\kappa, \ol\gamma}^{-1}$ is a $\POM^D$-isomorphism $(\Zset^+[t, t^{-1}], D_{\mu})\to(\Zset^+[t, t^{-1}], D_{\mu'})$ where the generating intervals $D_\mu$ and $D_{\mu'}$
are as in section \ref{subsection_graded_Grothendieck}.
As $g(D_{\mu})\subseteq D_{\mu'}$ and $g^{-1}(D_{\mu'})\subseteq D_{\mu},$ $D_{\mu}$ and $D_{\mu'}$ have the same cardinality, so $\kappa=\kappa'.$

The isomorphism $g$ is uniquely determined by the image of 1. Assume that $g(1)=p(t)$ for some Laurent polynomial $p(t)\in \Zset^+[t, t^{-1}].$ As $p(t)\in D_{\mu'},$ $p(t)$ contains no negative powers of $t.$ Analogously, if $q(t)$ is the polynomial $g^{-1}(1),$ then $q(t)$ does not contain any negative powers of $t$ as $q\in D_{\mu}.$ Since $1=g(g^{-1}(1))=g(q(t))=q(t)g(1)=q(t)p(t),$ the degrees of $p$ and $q$ add up to zero which implies that $p(t)=a$ and $q(t)=b$ are constant polynomials for some $a,b\in \Zset^+$. Since $ab=1,$ we have that $a=b=1$ which shows that $g(1)=1.$

If $\kappa$ is finite, equating the order-units $\sum_k \mu_k t^{k}$ and  $\sum_k \mu'_k t^{k}$ implies the relation $\mu_k=\mu'_k.$ If $\kappa$ is infinite, $g$ maps the set of monomials of the form $l_k t^{k},$ where $l_k$ is the cardinality of a finite subset of $\mu_k,$ to the set of monomials of the form $l'_k t^{k}$ where $l'_k$ is the cardinality of a finite subset of $\mu_k'.$ As $g(1)=1,$ this implies that the cardinality of the set of all finite sets of $\mu_k$ is equal to the cardinality of the set of all finite sets of $\mu_k'.$ If $\mu_k$ is finite, this implies that $\mu'_k$ is finite and that $\mu_k=\mu_k'$. If $\mu_k$ is infinite, this implies that $\mu_k$ is infinite. As the cardinality of the set of all finite subsets of an infinite cardinal is equal to that cardinal, we have that
$\mu_k=\mu_k'.$ Thus, $\mu_k=\mu_k'$ for any $k\in \Zset^+.$

Let $\pi:\kappa\to\kappa$ be a bijection which permutes the images of $\ol\gamma$ so they are listed in the non-decreasing order and let $\pi'$ be analogous such bijection for $\ol\gamma'.$ Let $\phi_{\pi}$ and $\phi_{\pi'}$ be graded $*$-isomorphisms induced by the maps $e_{\alpha0}\mapsto e_{\pi(\alpha)0}$ and $e_{\alpha0}\mapsto e_{\pi'(\alpha)0}$ respectively. If $\ol\mu: \kappa\to\kappa$ denotes the map $\ol\gamma\pi,$ we have that $f_{\kappa, \ol\gamma}=f_{\kappa,\ol\mu}\ol\phi_\pi$ and that $f_{\kappa, \ol\gamma'}=f_{\kappa,\ol\mu}\ol\phi_{\pi'}.$ Hence,  $g=f_{\kappa, \ol\mu}\ol\phi_{\pi'} f\ol\phi_{\pi}^{-1}f_{\kappa, \ol\mu}^{-1}$ is the identity which implies that $\ol\phi_{\pi'} f\ol\phi_{\pi}^{-1}$ is the identity so that $f=\ol\phi_{\pi'}^{-1} \ol\phi_{\pi}=\ol{\phi_{\pi'}^{-1} \phi_{\pi}}.$ This shows that if we take  $\rho={\pi'}^{-1}\pi$ and $\phi=\phi_{\pi'}^{-1} \phi_{\pi},$ then $\rho$ is as required in the statement of (1) and $\phi$ is a graded $*$-isomorphism such that $f=\ol\phi.$ The isomorphism $\phi$ is obtained by compositions of two maps as in part (1) of Lemma \ref{lemma_shifts_general}, so it maps a standard matrix unit to a standard matrix unit.

To show part (2), let $f_{\kappa,m, \ol\gamma}$ be the $\POM^D$-isomorphism $\V^\Gamma(\M_\kappa(K[x^m, x^{-m}])(\ol\gamma))\to \Zset^+[t,t^{-1}]/$ $(t^m=1)$ induced by the map $[e_{00}]\mapsto 1$ and let $f_{\kappa',m',\ol\gamma'}$ be analogous such isomorphism for $\kappa', m',$ and $\ol\gamma'.$
The assumption of part (2) implies that $g=f_{\kappa',m',\ol\gamma'}ff_{\kappa,m, \ol\gamma}^{-1}$ is a $\POM^D$-isomorphism $(\Zset^+[t, t^{-1}]/(t^m=1), D_{\mu})\to(\Zset^+[t, t^{-1}]/(t^{m'}=1), D_{\mu'}).$ As $g(D_{\mu})\subseteq D_{\mu'}$ and $g^{-1}(D_{\mu'})\subseteq D_{\mu},$ $D_{\mu}$ and $D_{\mu'}$ have the same cardinality, so $\kappa=\kappa'.$ As $t^{m}=t^mg(g^{-1}(1))=g(t^m g^{-1}(1))=g(g^{-1}(1))=1,$ the relation $t^m=1$ holds in $\Zset^+[t,t^{-1}]/(t^{m'}=1).$ Thus, $m\geq m'.$ Analogously,  $m'\geq m$, so $m=m'.$

The map $g$ is uniquely determined by the image of 1. Assume that $g(1)=p(t)$ for some $p(t)\in \Zset^+[t, t^{-1}]/(t^m=1).$ As $p(t)\in D_{\mu'},$ $p(t)$ contains no negative powers of $t.$ Analogously, if $q(t)\in \Zset^+[t, t^{-1}]/(t^m=1)$ is  $g^{-1}(1),$ then $q(t)$ does not contain any negative powers of $t.$ As $1=g(g^{-1}(1))=g(q(t))=q(t)g(1)=q(t)p(t),$ $p=t^{i}$ and $q=t^{m-i}$ for some $i\in m.$ Hence, $g(1)=t^{i}.$

Let $\sigma$ be the cyclic permutation  given by $k\mapsto k+_m i$ (recall that $+_m$ denotes the addition modulo $m$). If $\kappa$ is finite, equating the order-units $\sum_{k=0}^{m-1} \mu_k t^{k}$ and $\sum_{k=0}^{m-1} \mu'_k t^{k+i}$ implies $\mu_k=\mu'_{\sigma(k)}$. If $\kappa$ is infinite, $g$ maps the set of monomials of the form $l_k t^{k}$ for $l_k$ the cardinality of any finite subset of $\mu_k$ to the set of monomials of the form $l'_k t^{k+i}$ for $l'_k$ the cardinality of any finite subset of $\mu_{k+_m i}'.$ The first set has cardinality $\mu_k$ and the second set has cardinality $\mu_{k+_m i}',$ so $\mu_k=\mu_{k+_m i}'$ for any $k\in \Zset^+.$

Let $\ol\mu$ be the list obtained by permuting the images of $\ol\gamma$ so that when they are considered modulo $m,$ they are listed in a non-decreasing order. Let
$\pi:\kappa\to\kappa$ be the corresponding bijection and let $\pi'$ be the analogous bijection for $\ol\gamma'.$ Let $\ol\mu=\ol\gamma\pi$ and $\ol\mu'=\ol\gamma'\pi'.$ Let $\rho_1: \kappa\to \kappa$ denote a permutation so that
$\sigma\ol\mu\rho_1=\ol\mu'.$
Let $\phi_{\rho_1,\sigma}$ be the graded algebra isomorphism $\phi_{\rho_1,\sigma}:\M_\kappa(K[x^m, x^{-m}])(\ol\mu)\to \M_\kappa(K[x^m, x^{-m}])(\ol\mu')$ determined by first permuting the shifts so that the list $\mu_k=\mu'_{\sigma(k)}$ moves from the $k$-th place to the $(k-_mi)$-th place and then adding $i$ modulo $m$ to each of the shifts producing exactly the $\ol\mu'$ list. Note that this second operation can be obtained by considering the graded $*$-isomorphism induced by the map $e_{\alpha\beta}\mapsto x^{(\sigma(\ol\mu(\beta))-\sigma(\ol\mu(\alpha)))m}e_{\alpha\beta}.$ As any element of the form $x^{lm}, l\in\Zset$ is unitary, the map $\phi_{\rho_1,\sigma}$ is a $*$-isomorphism.

If $\phi_{\pi}$ and $\phi_{\pi'}$ denote the graded $*$-isomorphisms induced by the bijections $\pi$ and $\pi',$ we have that $f_{\kappa,m,\ol\mu}\ol\phi_{\pi}=f_{\kappa,m,\ol\gamma}$
and $f_{\kappa,m,\ol\mu'}\ol\phi_{\pi'}=f_{\kappa,m,\ol\gamma'}$
so that
$g =f_{\kappa, m, \ol\mu}\ol\phi_{\pi'} f\ol\phi_{\pi}^{-1}f_{\kappa, m, \ol\mu}^{-1}= f_{\kappa, m, \ol\mu'}\ol\phi_{\rho_1,\sigma}f_{\kappa, m, \ol\mu}^{-1}.$
Thus, \[\ol\phi_{\pi'} f\ol\phi_{\pi}^{-1}=\ol\phi_{\rho_1,\sigma}\mbox{ which implies that } f=\ol\phi_{\pi'}^{-1}\ol\phi_{\rho_1,\sigma}\ol\phi_{\pi}.\]
This shows that for $\phi=\phi_{\pi'}^{-1}\phi_{\rho_1,\sigma}\phi_{\pi},$ we have that $f=\ol\phi.$ Letting $\rho=\pi'^{-1}\rho_1\pi,$ we obtain a bijection as needed and we have that
$\phi(e_{\alpha\beta})=x^{(\sigma(\ol\mu(\beta))-\sigma(\ol\mu(\alpha)))m}e_{\rho(\alpha)\rho(\beta)}.$
\end{proof}

\subsection{1-S-NE graphs and their canonical forms}\label{subsection_cofinal}
Recall that $E$ is a 1-S-NE graph if and only if $E$ is a cofinal graph with a sink or a cycle without exits. If a cofinal graph $E$ has a sink $v_0$ and if $P^{v_0}$ is the set of paths ending in $v_0,$ let $\kappa$ be the cardinality of $P^{v_0}$ so that we can index the elements of $P^{v_0} $ by the elements of $\kappa$ and write $P^{v_0}$ as $\{p_\alpha\mid \alpha\in\kappa\}$. We can chose a bijection $\kappa\to P^{v_0}$ indexing the elements of $P^{v_0}$ so that 0 corresponds to the zero-length path $v_0.$ Let $\ol\gamma: \kappa\to \Zset^+$ map $\alpha$ to the length of the path $p_\alpha.$
If $e_{\alpha\beta}, \alpha,\beta\in\kappa$ are the standard matrix units of $\M_{\kappa}(K)(\ol\gamma),$
the correspondence $p_\alpha p_\beta^*\mapsto$ $e_{\alpha\beta}$ extends to a graded $*$-isomorphism $L_K(E)\cong_{\gr} \M_{\kappa}(K)(\ol\gamma)$ (\cite[Proposition 5.1]{Roozbeh_Lia_Ultramatricial} has more details). The paths from $P^{v_0}$ corresponds to the standard matrix units in the first column and $v_0$ corresponds to $e_{00}$.

If $\mu_k$ is the cardinality of $\ol\gamma^{-1}(k),$ the list of cardinals $\mu_0,  \mu_1, \mu_2,\ldots$, is such that $\mu_0=1$ (because $p_0=v_0$ is the only element of $P^{v_0}$ of length zero) and if $\mu_k\neq 0$ for some $k>0,$ then $\mu_{i}\neq 0$ for all $i\leq k$ (because the suffices of a path of length $k$ have lengths $i\leq k$). One can generalize the proof of \cite[Proposition 3.2]{Lia_LPA_realization} (given for $\kappa$ finite) to show the converse: for every $\kappa$ and a list of cardinals $\mu_k, k\in \Zset^+$ with the above two properties, there is a cofinal graph, which we denote by $E_{\can}$, such that that $L_K(E_{\can})\cong_{\gr}\M_\kappa(K)(\ol\gamma)$ where $\gamma$ is such that $|\ol\gamma^{-1}(k)|=\mu_k$ for any $k\in \Zset^+.$

The graph $E_{\can}$ can be constructed by the following inductive process. Let  $E_0$ be a single vertex $v_{00},$ let $E_1$ be obtained by adding $\mu_1$ new vertices $v_{1\alpha}, \alpha\in \mu_1$ and $\mu_1$ new edges $e_{1\alpha}$ where $\so(e_{1\alpha})=v_{1\alpha}$ and $\ra(e_{1\alpha})=v_{00}$ for all $\alpha\in\mu_1.$ Let $E_2$ be obtained by adding $\mu_2$ new vertices $v_{2\alpha}, \alpha\in \mu_2$ and $\mu_2$ new edges originating at the new vertices and ending at $v_{10}.$ If  $E_k, k\in \omega,$ is the graph obtained by continuing this process, let $E_{\can}$ be the directed union of $E_k, k\in \omega.$ Note that all vertices of $E_{\can}$ except the sink $v_{00}$ are regular and they emit exactly one edge.

If there is $k$ such that $\mu_i=0$ for all $i>k$ and $\mu_k\neq 0,$ then $k$ is the {\em spine length} and the path $e_{k0}\ldots e_{20}e_{10}$ is the {\em spine} of $E_{\can}$. Otherwise, $E_{\can}$ has the spine of infinite length and the left-infinite path $\ldots e_{20}e_{10}$ is the spine of $E_{\can}.$ The $\mu_i-1$ edges $e_{i\alpha},$ $\alpha\in \mu_i-\{0\}$ are referred to as the {\em $(i-1)$-tails}. This terminology reflects the fact that the $i$-tails end at $v_{i0}.$

To make graphical representation of the tails clearer, we introduce the following abbreviation: if $v$ receives $k$ edges originating at sources, no matter whether $k$ is finite or infinite cardinal, we depict this as $\xymatrix{\bullet\ar[r]^{(k)}&\bullet^v}.$ So, a canonical form with infinite spine length  can be represented by the graph below.
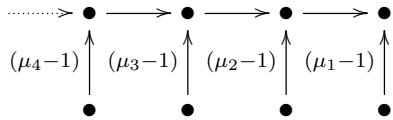
\[
\xymatrix{\ar@{.>}[r]&\bullet\ar[r]&\bullet\ar[r]&\bullet\ar[r]&\bullet\\&\bullet\ar[u]^{(\mu_4-1)}&\bullet\ar[u]^{(\mu_3-1)}&\bullet\ar[u]^{(\mu_2-1)}&\bullet\ar[u]^{(\mu_1-1)}}\]

If $E$ is a 1-S-NE graph and if its algebra has the matrix representation $\M_{12}(1,1+2,1+5,1+1),$ then our abbreviated graphical representation of $E_{\can}$ is the first graph below. Its  spine length is three.
As another example, the third graph below is a canonical form of the second graph.
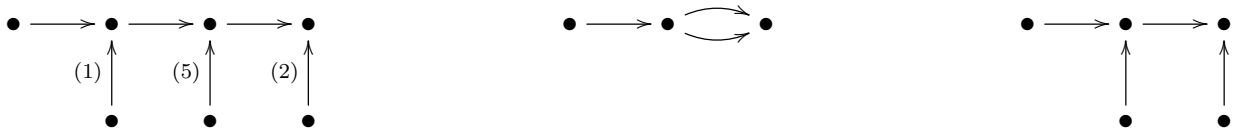
\[\xymatrix{\bullet\ar[r]&\bullet\ar[r]&\bullet\ar[r]&\bullet\\&\bullet\ar[u]^{(1)}&\bullet\ar[u]^{(5)}&\bullet\ar[u]^{(2)}}\hskip3cm\xymatrix{\bullet\ar[r]&\bullet\ar@/^/[r]\ar@/_/[r]&\bullet}\hskip3cm \xymatrix{\bullet\ar[r]&\bullet\ar[r]&\bullet\\&\bullet\ar[u]&\bullet\ar[u]}
\]
The notation $\xymatrix{\bullet\ar[r]^{(k)}&\bullet}$ is not to be confused with $\xymatrix{\bullet\ar[r]^k&\bullet}$ since the first notation indicates that there are $k$ sources each emitting a single edge to the terminal vertex and the second indicates that there are two vertices and the source vertex emits $k$ edges into the terminal vertex.

This construction enables one to come up with a way to associate a graph, unique up to a graph isomorphism, to all graphs with the algebras in the same graded $*$-isomorphism class:
if $E$ is such that $L_K(E)$ is graded $*$-isomorphic to $\M_{\kappa}(\ol\gamma),$ then we can repeat this construction using possibly different bijection
$\ol\gamma^{-1}(k)$ and paths of length $k$ in $E$ and obtain some $E_{\can}'.$ However, the only difference between $E_{\can}$ and $E_{\can}'$ is the labeling of the edges and vertices and, by construction, $E_{\can}'\cong E_{\can}.$ Because of this, a {\em canonical form} $E_{\can}$ is unique up to a graph isomorphism.

If a cofinal graph $E$ has a cycle $c$ of length $m>0$ without exits and $v_0=\so(c)$, let $P^{v_0}$
be the set of paths ending at $v_0$ and let
\begin{center}
$P^{v_0}_{\not c}$ be the set of paths ending in $v_0$ which do not contain the cycle $c.$
\end{center}
If $\kappa$ is the cardinality of $P^v_{\not c},$ we index the elements of $P^v_{\not c}$ by the elements of $\kappa$ so that $p_0=v_0$ is the path of length zero. Then, let $\ol\gamma: \kappa\to \Zset^+$ be the map $\alpha\mapsto |p_\alpha|.$
The correspondence mapping $p_\alpha  p_\beta^*\mapsto e_{\alpha\beta}$ extends to a graded $*$-isomorphism $L_K(E)\cong_{\gr}\M_\kappa(K[x^{m},x^{-m}])(\ol\gamma)$ (\cite[Proposition 5.1]{Roozbeh_Lia_Ultramatricial} has more details) which maps $v_0$ to $e_{00}$. The paths from $P_{\not c}^v$ corresponds to the standard matrix units in the first column.

Let $\mu_k$ be the cardinality of $\ol\gamma^{-1}(k)$ for $k\in m.$ The list of cardinals $\mu_0, \mu_1, \ldots, \mu_{m-1}$ is such that $\mu_k>0$ for each $k\in m$ since
there is a paths of length $k$ within the cycle $c.$ By generalizing the proof of
\cite[Proposition 3.4]{Lia_LPA_realization} to infinite cardinals, one can show the converse:
for any $\kappa,$ any integer $m>0$ and a list of nonzero cardinals $\mu_0, \mu_1,\ldots, \mu_{m-1}$ there is a cofinal graph $E_{\can}$ with a cycle without exits of length $m$ such that  $L_K(E_{\can})\cong_{\gr}\M_\kappa(K[x^m, x^{-m}])(\mu_0, \ldots, \mu_{m-1}).$

The graph $E_{\can}$ can be obtained as follows. Let $E_0$ be an isolated cycle $c$ of length $m$ and let $v_0, v_1,\ldots v_{m-1}$ be the vertices of $c$ listed in the order they appear in $c.$ The graph $E_{\can}$ is obtained by adding
$\mu_i-1$ new vertices $v_{i\alpha}$ and $\mu_i-1$ new edges $e_{i\alpha},$ $\alpha\in\mu_i$ such that $\so(e_{i\alpha})=v_{i\alpha}$ and $\ra(e_{i\alpha})=v_{i-_m1}$ for all $i\in m.$ Such graph $E_{\can}$ is a {\em canonical form} of any graph whose algebra
is graded $*$-isomorphic to the matrix algebra $\M_\kappa(K[x^m, x^{-m}])(\mu_0,\ldots, \mu_{m-1})$. It is unique up to a graph isomorphism and it can be represented as the graph below.
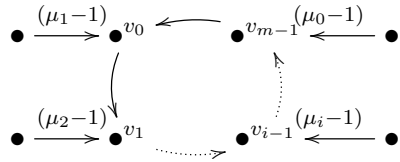
\[\xymatrix{  \bullet\ar[r]^{(\mu_1-1)}&
{\bullet}^{v_0}\ar@/_/[d]&
{\bullet}^{v_{m-1}}\ar@/_/[l]&  \bullet\ar[l]_{(\mu_0-1)}&\\
\bullet\ar[r]^{(\mu_2-1)}&
{\bullet}^{v_1}\ar@{.>}@/_/[r]&
{\bullet}^{v_{i-1}}\ar@{.>}@/_/[u]&  \bullet\ar[l]_{(\mu_{i}-1)}&} \]

For example, the second graph in Example \ref{example_canonical_representation} below is a canonical form of the first graph.

If $m>1,$ let $v_0, v_1, \ldots, v_{m-1}$ be the vertices listed in the order they appear in $c$ and let $c_i, i\in m,$ be the element of the equivalence class $[c]$ with $\so(c_i)=v_i.$
Consideration of $c_i$ instead of $c$ impacts the values of $\mu_k$ only up to the cyclic permutation given by $k\mapsto k+_m i.$ The change from $c$ to $c_i$ corresponds to the graded isomorphism on the matricial level mapping $e_{00}$ onto $e_{\alpha\alpha}$ for some $\alpha\in\kappa$ such that $\ol\gamma(\alpha)=i.$ The induced map on $\Zset^+[t,t^{-1}]/(t^m=1)$  corresponds to the map with $1\mapsto t^i.$

Let us consider an example illustrating the switch from $c=c_0$ to another $c_i\in [c]$.

\begin{example} Let $E$ and $F$ be the two graphs below and let $c$ and $c'$ be their cycles considering starting at $v_0$ and $v_0',$ respectively. 
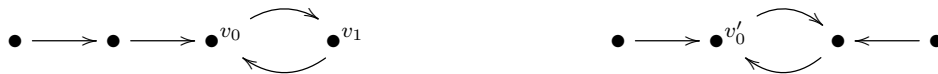
\[\xymatrix{\bullet\ar[r]& \bullet \ar[r]& \bullet^{v_0} \ar@/^1pc/[r] & \bullet^{v_1} \ar@/^1pc/ [l]}\hskip3cm\xymatrix{\bullet \ar[r]& \bullet^{v_0'} \ar@/^1pc/ [r] & \bullet \ar@/^1pc/ [l]&\bullet\ar[l] }\]
The paths in $P^{v_0}_{\not c}$ have lengths  0, 1, 1, and 2. If $c_1$ is the cycle based in $v_1,$ the paths in $P^{v_1}_{\not c_1}$ have the lengths, 0, 1, 2 and 3. Hence, both $\M_4(K[x^2, x^{-2}])(0,1,1,2)$ and $\M_4(K[x^2, x^{-2}])(0,1,2,3)$ are matricial representations of $L_K(E).$ The two representations are graded $*$-isomorphic by Lemma \ref{lemma_on_shifts} since, when the shifts are considered modulo $2$ and listed in non-decreasing order, the resulting list is $0,0,1,1$ in each case.
The change from $c$ to $c_1$ corresponds to mapping the unit $e_{00}$ to $e_{11}.$ This change induces $[e_{00}]\mapsto [e_{11}]=t[e_{00}]$ on the $\Gamma$-monoid level, so
this corresponds to the automorphism of $\Zset^+[t,t^{-1}]/(t^2=1)$ mapping 1 to $t$.

For the second graph, the lengths of paths in $P^{v_0'}_{\not c'}$ are $0,1,1,2$ and, considered modulo 2 and listed in non-decreasing order, we obtain $0,0,1,1.$ As this is the same list as the list for $E$, we have that $L_K(E)
\cong_{\gr}L_K(F)$ by Lemma \ref{lemma_on_shifts}. Mapping $P^{v_0}_{\not c}$ to $P^{v_0'}_{\not c'}$ bijectively and such that the lengths of paths remain the same modulo 2, induces such a graded isomorphism.

This example also illustrates that the Leavitt path algebras of two graphs can be graded isomorphic without the graphs being out-split equivalent (i.e obtained one from the other by finitely many out-splits and out-amalgamations).
\label{example_canonical_representation}
\end{example}

To integrate the terminology, we would like to unify the cases when a 1-S-NE graph has a sink and when it has a cycle by treating a path of zero length as a cycle of length zero. We say that such a cycle is an {\em improper cycle.} If $E$ is a 1-S-NE graph with the terminal cluster $c^0$ for a cycle $c$ (possibly improper) and if $m=|c|,$ then $m=0$ indicates that $E$ has a sink and $m>0$ indicates that $E$ has a cycle with no exits. If $v_0=\so(c)$ and $m=0,$ we let $P^{v_0}_{\not c}=P^{v_0}.$

We say that 1-S-NE graphs $E$ and $F$ are {\em 1-S-NE equivalent} and write $E\approx F$ if $E_{\can}\cong F_{\can}.$

Next, we show that the GCC holds for 1-S-NE graphs. This is known to hold for some special types of 1-S-NE graphs (the introduction has more details) but not for all 1-S-NE graphs and all underlying fields. In addition, condition (2) of the proposition below has not been previously considered together with (1) and (3).

\begin{proposition} {\bf The GCC holds for 1-S-NE graphs.}
The following conditions are equivalent for 1-S-NE graphs $E$ and $F.$
\begin{enumerate}[\upshape(1)]
\item There is a $\POM^D$-isomorphism $f:M_E^\Gamma\to M_F^\Gamma.$

\item $E\approx F.$

\item There is a graded $*$-isomorphism $\phi: L_K(E)\to L_K(F).$
\end{enumerate}
If $(1)$ holds, the isomorphism $\phi$ from condition $(3)$ can be found so that $f=\ol\phi.$
\label{proposition_cofinal_graphs}
\end{proposition}
\begin{proof}
Since both $L_K(E)$ and $L_K(E_{\can})$ are graded $*$-isomorphic to the same graded matrix algebra by the definition of $E_{\can},$ we have that
(2) $\Rightarrow$ (3). The implication (3) $\Rightarrow$ (1) is direct, so it remains to show (1) $\Rightarrow$ (2) and the last sentence of the theorem.

Assume that (1) holds and let $m$ be the length of the terminal cycle of $E$. If $m=0$,
then every nonzero element of $M_E^\Gamma$ is incomparable. So, every element of $M_F^\Gamma$ is incomparable, which implies that $F$ has a sink. If $m>0,$ then every nonzero element of $M_E^\Gamma$ is periodic with the period $m.$ The existence of $f$ implies that every nonzero element of $M_F^\Gamma$ is periodic with the period $m.$ Hence, $F$ has a cycle without exits of length $m.$

Let $\ol\mu$ and $\ol\mu'$ be the maps corresponding to the shifts of the matrix representations $\M_E$ and $\M_F$ and let $\phi_E$ and $\phi_F$ be graded $*$-isomorphisms of the Leavitt path algebras and their matrix representations. The map $\ol{\phi_F}f\ol{\phi_E}^{-1}$ is a $\POM^D$-isomorphism.
By Proposition \ref{proposition_matrix_rings_general}, the cardinalities $\kappa$ and $\mu_k, k\in \Zset^+$ from Proposition \ref{proposition_matrix_rings_general} match and there is $i\in m$ such that $\mu_k$ and $\mu'_{k+_mi}$ for $k\in \Zset^+$ have the same values modulo $m.$ If $i\neq 0,$ we can choose a different matrix representation of $L_K(F)$ so that the images of $\phi_E$ and $\phi_F$ is the same matrix algebra. If $\operatorname{id}$ is the identity map on this matrix algebra, then $\ol{\phi_F}f\ol{\phi_E}^{-1}=\ol{\operatorname{id}}.$  Thus,  $f=\ol{\phi_F^{-1}\operatorname{id}\phi_E},$ so we realized $f$ by a graded $*$-isomorphism. Since the two graph algebras have the same matrix algebra representation, we have that $E_{\can}\cong F_{\can}$ and so $E\approx F$ holds.
\end{proof}

\subsection{Relative canonical form}\label{subsection_relative_for_n=1}
If $E$ is a 1-S-NE graph, let $v_0, \ldots, v_{m-1}$ be the vertices of a terminal cycle $c.$ In addition to the sets
$P^{v_0}$ and $P^{v_0}_{\not c}$, we also consider the following set for $j\in m$ and a positive integer $k.$
\begin{enumerate}
\item[] $\Pa^{v_j}_{k}$ is the set of paths of length $k$ ending in $v_j$ which share no edge with the cycle $c.$
\[\mbox{ Let }\Pa^{v_j}=\bigcup_{0<k\in\omega} \Pa^{v_j}_k,\mbox{ and let $E^{v_j}$ be the subgraph generated by the vertices of the paths in $\Pa^{v_j}.$}\]
\end{enumerate}
We use calligraphic $\Pa$ instead of $P$ to highlight the difference between $P^{v_j}$ and $\Pa^{v_j}$: the first set contains all paths ending at $v_j$ while the second can be strictly smaller.
The sets $\Pa^{v_j}$ and $P^{v_j}_{\not c}$ can also be different. The graph $E^{v_j}$ is a 1-S-NE graph and $v_j$ is its sink.

Let $E=E_{\tot}$ be a 1-S-NE graph and $V\subseteq E^0.$ We define a graph $E_{\can, V}$ which is canonical only {\em relative} to the root of $V$ and we call it the {\em canonical form relative to $V$.} If $V=\emptyset,$ $E_{\can, V}=E_{\can}$ so this construction also presents a specific operation transforming a 1-S-NE graph $E$ to $E_{\can}.$
If $V$ contains a vertex of the terminal cycle $c$ or the sources of $\Pa_1^{v_j}$ paths for all $j\in m,$ then $E_{\can, V}=E.$

Our main application of this construction is ensuring that the paths connecting one cycle with the other in an $n$-S-NE graph are as short and direct as possible. In particular, we use this construction for the porcupine-quotient graph $\ol{c_{j+1}^0}/\ol{c_j}$ of an $n$-S-NE graph and its two cycles $c_j$ and $c_{j+1}$ such that $c_j$ emits exits to $c_{j+1}$ and with $V$ being the set of vertices which are not in $\ol{c_j^0}.$ The outcome of such constructions is a {\em direct-exits form} of the $(j+1)j$-connecting part of the graph.

We let $E_V$ be the subgraph of $E$ generated by $c^0\cup R(V).$ We define
$E_{\can, V}$ as a graph which contains $E_V$ as well as some  new vertices and edges. We introduce these new elements simultaneously with creating a bijection $\sigma$ of the set $P^{v_0}_{\not c}$ of $E$ and $P^{v_0}_{\not c}$ of $E_{\can, V}.$

{\bf The case $\mathbf{m>0}.$} For any $k\geq 0,$ we let $c_{(m-_mk)0}$ be the part of $c$ from $v_{m-_mk}$ to $v_0$ if $m-_mk\neq 0$ and $c_{(m-_mk)0}=v_0$ if $m-_mk=0.$ With this notation, we let $\sigma(c_{j0})=c_{j0}$ for $j\in m$ and  $\sigma(hc_{j0})=hc_{j0}$ for  $h\in \Pa_1^{v_j}$ such that $\so(h)\in R(V).$

For $k>0,$ and $p=gq\in \Pa_k^{v_j}$ where $g$ is an edge and $q\in \Pa_{k-1}^{v_j}$ is such that $\so(q)\notin R(V),$
we consider the cases $\so(g)\notin R(V)$ and $\so(g)\in R(V).$

If $\so(g)\notin R(V),$ let $v_g$ be a new vertex and $e_g$ be a new edge with the source $v_g$ and range $v_{j-_mk}.$
Thus, one can consider that the path $p$ of $E$ is replaced by the tail $e_g$ in
$E_{\can, V},$ so we let $\sigma(pc_{j0})=e_gc_{(j-_m|q|)0}.$ Note that the length of $pc_{j0}$ is $1+|q|+m-_m j$ and the length of $e_gc_{(j-_m|q|)0}$ is $1+m-_m (j-_m |q|)=1+m-_mj+_m|q|.$ So, $|\sigma(pc_{j0})|=|pc_{j0}| (\mymod m).$

If $\so(g)\in R(V)$, let $e_g$ be a new edge with the source $\so(g)$ and range $v_{j-_mk}$ and let $\sigma(p)$ be defined as in the previous case. Requiring that $\so(g)=\so(e_h)$ ensures that (E1) holds for the images of the map $\phi$ from the proof of Proposition \ref{proposition_canonical_tails}.

If $p=rgq$ is such that all vertices of $r$ are in $R(V),$ $g\in E^1$ is such that $\ra(g)\notin R(V)$ and $q\in \Pa_{k}^{v_j}$ for some $k$ and $j\in m,$ we let $\sigma(rgq)$ be $r\sigma(gq).$ With these definitions, the map $\sigma$ becomes defined on $P_{\not c}^{v_0}$ and it bijectively maps it to the set $P_{\not c}^{v_0}$ of $E_{\can, V}.$

For example, if $E$ is the first graph below and $V=\{v\},$ $E_{\can, V}$ is the second graph below.
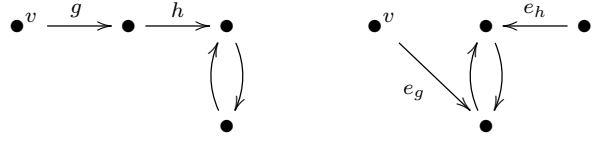
\[
\xymatrix{\bullet^v\ar[r]^g&\bullet\ar[r]^{h}&\bullet\ar@/^/[d]\\
&&\bullet\ar@/^/[u]}\hskip1.5cm
\xymatrix{\bullet^v\ar[rd]_{e_g}&\bullet\ar@/^/[d]&\bullet\ar[l]_{e_{h}}\\
&\bullet\ar@/^/[u]}\]

{\bf The case $\mathbf{m=0}.$} If there is a maximum of lengths of paths from a vertex in $V$ to $v_0,$ let us denote it with $k$. If there is no such maximum $k$, let $k=\omega.$

We let $e_{k-1}\ldots e_0$ ($\ldots e_1e_0$ if $k=\omega$) be a new path containing new edges and new vertices except its range $v_0$. We refer to such path as the {\em spine of $E$ relative to $V$}.
Pick a path $p_1$ in the set of paths in $\Pa_1^{v_0}$ with sources not in $R(V)$ and let $P_1$ be the remaining set of paths in $\Pa_1^{v_0}$ with sources not in $R(V)$, if any.
If there is a preference for one such path $p_1$ (see the last sentence of Proposition \ref{proposition_canonical_tails}), we can make a choice so this preference is satisfied.

For each $p\in P_1,$  we add a tail $e_p$ to $v_0$ (i.e. a new edge $e_p$ ending in $v_0$ and a new vertex $v_p$ as the source of $e_p$) and we let $\sigma(p)=e_p.$ We also let $\sigma(p_1)=e_0.$ If $p$ is a path in $\Pa_1^{v_0}$ such that $\so(p)\in R(V),$ then this path is also a path of $E_V$ and we let $\sigma(p)=p.$

If $k>1$ we continue this process by considering the set of paths $p=gh$ in $\Pa_2^{v_0}$
such that $\so(h)\notin R(V),$ we again consider the cases $\so(g)\notin R(V)$ and
$\so(g)\in R(V).$

If $\so(g)\notin R(V),$ then the set $P_2$ of all paths in $\Pa_2^{v_0}$ such that neither the source or the range of the first edge is in $R(H)$ is nonempty.
Let $p_2$ be arbitrary element of $P_2$. If there is a preference for one such path $p_2$, we can make a choice so this preference is satisfied.

Let $P_2'=P_2-\{p_2\}.$ If $P_2'$ is nonempty and if $p\in P_2',$ we add a new edge $e_g$ and its new source $v_g=\so(e_g)$ to be a tail to $\so(e_0),$ we let $\sigma(p)=e_ge_0,$ and we let $\sigma(p_2)=e_1e_0.$

If $\so(g)\in R(V),$ we add a new edge $e_g$ with $\so(g)$ as its source and $\so(e_0)$ as its range (recall that $e_0$ is the last edge of our newly added spine). In addition, we let $\sigma(p)=e_ge_0.$

If $k=1,$ then there are no paths $p=gh\in\Pa_2^{v_0}$ such that $\so(g)\in R(V)$ and $\ra(g)\notin R(V).$ So, in this case, it is sufficient to consider the case when $\so(g)\notin R(V)$ and this case is similar to the consideration of the same case if $k>1$:
the condition  $\so(g)\notin R(V)$ rules out the possibility $\so(h)\in R(V),$ so we have that $\so(h)\notin R(V).$ In this case, we let $\sigma(p)=e_ge_0$ for a new edge $e_g$ starting at a new source $v_g$ and ending in $\so(e_0).$

We continue this process by considering $\Pa_3^{v_0}$ if needed. Eventually, the map $\sigma$ becomes defined on the entire set $P^{v_0}$ and it bijectively maps it to the set $P^{v_0}$ of $E_{\can, V}.$

For example, let $E$ be the first graph below so that $E$ consists of paths $p_n$ of length $n$ ending at $v_0$ for every $n$ such that $p_n$ and $p_l$ have no common vertices or edges except their range $v_0$ if $n\neq l$. If $\so(p_n)=v_n$ for $n\in \omega,$ let $V=\{v_1,v_2,\ldots \}$ be the set of the sources of $E.$ Then, $E_{\can, V}$ is the second graph below.
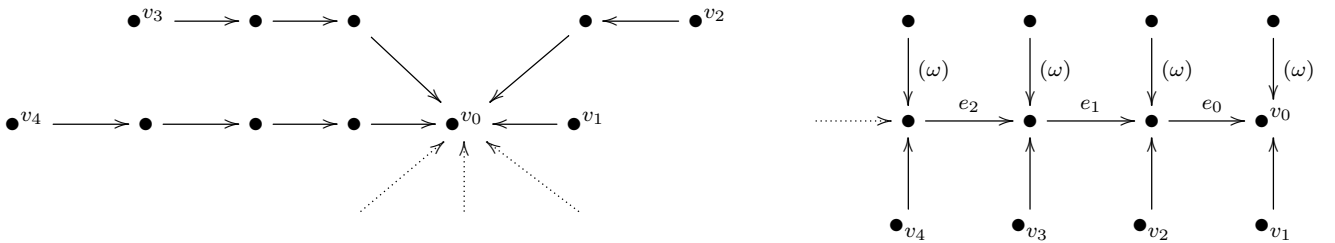
\[
\xymatrix{&\bullet^{v_3}\ar[r]&\bullet\ar[r]&\bullet\ar[dr]&&\bullet\ar[dl]&
\bullet^{v_2}\ar[l]\\\bullet^{v_4}\ar[r]&\bullet\ar[r]&\bullet\ar[r]&\bullet\ar[r]&\bullet^{v_0}&\ar[l]\bullet^{v_1}&\\&&&\ar@{.>}[ur]&\ar@{.>}[u]&\ar@{.>}[ul]&}\hskip.8cm
\xymatrix{&\bullet\ar[d]^{(\omega)}&\bullet\ar[d]^{(\omega)}&\bullet\ar[d]^{(\omega)}&\bullet\ar[d]^{(\omega)}\\\ar@{.>}[r]
&\bullet\ar[r]^{e_2}&\bullet\ar[r]^{e_1}&\bullet\ar[r]^{e_0}&\bullet^{v_0}\\
&\bullet_{v_4}\ar[u]&\bullet_{v_3}\ar[u]&\bullet_{v_2}\ar[u]&\bullet_{v_1}\ar[u]}\]

We show that the algebras of $E$ and $E_{\can, V}$ are  graded $*$-isomorphic.

\begin{proposition} {\bf Canonical path rearrangement.}
If $E=E_{\tot}$ is a 1-S-NE graph and $V\subseteq E^0,$ then there is an operation $E\to E_{\can, V}$ which extends to a graded $*$-isomorphism $L_K(E)\to L_K(E_{\can, V}).$

If $E, F,$ and $f$ are as in Proposition \ref{proposition_cofinal_graphs}, there are operations $\phi_E: E\to E_{\can},$ $\iota: E_{\can}\cong F_{\can},$ and $\phi_F: F\to F_{\can},$ and $f$ can be realized as $\ol{\phi_F^{-1}\iota\phi_E}=f.$

If $m=0$ and $f_k,f_{k-1},\ldots, f_1$ are edges not in the subgraph generated by $R(V),$ and such that the length of the path from $\ra(f_i)$ to the sink is $i-1,$ for $i=1,\ldots,k$, then we can have that the edges $f_k,f_{k-1},\ldots, f_1$ are mapped to the edges of the suffix of length $k$ of the spine of $E_{\can, V}.$
\label{proposition_canonical_tails}
\end{proposition}
\begin{proof}
Let us consider the case $m>0$ first. Keeping our previous notation, let $\sigma$ be the bijection introduced along with $E_{\can, V}$.
We define a map $\phi$ on  $E^0\cup E^1$ by mapping the vertices and edges of $E_V$ identically onto themselves. For $p=gq\in \Pa_k^{v_j},$ $g\in E^1,$ and $q\in \Pa_{k-1}^{v_j}$ such that $\so(q)\notin R(V)$, let
\begin{center}
$\phi(g)=e_gc_{(j-_mk)0}\sigma(qc_{j0})^*.\;$
If  $\so(g)\notin R(V),$ we also let $\phi(\so(g))=v_g.$
\end{center}
Defining $\phi(g^*)$ as $\phi(g)^*$ ensures that the algebra extension of $\phi$ will be a $*$-homomorphism. It is direct to check that (V) holds. To check (E1), let $g$ be an edge as above such that $\so(g)\notin R(V)$. If $k>2$ and $q=hr$ for $h\in E^1$, we have that
\[\phi(g)\phi(\ra(g))=e_gc_{(j-_mk)0}\sigma(qc_{j0})^*v_h=e_gc_{(j-_mk)0}\sigma(qc_{j0})^*=\phi(g)\]
where the middle equality holds since $\sigma(qc_{j0})$ is a path which originates at the source of the new edge with $v_h$ as its source, so $\sigma(qc_{j0})^*v_h=\sigma(qc_{j0})^*.$ If $k=2,$ then $q=h,$ $\sigma(qc_{j0})=qc_{j0}=hc_{j0}$ and
\[\phi(g)\phi(\ra(g))=e_gc_{(j-_mk)0}c_{j0}^*h^*\phi(\so(h))=e_gc_{(j-_mk)0}c_{j0}^*h^*\so(h)=e_gc_{(j-_mk)0}\sigma(hc_{j0})^*=\phi(g)\]
If $\so(g)\in R(V),$
the argument is very similar. The relation $\phi(\so(g))\phi(g)=\phi(g)$ is direct to check.

By the definition of $\phi$ on the ghost edges, (E2) holds since (E1) holds. If $g', k',$ and $j'$ are analogous to $g, k,$ and $j$ above, then (CK1) holds since
\[\sigma(q'c_{j'0})c_{(j'-_mk')0}^*e_{g'}^*e_gc_{(j-_mk)0}\sigma(qc_{j0})^*\]
is zero unless $g=g', j=j', k=k'$ and $q=q'$ and, in that case, the above expression is
\[\sigma(qc_{j0})c_{(j-_mk)0}^*e_{g}^*e_gc_{(j-_mk)0}\sigma(qc_{j0})^*=\sigma(qc_{j0})c_{(j-_mk)0}^*c_{(j-_mk)0}\sigma(qc_{j0})^*=\]\[\sigma(qc_{j0})
v_0\sigma(qc_{j0})^*=\so(\sigma(qc_{j0}))=\phi(\so(q))=\phi(\ra(g)).\]
If $v\in E^0$ is in $c^0\cup R(V)-V$ or it is a source of a path in $\Pa_1^{v_j},$ then it is direct to check that (CK2) holds. Otherwise, $v$ emits a single edge, say $g$ with its range not in $R(V)$. There is unique $k$ and $j$ and unique path $q\in \Pa_{k-1}^{v_j}$ with $\so(q)=\ra(g)$. If $\so(g)\in R(V)$ then $\phi(v)=v$ and $\phi(v)=v_g$ otherwise. In either case, $\phi(v)$ is a vertex of $E_{\can, V}$ which emits a single edge $e_g,$ and $\phi(v)=\phi(g)\phi(g)^*$ since
\[e_gc_{(j-_mk)0}\sigma(qc_{j0})^*\sigma(qc_{j0})c_{(j-_mk)0}^*e_{g}^*=
e_gc_{(j-_mk)0}v_0c_{(j-_mk)0}^*e_{g}^*=e_gv_{j-_mk}e_g^*=e_ge_g^*.\]

By the Universal Property, $\phi$ extends to a homomorphism which is graded and a $*$-homomorphism by the definition of $\phi$ on vertices, edges and ghost edges. Using the inverse of $\sigma,$ we can obtain the inverse of $\phi$, so $\phi$ is a graded $*$-isomorphism.

If $m=0,$ we define a map $\phi$ on $E^0\cup E^1$ by mapping the vertices and edges of $E_V$ identically on themselves. So, for $p=gq\in \Pa_l^{v_0}$ with $0<l\leq k,$  $g\in E^1,$ and $q\in \Pa_{l-1}^{v_0}$, it is sufficient to consider the case $\so(q)\notin R(V)$. In this case, we let
\begin{center}
$\phi(g)=e_ge_{l-1}e_{l-2}\ldots e_0\sigma(q)^*.$ If $\so(g)\notin R(V),$ we also let
$\phi(\so(g))=v_g.$
\end{center}
By this definition,
$\phi(g)\phi(\ra(g))=e_ge_{l-1}e_{l-2}\ldots e_0\sigma(q)^*v_q=e_ge_{l-1}e_{l-2}\ldots e_0\sigma(q)^*=\phi(g)$ and checking that (E1) holds in other cases is direct. With
$\phi(g^*)=\phi(g)^*$ and (E1) holding, (E2) holds. Checking (CK1) and (CK2) is similar to the $m>0$ case. The argument that $\phi$ extends to a graded $*$-isomorphism is also the same as in the $m>0$ case.

If $E, F,$ and $f$ are as specified, the previous part of the proposition establishes the existence of graph operations $\phi_E: E\to E_{\can}$  and $\phi_F: F\to F_{\can}$ which extend to graded $*$-isomorphisms of the algebras. The existence of $f$ and the proof of Proposition \ref{proposition_cofinal_graphs} implies that there is $\iota: E_{\can}\cong F_{\can}$ such that $\ol{\phi_F}f\ol\phi_E^{-1}=\ol\iota$. Hence, $f=\ol{\phi_F^{-1}\iota\phi_E}$

The last sentence of the proposition holds by our ability to choose paths $p_1, p_2,\ldots, p_k$ so that their first edges are $f_k, f_{k-1}, \ldots, f_1$ respectively when forming the new spine.
\end{proof}

\section{2-S-NE graphs}\label{section_n=2}

In this section, we define a canonical form of a countable 2-S-NE graph and prove Theorem \ref{theorem_n=2}, the main result for this class of graphs.

If $E$ is a 2-S-NE graph, there is an admissible pair $(H, S)$ such that $(\emptyset, \emptyset)\leq (H, S)\leq (E^0, \emptyset)$ is a composition series of $E$. Since $E/(H,S)$ is cofinal, $S=B_H.$ As $H$ is nontrivial and $P_{(H,S)}$ is cofinal, $S=\emptyset$ and so $B_H=\emptyset.$ The set $H$ contains no infinite emitters as $H$ is the saturated closure of a sink or a cycle without exits.
Since $E/H$ is cofinal, it has at most one sink, so $E^0-H$ contains at most one infinite emitter $v$. Since $E/H$ is row-finite, $v$ does not emit infinitely many edges to $E^0-H$. As $B_H=\emptyset$, $v$ emits zero edges to $E^0-H,$ so $v$ emits all the edges it emits to $H.$ As both $P_H$ and $E/H$ are cofinal,
$E$ cannot have more than one infinite emitter, more than two cycles, or more than an infinite emitter and a cycle (the existence of any such elements would imply that the composition series is longer than two). The graph $E$ has either one or two terminal clusters.

\subsection{2-S-NE graphs with two terminal clusters -- the easy case}
If there are two terminal clusters $c_1^0$ and $c_2^0$, then $E$ does not have an infinite emitter or a cycle with exits. The algebra $L_K(E)$ is graded $*$-isomorphic to the direct sum of $I(\ol{c_1^0} )$ and $I(\ol{c_2^0})$. Say that $H=\ol{c_1^0}$ so that $I(\ol{c_2^0})\cong_{\gr} L_K(E/H).$  The total out-split $E_{\tot}$ in this case consists of two disconnected 1-S-NE graphs
whose algebras are graded $*$-isomorphic to $I(\ol{c_1^0})$ and $I(\ol{c_2^0})$, so we can consider $E_{\tot}$ instead of $E$. For example, if $E$ is the graph
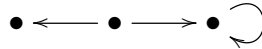
\[\xymatrix{\bullet&\bullet\ar[l]\ar[r]&\bullet\ar@(ru,rd)}\]
then $E_{\tot}$
consists of two connected components $
\xymatrix{\bullet&\bullet\ar[l]}$ and $\xymatrix{\bullet\ar[r]&\bullet\ar@(ru,rd)}\;\;\;\;\;.$

If $E=E_{\tot}$ is a 2-S-NE graph with two terminal clusters and $E_1$ and $E_2$ are the two disconnected graphs forming $E$, we can consider the 1-S-NE canonical forms of $E_1$ and $E_2$ and we let $E_{\can}$ be the union of these two canonical forms. With this definition, if $E$ and $F$ are two
2-S-NE graphs with two terminal clusters and if
$F_{\can}=F_1\cup F_2$ and $F_1\cap F_2=\emptyset$, we define the relation $E\approx F$ by
$E_1\approx F_1$ and $E_2\approx F_2$ or $E_1\approx F_2$ and $E_2\approx F_1.$
Thus, we have that
\[E\approx F\;\;\;\Longleftrightarrow\;\;\;E_{\can}\cong F_{\can}.\]

\subsection{2-S-NE graphs with a unique terminal cluster} Having the easy case out of the way, we remain primarily interested in the case when $E$ has only one terminal cluster. In this case, $H$ is the saturated closure of that cluster and $E^0-H$ contains an infinite emitter or a cycle with exits to $H$.
We introduce some notation which we use for this type of graphs.
Let $c^0$ be the terminal cluster of $E/H$ and $d^0$ be the terminal cluster of $E$ (so $H=\ol{d^0}$). Recall that we allow the case that $c$ and $d$ are improper cycles so that this scenario encompasses the situation that $c^0$ is an infinite emitter or that $d^0$ is a sink.
Since $d^0$ is the only terminal cluster of $E$, $c^0\subseteq R(H)$ (otherwise $c^0$ would be the terminal cluster of $E$ also), so there are paths from $c^0$ to $d^0$. We say that $p$ is a {\em $c$-to-$d$ path} if the source of $p$ is in $c^0,$ the range in $d^0$ and no edge of $p$ is on $c$ or $d$.
The examples with the Toeplitz graph from the introduction  illustrates that graphs whose algebras are graded $*$-isomorphic may have different number of $c$-to-$d$ paths of certain length.

Let $n=|d|$ and $m=|c|.$ If $m>0,$ then $c$ is a proper cycle with at least one exit to $H$ and such that all its exits have ranges in $H.$ The cycle $c$ is without exits in $E/H$ and $E$ is row-finite.  In this case, $E$ is a {\em cycle-to-cycle} graph if $n>0$ and a {\em cycle-to-sink} graph if $n=0.$ If $m=0,$ $c$ is an improper cycle and $v_0=\so(c)$ is an infinite emitter of $E$ which emits all of its edges to $H$ and which is a sink in $E/H.$ In this case, $E$ is an {\em infinite-emitter-to-cycle} graph if $n>0$ and an {\em infinite-emitter-to-sink} graph if $n=0.$ If $m=0,$ we say that the edges which $v_0$ emits are {\em exits}. This enables us to unify the terminology in the $m=0$ and $m>0$ cases.

For example, the first graph below is cycle-to-cycle ($m>0$ and $n>0$), the second (the Toeplitz graph) is cycle-to-sink ($m>0$ and $n=0$), the third is infinite-emitter-to-cycle ($m=0$ and $n>0$),   and the fourth is infinite-emitter-to-sink ($m=n=0$).
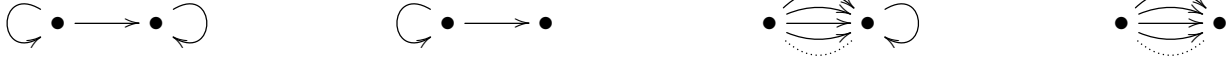
\[
\xymatrix{\bullet\ar@(lu,ld)\ar[r]&\bullet\ar@(ru,rd)}\hskip3.4cm
\xymatrix{\bullet\ar@(lu,ld)\ar[r]&\bullet}\hskip2.5cm\xymatrix{\bullet
\ar@{.} @/_1pc/ [r] _{\mbox{ } } \ar@/_/ [r] \ar [r] \ar@/^/ [r] \ar@/^1pc/ [r] &\bullet\ar@(ru,rd)}\hskip2.9cm\xymatrix{\bullet
\ar@{.} @/_1pc/ [r] _{\mbox{ } } \ar@/_/ [r] \ar [r] \ar@/^/ [r] \ar@/^1pc/ [r] &\bullet}\]

For the rest of this section, let $E=E_{\tot}$ be a 2-S-NE graph with a single terminal cluster, $H$ be its nonempty and proper hereditary and saturated set,  $m=|c|$ be the length of the terminal cycle $c$ of the 1-S-NE invariant of $E/H,$ and $n=|d|$ be the length of the terminal cycle $d$ of the 1-S-NE graph $P_H.$ Since $E=E_{\tot},$ every vertex which not terminal for neither $E$ nor $E/H$ emits exactly one edge and any non-terminal vertex of $E/H$ (if any) emits its only emitted edge to $E^0-H.$

We start by a series of operations on 2-S-NE graphs.  Some of these operations are out-splits and out-amalgamations, so they belong to the class characterized as the ``graph moves'' of symbolic dynamics. However, some of the operations we consider are not ``moves'' in the sense used in the current literature. All of these operations have properties described in section \ref{subsection_2SNE_intro}, so they induce a graded $*$-isomorphism of the graph algebras. The process from section \ref{subsection_relative_for_n=1} of transforming a 1-S-NE graph into its canonical form is an example of such an operation for 1-S-NE graphs.

\subsection{Direct-exit forms and connecting matrix}
Let $E=E_{\tot}$ be such that $d$ is a proper cycle (so $n>0$). Consider the porcupine graph $P_H$ and let $V$ be the set of all vertices of $P_H$ which are not in $H.$ Thus, $V=R(V).$ We define a 2-S-NE graph $E_{\direct}$ which we call the {\em direct-exit form} of $E$. We let $E_{\direct}$ consists of the part of $(P_H)_{\can, V}$ (defined in section \ref{subsection_relative_for_n=1}) outside of $R(c^0)$ and we let the rest of $E_{\direct}$ be $E/H.$ If $\phi_H$ is the graded $*$-isomorphism of $P_H$ and $(P_H)_{\can, V}$ (see section \ref{subsection_relative_for_n=1} and Proposition \ref{proposition_canonical_tails}), we define a map $\phi$ on the vertices and edges of $E$ by
\[\phi(v)=\phi_H(v)\mbox{ if } v\in H\;\;\mbox{ and }\;\;\phi(v)=v\mbox{ if } v\in E^0-H\mbox{ and }\]
\[\phi(e)=\phi_H(e)\mbox{ if } \ra(e)\in H\;\;\mbox{ and }\;\;\phi(e)=e\mbox{ if } \ra(e)\in E^0-H.\]
It is direct to check that the axioms hold for these images by the definition of $(P_H)_{\can, V}$ and $\phi_H$. Thus, $\phi$ extends to a graded $*$-homomorphism which is invertible since $\phi_H$ is invertible.

We point out some properties of $E_{\direct}.$ First, all $c$-to-$d$ paths have length one and the notation ``dir'' for ``direct'' emphasized that the $c$-to-$d$ paths are as direct as possible. Second, all edges with sources in $H-d^0$ end at vertices of $d.$ We refer to such edges ending at $w_i$ as the {\em $i$-tails} of $E_{\direct}$.

For example, the second graph is the direct-exit form of the first and the fourth graph is such form of the third graph below.
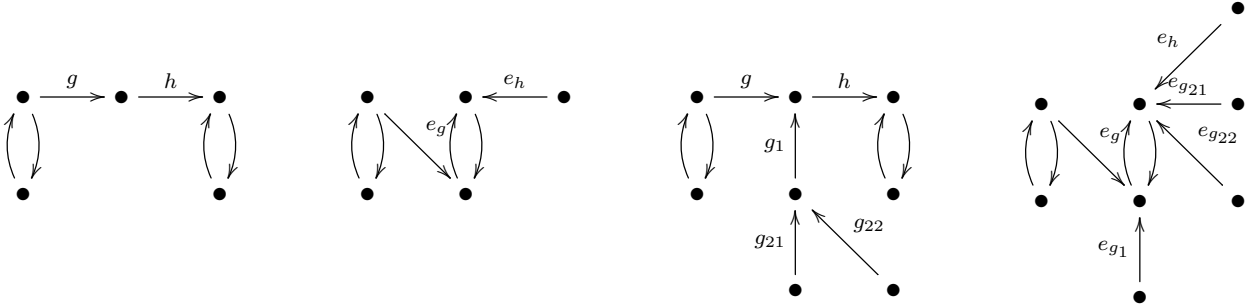
\[
\xymatrix{&&\\\bullet\ar@/^/[d]\ar[r]^g&\bullet\ar[r]^{h}&\bullet\ar@/^/[d]\\
\bullet\ar@/^/[u]&&\bullet\ar@/^/[u]}\hskip1.5cm
\xymatrix{&&\\\bullet\ar@/^/[d]\ar[rd]^{e_g}&\bullet\ar@/^/[d]&\bullet\ar[l]_{e_{h}}\\
\bullet\ar@/^/[u]&\bullet\ar@/^/[u]}\hskip1.3cm
\xymatrix{\\\bullet\ar@/^/[d]\ar[r]^g&\bullet\ar[r]^h&\bullet\ar@/^/[d]\\
\bullet\ar@/^/[u]&\bullet\ar[u]^{g_1}&\bullet\ar@/^/[u]\\&\bullet\ar[u]^{g_{21}}&\bullet\ar[ul]_{g_{22}}}\hskip1.5cm
\xymatrix{&&\bullet\ar[dl]_{e_h}\\\bullet\ar@/^/[d]\ar[dr]^{e_g}&\bullet\ar@/^/[d]&\bullet\ar[l]_{e_{g_{21}}}\\
\bullet\ar@/^/[u]&\bullet\ar@/^/[u]&\bullet\ar[ul]_{e_{g_{22}}}\\&\bullet\ar[u]^{e_{g_1}}&&}\]

Let us consider an example with $m=0.$ Let $E_1$ be the first graph below and let $E_3$ be the graph obtained by replacing the $c$-to-$d$ paths of $E_1$ by the paths of length three (the second graph below).
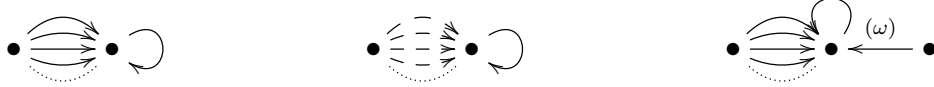
\[\xymatrix{\bullet\ar@{.} @/_1pc/ [r] _{\mbox{ } } \ar@/_/ [r] \ar [r] \ar@/^/ [r] \ar@/^1pc/ [r] &\bullet\ar@(ru,rd)}\hskip3cm \xymatrix{\bullet\ar@{.} @/_1pc/ [r] _{\mbox{ } } \ar@{-->}@/_/ [r] \ar@{-->} [r] \ar@{-->}@/^/ [r] \ar@{-->}@/^1pc/ [r] &\bullet\ar@(ru,rd)}\hskip3cm
\xymatrix{\bullet\ar@{.} @/_1pc/ [r] _{\mbox{ } } \ar@/_/ [r] \ar [r] \ar@/^/ [r] \ar@/^1pc/ [r] &\bullet\ar@(ur,ul)&\bullet\ar[l]_{(\omega)}}
\]
We have that $E_1=(E_1)_{\direct}$ and the last graph is the direct-exit form of $E_3.$ If $E_k$ is obtained by replacing the $c$-to-$d$ paths of $E_1$ by the paths of length $k>1$, the last graph is the direct-exit form of $E_k$ also. The Leavitt path algebras of $E_1$ and of $E_k, k>1$ are not graded $*$-isomorphic  because the first algebra is unital and the second is not. A direct argument for the algebras of $E_3$ and $E_k$ being graded $*$-isomorphic is that the two graphs have the same direct-exit form.

If $E$ is a direct-exit graph with $n>0$, we let $a_{ji},$ $j\in m, i\in n,$ be the number of edges $v_j$ emits to $w_i.$ If $m>0$ this number is finite and we refer to the $m\times n$ matrix $[a_{ji}]$ as the $c$-to-$d$ {\em connecting matrix}.
This matrix is dependent on the choice of $d\in [d]$ and $c\in [c],$ so the reference to $c$ and $d$ is needed. If $c'\in[c]$ is a different element, the $c'$-to-$d$ connecting matrix can be obtained by applying a degree $m$ cyclic permutation of rows of the original matrix. Similarly, if $d'\in [d]$ is a different element then the $c$-to-$d'$ connecting matrix is obtained by permuting the columns of the $c$-to-$d$ matrix by a cyclic permutation of length $n$.

If $m=0$, we allow $\omega$ to be the entry of a connecting matrix. In this case, the connecting matrix is a $1\times n$ matrix over $\omega\cup\{\omega\}$ such that $a_{0i}$ is the cardinality of the exits ending at $w_i.$

Let us move on to the $n=0$ case.  If $E=E_{\tot}$ has a sink,
we define $E_{\direct}$ analogously by letting $V=R(V)$ be the set of vertices of $P_H$ which are not in $H$ and pasting the graphs $(P_H)_{\can, V}$ and $E/H$ together with rearranging paths using Proposition \ref{proposition_canonical_tails} just as in the $n>0$ case.

We also let $H_t$ (where $t$ is for ``tails'') be the subgraph of $E_{\direct}$ generated by $H$ and $k_t$ be its spine length. The graph $H_t$ is a 1-S-NE graph and $w_0$ is its sink. We say that $H_t$ is the {\em tail graph}.
If $n>0,$ we define the tail graph analogously as the subgraph generated by vertices in $H$. Such a definition enables us to unify the arguments of both $n>0$ and $n=0$.

For example, if $E$ is the first graph below, the second graph is $E_{\direct}$ and the third graph is $H_t.$
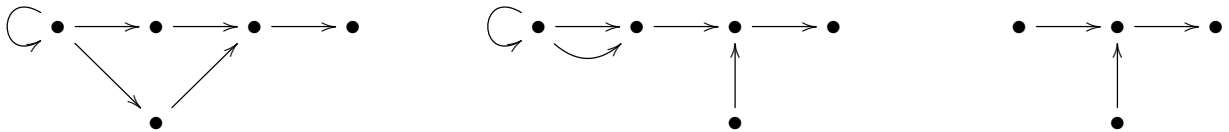
\[\xymatrix{\bullet\ar@(lu,ld)\ar[r]
\ar[rd]&\bullet\ar[r]&\bullet\ar[r]&\bullet\\&\bullet\ar[ur]&&}\hskip2cm\xymatrix{\bullet\ar@(lu,ld)\ar[r]
\ar@/_1pc/[r]&\bullet\ar[r]&\bullet\ar[r]&\bullet\\&&\bullet\ar[u]&}\hskip2cm\xymatrix{\bullet\ar[r]&\bullet\ar[r]&\bullet\\&\bullet\ar[u]&}\]

Let $w_i$ be the vertex on the spine of $H_t$ which is at length $i$ from $w_0.$ If the set of $i$-values such that $w_i$ receives an exit has the largest element $k,$ we say that $k$ is the {\em spine length of $E$} and the path from $w_k$ to $w_0$ is the {\em spine of $E$}. If such largest element does not exist, then $n=0$ necessarily and the tail graph has infinite spine length $k_t=\omega$. In this case, we let  $k=k_t$ and we let the spine of $E$ be the spine of the tail graph.
If $m>0$ and $n=0,$ $k$ is necessarily finite. In the example above, $k=k_t=2.$ Some of the following examples exhibit graphs with $k<k_t$ and some of graphs with $k=k_t=\omega.$

For a cycle-to-sink direct-exit graph with a spine of length $k$, we define its $c$-to-$d$ {\em connecting matrix} as the $m\times (k+1)$ matrix $[a_{ji}]$ where $a_{ji}$ is the number of edges $v_j$ emits to $w_i.$
It is dependent on the choice of $c\in [c]$ and a different choice of an element of $[c]$ results in a matrix with rows permuted by a cyclic permutation of degree $m$. For example, below is an $m=2, n=0$ direct-exit graph with the spine of length $k=3$ and the connecting matrix $
\left[\begin{array}{cccc}
1 & 0 & 4 & 2 \\
0 & 3 & 0 & 1
\end{array}\right].$
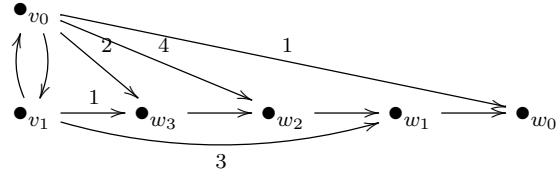
\[\xymatrix{\bullet_{v_0}\ar[dr]^2\ar[drr]^4\ar[drrrr]^1\ar@/^/[d]&&&\\
\bullet_{v_1}\ar[r]^1\ar@/^/[u]\ar@/_1pc/[rrr]_3&\bullet_{w_3}\ar[r]&\bullet_{w_2}\ar[r]&\bullet_{w_1}\ar[r]&\bullet_{w_0}}\]

If $n=m=0$ and if $k$ is the spine length (possibly $\omega$),
the connecting matrix is a  $1\times (k+1)$ matrix with the cardinality of the exits ending at $w_i$ at the $i$-th spot.
The first graph below is a direct-exit graph with a finite spine of length two and connecting matrix $[\omega\; 0\; \omega]$. The second  graph below is a direct-exit graph with the spine length infinite  and the connecting matrix $[\omega\; 0\; \omega\; 0\; \omega\; 0\;\ldots].$
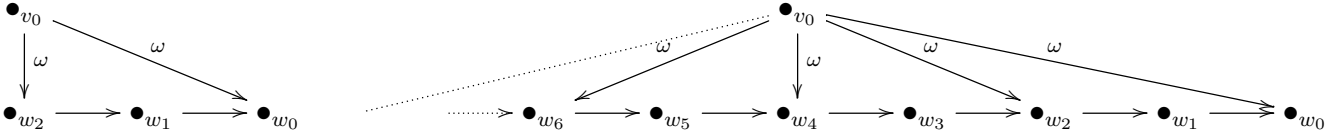
\[\xymatrix{\bullet_{v_0} \ar[drr]^\omega \ar[d]^\omega&\\
\bullet_{w_2}\ar[r]&\bullet_{w_1}\ar[r]&\bullet_{w_0}}\hskip.5cm
\xymatrix{&&&&\bullet_{v_0}\ar@{.}[dllll]\ar[dll]_\omega \ar[drr]^\omega \ar[drrrr]^\omega\ar[d]^\omega&&&\\
&\ar@{.>}[r]&\bullet_{w_6}\ar[r]&\bullet_{w_5}\ar[r]&\bullet_{w_4}\ar[r]&\bullet_{w_3}\ar[r]&\bullet_{w_2}\ar[r]&\bullet_{w_1}\ar[r]&\bullet_{w_0}}\]

If $E$ is the first graph below, its $c$-to-$d$ part is ``direct'' but the path originating at a source in $H$ and ending in $w_0$ is not ``canonical''. The second graph is $E_{\direct}.$
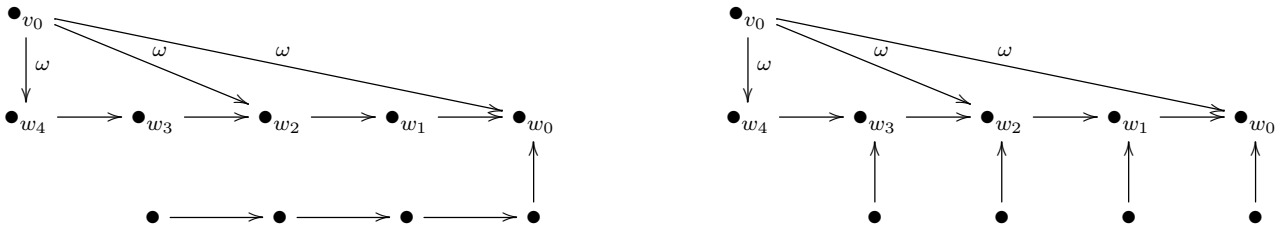
\[\xymatrix{\bullet_{v_0} \ar[drr]^\omega \ar[drrrr]^\omega\ar[d]^\omega&&&\\
\bullet_{w_4}\ar[r]&\bullet_{w_3}\ar[r]&\bullet_{w_2}\ar[r]&\bullet_{w_1}\ar[r]&\bullet_{w_0}\\
&\bullet\ar[r]&\bullet\ar[r]&\bullet\ar[r]&\bullet\ar[u]}\hskip2cm
\xymatrix{\bullet_{v_0} \ar[drr]^\omega \ar[drrrr]^\omega\ar[d]^\omega&&&\\
\bullet_{w_4}\ar[r]&\bullet_{w_3}\ar[r]&\bullet_{w_2}\ar[r]&\bullet_{w_1}\ar[r]&\bullet_{w_0}\\&\bullet\ar[u]&\bullet\ar[u]&\bullet\ar[u]&\bullet\ar[u]}\]
If the length of the path with all its vertices not in $T(v_0)$ which ends at $w_0$ were 7 instead of 4, the following graph would be $E_{\direct}$. This graph is also an example of a graph with $k=4<k_t=7.$
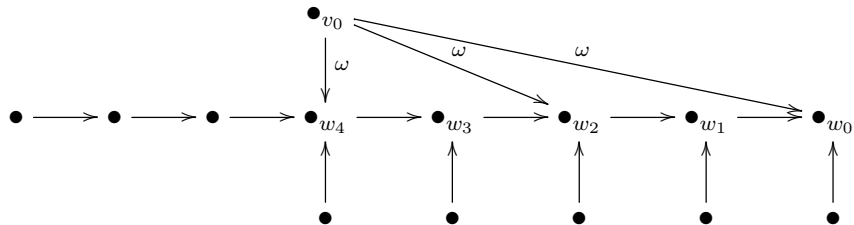
\[
\xymatrix{&&&\bullet_{v_0} \ar[drr]^\omega \ar[drrrr]^\omega\ar[d]^\omega&&&\\
\bullet\ar[r]&\bullet\ar[r]&\bullet\ar[r]&\bullet_{w_4}\ar[r]&\bullet_{w_3}\ar[r]&\bullet_{w_2}\ar[r]&\bullet_{w_1}\ar[r]&\bullet_{w_0}\\
&&&\bullet\ar[u]&\bullet\ar[u]&\bullet\ar[u]&\bullet\ar[u]&\bullet\ar[u]}\]

\subsection{Connecting polynomial}
Let $E=E_{\direct}$ be a direct-exit 2-S-NE graph with $m>0$ and let $k$ be the spine length in the case that $n=0.$ Let
\[a_{E}(t)=\sum_{j\in m, i\in n} a_{ji}\,t^{j+1+n-_ni}\mbox{ if $n>0\;\;$ and }\;\;a_{E}(t)=\sum_{j\in m, i\leq k} a_{ji}\,t^{j+1+i}\mbox{ if $n=0.$}\]
We refer to $a_E$ as a {\em connecting polynomial}.  It depends on  $c\in[c]$ and $d\in [d]$ but if it is clear which cycles we use, we will shorten the notation to $a_E.$ The following lemma focuses on this polynomial.

\begin{lemma}
Let $E$ be a direct-exit 2-S-NE graph with $m>0$ and
let $a_{E}$ be its connecting polynomial computed using $c$ and $d.$ If $v_0=\so(c)$ and  $l\in \Zset^+,$  then
\[[v_0]=t^{lm}[v_0]+\sum_{j=0}^{l-1}t^{jm}a_E[w_0].\]
\label{lemma_monoid_of_canonical_form}
\end{lemma}
\begin{proof}
We have that
$v_0\to t^mv_0+a_Ew_0\to t^{2m}v_0+t^ma_Ew_0+a_Ew_0\to\ldots\to t^{lm}v_0+\sum_{j=0}^{l-1} t^{jm}a_Ew_0 $ holds in $\F^\Gamma_E$ for any nonnegative integer $l,$ so $[v_0]=t^{lm}[v_0]+\sum_{j=0}^{l-1} t^{jm}a_E[w_0]$ holds in $M_E^\Gamma$.
\end{proof}

We consider the $m=0$ case next and prove a lemma analogous to Lemma \ref{lemma_monoid_of_canonical_form}. If $m=0$, the concept of a single connecting polynomial is replaced by polynomials defined for each nonempty and finite $Z\subseteq\so^{-1}(v_0)$ where $v_0$ is the infinite emitter. For such $Z$, let $P_Z$ be the set of paths in $P_{\not d}^{w_0}$ which have the first edge in $Z$ and let $a_Z\in \Zset^+[t]$ be
the polynomial $\sum_{p\in P_Z} t^{|p|}.$ We let $q_{\emptyset}$ denote $v_0$ to unify the treatment.
If we define the relation $\leq$ on the set of polynomials such that $a\leq b$ if $a+c=b$ for some polynomial $c,$ we have that $Z\subseteq W$ if and only if $a_Z\leq a_W.$

\begin{lemma}
Let $E=E_{\direct}$ be a 2-S-NE graph with $m=0$ and let $v_0$ be the infinite emitter of $E$.
For any finite $Z\subseteq \so^{-1}(v_0),$ $[v_0]=[q_Z]+a_Z[w_0].$

Let $b\in \Zset^+[t],$ $[x]\in M_E^\Gamma,$ and let $Z$ be finite subset of $\so^{-1}(v_0).$
If $[q_Z]=[x]+b[w_0]$ holds, then there is finite $W\subseteq \so^{-1}(v_0)$ disjoint from $Z$ such that $[x]=[q_{Z\cup W}]$ and $b=a_W[w_0].$
\label{lemma_monoid_of_canonical_form_m=0}
\end{lemma}
\begin{proof}
The first part holds since $v_0\to q_Z+a_Zw_0$ holds in $\F^\Gamma_E.$ To show the second part, assume that $[q_Z]=[x]+b[w_0]$ holds for some $b, [x],$ and $Z$ as in the statement of the lemma. This implies that $q_Z\to y$ and $x+bw_0\to y$ for some
$y\in \F_E^\Gamma$ by the Confluence Lemma (Lemma \ref{lemma_confluence}). The first relation implies that such $y$ can be chosen to have the form $y=q_{Z\cup W'}+a_{W'}t^{kn}w_0$ for some finite $W'$ disjoint from $Z$ and some nonnegative integer $k.$ Since $a_{W'}t^{kn}[w_0]=a_{W'}[w_0]$ and $w_0$ is a terminal vertex,
the relation $x+bw_0\to y$ implies that such $W'$ can be chosen so that $b[w_0]$ is a summand of $a_{W'}[w_0].$ As any summand of $a_{W'}[w_0]$ is of the form $a_{W}[w_0]$ for some finite subset $W$ of $W',$ we have that $b[w_0]=a_{W}[w_0]$ for some such $W.$ Thus, we have that
\[[x]+a_{W}[w_0]=[x]+b[w_0]=[q_Z]=[q_{Z\cup W}]+a_W[w_0].\]
By canceling $a_W[w_0],$ we obtain that $[x]=[q_{Z\cup W}].$
\end{proof}

\subsection{The exit moves}
Let $E$ be a direct-exit graph with $m>0$ and let $c\in[c]$ be such that $v_j\in c$ emits an exit $e.$ We consider the out-split with respect to $\{e\}$ and $\so^{-1}(v_j)-\{e\}$ followed by the total out-split of the new graph. We continue to use the label $v_j$ for the second new vertex in the out-split because this vertex is between $v_{j-_m1}$ and $v_{j+_m1}$ in the cycle with exits of the new graph which we continue to call $c.$ With these labels, the parts $E/H$ and $d$ are unchanged in the out-split graph. If $v_e$ is the second new vertex, it receives one new edge $g_e$ from $v_{j-_m1}$ and a copy of the subgraph generated by $R(v_j)-T(v_j)\cup\{v_j\}$ ending at $\so(g_e)$ instead of $v_j.$ The vertex $v_e$ emits the edge which we continue to call $e$ since it ends in $\ra(e)$. We refer to the resulting graph as the {\em $e$-blow-up} of $E.$

For example, if $E$ is the Toeplitz graph $\;\;\;\;\xymatrix{\bullet\ar@(lu,ld)\ar[r]^e&\bullet},$ then $\;\;\;\;\xymatrix{\bullet\ar@(lu,ld)\ar[r]&\bullet\ar[r]&\bullet}$
is the $e$-blow-up of $E.$ If $f$ is the newly added exit from the cycle, then
$\;\;\;\;\xymatrix{\bullet\ar@(lu,ld)\ar[r]&\bullet\ar[r]&\bullet\ar[r]&\bullet}$ is the $f$-blow-up of the previous graph. As another example, let $E$ be the first graph below. The second graph is the $e$-blow-up and the third graph is the $f$-blow-up.
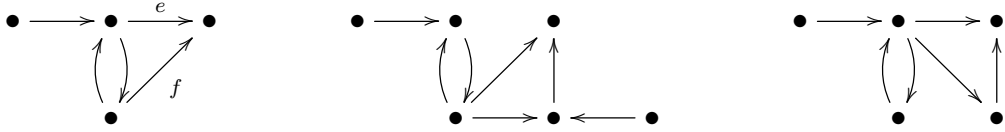
\[\xymatrix{\bullet\ar[r]&\bullet\ar@/^/[d]\ar[r]^e&\bullet\\
&\bullet\ar@/^/[u]\ar[ur]_f&}\hskip1.5cm\xymatrix{\bullet\ar[r]&\bullet\ar@/^/[d]&\bullet&\\&
\bullet\ar[r]\ar[ur]\ar@/^/[u]&\bullet\ar[u]&\bullet\ar[l]}\hskip1.5cm\xymatrix{\bullet\ar[r]&\bullet\ar[r]\ar@/^/[d]\ar[dr]&\bullet\\
&\bullet\ar@/^/[u]&\bullet\ar[u]}\]

If $m=0$ and $e$ is an edge $v_0$ emits, the {\em $e$-blow-up} is the graph obtained by the out-split with respect to $\{e\}$ and $\so^{-1}(v_0)-\{e\}$ and then considering the total out-split of the obtained graph. For example, if $E$ is the first graph below and $e$ is the first edge of the $c$-to-$d$ paths of length two, then the second graph is the $e$-blow-up of $E$.
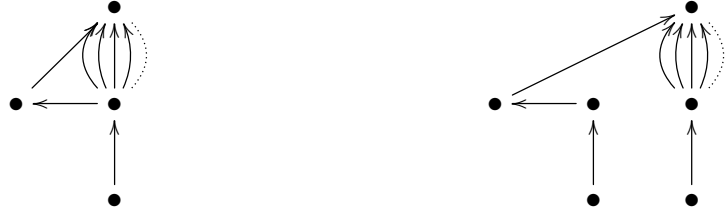
\[\xymatrix{&\bullet\\\bullet\ar[ur]&
\ar@{.} @/_1pc/ [u] _{\mbox{ } } \ar@/_/ [u] \ar [u] \ar@/^/ [u] \ar@/^1pc/ [u] \bullet\ar[l]\\&
\bullet\ar[u]}\hskip4cm \xymatrix{&&\bullet\\
\bullet\ar[urr]&\bullet\ar[l]&\ar@{.} @/_1pc/ [u] _{\mbox{ } } \ar@/_/ [u] \ar [u] \ar@/^/ [u] \ar@/^1pc/ [u] \bullet\\&\bullet\ar[u]&\bullet\ar[u]}\]

Performing a blow-up with respect to an exit $e$ results in a graph which may not be direct-exit any more. Let $E_e$ be the graph obtained by transforming the resulting graph to a direct-exit graph. We call such an operation an {\em exit move} and write $E\to_1 E_e$.

For example, if $E$ is the first graph below, the second graph is the blow-up with respect to the only exit and the third graph is the direct form of the second, hence, the exit move of $E$.

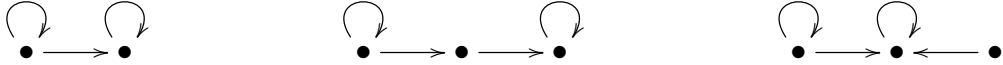
\[
\xymatrix{\bullet\ar[r]\ar@(lu,ru)&\bullet\ar@(lu,ru)}
\hskip2.7cm \xymatrix{\bullet\ar[r]\ar@(lu,ru)&\bullet\ar[r]&\bullet\ar@(lu,ru)}
\hskip2.7cm
\xymatrix{\bullet\ar[r]\ar@(lu,ru)&\bullet\ar@(lu,ru)&\bullet\ar[l]}\]

If $E\to_1 E_e$ and $E_e\to_1 F=(E_e)_g$ for some exit $g,$ we write $E\to_2 F$.
For example, if both exits of the first graph are moved, the second graph is obtained.
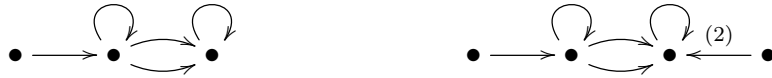
\[\xymatrix{\bullet\ar[r]&\bullet\ar@(lu,ru)\ar@/^/[r]
\ar@/_/[r]&\bullet\ar@(lu,ru)}\hskip3cm
\xymatrix{\bullet\ar[r]&\bullet\ar@(lu,ru)\ar@/^/[r]
\ar@/_/[r]&\bullet\ar@(lu,ru)&\bullet\ar[l]_{(2)}}\]

We consider the impact of an exit move to the number of tails and the connecting matrix. If $n>0,$ we let $l_i$ denote the cardinality of the set of $i$-tails for $i\in n.$ If $n=0$ and $k_t$ is the spine of the tail-graph $H_t,$ we let $l_i$ denote the cardinality of the set of $i$-tails for $i\leq k_t.$

An exit with source $v_j, j\in m$ and range $w_i$ is an {\em $ji$-exit} where $i\in n$ if $n>0$ and $i\leq k$ if $n=0$.
If $m>0, n>0$ and an $ji$-exit $e$ is moved, the only change in the connecting matrix is that the $a_{ji}$-value decreases by one and the $a_{(j-_m1)(i-_n1))}$-value increases by one. The cardinality $l_i$ increases by at least one (corresponding to $e$). To describe the increase in the number of $i'$-tails, we extend the definition of $\Pa_k^{v_j}$ to include the value $k=0$ and let this set be the moved exit $e.$ If $\Pa^{v_j}_k$ is nonempty and $i'\in n$ is such that $k+i'=i(\mymod n),$ then $w_{i'}$ gets $|\Pa^{v_j}_k|$ new tails. Thus, the number of $i'$-tails in the resulting graph is
\begin{equation}
l_{i'}+\sum_{k\in\{k\mid k=i-i'(\mymod n)\}}|\Pa^{v_j}_k|
\label{equation_i-tails_n>0}
\end{equation}
where the standard cardinal arithmetic holds if either $l_{i'}$ or $|\Pa^{v_j}_k|$ is infinite.

If $m>0$, $n=0,$ $k_t$ is the spine length of the tail graph $H_t$, $k_j$ is the spine length of $E^j,$ and a $ji$-exit is moved for some $j\in m, i\leq k,$
then $a_{ji}$ decreases by one and $a_{(j-_m1)(i+1)}$ increases by one. If $i=k,$ the length of the spine of the new graph increases by one. If $k_j>k_t,$ the length of the spine graph increases. The number of $i'$-tails in the new graphs is equal to
\begin{equation}
l_{i'}+|\Pa^{v_j}_{k'}|\;\;\mbox{  for } k'\mbox{ such that }k'=i'-i.
\label{equation_i-tails_n=0}
\end{equation}
Formula (\ref{equation_i-tails_n>0}) has $i-i'$ and formula (\ref{equation_i-tails_n=0}) has $i'-i$. This difference is present since the distance from $w_i$ to $w_0$ is $n-i$ if $n>0$  and it is  $i$ if $n=0.$

If $m=0$ and $n>0$ and if a $0i$-exit is moved, the $0i$-value of the connecting matrix of the resulting graph is $a_{0i}-1$ (here we use cardinal arithmetic to have that $\omega-1=\omega$ if $a_{0i}=\omega$). No other entries of the matrix are changed.
The number of $i'$-tails in the new graph is given by the same formula as in the $n>0$ and $m>0$ case except that the only possible value of $j$ is 0.

If $m=n=0$, $k$ is the spine length, and a $0i$-exit is moved, the $0i$-value of its connecting matrix is $a_{0i}-1$. The formula for the number of the $i'$-tails in the new graph is the same as in the $m>0$ and $n=0$ case except that the only possible $j$-value is zero.

Because $\Pa_0^{v_j}$ is nonempty, if $l_i$ is finite, then a move of any $ji$-exit increases the number of $i$-tails in the resulting graph. If $l_i$ is infinite, the move of an exit may produce a graph isomorphic to the original graph. Example \ref{example_moving_exits_invariant} below exhibits some such graphs.

\begin{example} Moving any of the exits of the two graphs below produces a graph isomorphic to the initial graph.
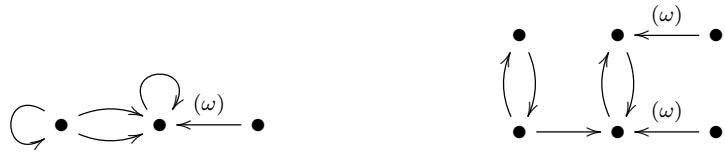
\[
\xymatrix{&&\\\bullet\ar@/^/[r]\ar@/_/[r]\ar@(lu,ld)&\bullet\ar@(lu,ru)&\bullet\ar[l]_{(\omega)}}
\hskip3cm
\xymatrix{\bullet\ar@/^/[d]&\bullet\ar@/^/[d]&\bullet\ar[l]_{(\omega)}\\
\bullet\ar[r]\ar@/^/[u]&\bullet\ar@/^/[u]&\bullet\ar[l]_{(\omega)}}
\]
\label{example_moving_exits_invariant}
\end{example}

We prove a short lemma we use for the cycle-to-cycle graphs.

\begin{lemma}
Let $E=E_{\direct}$ be a cycle-to-cycle graph with $|c|=m, |d|=n,$ and let $G$ be the greatest common divisor of $m$ and $n$. A $ji$-exit can be moved to become a $lk$-exit of the resulting graph if and only if $i-j=l-k (\mymod G).$
\label{lemma_feasibility_of_an_exit_move}
\end{lemma}
\begin{proof}
Let $m=Gm'$ and $n=Gn'$ for $m'$ and $n'$ which are mutually prime.

By the definition of an exit move, a $ji$-exit can be moved to become a $lk$-exit of the resulting graph if and only if there is $k'\in \Zset$ such that $j-k'=k(\mymod m)$ and $i-k'=l(\mymod n).$ Assuming this holds, let $k'=j-k+m_0m=i-l+n_0n$ for some $m_0, n_0\in \Zset.$ We have that $j-k=i-l+G(n_0n'-m_0m')$ which shows that $j-k=i-l(\mymod G)$ and so $i-j=l-k(\mymod G).$

Conversely, suppose that $i-j=l-k(\mymod G)$ so that $j-k=i-l(\mymod G)$ and let $j-k=i-l+k''G$ for some $k''\in\Zset.$
Since $m'$ and $n'$ are mutually prime, let $m_0,n_0\in \Zset$ be such that $1=m'm_0+n'n_0.$ Hence,  $k''G=k''Gm'm_0+k''Gn'n_0=k''mm_0+k''nn_0$ so that
$j-k=i-l+k''mm_0+k''nn_0$.
Let $k'=i-l+k''nn_0$ so that   $j-k=k'(\mymod m)$ and that $i-l=k'(\mymod n).$ Thus,
$j-k'=k(\mymod m)$ and $i-k'=l(\mymod n)$ showing that a $ji$-exit can be moved to become a $lk$-exit.
\end{proof}

\subsection{Reduction and reducibility}
We would like to introduce the inverse of an exit move which we refer to as {\em reduction}. Let us start with such consideration for cycle-to-cycle graphs.

Assume that $E$ is a cycle-to-cycle graph for which there are $j\in m$ and $i\in n$ such that $a_{(j-_m1)(i-_n1)}\neq 0$ and such that
\begin{equation}
l_{i'}\geq \sum_{k\in\{k\mid k=i-i'(\mymod n)\}}|\Pa^{v_j}_k|
\label{equation_reduceable_n>0}
\end{equation}
holds for every $i'\in n.$
In this case, there is an operation inverse to a move of a $ji$-exit and we say that $E$ is {\em $ji$-reducible}. Note that the relation  (\ref{equation_reduceable_n>0}) trivially holds for $i'$ such that $l_{i'}=\omega.$ We say that $E$ is {\em reducible} if there are $i\in n$ and $j\in m$ such that $E$ is $ji$-reducible.

If $E$ is $ji$-reducible and $[a_{ji}]$ is the connecting matrix of $E$ corresponding to $c\in [c]$ and $d\in [d],$ we define {\em a single $ji$-reduction} $E_{\red, 1, ji}$ of $E$ as follows. The quotient of $E_{\red, 1, ji}$ is the same as the quotient $E.$ The $c$-to-$d$ part of $E_{\red, 1, ji}$ is determined by the connecting matrix $a'_{j'i'}$ given by
\[a_{ji}'=a_{ji}+1, \;\;\;a_{(j-_m1)(i-_n1)}'=a_{(j-_m1)(i-_n1)}-1\mbox{ and }\]\[a'_{j'i'}=a_{j'i'}\mbox{ if }(j'\neq j\mbox{ or }i'\neq i)\mbox{ and }(j'\neq j-_m1\mbox{ or }i'\neq i-_n1).\]
The graph $E_{\red, 1, ji}$
has the number of $i'$-tails
equal to
\[
l_{i'}-\sum_{k\in\{k\mid k=i-i'(\mymod n)\}}|\Pa^{v_j}_k|
\]
if $l_{i'}$ is finite and it has $\omega$ $i'$-tails otherwise. If $l_{i'}$ is finite, the cardinal subtraction is subtraction of finite numbers because  if $l_{i'}$ is finite and  relation (\ref{equation_reduceable_n>0}) holds, then $|\Pa^{v_j}_k|$ is finite for any $k$ such that $k=i-i'(\mymod n).$ The scenario when both sides of relation (\ref{equation_reduceable_n>0})
are $\omega$ is considered in section \ref{subsection_tail_cutting}.

By the definition of $E_{\red, 1, ji},$ we have that $E_{\red, 1, ji}\to_1 E$ holds. Hence, the algebras of $E$ and $E_{\red, 1, ji}$ are in the same graded $*$-isomorphism class. Thus, if $E$ is $ji$-reducible, there is a graph $F$ such that $F\to_1 E$ is the move of a $ji$-exit. The converse also holds: if there is such a graph $F,$ then $a_{(j-_m1)(i-_n1)}\neq 0$ and relation (\ref{equation_reduceable_n>0}) holds.
We use this condition to define reducibility for S-NE graphs with composition length larger than two in section \ref{section_n>2}.

If $E_{\red, 1, ji}$ is $j'i'$-reducible, we let $E_{\red, 2, ji, j'i'}$ be the graph $(E_{\red, 1, ji})_{\red, 1, j'i'}.$ Continuing this argument, we define $E_{\red, l, j_1i_1,\ldots, j_li_l}$ as $(E_{\red, l-1, j_1i_1,\ldots, j_{l-1}i_{l-1}})_{\red, 1, j_li_l}.$ We write $E_{\red, l, ji}$  for the graph $E_{\red, l, (j-_m(l-1))(i-_n(l-1))\ldots, (j-_m1)(i-_n1), ji}$
where $E_{\red, l, (j-_m(l-1))(i-_n(l-1))\ldots, (j-_m1)(i-_n1), ji}\to_l E$ is the move of the same $ji$-exit $l$ times.

Let $E$ be a graph in which  $l_i$ is finite for at least one $i.$ If $E$ is not reducible, we say that $E$ is {\em reduced} and write $E=E_{\red}.$ If $E$ is reducible, the process of reduction terminates after finitely many steps producing a {\em reduction} $E_{\red}$ of $E$. If $E$ is a graph with $l_i=\omega$ for all $i\in n,$ we say that $E$ is $ji$-reduced for every $j\in m$ and $i\in n.$ We also say that $E$ is {\em reduced} and write $E=E_{\red}.$ Thus, the case $l_i=\omega$ for all $i\in n$ is the only case when a graph is {\em both reduced and reducible}.

We are interested in one particular type of a reduced graph -- a graph reduced only with respect to full revolutions of exits. If $m>0,$ $n>0,$ and if $L$ is the least common multiple of $m$ and $n$, let $N_{jii'}$ be the number of $i'$-tails added to $w_{i'}$ when a $ji$-exit is moved $L$-times. An explicit formula for $N_{jii'}$ can be obtained using formula (\ref{equation_i-tails_n>0}) but we do not display it since we do not use it.

We repeat the process of defining a reducible and a reduced graph, but by considering reduction of exits only for a full $L$-revolution instead of a single exit move. In particular, if $E$ is such that there are $j\in m$ and $i\in n$ such that $a_{ji}\neq 0$ and such that
$l_{i'}\geq N_{jii'}$
holds for every $i'\in n$, then $E$ is {\em $L,ji$-reducible}. We say that $E$ is {\em $L$-reducible} if it is $L, ji$-reducible for some $i\in n, j\in m.$

Let $E$ be a graph in which $l_i$ is finite for at least one $i.$ If $E$ is not $L$-reducible, then $E$ is {\em $L$-reduced}. If $E$ is $L$-reducible, the process of reduction terminates after finitely many steps and it produces an $L$-reduced graph
which we refer to as an
{\em $L$-reduction} of $E$. If $E$ is a graph with $l_i=\omega$ for all $i,$ we say that $E$ is {\em $L$-reduced}. We can unify the two cases by stating that $E$ is $L$-reduced if and only if for every $j\in m$ and $i\in n,$ if $E_{\red, L, ji}$ is defined, then $E\cong E_{\red, L, ji}.$

Let $E$ be the graph below where $l_0$ and $l_1$ countable cardinalities. This graph is reduced if and only if both $l_0$ and $l_1$ are in $\{0, \omega\}.$ This is because if $0<l_0<\omega,$ then $E$ can be $01$-reduced and, if $0<l_1<\omega$ then $E$ can be $00$-reduced. Two consecutive moves of any of the exits create one 0-tail and one 1-tail, so $E$ is 2-reduced if and only if at least one of $l_0$ and $l_1$ is in $\{0, \omega\}.$
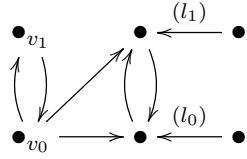
\[
\xymatrix{\bullet_{v_1}\ar@/^/[d]&\bullet\ar@/^/[d]&\bullet\ar[l]_{(l_1)}\\
\bullet_{v_0}\ar[r]\ar[ur]\ar@/^/[u]&\bullet\ar@/^/[u]&\bullet\ar[l]_{(l_0)}}
\]

Next, let $E$ be a graph with an infinite-emitter. If there is $F$ such that $F\to_1 E$ and $E$ is obtained by a move of a $0i$-exit, we say that $E$ is $0i$-reducible. We define a {\em $0i$-reduction,} {\em $0i$-reducibility,} and {\em being reduced} analogously as for $m>0$. Instead of $L$-reductions, we consider reductions with respect to $0i$-exits such that $a_{0i}=\omega$ and we call such reductions {\em $\omega$-reductions}. Thus, $E$ is $\omega$-reducible if there is a graph $F$ such that $F\to_1 E$ holds and it is a move of a $0i$-exit with $a_{0i}=\omega.$
We say that $E$ is {\em $\omega$-reduced} if $E$ is either not $\omega$-reducible or if it is isomorphic to any of its $\omega$-reductions.
For example, let $E$ be an infinite-emitter-to-sink graph with quotient consisting of a single vertex, with the infinite spine length, and such that $a_{0i}=a_{0i}'=\omega$ for all $0\neq i\in \omega$ and with $a_{00}<\omega.$
Then $E$ is reduced if any only if $l_i\in \{0, \omega\}$ for $i\geq 0$ and $E$ is $\omega$-reduced if and only if  $l_i\in \{0, \omega\}$ for $i>0.$

A cycle-to-sink graph $E$ is {\em $ji$-reducible} if $F\to_1 E$ for a graph $F$ such that $E$ is obtained by moving a $ji$-exit of $F$ where $j\in m$ and $i$ is less than or equal to the spine length of $F$.
We define the concepts of a {\em $ji$-reduction,} of {\em $ji$-reducibility} and of {\em being reduced} analogously as when $n>0$.

If $k$ is the spine length of  a cycle-to-sink graph $E$, and if all exits except those with range $w_k$ are moved until they end in $w_k$, then the spine length of the resulting graph is still $k$. We are interested in graphs reduced up to the exit moves which would increase the spine length. In particular, we say that $E$ is {\em spine-reducible} if there is $F$ and $l$ such that $F\to_l E$ and $F$ has shorter spine length.
If $E$ is not spine-reducible, we say that $E$ is {\em spine-reduced}.

For example, let us consider the class of graphs with their quotients being a single loop and with $H_t$ having the spine of infinite length with infinitely many $i$-tails for all $i\in \omega.$ Let $E_k$ be one such graph with the connecting matrix $[a_{0i}],$ $i=0, \ldots, k$  and with $k>0.$ Let $E_{k-1}$ be another such graph with the same quotient and the tail graph as $E_k$ and with the first $k-2$ entries of the connecting matrix the same, with the spine length $k-1$ and with $a_{0(k-1)}+a_{0k}$ at the $k-1$-spot of the connecting matrix. The graph $E_k$ is the result of moving $a_{0k}$-many of $0(k-1)$-exits of $E_{k-1}.$ We can continue this process to obtain $E_{k-2}$ so that $E_{k-2}\to_{a_{k-1}} E_{k-1}
$ and so on until $E_0$ is obtained. The spine length of  $E_0$ is zero, its connecting matrix is $\left[\sum_{i=0}^k a_{0i}\right]_{1\times 1},$ and $E_0$ is spine-reduced.

\subsection{Tail cutting}\label{subsection_tail_cutting}
Reducing a graph limits the number of tails if the number of tails obtained by an exit move is finite. To be able to control the infinite case also, we consider another type of the tail number minimization. The following example illustrates this.

\begin{example}
Let $E, E'$ and $F$ be the three graphs below. The graph $E'$ is obtained by moving the only exit of $E.$ The graph $E'$ is {\em also} obtained by moving the (only) exit of $F.$ All three graphs are reduced.
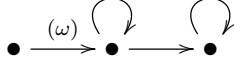
\[\xymatrix{\bullet\ar[r]^{(\omega)}&\bullet\ar[r]\ar@(lu,ru)&\bullet\ar@(lu,ru)&}\hskip2cm
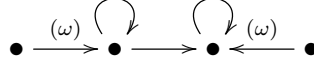
\xymatrix{\bullet\ar[r]^{(\omega)}&\bullet\ar[r]\ar@(lu,ru)&\bullet\ar@(lu,ru)&\bullet\ar[l]_{(\omega)}}
\hskip2cm
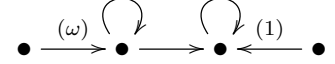
\xymatrix{\bullet\ar[r]^{(\omega)}&\bullet\ar[r]\ar@(lu,ru)&\bullet\ar@(lu,ru)&\bullet\ar[l]_{(1)}}\]

Let $v_0$ be the exit-emitter in all three graphs and let $w_0$ be the vertex of the terminal cycle. Let $\phi_E:E\to E'$ and $\phi_F: F\to E'$ be the exit moves. They both induce $\POM^D$-isomorphisms given by $[w_0]\mapsto [w_0]$ and $[v_0]\mapsto [v_0]+t[w_0]=[v_0]+[w_0]$ with the inverses such that $[v_0]\mapsto t[v_0]+t[w_0]=t[v_0]+[w_0].$ So, the $\POM^D$-isomorphism $f=\ol\phi_E^{-1}\ol\phi_F:M_F^\Gamma\to M_E^\Gamma$ has $[v_0]\mapsto [v_0]$ and, hence, it is the identity map.
If $g$ is the exit, $\{g_n\mid n\in \omega\}$ are the edges originating at sources  and ending in $v_0,$   $u_n=\so(g_n),$ $f_1$ is the 0-tail of $F,$ and $w_1=\so(f_1),$ then a graded $*$-isomorphism $\iota_{\cut}$ which realizes $f$ can be obtained by
\[f_1\mapsto g_0gd^*,\;\;\;\;g_n\mapsto g_n-g_ngg^*+g_{n+1}gg^*,\] by
$w_1=f_1f_1^*\mapsto \iota_{\cut}(f_1)\iota_{\cut}(f_1)^*=g_0gg^*g_0^*$ and
$u_n=g_ng_n^*\mapsto  \iota_{\cut}(g_n)\iota_{\cut}(g_n)^*=u_n-g_ngg^*g_n^*+g_{n+1}gg^*g_{n+1}^*$ and by mapping all other vertices and edges of $F$ identically onto themselves.
It is direct to check that such a correspondence extends to a graded $*$-monomorphism with the inverse $\iota_{\cut}^{-1}$ such that
\[g_0\mapsto g_0-g_0gg^*+f_1dg^*\;\mbox{ and }g_n\mapsto g_n-g_ngg^*+g_{n-1}gg^*\mbox{ for }n>0\]
and $u_0=g_0g_0^*\mapsto \iota_{\cut}^{-1}(g_0)\iota_{\cut}^{-1}(g_0)^*=u_0-g_0gg^*g_0^*+w_1$ and $u_n=g_ng_n^*\mapsto  \iota_{\cut}^{-1}(g_n)\iota_{\cut}^{-1}(g_n)^*=u_n-g_ngg^*g_n^*+g_{n-1}gg^*g_{n-1}^*$ for $n>0.$
\label{example_number_of_tails}
\end{example}

In examples like these,
if $E'$ is obtained from an exit move which increased the number of $i$-tails by $\omega$ for some $i,$ we would like to highlight the fact that $E$ is obtained from $E'$ by the inverse of such exit move which is decreasing the number of tails by the {\em maximal} possible number.
In contrast, the inverse of the exit move producing $F$ from $E'$ is not such. So, we would like $E$ to be the cut-form of both $F$ and $E'.$ Note that the graphs $E$ and $F$ have isomorphic quotients and equal connecting matrices, so only the number of tails differentiates them.

If graphs have a sink, the tail cutting process requires some additional consideration so we consider graphs with a sink first.
We let $E_{\cut}=E$ if $E/H$ is finite. Let $k_t$ be the length of the spine of the tail graph and let $C(i)$ stands for the following statement for $i>0$.
\begin{itemize}
\item[$C(i)$] There are $j,j'\in m, i'\in k_t,$ and $k',l\in\omega$ such that
$|\Pa^{v_j}_l|=\omega,$$a_{j'i'}\neq 0,$ and $l+k'|c|+d(j,j')+i'=i$
\end{itemize}
where $d(j,j')$ stands for the length of the shortest path in $c$ from $v_j$ to $v_j'.$ If $C(i)$ holds, we say that $i$-tails are {\em cuttable}. Note that here we are treating the edge of the spine from $w_{i+1}$ to $w_i$ as an $i$-tail. This will enable us to decrease the length of the spine of the tail graph if $(k_t-1)$-tails are cuttable.

If $m=0,$ the condition $C(i)$ simplifies to:
\begin{itemize}
\item[] There are $i'\in k_t,$ and $l\in\omega$ such that
$|\Pa^{v_0}_l|=\omega,\;\;$$a_{0i'}\neq 0,\;$ and $\;l+i'=i.$
\end{itemize}

The following example illustrates our approach to the type of the tail cutting which decreases the spine length of the tail graph for 2-S-NE graphs with a sink.
\begin{example}
Let $E_1$, $E_2$ and $E_{\omega}$ be the three graphs below. The tail graph is the only distinguishing element of the three graphs and their $\Gamma$-monoids are indistinguishable.
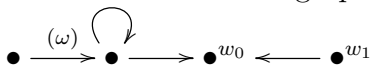
\[\xymatrix{\bullet\ar[r]^{(\omega)}&\bullet\ar[r]\ar@(lu,ru)&\bullet^{w_0}&\bullet^{w_1}\ar[l]}\hskip3cm
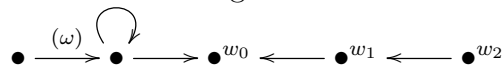
\xymatrix{\bullet\ar[r]^{(\omega)}&\bullet\ar[r]\ar@(lu,ru)&\bullet^{w_0}&\bullet^{w_1}\ar[l]&\bullet^{w_2}\ar[l]}\]
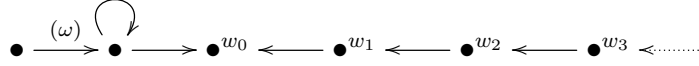
\[\xymatrix{\bullet\ar[r]^{(\omega)}&\bullet\ar[r]\ar@(lu,ru)&\bullet^{w_0}&\bullet^{w_1}\ar[l]&\bullet^{w_2}\ar[l]&\bullet^{w_3}\ar[l]&\ar@{.>}[l]}\]

The spine length of the quotient graphs is one. We consider one edge of the quotient graph $g_t$ to be the spine of the quotient graph and we let $u_t=\so(g_t).$ Let $u_n$ be the other sources of the quotient graph and $g_n$ be the edges for $n\in \omega$. Let $e$ be the loop and $g$ the exit in each graph and let $f_n$ be the edge which $w_n$ emits in each graph for $n=1$ in $E_1,$ and $n=1,2$ in $E_2,$ and $n>0$ in $E_{\omega}.$

In this case, all $i$-tails for $i>0$ are cuttable in $E_2$ and $E_\omega.$ These graphs have no tails for $i>0$ but we can shorten the tail graph spine. The porcupine graphs of all three graphs are 1-S-NE equivalent to the graph below and the existence of a bijection of the sets $P^{w_0}$ of these graphs enables us to define the tail cutting maps.
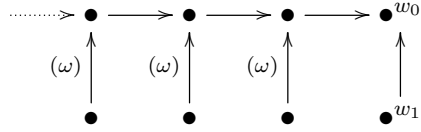
\[\xymatrix{\ar@{.>}[r]&\bullet\ar[r]&\bullet\ar[r]&\bullet\ar[r]&\bullet^{w_0}\\
&\bullet\ar[u]^{(\omega)}&\bullet\ar[u]^{(\omega)}&\bullet\ar[u]^{(\omega)}&\bullet^{w_1}\ar[u]}\]

We obtain one such bijection by matching $f_2f_1$ of $E_2$ with $g_0g$ of $E_1$ and
$g_ng$ of $E_2$ with $g_{n+1}g$ of $E_1$ for $n\geq 0.$
This
enables to define
the tail-cutting operation $\iota_{\cut}: E_2\to_{\cut} E_1$ by mapping the vertices $v_0, w_0, w_1,$ and $u_t$ and the edges $e,g, f_1,$ and $g_t$ identically on themselves and by mapping the edges $f_2$ and $g_n$ as follows.
\[f_2=f_2f_1f_1^*\mapsto g_0gf_1^* \;\mbox{ and }\;g_n\mapsto g_n-g_ngg^*+g_{n+1}gg^*\]
We let $w_2=f_2f_2^*\mapsto \iota_{\cut}(f_2)\iota_{\cut}(f_2)^*=g_0gg^*g_0^*$ and
$u_n=g_ng_n^*\mapsto \iota_{\cut}(g_n)\iota_{\cut}(g_n)^*=u_n-g_ngg^*g_n^*+g_{n+1}gg^*g_{n+1}^*.$ With these definitions, $\iota_{\cut}(g_ng)=g_{n+1}g$ for $n\geq 0.$
We define the inverse $\iota_{\cut}^{-1}: E_1\to_{\cut, -1} E_2$ also similarly as in Example \ref{example_number_of_tails} and it is such that
\[g_0\mapsto g_0-g_0gg^*+f_2f_1g^*,\;\;\;\;g_n\mapsto g_n-g_ngg^*+g_{n-1}gg^*\mbox{ for }n>0\]
which determines the images of $u_0$ and $u_n, n>0$ to be $u_0-g_0gg^*g_0^*+w_2$ and $u_n-g_ngg^*g_n^*+g_{n-1}gg^*g_{n-1}^*$ for $n>0$ in order for (CK2) to hold. We have that
$\iota_{\cut}^{-1}( g_0g)=f_2f_1$ and
$\iota_{\cut}^{-1}(g_ng)=g_{n-1}g$ for $n>0.$
The maps $\ol\iota_{\cut}$ and $\ol{\iota_{\cut}^{-1}}$ are identities since $[g_{n}gg^*g_{n}^*]=t^2[w_0]=[g_lgg^*g_l^*]$ for any $n, l\geq 0$.

Let us consider $E_\omega$ and $E_1$ and the bijection of paths ending at $w_0$ in the two graphs given by matching $f_2f_1$ with
$g_0g$, $f_3f_2f_1$ with $g_1eg,$ $f_4f_3f_2f_1$ with $g_2e^2g,$ and so on.
This bijection gives rise to the tail-cutting operation $\iota_{\cut}: E_\omega\to_{\cut} E_1$  which map the vertices $v_0, w_0, w_1$ and the edges $e,g, f_1$ identically on themselves and which maps the remaining edges by the formulas below. For those formulas, we let $e^0$ stand for $v_0$ and $e^{-n}$ stand for $(e^*)^n$ for $n>0$. We also shorten the notation of $e^{n}gg^*e^{-n}$ to $e_n$ for $n>0$ and we let $e_0=gg^*$ so that the idempotents $e_n$ are mutually orthogonal. With this notation, $\iota_{\cut}$ is defined so that
\[f_n=(f_n\ldots f_1)(f_{n-1}\ldots f_1)^*\mapsto g_{n-2}e^{n-2}g(g_{n-3}e^{n-3}g)^*=g_{n-2}ee_{n-3}g_{n-3}^*\;\;\mbox{ for }n>2,\]\[f_2=f_2f_1f_1^*\mapsto g_0gf_1^*,\;\mbox{ and }g_n\mapsto
g_n-\sum_{j=0}^{n}g_{n}e_j+\sum_{j=0}^{n}g_{n+1}e_j\mbox{ for }n\geq 0.
\]
In order for (CK2) to hold for the images, we let $w_n=f_nf_n^*\mapsto \iota_{\cut}(f_n)\iota_{\cut}(f_n)^*= g_{n-2}e_{n-2}g_{n-2}^*, n\geq 2$ and
$u_n=g_ng_n^*\mapsto\iota_{\cut}(g_n)\iota_{\cut}(g_n)^*=u_n-\sum_{j=0}^{n}g_{n}e_jg_n^*+\sum_{j=0}^{n}g_{n+1}e_jg_{n+1}^*.$ With these definitions, we have that $\iota_{\cut}(g_ne_l)=g_{n+1}e_l$ for $n\geq 0$ and $l\leq n$ and $\iota_{\cut}(g_ne_l)=g_{n}e_l$ for $l>n$. We aim to define the tail-creating operation $\iota_{\cut}^{-1}: E_1\to_{\cut, -1} E_\omega$, inverse to $\iota_{\cut},$ so that
$\iota_{\cut}^{-1}(g_ne_l)$ is $g_{n}e_l$ for $l>n,$ $g_{n-1}e_l$ for $l<n,$ and $f_{n+2}\ldots f_1g^*e^{-n}$ for $l=n.$ We achieve that by letting
\[g_0\mapsto g_0-g_0e_0+f_2f_1g^*,\;\; \;\;g_n\mapsto g_n-\sum_{j=0}^{n}g_{n}e_j+\sum_{j=0}^{n-1}g_{n-1}e_j+f_{n+2}\ldots f_1g^*e^{-n}\mbox{ for }n>0.\]
It is direct to check that the maps extend to mutually inverse graded $*$-homomorphisms.

If $E_{\omega,\omega}$
is the graph below, its algebra is still graded $*$-isomorphic to that of $E_1.$
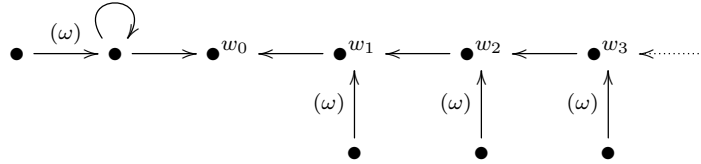
\[\xymatrix{\bullet\ar[r]^{(\omega)}&\bullet\ar[r]\ar@(lu,ru)&\bullet^{w_0}&\bullet^{w_1}\ar[l]&\bullet^{w_2}\ar[l]&\bullet^{w_3}\ar[l]&\ar@{.>}[l]\\&&&\bullet\ar[u]^{(\omega)}&\bullet\ar[u]^{(\omega)}&\bullet\ar[u]^{(\omega)}&}\]
We can choose to map $f_2f_1$ to $g_0g,$ $f_3f_2f_1$ to $g_2eg,$ $f_4f_3f_2f_1$ to $g_4e^2g$ and, continuing this trend,
$f_{l}\ldots f_1$ to $g_{2(l-2)}e^{l-2}g=g_{2l-4}e^{l-2}g$ for $l\geq 2.$
If $w_{(l,n)}$ are the sources of tails $f_{(l,n)}$ ending at $w_l$ for $l\geq 0,$
we can choose to map
$f_{(0,n)}f_1$ to $g_{2n}gf_1^*$ and
\[f_{(l,n)}f_l\ldots f_1\;\;\mbox{ to }\;\;g_{2(l-2)+2n+2}e^{l-1}g=g_{2l+2n-2}e^{l-1}g\mbox{ for }l\geq 2\mbox{ and }n\geq 0.\]
Let $\iota_{\cut}$ map $v_0, w_0, w_1, u_t, e, g, f_1,$ and $g_t$ identically onto themselves and let
\[f_l\mapsto g_{2l-4}ee_{l-3}g_{2l-6}^*\;\mbox{ for }l>2,\;\;f_2\mapsto g_0gf_1^*,\;\;f_{(l,n)}\mapsto g_{2l+2n-2}ee_{l-2}g_{2l-4}^*\;\mbox{ for }l\geq 2, n\geq 0\mbox{ and }\]
\[
g_l\mapsto
g_l-\sum_{j=0}^{\lfloor\frac{l}{2}\rfloor}g_{l}e_j+\sum_{j=0}^{\lfloor\frac{l}{2}\rfloor}g_{2l+1}e_j\mbox{ for }l\geq 0
\]
where $\lfloor\frac{l}{2}\rfloor$ is the floor function returning the value $\frac{l}{2}$ if $l$ is even and $\frac{l-1}{2}$ if $l$ is odd. As in the previous examples, we want (CK2) to hold for the images of the sources of the edges with already defined images, so we let $w_l\mapsto g_{2l-4}e_{l-2}g_{2l-4}^*$  for $l\geq 2,$ $w_{(l,n)}\mapsto g_{2l+2n-2}e_{l-1}g_{2l+2n-2}^*,$ and $u_l\mapsto u_l-\sum_{j=0}^{\lfloor\frac{l}{2}\rfloor}g_{l}e_jg_l^*+\sum_{j=0}^{\lfloor\frac{l}{2}\rfloor}g_{2l+1}e_jg_{2l+1}^*.$
It is direct to check that $\iota_{\cut}$ extends to a graded $*$-monomorphism. The inverse $\iota_{\cut}^{-1}$ can be obtained by considerations
of the $\iota_{\cut}$-images of the elements of the form $g_le_n$ for $l,n\geq 0$ and have that $\iota_{\cut}^{-1}$ maps $g_le_n$ to $g_le_n$ if $n>\lfloor\frac{l}{2}\rfloor,$ to $g_ke_n$ if $l=2k+1$ and $n\leq k,$ to $f_{(n+1, k-n)}f_{n+1}\ldots f_1g^*e^{-n}$ if $l=2k$ and $n<k,$ and to $f_{k+2}\ldots f_1g^*e^{-k}$ if $l=2k$ and $n=k.$ This enables us to define $\iota_{\cut}^{-1}(g_l)$ by the condition  $g_l=\iota_{\cut}^{-1}\iota_{\cut}(g_l)=\iota_{\cut}^{-1}(g_l-\sum_{j=0}^{\lfloor\frac{l}{2}\rfloor}g_{l}e_j+\sum_{j=0}^{\lfloor\frac{l}{2}\rfloor}g_{2l+1}e_j)$ so that
\[\iota_{\cut}^{-1}(g_l)=g_l+\sum_{j=0}^{\lfloor\frac{l}{2}\rfloor}\iota_{\cut}^{-1}(g_{l}e_j)-\sum_{j=0}^{\lfloor\frac{l}{2}\rfloor}\iota_{\cut}^{-1}(g_{2l+1}e_j).\]
Note that the edge $f_1$ cannot be cut from any of the graphs since the $0$-tails are not cuttable.
\label{example_shortening_tails}
\end{example}

The approach to obtaining the cut maps and their inverses in the previous example generalizes to arbitrary graphs. We continue to consider 2-S-NE graphs with one terminal cycle and with sinks.

\begin{definition}
Let $l_i$ be the number of $i$-tails of a 2-S-NE graph $E$ with one terminal cycle and a sink for $i<k_t.$ We define {\em the cut form $E_{\cut}$} of $E$
so that its quotient, $d$ and the $c$-to-$d$ part are the same as for $E$ and the number of tails is specified as follows. For $i<k,$ $E_{\cut}$ has $l_i$ $i$-tails if $i$-tails are not cuttable and it has zero $i$-tails otherwise. For the rest of the tails, if any, we consider the following cases.
\begin{enumerate}
\item There is no $k_0\in k_t, k_0\geq k$ such that $i$-tails are cuttable for all $i\in k_t, i\geq k_0.$ Then, the number of $i$-tails is $l_i$ if $i$-tails are not cuttable and zero otherwise.

\item There is $k_0\in k_t, k_0\geq k$ such that $i$-tails are cuttable for all $i\in k_t, i\geq k_0.$ The length of the spine graph of $E_{\cut}$ is $k_0$ in this case. For $ i<k_0,$ the number of $i$ tails is $l_i$ if $i$-tails are cuttable and zero otherwise.
\end{enumerate}

We say that a graph $E$ is {\em cut} or that it is in its {\em cut form} if $E=E_{\cut}.$
\label{definition_cut_form}
\end{definition}

For example, for $E_2,$ $E_\omega,$ and $E_{\omega,\omega}$ of the previous example, $k_0=1$ and the tail spine length of $E_2, E_\omega,$ and $E_{\omega,\omega}$ can be reduced to one.

\begin{proposition}
If $E$ is a graph with a sink, the map $\iota_{\cut}: E\to_{\cut} E_{\cut}$ extends to a graded $*$-isomorphism of the corresponding algebras.
\label{proposition_cut_form}
\end{proposition}
\begin{proof}
We fix some notation first. Since $E$ has a sink, $|d|=0,$ so we can use ``$n$'' for other values throughout the proof.

For $i\in k_t,$ let $l_i$ be the cardinality of $i$-tails, for $i\leq k_t,$ let $w_i$ be the vertices on the spine of $H_t$ emitting the edge $f_i$ on the spine and let $f_{(i, n)}$ be the $i$-tails for $n\in l_i$ and $w_{{(i,n)}}$ be their sources. Let $C\subseteq k_t$ be the set of all $i$ such that $i$-tails are cuttable.

Recall the definitions of the sets $\Pa^{v_j}_{k},\;j\in m$
and $\Pa^{v_j}=\bigcup_{0<k\in\omega} \Pa^{v_j}_k$
and graphs $E^{v_j}, j\in m,$
from section \ref{subsection_relative_for_n=1}. We can assume that $E$ is a graph such that $E^{v_j}$ is in the 1-S-NE canonical form for all $j\in m$ (we use this construction in the next section also).

For $j\in m$ if $m>0$ and $j=0$ if $m=0,$ let $k_j$ be the spine length of
$E^{v_{j}}.$ For $l\in k_j, l>0,$ there is a unique path $p_{j,l-1}$ of length $l-1$ ending at $v_j$ by the definition of 1-S-NE canonical form of a 1-S-NE graph with a sink. Let $g_{j,l,n}, n\in |\Pa^{v_j}_l|-1$ be the edge of
$E^{v_j}$ ending in $\so(p_{j, l-1})$ and originating at a source $u_{j, l, n}=\so(g_{j,l,n}).$ So,  $g_{j,l,n}, n\in |\Pa^{v_j}_l|-1$ are the tails to $\so(p_{j, l-1})$ and the remaining element of the set of the first edges of paths in $\Pa^{v_j}_l$ is on the spine of $E^{v_j}$.

Let $c_{j}$ be the element of $[c]$ starting at $v_{j},$ let $c_{jj'}$ be the shortest path from $v_{j}$ to $v_{j'},$ let  $g_{j'i'}$ be any $j'i'$ exit if $a_{j'i'}\neq 0,$ and let $d_{i}$ be the part of the tail graph spine from $w_{i}$ to $w_0.$

For $i\in C$, there are $j_i, l_i, k_i, j_i',$ and $i_i'$ such that $|\Pa_{l_i}^{v_{j_i}}|=\omega, a_{j_i'i_i'}\neq 0,$ and $l_i+k_i|c|+d(j_i, j_i')+i_i'=i.$
This enables us to map
$f_{i+1}d_{i}$ if $i\leq n_0$ in case (2) and
$f_{(i, n)}d_{i},$ in any case, to
\[g_{j_i, l_i, \sigma(n)}p_{l_i-1}c_{j_i}^{k_i}c_{j_ij'_i}g_{j'_ii'_i}d_{i'_i}\] where $\sigma(n)$ is appropriately chosen value we specify in the rest of the proof and which depends on cases (1) and (2) and on whether $l_i$ is finite or not. No matter our choice of $\sigma(n),$ the lengths of the corresponding paths match because $l_i+k_i|c|+d(j_i, j_i')+1+i_i'=i+1.$

Let us consider the case $m>0$
first. Let $JL=\{(j,l)\in m\times k_j\mid j=j_i, l=l_i$ for some $i\in C\}$ and, for $(j,l)\in JL,$ let \[C_{j,l}=\{i\in C\;\mid\; j=j_i,\; l=l_i\}.\]
Note that any of $C,$ $JL,$ and $C_{j,l}$ can be infinite.
For any $(j,l)\in JL,$ let us index the elements of $C_{j,l}$ as $i_{j,l,0}, i_{j,l,1},\ldots $ where the list ends after finitely many steps if the cardinality of $C_{j,l} $ is finite and it is infinite otherwise.
The indexing of $C_{j,l}$ enables us to assign a unique triple $(j,l,n)$ such that $(j,l)\in JL$ and $n\in |C_{j,l}|$ to any $i\in C$. We denote this correspondence by $i\mapsto (j_i, l_i, n_i)$ and its inverse by $(j,l,n)\mapsto i_{(j,l,n)}.$ Note that even when $C_{j,l}$ is infinite, for any $n\in\omega,$ the set \[C_{j,l,\leq n}=\{i\in C_{j,l}\mid n_i\leq n\}\] is finite. If $C_{j,l}$ is finite and $n\geq |C_{j,l}|$, then $C_{j,l, \leq n}=\{i_{j,l,0}, \ldots, i_{j,l,|C_{j,l}|-1}\}$. If $C_{j,l}$ is infinite,  $C_{j,l, \leq n}=\{i_{j,l,0}, \ldots, i_{j,l,n}\}$ for any $n\in\omega$.

We start to create the tail-cutting map $\iota_{\cut}: E\to_{\cut} E_{\cut}.$
Let us shorten
\begin{center}
$p_{l_{i}-1}c_{j_{i}}^{k_{i}}c_{j_{i}j'_{i}}g_{j'_{i}i'_{i}}d_{i'_i}$ to $q_i\;\;$ and $\;\;q_iq_i^*$ to $e_{i}.$
\end{center}
Since the quintuple $(l_i, j_i, k_i, j'_i, i'_i)$ is unique for $i\in C,$ the idempotents $e_i$ are  orthogonal to each other and $e_iq_i=q_i.$
With the introduced abbreviations, for $i\in C$ and $n'\in l_i,$ we define $\iota_{\cut}$ on paths of the form $f_{(i,n')}d_{i}$ by
\[f_{(i,n')}d_{i}\mapsto
\left\{
\begin{array}{ll}
g_{j_{i}, l_{i}, 2n_i+2n'}\;\;\,q_i&\mbox{ if (1) holds or if (2) holds and $i<k_0$ }\\
g_{j_{i}, l_{i}, 2n_i+2n'+2}q_i&\mbox{ if (2) holds and $i\geq k_0.$}
\end{array}
\right.\]
If (1) holds, then we map $d_i$ to $d_i$ for any $i\in k_t.$ If (2) holds, then we map $d_i$ to $d_i$ for $i< k_0$ and, for $i=k_0+n, n\geq 0,$ we treat $f_{k_0+n+1}=f_{i+1}$ as a $k_0+n$ tail and let
\[d_{i+1}=f_{i+1}d_i\mapsto g_{j_{i}, l_{i}, 2n_i}q_{i},\] so this agrees with the formula for the image of $f_{(i,n')}d_i$ in the sense that it explains the lag of ``+2'' in the formula for the image of $f_{(i, n')d_i}$ the case that (2) and $i\geq k_0$ hold.

We extend $\iota_{\cut}$ to $E^0\cup E^1$ by the formulas which ensure that the axioms hold as follows. We let $\iota_{\cut} $ be the identity on the vertices and edges of $c, d,$ any $c$-to-$d$ path,  any non-cuttable $i$-tails for $i\in k_t,$ the spines of $E^{v_j}$ for $j\in m$ and on the tails of $\so(p_{j, l-1})$ and their sources for all $j$ and $l$ such that $(j,l)\notin JL.$ For the remaining tails, we let
\[f_{(i,n')}=f_{(i,n')}d_id_i^*\mapsto \iota_{\cut}(f_{(i,n')}d_i)\iota_{\cut}(d_i)^*, \;\;\;
w_{(i,n')}=f_{(i,n')}f_{(i,n')}^*\mapsto \iota_{\cut}(f_{(i,n')})\iota_{\cut}(f_{(i,n')})^*\;\;\mbox{ for } n\in l_i\]
If (2) holds and $i\geq k_0,$   we let
$f_{i+1}=d_{i+1}d_{i}^*
\mapsto
\iota_{\cut}(d_{i+1})\iota_{\cut}(d_{i})^*$ and
\[w_{i+1}
\mapsto \iota_{\cut}(f_{i+1})\iota_{\cut}(f_{i+1})^*=
\iota_{\cut}(d_{i+1})\iota_{\cut}(d_{i})^*\iota_{\cut}(d_{i})
\iota_{\cut}(d_{i+1})^*=
\iota_{\cut}(d_{i+1})\iota_{\cut}(d_{i+1})^*.\] With these definitions, for $n\in\omega$,
\[\iota_{\cut}(f_{i+1})^*\iota_{\cut}(f_{i+1})=\iota_{\cut}(d_{i})\iota_{\cut}(d_{i+1})^*\iota_{\cut}(d_{i+1})\iota_{\cut}(d_{i})^*=\iota_{\cut}(d_{i})\iota_{\cut}(d_{i})^*=\iota_{\cut}(w_{i})\]
and one checks that
\[\iota_{\cut}(f_{(i, n')})^*\iota_{\cut}(f_{(i, n')})=\iota_{\cut}(d_i)\iota_{\cut}(f_{(i,n')}d_i)^*\iota_{\cut}(f_{(i,n')}d_i)\iota_{\cut}(d_i)^*=\iota_{\cut}(d_i)\iota_{\cut}(d_i)^*=\iota_{\cut}(w_{i}).\]

We define the images on the rest of the vertices and edges of $E^{v_j}$ next. For $n'\in \omega$ we consider the floor function $\lfloor \frac{n'}{2}\rfloor,$ returning the value $\frac{n'}{2}$ if $n'$ is even and $\frac{n'-1}{2}$ if $n'$ is odd (same as in Example \ref{example_shortening_tails}).
For $(j,l)\in JL$ and $n'>0,$ we let
\[g_{j,l, n'}\mapsto
g_{j,l, n'}-\sum_{i\in C_{j,l,\leq\lfloor\frac{n'}{2}\rfloor}}g_{j,l, n'}e_{i}
+\sum_{i\in C_{j,l, \leq \lfloor \frac{n'}{2}\rfloor}}g_{j,l, 2n'+1}e_{i},\] and $
u_{j,l, n'}=g_{j,l,n'}g_{j,l,n'}^*\mapsto
\iota_{\cut}(g_{j,l,n'})\iota_{\cut}(g_{j,l, n'})^*.
$
Defining $\iota_{\cut}(g^*)$ to be $\iota_{\cut}(g)^*$ ensures that the extension is a $*$-map. We have that $\iota_{\cut}(e_i)=e_i$ for every $i\in C.$ Thus, for $(j, l)\in JL$ and $i\in C$
\begin{center}
\begin{tabular}{lll}
$\iota_{\cut}(g_{j,l,n'}e_{i})$ is & 0 & if $i\notin C_{j,l}$ (and $g_{j,l,n'}e_{i}=0$ in this case also)\\
&$g_{j,l,n'}e_{i}$ &if $i\in C_{j,l}-C_{j,l,\leq\lfloor \frac{n'}{2}\rfloor}$\\
&
$g_{j,l,2n'+1}e_{i}$& if $i\in C_{j,l, \leq\lfloor\frac{n'}{2}\rfloor}$.
\end{tabular}
\end{center}

It is direct to check that the axioms hold so the above map extends to a graded $*$-homomorphism by the Universal Property. The extension is injective by the Graded Uniqueness Theorem.

We aim to define the inverse of $\iota_{\cut}$ next. First we consider the values of such inverse on the elements of the form $g_{j,l,n'}e_i$ for $(j,l)\in JL$ and $i\in C_{j,l}.$ For $i\in C_{j,l,\leq \lfloor \frac{2n'}{2}\rfloor}=C_{j,l, \leq n'},$ $i=i_{(j,k,n_i)}$ for some $n_i\leq n'.$ Hence, $n'=n_i+n$ for some $n\geq 0.$ This enables us to define $\iota_{\cut}^{-1}(g_{j,l,n'}e_{i})$ as
\begin{center}
\begin{tabular}{ll}
$g_{j,l,n'}e_{i}$ &if $i\in C_{j,l}-C_{j,l,\leq\lfloor \frac{n'}{2}\rfloor},$\\
$g_{j,l,n}\;e_{i}$ & if $n'=2n+1$ and $i\in C_{j,l,\leq \lfloor \frac{n'}{2}\rfloor},$\\
$f_{(i,n)}d_{i}q_i^*$ & if $i\in C_{j,l,\leq \lfloor \frac{n'}{2}\rfloor},$ $n'=2n_i+2n$ for $n\geq 0$ if (1) holds or (2) holds and $i<k_0,$\\
$f_{(i,n)}d_{i}q_i^*$ & if $i\in C_{j,l,\leq \lfloor \frac{n'}{2}\rfloor}$ and $n'=2n_i+2n+2$ for $n\geq 0$ if (2) holds and $i\geq k_0,$\\
$d_{i+1}q_i^*$ & if $i\in C_{j,l,\leq \lfloor \frac{n'}{2}\rfloor}$ and $n'=2n_i$ if (2) holds and $i\geq k_0.$
\end{tabular}
\end{center}
With $\iota_{\cut}^{-1}$ defined on $g_{(j,l, n')}e_i,$ we automatically have that
$g_{(j,l, n')}=\iota_{\cut}^{-1}(\iota_{\cut}(g_{(j,l, n'))})$ if we let
\[\iota_{\cut}^{-1}(g_{(j,l, n')})=g_{j,l, n'}+\sum_{i\in C_{j,l,\leq\lfloor\frac{n'}{2}\rfloor}}\iota_{\cut}^{-1}(g_{j,l, n'}e_{i})
-\sum_{i\in C_{j,l, \leq \lfloor \frac{n'}{2}\rfloor}}\iota_{\cut}^{-1}(g_{j,l, 2n'+1}e_{i})
\]
and one checks that $g_{(j,l, n')}=\iota_{\cut}(\iota_{\cut}^{-1}(g_{(j,l, n'))}).$ Finally, we let $u_{(j,l, n')}=g_{(j,l, n')}g_{(j,l n')}^*$ be mapped to $
\iota_{\cut}^{-1}(g_{(j, l, n')})\iota_{\cut}^{-1}(g_{(j, l n')})^*
$ for $(j,l)\in JL$ and $n'\in\omega$ which ensures that (CK2) holds for these images. It is direct to check that $\iota_{\cut}^{-1}$ and $\iota_{\cut}$ are inverse to each other.

If $m=0,$ the only possible value of $j$ is zero, so some definitions simplify and some formulas are shorter. For example, the elements $q_i$ become
$p_{l_{i}-1}g_{0i'_{i}}d_{i'_i},$ the set $JL$ becomes the set $L=\{l\in k_0\mid l=l_i$ for $i\in C\}$ and the sets $C_{j,l}$ become $C_l=\{i\in C\mid l_i=l\}$ for $l\in L.$ With these modifications, the definitions of
$\iota_{\cut}^{-1}$ and $\iota_{\cut}$ are analogous to those in the case $m>0.$ Note that if the spine length $k$ is infinite, the case (2) cannot happen (because $k=k_t$ in that case).
\end{proof}

Let us move on to the case when $E$ has a proper terminal cycle in which case we use $n$ only for $|d|>0.$ Let $E_{\cut}$ be the graph with the same quotient and $c$-to-$d$ part as $E$.
If $E/H$ is finite, $E_{\cut}=E.$ If $E/H$ is infinite, let $d(j,j')$ stands for the length of the shortest path in $c$ from $v_j$ to $v_j'$ and let $C(i)$ be the following statement.
\begin{itemize}
\item[$C(i)$] There are $j,j'\in m, i'\in n,$ and $k',l\in\omega$ such that
$|\Pa^{v_j}_l|=\omega,\;\;$$a_{j'i'}\neq 0,\;$ and $\;l+k'|c|+d(j,j')+n-i'=n-i(\mymod n).$
\end{itemize}

If $m=0,$ we have that $j=j'=|c|=0$ in condition $C(i)$ so it simplifies as follows.
\begin{itemize}
\item[] There are $i'\in n,$ and $l\in\omega$ such that
$|\Pa^{v_0}_l|=\omega,\;\;$$a_{0i'}\neq 0,\;$ and $\;l+n-i'=n-i(\mymod n).$
\end{itemize}

We say that the $i$-tails are {\em cuttable} if $C(i)$ holds or if there is $j\in m$ such that $\Pa_l^{v_j}$ is nonempty for infinitely many $l$ (equivalently, the spine of $E^{v_j}$ is infinite). In this last case, we also say that $i$-tails are cuttable for every $i\in n.$
Using the introduced terminology, if $l_i$ is the number of $i$-tails of $E$, then $E_{\cut}$ has $l_i$ tails if $i$-tails are not cuttable and it has zero $i$-tails otherwise.

Condition $C(i)$ for $n>0$ is completely analogous to condition $C(i)$ for $n=0$ except that $n-i$ and $n-i'$ in the last equation for $n>0$ are replaced by $i$ and $i'$ for $n=0$. This is due to the fact that the distance from $w_i$ to $w_0$ is $n-i$ if $n>0$ and it is $i$ if $n=0.$ Since $\mymod 0$ is the trivial relation, the presence of $\mymod n$ in the $n>0$ case matches the absence of $\mymod 0$ in the $n=0$ case.

If $n>0,$ then $n$ plays the role of both $k$ and $k_t$ and condition (2) never holds. If $E$ is a graph with $n>0$ and such that $E^{v_j}$ has finite spine length for all $j\in m,$ then the proof of Proposition \ref{proposition_cut_form} carries when using $n$ instead of $k_t.$ The proof is simplified since condition (1) is always in effect. This enables us to have the tail cutting operation $\iota_{\cut}:E\to_{\cut} E_{\cut}$ which cuts all cuttable tails of $E$ and produces $E_{\cut}$ for graph with proper terminal cycle.

Assume that $E$ is a graph with $n>0$ and such that $E^{v_j}$ has infinite spine length for some $j\in m.$ If $m>0$ and $L$ is the least common multiple of $m$ and $n,$ let $E'$ be the graph obtained by moving any exit of $E$ $L$ times. If $m=0,$ let $E'$ be the graph obtained by moving any $0i$-exit of $E$ with $a_{0i}=\omega.$ In each case, $E'$ has infinitely many $i$-tails for every $i\in n$ and $E_{\cut}$ is defined as the graph with zero $i$-tails for every $i\in n.$ Let $\phi_E$ be the move operation which transforms $E$ to $E'$ and let $\phi_{\cut}$ be the same
move but applied to $E_{\cut}$. It produces $E'$ also since infinitely many tails are created by such operation. Thus, we can let $\iota_{\cut}:E\to_{\cut} E_{\cut}$ be $\phi_{\cut}^{-1}\phi_E$ so that $\iota_{\cut}^{-1}=\phi_E^{-1}\phi_{\cut}.$

\subsection{Canonical quotients} \label{subsection_canonical_quotients}
Next we turn to the quotient $E/H$ of a 2-S-NE graph $E=E_{\direct}$ with one terminal cluster. We introduce some  terminology needed in the $m>0$ case. We say that $E$ is a {\em single exit-emitter} graph if only one vertex of $c$
emits exits. If $m=0,$ this trivially holds.

Let $v_0, \ldots, v_{m-1}$ be the vertices of the emitting cycle $c$ of $E$. We consider the following construction. First, we move the exits so that only $v_0$ emits them.
Then, we let $S_{0}$ be the 1-S-NE graph obtained by considering the quotient but without the edge $v_0$ emits in $c$. The vertex $v_0$ is the sink of this graph and the presence of ``S'' in this notation indicates that $S_{0}$ is a graph with a sink. We call it the {\em $0$-sink-graph} or, if it is clear which $c\in [c]$ we consider, a {\em sink-graph}. By considering other elements of $[c],$ for $j\in m$ we have the {\em $j$-sink-graph}.

By the last sentence of Proposition \ref{proposition_canonical_tails}, we can rearrange the paths in the quotient so that the part of $c$ from $v_{1}$ to $v_0$ is on the spine. Here ``1'' in $v_1$ is considered modulo $m$. If the spine is longer than $m$, we think of $v_1$ as the entry point of the spine to the cycle. With this rearrangement, we use
$S_{10}$ for the sink-graph of the resulting graph to indicate both the entry point as well as the single exit emitter location: ``1'' receives the spine and ``0'' emits the exits.

If the spine of $S_{10}$ is longer than $m-1$ and if the number of the $j$-tails is at least one for $1\neq j\in m,$ then one can rearrange the paths so that the entry point of the spine is moved $m-1$ vertices in the direction of $c$ and the length is by $m-1$ longer. This can also be achieved by considering the following $m-1$ in-split plus moves: one at each $v_j$ for $j=m-1, m-2, \ldots 1.$ The diagram below represents one step of this process.
The first graph is the mother graph $M$ in which $qe$ is on the spine and the non-relevant tails of the spine and the cycle are not represented. The in-split minus of $v_j$ in $M$ with respect to the partitions $\{\ra^{-1}(\ra(e)), \emptyset\}$ and $\{\ra^{-1}(\ra(e))-\{e\}, \{e\}\}$ produces the second and the third graph. We refer to such a move as {\em moving the entry point of the spine from $v_{j}$ to $v_{j+_m1}.$}
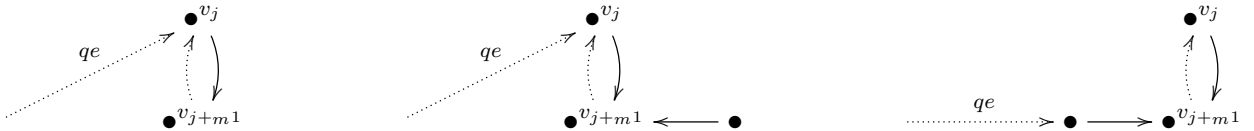
\[
\xymatrix{&&\bullet^{v_j}\ar@/^/[d]\\\ar@{.>}^{qe}[urr]&&\bullet^{v_{j+_m1}}\ar@{.>}@/^/[u]}
\hskip1.8cm
\xymatrix{&&\bullet^{v_{j}}\ar@/^/[d]\\\ar@{.>}[urr]^{qe}&&\bullet^{v_{j+_m1}}\ar@{.>}@/^/[u]&\bullet\ar[l]}
\hskip1.8cm
\xymatrix{&&&\bullet^{v_{j}}\ar@/^/[d]\\\ar@{.>}[rr]^{qe}&&\bullet\ar[r]&\bullet^{v_{j+_m1}}\ar@{.>}@/^/[u]}
\]

Let $k_2$ be the length of the spine of the sink-graph $S_{10}$. The presence of ``2'' in the subscript of $k_2$ notes that this is the spine length of the spine of $H_2/H_1$ and we follow this labeling when considering $n$-S-NE-graphs for $n>2$. Since we make the spine be a part of the cycle of length $m-1,$ $k_2\geq m-1.$ For $i\leq k_2,$ let $u_i$ be the vertices on the spine and $n_i$ be the cardinality of $\ra^{-1}(u_i).$  So, for $j\in m,$ $u_j=v_{m-_mj}.$
Since the edge $u_i$ receives in the spine is counted in $n_i,$ we have that $n_i\geq 1$ for all $i\leq k_2.$

With this notation, we consider the feasibility of extending the spine using an in-split plus move we specify next and which is possible if the following two conditions hold. When considering condition (1), it is relevant that $E=E_{\direct}.$
\begin{enumerate}
\item There is a tail ending at $w_i$ for every exit $e$ which ends at $w_i.$

\item Every $v_j, j\in m,$ has a tail.
\end{enumerate}

We express condition (2) in terms of cardinalities $n_j$ in order to generalize it to graphs with composition length larger than two. Expressed in this way, condition (2) is stating that  $n_j>1$ for every $j\in m$ and, if $k_2>m-1,$ then $n_{m-1}>2$.

We demonstrate how the spine can be extended if $E$ satisfies these conditions. Since $n_j>1$ for $j\in m,$ $E$ allows moving the entry point of the spine from $v_1$ to $v_0$ (as ``1'' in $v_1$ is considered modulo $m$, if $m=1$ then this step is trivial). This uses up one tail for each $1\neq j\in m.$ Let $E'$ be the resulting graph. In $E',$ $v_j$ has one less tail for all $j\neq 1$ and the part of the spine outside of $c$ is extended by a line of length $m-1$ ending in $v_0.$

Since the number of tails of $v_1$ is unimpacted by this move and since (1) holds, $E'$ is the result of an in-split minus move of the following mother graph $M$. The connecting parts of $E'$ and $M$ are the same. The quotient of $M$ is the same as that of $E'$ except that the tail at $v_1$ is deleted. The tail graph of $M$ is the same except that the tails from condition (1) are deleted. In this case, $E'$ is the result of the in-split minus move of $M$ at $v_0$ with respect to the partition $\{\ra^{-1}(v_0), \emptyset\}$ followed by the total out-split.

If $e$ is the edge on the spine with range $v_0$ in $E'$, we can consider the in-split minus of $M$ with respect to $\{\ra^{-1}(v_0)-\{e\}, \{e\}\}$ and let $F$ be the total out-split of the resulting graph. So, the entry point of the spine in  $F$ is moved from $v_0$ back to $v_{1}$ and its length is extended by a line of length $m$ inserted between $v_1$ and $u_m$.

Comparing the quotients of $E$ and $F,$ the $n_j$-value of $F$ is by one less than $n_j$ for $E$ for every $j\in m$ and the spine is longer by a line of length $m$ inserted between $u_m$ and $v_1=u_{m-1}.$ We refer to the operation $E\to F$ as a one-step {\em spine extending}. By considering composition of such operations, if possible by (1) and (2), we can extend the spine by multiples of $m$.

\begin{example}
\label{example_with_different_E/H_parts}
Let $E$ be the first graph below.

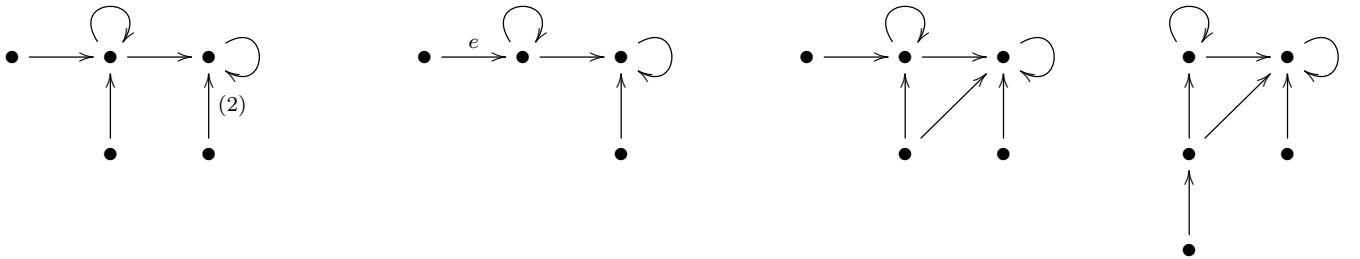
\[\xymatrix{\bullet\ar[r]&\bullet\ar[r]\ar@(lu,ru)&\bullet\ar@(ru,rd)\\&\bullet\ar[u]&\bullet\ar[u]_{(2)}}\hskip2cm
\xymatrix{\bullet\ar[r]^e&\bullet\ar[r]\ar@(lu,ru)&\bullet\ar@(ru,rd)\\&&\bullet\ar[u]}
\hskip2cm
\xymatrix{\bullet\ar[r]&\bullet\ar[r]\ar@(lu,ru)&\bullet\ar@(ru,rd)\\&\bullet\ar[u]\ar[ur]&\bullet\ar[u]}
\hskip2cm
\xymatrix{\bullet\ar[r]\ar@(lu,ru)&\bullet\ar@(ru,rd)\\\bullet\ar[u]\ar[ur]&\bullet\ar[u]\\\bullet\ar[u]&}
\]
The presence of one $v_0$ tail and one $w_0$ tail makes $E$ eligible for the spine extending. The
second graph is the mother graph and the third and the fourth graphs are two in-split plus equivalent graphs obtained by the in-split minus moves on the mother graph.

The total out-split of the third graph is $E$ and the first graph below is the total out-split of the last graph in the previous diagram.
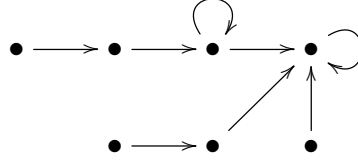
\[
\xymatrix{\bullet\ar[r]&\bullet\ar[r]&\bullet\ar[r]\ar@(lu,ru)&\bullet\ar@(ru,rd)\\&\bullet\ar[r]&\bullet\ar[ur]&\bullet\ar[u]}
\]
\end{example}

As another example, the graph below permits arbitrary spine extending. Extending the spine three times produces the second graph below.

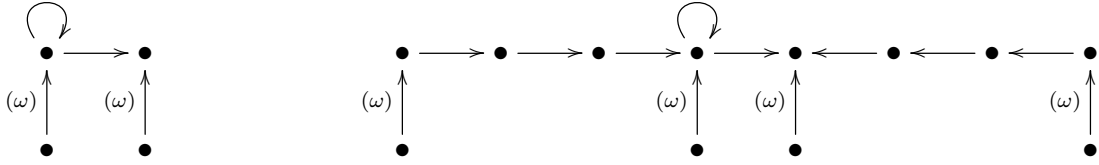
\[
\xymatrix{\bullet\ar[r]\ar@(lu,ru)&\bullet\\\bullet\ar[u]^{(\omega)}&\bullet\ar[u]^{(\omega)}}
\hskip2.5cm
\xymatrix{\bullet\ar[r]&\bullet\ar[r]&\bullet\ar[r]&\bullet\ar[r]\ar@(lu,ru)&\bullet
&\bullet\ar[l]
&\bullet\ar[l]
&\bullet\ar[l]
\\\bullet\ar[u]^{(\omega)}
&&
&\bullet\ar[u]^{(\omega)}&\bullet\ar[u]^{(\omega)}
&&
&\bullet\ar[u]^{(\omega)}}
\]

If $E$ permits arbitrary spine extending, then condition (2) implies that $n_j=\omega$ for every $j\in m.$ Extending the spine $l$ time inserts a line of length $lm$ between $u_{m-1}$ and $u_{(l+1)m-1}$ and the original part of the spine is below $u_{(l+1)m-1}$. The presence of this segment of length $lm,$ enables us to be able to tell whether the spine can be shortened by it. So, this enables us to
reduce the spine length to its smallest possible length -- when there is no line of length $lm$ for any $l>0.$ So, either the spine length is $m-1$ or there is $i$ such that $n_i>1$ and $m\leq i<2m$. In this case, we say that the spine of $E$ is {\em as short as possible}.
For example, the second graph in the last example above does not have the spine as short as possible and the first graph does. The terminology ``as short as possible'' is somewhat misleading if the spine of the sink-graph is infinite because it remains infinite even though the line of some finite length may be deleted from it and the two remaining pieces reconnected.

We say that $E$ is {\em I-plus-reduced} if one of the conditions (a) or (b) below holds.
\begin{enumerate}[\upshape(a)]
\item $E$ does not permit a spine extending.

\item $E$ permits arbitrary spine extending and the spine of $E$ is as short as possible.
\end{enumerate}

We focus such graphs in Proposition \ref{proposition_E_without_H} and the proof of this proposition implies that the conditions (1) and (2) are also necessary for the spine extending.

If the graph $E$ is I-plus-reduced and $v_j$ is the single exit-emitter, we use $E_{\can\quot, j}$ for $E$.
The above process shows that every $E$ can be converted to a graph which is I-plus-reduced and we write $E\to E_{\can\quot, j}$ for such a graph. The quotient of
$E_{\can\quot, j}$ is a {\em canonical quotient} of $E.$

If $m>1,$ the sink-graphs in $E_{\can\quot, j}$ and $E_{\can\quot, j'}$ do not have to be isomorphic as the following example shows. Let $E$ be the first graph below (where $l_0$ and $l_1$ can be finite or infinite), let $c$ be the cycle with $\so(c)=v_0$ and $c_1$ is the other element of $[c].$ The second graph is $E_{\can\quot, 0}$ and the third is $E_{\can\quot, 1}.$
\[\xymatrix{&\bullet\ar@/^/[d]\ar[dr]^2&\bullet\ar@/^/[d]&\bullet\ar[l]_{(l_1)}\\\bullet\ar[r]&
\bullet_{v_0}\ar[r]\ar@/^/[u]&\bullet\ar@/^/[u]&\bullet\ar[l]_{(l_0)}}\hskip1.5cm
\xymatrix{&\bullet\ar@/^/[d]&\bullet\ar@/^/[d]&\bullet\ar[l]_{(l_1)}\\\bullet\ar[r]&
\bullet\ar[ur]^2\ar[r]\ar@/^/[u]&\bullet\ar@/^/[u]&\bullet\ar[l]_{(l_0+2)}}
\hskip1.5cm
\xymatrix{&\bullet\ar@/^/[d]\ar[r]\ar[dr]^2&\bullet\ar@/^/[d]&\bullet\ar[l]_{(l_1+1)}\\\bullet\ar[r]&
\bullet\ar@/^/[u]&\bullet\ar@/^/[u]&\bullet\ar[l]_{(l_0+1)}}\]
The sink-graph of $E_{\can\quot, 0}$ is $
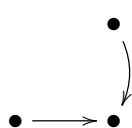
\xymatrix{&\bullet\ar@/^/[d]\\\bullet\ar[r]&
\bullet}$ and the sink-graph of $E_{\can\quot, 1}$ is $
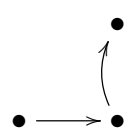
\xymatrix{&\bullet&\\\bullet\ar[r]&
\bullet\ar@/^/[u]&}$.

\subsection{Canonical forms} For $j\in m,$ we define a canonical form by reducing and cutting
$E_{\can\quot, j}$. Reductions and cuts do not impact the quotient of $E_{\can\quot, j}$ -- they only ensure that the number of tails is as small as possible. After this, we move exits to certain form depending on $m$ and $n$ values which are specified in the definition below and which make $v_j$ be a single exit-emitter again if $n>0.$ Finally, we $L$-, spine- or $\omega$-reduce (whatever is applicable based on the $m$ and $n$ values) and cut the tails to regain their  minimality.
We shorten $(E_{\red})_{\cut}$ to $E_{\red, \cut}.$

\begin{definition} Let $E$ be a direct-exit 2-S-NE graph, $j\in m$ (where this stands for $j=0$ if $m=0$). Move the exits of $(E_{\can\quot, j})_{\red, \cut}$, if needed, as specified below and then, let
$E_{\can}$ be the $L$-, spine- or $\omega$-reduced and cut form, whatever is applicable, of the resulting graph.
\begin{itemize}
\item If $m>0$ and $n>0,$ then the required form is obtained by moving the exits so that the resulting graph is a single exit-emitter and that the ranges of the exits are among $G=$GCD$(m,n)$ consecutive vertices of $d.$ In particular, for $c\in [c]$ and $d\in [d]$ such that $v_j=\so(c)$ is a single exit-emitter and the vertices from $w_{n-G+1}$ to $w_0$ are receiving exits. This is possible by Lemma \ref{lemma_feasibility_of_an_exit_move}. Thus, $E_{\can}$ depends on the choice of $c$ and $d.$

\item If $m>0$ and $n=0,$ the required form is obtained by moving the exits which do not end in the spine source $w_k$ to end in $w_k.$ In particular, a $ji$-exit with $i<k$ would become a $(j-_m{k-i})(i+(k-i))$-exit.

\item If $m=0$ and $n>0$ or
$m=n=0$ and $k<\omega$,
the required form is obtained by moving all the $0i$-exits if $w_i$ receives nonzero and finitely many of such exits. Thus,
for $i\in n$ if $n>0$ and for $i\leq k$ if $n=0$, $w_i$ receives
either zero or infinitely many exits.

\item If $m=n=0$ and $k=\omega$,
no exit moves are needed and $(E_{\can\quot, j})_{\red, \cut}$ is canonical.
\end{itemize}
If $E$ is a 2-S-NE graph, we let $E_{\can}=(E_{\direct})_{\can}.$
We say that $E$ is {\em canonical} or that it is {\em in a canonical form} if $E\cong F_{\can}$ for some graph $F.$
\label{definition_canonical_forms_n=2}
\end{definition}

For example, let $E, E_0,$ and $E_1$ be three graphs below in that order. Both $E_0$ and $E_1$ are canonical forms of $E.$
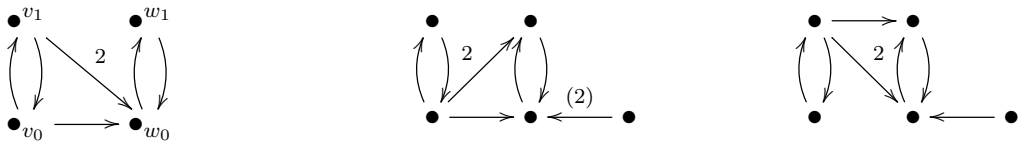
\[\xymatrix{\bullet^{v_1}\ar@/^/[d]\ar[dr]^2&\bullet^{w_1}\ar@/^/[d]&\\
\bullet_{v_0}\ar[r]\ar@/^/[u]&\bullet_{w_0}\ar@/^/[u]&}
\hskip2cm
\xymatrix{\bullet\ar@/^/[d]&\bullet\ar@/^/[d]&\\
\bullet\ar[ur]^2\ar[r]\ar@/^/[u]&\bullet\ar@/^/[u]&\bullet\ar[l]_{(2)}}
\hskip2cm
\xymatrix{\bullet\ar@/^/[d]\ar[r]\ar[dr]^2&\bullet\ar@/^/[d]&\\
\bullet\ar@/^/[u]&\bullet\ar@/^/[u]&\bullet\ar[l]}
\]
Note that $E_0$ is obtained by moving both $10$-exits and $E_1$ is obtained by moving the $00$-exit. A graded $*$-isomorphism of the algebras of $E_0$ and $E_1$ can be obtained by composing the inverse of the moves $E\to E_0$ with the moves $E\to E_1.$
We generalize this observation in the following lemma whose proof is contained in its statement.

\begin{lemma}
Let $E=E_{\direct}$ be  such that $F$ and $G$ are its canonical forms.
Let $\phi_F: E\to F$ and $\phi_G: E\to G$ be the maps used to obtain the canonical forms $F$ and $G$ from $E$. Then, $\phi_{G}\phi_F^{-1}: F\to G$ induces a graded $*$-isomorphism of the corresponding algebras.
\label{lemma_maps_between_canonical_forms}
\end{lemma}

Although a graph $E$ can have non-isomorphic canonical forms, they do determine the graded $*$-isomorphism class as Theorem \ref{theorem_n=2} shows: the algebras of $E$ and $F$ have graded $*$-isomorphic algebras if and only if representatives of $[c]$ and $[d]$ can be found so that the canonical forms determined using those representatives are isomorphic.

We consider another example illustrating the impact of the number of tails on the graded $*$-isomorphism class of the corresponding algebra. Let us consider graphs of with $m=n=2$ and with the quotient consisting of one edge ending in the cycle. A canonical form of such a graph  has the form as in the figure below.
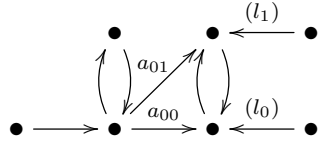
\[
\xymatrix{&\bullet\ar@/^/[d]&\bullet\ar@/^/[d]&\bullet\ar[l]_{(l_1)}\\
\bullet\ar[r]&\bullet\ar[ur]^{a_{01}\hskip-.2cm}\ar[r]^{a_{00}}\ar@/^/[u]&\bullet\ar@/^/[u]&\bullet\ar[l]_{(l_0)}}\]
If both $a_{00}$ and $a_{01}$ are nonzero, then $l_0, l_1\in \{0,1,2, \omega\}$ and the possible values for the pair $(l_0,l_1)$  include any with at least one entry zero and $(1,1), (1,\omega), (\omega, 1),$ and $(\omega,\omega).$ If $a_{00}\neq a_{01}$, then all the different possibilities for the $l_0$ and $l_1$ values correspond to different graded $*$-isomorphism classes of the corresponding algebras. If $a_{00}=a_{10},$ then the map interchanging $w_0$ and $w_1,$ their tails and the exits they receive and which is the identity on other graph elements is a graph isomorphism. This makes the algebra of a graph with $(l_0,l_1)$ graded $*$-isomorphic to the algebra with $(l_1,l_0),$ so the
number of graded $*$-isomorphism classes is smaller than when $a_{00}\neq a_{01}.$

If $\omega$ tails are added to the exit-emitter and if $a_{00}\neq 0,$ then the 1-tails are cuttable, so $l_1=0$ in any canonical form. The 0-tails are cuttable if and only if $a_{01}\neq 0.$
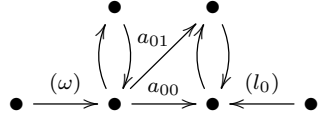
\[
\xymatrix{&\bullet\ar@/^/[d]&\bullet\ar@/^/[d]&\\\bullet\ar[r]^{(\omega)}&
\bullet\ar[ur]^{a_{01}\hskip-.2cm}\ar[r]^{a_{00}}\ar@/^/[u]&\bullet\ar@/^/[u]&\bullet\ar[l]_{(l_0)}}\]
If $a_{01}\neq 0,$ a canonical form has $l_0=l_1=0.$ If $a_{01}=0$, then $l_0\in\{0,1,\omega\}.$ So, there are exactly three graded $*$-isomorphism classes of algebras of such graphs.

The Toeplitz graph is a canonical form of any graph with quotient consisting of a single loop, the spine of the tail graph being the spine of the graph and the tail graph having no tails.

Let us look at an example with $m=0$ and $n>0.$ Let $E_{0}$ be the first graph below where $0\leq a_{01}<\omega.$  In this graph, the 1-tails are cuttable and 0-tails are cuttable if and only if $a_{01}\neq 0.$
If $a_{01}\neq 0,$ then a cut form has neither 0- nor 1-tails and a canonical form, the second graph below, is obtained by moving all $a_{01}$ exits and then cutting the resulting graph. If $a_{01}=0,$ then only 1-tails are cuttable but then 0-tails are reducible to zero if their number is finite. In this case, the second and the third graphs are possibilities for a canonical form based on whether $l_0$ is finite or not.
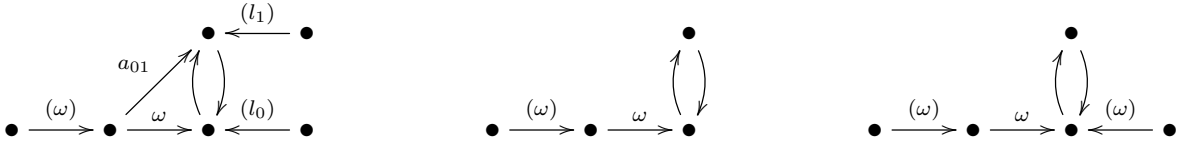
\[
\xymatrix{&&\bullet\ar@/^/[d]&\bullet\ar[l]_{(l_1)}\\
\bullet\ar[r]^{(\omega)}&\bullet\ar[r]^\omega\ar[ur]^{a_{01}}&\bullet\ar@/^/[u]&\bullet\ar[l]_{(l_0)}}\hskip2cm
\xymatrix{&&\bullet\ar@/^/[d]\\
\bullet\ar[r]^{(\omega)}&\bullet\ar[r]^\omega&\bullet\ar@/^/[u]}\hskip2cm \xymatrix{&&\bullet\ar@/^/[d]&\\
\bullet\ar[r]^{(\omega)}&\bullet\ar[r]^\omega&\bullet\ar@/^/[u]&\bullet\ar[l]_{(\omega)}}
\]

For an example with $m=n=0$ and $k=\omega,$ let $E$ be the graph below so that its connecting matrix is $[\omega\; a_{01}\; \omega\; a_{03}\; \omega\; a_{05}\ldots]$) where $a_{0(2k+1)}$ is finite for all $k\in\omega.$ Let $0<l_{i}<\omega$ for $i\geq 0.$
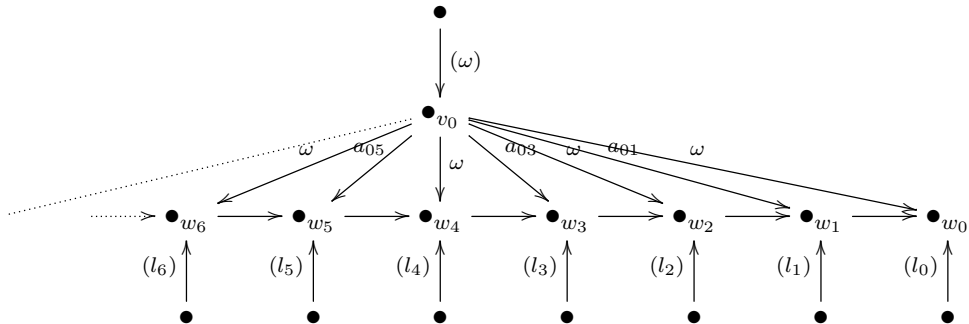
\[\xymatrix{&&&&\bullet\ar[d]^{(\omega)}&&&\\&&&&\bullet_{v_0}\ar@{.}[dllll]\ar[dll]_\omega\ar[dl]_{a_{05}\hskip-.2cm} \ar[drr]^\omega\ar[dr]^{a_{03}\hskip.5cm} \ar[drrrr]^\omega\ar[d]^\omega\ar[drrr]^{a_{01\hskip.5cm}}&&&\\
&\ar@{.>}[r]&\bullet_{w_6}\ar[r]&\bullet_{w_5}\ar[r]&\bullet_{w_4}\ar[r]&\bullet_{w_3}\ar[r]&\bullet_{w_2}\ar[r]&\bullet_{w_1}\ar[r]&\bullet_{w_0}\\&&\bullet\ar[u]^{(l_6)}&\bullet\ar[u]^{(l_5)}&\bullet\ar[u]^{(l_4)}&\bullet\ar[u]^{(l_3)}&\bullet\ar[u]^{(l_2)}&\bullet\ar[u]^{(l_1)}&\bullet\ar[u]^{(l_0)}}\]
The $2k+1$-tails are cuttable for all $k\geq 0$ and $2k+2$-tails are cuttable if $a_{0(2k+1)}\neq 0.$
In addition, $2k$-tails are reducible to zero if their number is finite. So, the number of $2k$-tails determines the graded $*$-isomorphism class of the algebra of such a graph.

\subsection{2-S-NE equivalence} For 2-S-NE graphs $E$ and $F$, we define the following relation.
\[E\approx F\;\;\;\mbox{ if there are canonical forms $E_{\can}$ and $F_{\can}$ such that }E_{\can}\cong F_{\can}\]
and we say that $E$ and $F$ are {\em 2-S-NE equivalent} if $E\approx F$ holds. It is clear that the relation $\approx$ is reflexive and symmetric. We show that  transitivity follows from Lemma \ref{lemma_maps_between_canonical_forms}.

\begin{lemma} The relation $\approx$ is transitive on 2-S-NE graphs.
\label{lemma_approx_is_transitive}
\end{lemma}
\begin{proof}
If $E, F,$ and $G$ are graphs
such that $E\approx F$ and $F\approx G,$ then there are canonical forms $E_1$ of $E,$
$F_1$ and $F_2$ of $F,$ and $G_2$ of $G$ such that $E_1\cong F_1$ and $F_2\cong G_2.$
Let $\phi_{F_1F_2}: F_1\to F_2$ be the operation from Lemma \ref{lemma_maps_between_canonical_forms}. Applying the same operation to $G_2$ results in another canonical form $G_1$ of $G$ and $F_2\cong G_2$ implies $F_1\cong G_1.$ Thus, $E_1\cong F_1$ and $F_1\cong G_1$ hold, so $E_1\cong G_1$ holds.
\end{proof}

There are graphs $E$ and $F$ such that $E\approx F$ and $E\ncong F.$ For example, let $E$ be the first and $F$ the second graph. The last graph is a canonical form of both
$E$ and $F$ so we have that $E\approx F.$
\[
\xymatrix{\bullet\ar@/^/[d]\ar[dr]\ar[r]&\bullet&\\
\bullet\ar[r]\ar@/^/[u]&\bullet\ar[u]&}\hskip2cm
\xymatrix{\bullet\ar@/^/[d]&\bullet&\\
\bullet\ar[ur]\ar@/^/[r]\ar@/_/[r]\ar@/^/[u]&\bullet\ar[u]&}\hskip2cm
\xymatrix{\bullet\ar@/^/[d]\ar[dr]&\bullet&\bullet\ar[l]\\\
\bullet\ar@/^/[r]\ar@/_/[r]\ar@/^/[u]&\bullet\ar[u]&}\]

\subsection{Theorem \ref{theorem_n=2} and its proof}\label{subsection_the_proof}
In this section, we formulate and prove that GCC holds for countable 2-S-NE graphs and we realize every $\POM^D$-isomorphism by a graph operation.

\begin{theorem} {\bf The GCC holds for the 2-S-NE graphs.}
Let $E$ and $F$ be two countable 2-S-NE graphs.
The following conditions are equivalent.
\begin{enumerate}[\upshape(1)]
\item There is a $\POM^D$-isomorphism $f:M_E^\Gamma\to M_F^\Gamma.$

\item The relation $E\approx F$ holds.

\item There is a graded $*$-isomorphism $\phi:L_K(E)\to L_K(F).$
\end{enumerate}
If (1) holds, there are  canonical forms $E_{\can}$ and $F_{\can}$ and operations $\phi_E: E\to E_{\can},$ $\iota: E_{\can}\cong  F_{\can},$ and $\phi_F: F\to F_{\can},$ such that $\ol{\phi_F^{-1}\iota\phi_E}=f.$
\label{theorem_n=2}
\end{theorem}
The implications (2) $\Rightarrow$ (3) and (3) $\Rightarrow$ (1) are direct.
In the rest of this section, we show that (1) $\Rightarrow$ (2)  and that the last sentence of the theorem holds.

Before going any further, we fix some notation and establish a general set-up for two 2-S-NE graphs $E$ and $F$ and a $\POM^D$-isomorphism $f$ between their $\Gamma$-monoids. Let
$\phi_E: E\to E_{\can}$ and $\phi_F: F\to F_{\can}$
be operations transforming the graphs to their canonical forms. We can consider $\ol\phi_Ff\ol{\phi_E^{-1}}$ instead of $f$ and so we can assume that $E=E_{\can}$ and $F=F_{\can}.$

Let $E$ have composition factors $P_H$ and $E/H$ and let $\ol\iota_H$ be the map induced by the inclusion $\iota_H: I(H)\to L_K(E)$ and $\ol\pi_{E/H}$ be the map induced by the natural map $\pi_{E/H}: L_K(E)\to L_K(E/H).$ Let $G$ be a hereditary and saturated subset of $F^0$ such that $f\ol\iota_H(J^\Gamma(H))=J^\Gamma(G)$ and let $\iota_G$ and $\pi_{F/G}$ be analogous to $\iota_H$ and $\pi_{E/H}.$ By the choice of $G,$  $f\ol\iota_H: J^\Gamma(H)\to J^\Gamma(G)$ is a $\POM^D$-isomorphism and $f$ induces a $\POM^D$-isomorphism $f_{2/1}:M_{E/H}^\Gamma\to M_{F/G}^\Gamma$ such that the diagram below commutes.
\[\xymatrix{0\ar[r]& J^\Gamma(H)\ar[r]^{\hskip-.4cm\ol\iota_H}\ar[d]^{f\ol\iota_H}&\; M_E^\Gamma\;\ar[r]^{\hskip-.4cm\ol\pi_{E/H}}\ar[d]^{f}&\; M^\Gamma_{E/H}\;\ar[r]\ar[d]^{f_{2/1}}&0\\
0\ar[r] &J^\Gamma(G)\ar[r]^{\hskip-.4cm\ol\iota_G}& M_F^\Gamma\ar[r]^{\hskip-.4cm\ol\pi_{F/G}}& M^\Gamma_{F/G}\ar[r]&0} \]

{\bf The two-cluster case.} In the easy case when $E$ has two connected components, the first row of the above diagram splits and the existence of $f$ implies that there is a splitting of the bottom row.
Thus, any two generators of $M_F^\Gamma$ are $\Zset^+[t,t^{-1}]$-independent. This implies that that there are no paths from any vertex of the cluster of $F/G$ to $G$ (because there would be  two $\Zset^+[t,t^{-1}]$-dependent generators of $M_F^\Gamma$ otherwise).
As $F=F_{\tot},$ $F$ has two connected components, one 1-S-NE equivalent to $H$ and the other to $E/H$. Since $E$ and $F$ are canonical, their 1-S-NE components are canonical also and, hence, isomorphic by Propositions \ref{proposition_cofinal_graphs} and \ref{proposition_canonical_tails}. Thus, we have that $E\cong F.$ In addition,
there are $\chi: P_H\to P_G$ and $\psi: E/H\to F/G$ such that $\ol\chi=f\ol\iota_H$ and $\ol\psi=f_{2/1}$ by Propositions \ref{proposition_cofinal_graphs} and \ref{proposition_canonical_tails}. This implies the existence of a graph operation $\phi$ which is defined on $P_H$ as $\chi$ and on $E/H$ as $\psi$ and which is such that $\ol\phi=f.$

{\bf The one-cluster case.}
Having the easy case out of the way, let us assume that $E$ has only one terminal cluster. As the previous paragraph shows, $F$ also has only one terminal cluster.

Let $c, c', d,$ and $d'$ be the terminal cycles of $E/H,$ $F/G,$ $P_H$, and $P_G,$ respectively and let $|c|=m,$ $|c'|=m',$ $|d|=n,$ and $|d'|=n'$.
We let $c^0$ consists of $v_0,\ldots, v_{m-_m1}$ where these vertices are listed in the same order they appear in $c$ and we let $v'_0,\ldots, v'_{m'-_{m'}1}$ be analogously ordered vertices of $c',$ $w_0,\ldots, w_{n-_n1}$ of $d,$ and $w'_0,\ldots, w'_{n'-_{n'}1}$ of $d'.$

By Propositions \ref{proposition_cofinal_graphs} and \ref{proposition_canonical_tails}, the existence of $f\ol\iota_H$ implies that
$P_H\approx P_G,$ so
$n=|d|=|d'|=n'.$ By the same propositions, the existence of $f_{2/1}$ implies that
$E/H\approx F/G,$ so
$m=|c|=|c'|=m'.$

If $n=0,$ we have that $f$ maps $[w_0]$ to $[w_0'].$  If $n>0,$ for any $d\in [d]$ and $d'\in [d'],$ if $f([\so(d)])=t^i[\so(d')]$ for some $i\in n,$ we replace $d'$ with the element $d''\in [d']$ which originates at $w_{n-_ni}.$ If $d''\neq d',$ relabel the vertices so that $d''=d'.$ Thus, we have that $f([w_0])=[w_0']$ holds.

\begin{proposition} {\bf The Quotient Proposition.}
Let $E, F, H, G, v_0, v_0', w_0, w_0'$ and $f$ be as above (recall that  $f([w_0])=[w_0']$). If $f([v_0])=[v_0']$ and the connecting matrices (computed using $c, d, c'$  and $d'$) are equal, then there is an isomorphism $\iota: E/H\cong F/G$  such that $f_{2/1}=\ol{\iota}.$
\label{proposition_E_without_H}
\end{proposition}
\begin{proof}
If $m=0,$ the sink-graphs are the 1-S-NE quotients
so the existence of $f_{2/1}$
implies that $E/H\approx F/G.$ Since both $E/H$ and $F/G$ are canonical, we have that $E/H\approx F/G$ implies the existence of an isomorphism $\iota:E/H\cong F/G.$ This isomorphism maps the sink $v_0$ of $E/H$ to the sink $v_0'$ of $F/G$ and so $f_{2/1}([v_0])=[v_0']=[\iota(v_0)]$ which implies that $\ol\iota=f_{2/1}.$

With the case $m=0$ out of the way, it remains to consider the case $m>0.$
By what we showed so far,  $E$ and $F$ have the same connecting matrix, so the $c$-to-$d$ part is the same as the $c'$-to-$d'$ part.

We use the same notation as in section \ref{subsection_canonical_quotients}.
To summarize it,
$k_2$ is the length of the spine of the sink-graph $S_{10}$ of a canonical quotient, $u_i$ are the vertices on the spine for $i\leq k_2$ $n_i$ is the cardinality of $\ra^{-1}(u_i)$ and, for $j\in m,$ $u_j=v_{m-_mj}.$
Let $k'_2, u_i',$ and $n_i',$  be the analogous values for  $F$. We also let $\rho: P_{\not d}^{w_0}\to P_{\not d'}^{w_0'}$ be a bijection induced by the restriction $f_1$ of $f$ on the $\Gamma$-monoids of the porcupine graphs $P_H$ and $P_G.$

For each $i\leq k_2$, let $p_i$ be the part of the spine from $u_i$ to $v_0$ and let $p'_{i}$ be the same for $F$. We refer to the length $i$ of such $p_i$ as the {\em distance} from $u_i$ to $v_0$ and the length $i+1$ as the distance from a tail of $u_i$ to $v_0$ and, similarly, we define the distance for vertices in the quotient of $F$. For a path $p$ from a vertex in the quotient to $w_0$ and $l\geq 0,$ we say that $p$ an {\em $l$-route} if it contains $c^l$ and does not contain $c^{l+1}$.

If the sink-graphs are not isomorphic, there is $i\leq \min \{k_2, k'_2\}$ such that $n_i\neq n_i'$. Let $i_0$ be the smallest such $i$.
If $n_{i_0}>n_{i_0}',$ there are more 0-routes from the sources of the edges $u_{i_0}$ receives to $w_0$ in $E$ than from the sources of the edges $u'_{i_0}$ receives to $w_0'$ in $F$.

Since the connecting matrices are equal, if $i<i_0,$ $\rho$ can be chosen so that it maps all paths starting at vertices at distance $i<i_0$ to the matching paths in $F$.  On the other hand, if $i=i_0,$ the excess of the 0-routes in $E$ implies that $\rho$ has to map a 0-route to a path in the tail graph. Since $n_{i_0}>n_{i_0}'$ at least one tail of $u_{i_0}$ can be chosen so that $\rho$ maps a 0-route from the tail's source to a path in the tail graph. Hence, the tail graph of $F$ contains enough tails to accommodate such an image. In particular, for every exit $e$ which ends at $w_i$, $F$ contains a tail at $w_i'$. Thus, condition (1) for the spine extending (see section \ref{subsection_canonical_quotients}) holds.

Since $F$ contains at least two edges ending at $w_i'$ (the exit matching $e$ and the tail which $\rho$ maps to a 0-route), $w_i$ has to receive at least that many edges. As one of them is $e$, $w_i$ has to have at least one more tail. So, $E$ {\em also} satisfies condition (1) for spine extending.

Because $\rho$ interchanges a quotient route to a tail route, at least one tail of $u_{i_0}$ can be found so that its source $w$ is such that the image $f([w])$
cannot be matched to a single monoid generator in the correspondence $f$ creates on the generating intervals $D_E$ and $D_F$. Because of this,
\[f([w])=f(t^{i_0+1}[v_0])=t^{i_0+1}\left(t^{lm}[v_0']+\sum_{j=0}^{l-1}t^{jm}a_E[w_0']\right)\]
is not identified with $[w']$ for no vertex $w'$ in the quotient (in particular, with no the source of any 0-route) but it is a ``split'' into a sum of different monoid generators.
By considering the distance from $v_0'$ and by the choice of $i_0$, the first term $t^{i_0+1+lm}[v_0']$ has to correspond to a tail to $u_{i_0+lm}'$.

The existence of $u'_{i_0+lm}$ implies that $f(t^{i_0}[v_0])=t^{i_0}[v_0']=[u'_{i_0+lm}]+\sum_{j=0}^{l-1}t^{jm+i_0}a_E[w_0']$ and that there are vertices $u_{i}'$ for $i<i_0+lm$. In particular,  $f([v_0])=[v_0']=[u'_{lm}]+\sum_{j=0}^{l-1}t^{jm}a_E[w_0'].$ Hence, the image of $[v_0]$ allows a ``split'' of the same type as that of the map induced by an in-split plus.
This split caries to the images of $[v]$ for $v\in R(v_0)$: if $p$ is a path from $v$ to $v_0,$ then
$f([v])=f(t^{|p|}[v_0])=t^{|p|}[u'_{lm}]+\sum_{j=0}^{l-1}t^{jm+|p|}a_E[w_0'].$ So, for every such $v,$ there is a vertex $v'$ in the root of $u'_{lm}$ such that the length of the path from $v'$ to $u'_{lm}$ is $|p|$. Thus, we have that
\[n'_{i+lm}\geq n_i\mbox{ for all }i\geq 0.\]

For $jm+i$ where
$j=1,\ldots, l-1$ and $i\in m$, the elements $u_{jm+i}'$ are in the ``gap'' on the spine of $F$  and $f^{-1}([u_{jm+i}'])$
has to be of the form
$[pc^{j'}c^{-j'}p^*]$ in the quotient for some $j'$. Given that $u_{m+i}',$ for $i=1,\ldots, m$ has the distance $m+i$ from $v_0',$ the value of $j'$ is one in this case and the length of $p$ has to be $i.$ Hence,  such a path $p$ has to start in a tail of $u_{i-1}$ (it cannot start at $u_i$ because all the $j'$-routes from $u_{i}$ to $w_0$ are already matched up by $\rho$ with the matching $j'$-routes in $F$). This shows that every $u_{i-1}$ has to have a tail. Thus, $n_{i}>1$
for each $i\in m.$ If $k_2>m-1,$ the $j'$-routes from $u_{m}$ are matched up to the $j'$-routes from $u_{m}',$ so there has to be an edge $u_{m-1}$ receives outside of the spine. Hence, $n_{m-1}>2$.
This shows that $E$ satisfies condition (2) for the spine extending. Hence, the existence of any $i$ such that $n_i>n_i'$ implies that $E$ permits arbitrary spine extending.

Let us extend the spine of $E$ by $(l+1)m$ now, and let $E_{+l+1}$ be the resulting graph. By  considering $f\ol\phi_{+l+1}^{-1}$ where $\phi_{+l+1}$ is the spine extending, we have that the vertex on the spine of $E_{+l+1}$ at distance $lm$ from the exit emitter has no tails. As  $n_{lm}'>1,$ the difference in cardinalities of the edges the vertices of $E_{+l+1}$ and $F$ receive is in favor of $F$ now. Since our arguments so far did not require the graphs to be I-plus-reduced, we can apply them to $E_{+l+1}$ and $F$ and deduce that $F$ satisfies both conditions for the spine extending.

This shows that both $E$ and $F$ satisfy condition (b) -- their spines can be extended but are as short as possible. Hence, we have that $n_{i}=n_{i}'=\omega$ for all $i\in m.$
Since $n_{i_0}>n_{i_0}',$ $i_0$ is necessarily larger than $m-1$ so the spine length of $E$ is longer than $m-1$ and the formula $n_{i+lm}'\geq n_i$ implies that the spine of $F$ is at least by $lm$ longer than the spine of $E$. Hence, both graphs have spine longer than $m-1.$

Let $E_{+l}$ obtained by extending the spine of $E$ $l$ times. If $n_i''$ are the cardinalities of the edges $i$-th vertex on the spine of $E_{+l}$ receives, we have that $n_i'\geq n_{i}''$ for every $i.$
If $\sigma: P^{v_0}\to P^{v_0'}$ is the length preserving bijection of paths which $f_{2/1}$ induces and if $\sigma_l$ is such bijection induced by the spine extending map $\phi_{+l}:E\to E_{+l},$
we have that $\sigma\sigma_l^{-1}$ maps an $i$-tail onto an $i$-tail for every $i\geq 0.$
As $\sigma\sigma_l^{-1}$ is onto, the existence of $\sigma\sigma_l^{-1}$ implies that
$n_i'=n_{i}''$ for every $i.$ These values make $F$ isomorphic to the graph $E_{+l}.$ Let $\iota: E_{+l}\cong F$ be one such isomorphism (so $\iota(v_0)=v_0'$). As the spine of $F$ is as short as possible, we cannot have that $n_i'=1$ for all $m\leq i<(l+1)m.$ Thus, $l$ has to be zero which implies that $n_i'=n_i$ for all $i.$ In this case, $E_{+l}=E$ and there is an isomorphism $\iota: E\cong F$ such that $f_{2/1}=\ol\iota.$
\end{proof}

\begin{remark}
Note that the proof shows that $\iota$ is an isomorphism of the sink graphs of $E/H$ and $F/G$ which is stronger statement that the conclusion of the Quotient Proposition.
\label{remark_on_sink_graph}
\end{remark}

\begin{remark}
Note that if any of the tails of $u_i$ and $u_i'$ for $i>m-1$ is replaced by the left (finite or infinite) path with the same number of tails in both graphs (as one below) and the sink of such replacement identified with the range of the tail, then
the arguments of the proof still work out.
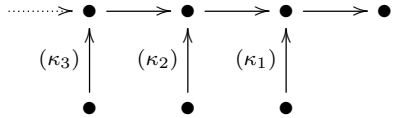
\[
\xymatrix{\ar@{.>}[r]&\bullet\ar[r]&\bullet\ar[r]&\bullet\ar[r]&\bullet
\\&\bullet\ar[u]^{(\kappa_3)}
&\bullet\ar[u]^{(\kappa_2)}&\bullet\ar[u]^{(\kappa_1)}&
}
\]
\noindent For $i\in m,$ if $n_i=n_i'=\omega,$ and if  such replacements are made to a (finite or infinite) number of tails of $u_i$ and $u_i'$ such that still infinitely many tails do not have such a replacement, then the same holds because the changes do not alter the validity of condition (2).
\label{remark_on_replacing_tails}
\end{remark}

The following lemma uses the assumptions that $E$ and $F$ are cut and I-plus-reduced and it does not hold without either of the two assumptions, as some of our earlier examples show.

\begin{lemma} {\bf The Cut Lemma.}
Let $E=E_{\cut}$ and $F=F_{\cut}$ be I-plus-reduced graphs such that
that there is $\iota: E/H\cong F/G$ which maps $v_0=\so(c)$ onto $v_0'=\so(c').$

If there is a $\POM^D$-isomorphism $f$ of the $\Gamma$-monoids such that $f_{2/1}=\ol\iota,$ that $f([w_0])=[w_0']$ for $w_0=\so(d)$ and $w_0'=\so(d'),$ and that the  connecting matrices, computed for $c,c',d,$ and $d',$ are equal, then $\iota$ can be extended to a graph isomorphism $\iota: E\cong F$ such that $f=\ol\iota.$
\label{lemma_f_identity_implies_cut_graphs_isomorphic}
\end{lemma}
\begin{proof}
The assumptions of the lemma enable us to extend $\iota$ to exits and their ranges and to the vertices and edges of $d$ if $n>0$, and of the suffix of the spine of the tail graph of length $k$ if $n=0.$

We show that $\iota$ can be extended to the rest of the spine (if any) and to the tails of the tail graph. Condition $C(i),$ considered on $E,$ depends only on the part of the graph on which $\iota$ is presently defined. Thus, for every $i\in n$ if $n>0$ (and every $i\leq k_t$ if $n=0$),
we have that $i$-tails of $E$ are cuttable if and only if $i$-tails of $F$ are cuttable.

If $n>0$ and $i\in n,$ let $P^{w_0}_{\not d, i}$ be the set of paths of $P^{w_0}_{\not d}$ of length $n-i+1$ modulo $n.$
If $n=0,$ then $P^{w_0}_{\not d}=P^{w_0}.$ Let $P^{w_0}_{\not d, i}$ be the set of paths of length $i+1$ ending at $w_0$.

Let us partition $P^{w_0}_{\not d, i}$ into two disjoint sets depending on the location of their sources:
\begin{center}
$P^{w_0}_{\quot, i}=\{p\in P^{w_0}_{\not d, i}\mid  \so(p)\in E^0-H\}\;\;$ and
$\;\;P_{\tails, i}^{w_0}=\{ p\in P^{w_0}_{\not d, i}\mid  \so(p)\in H\}.$
\end{center}
The length of a path $p\in P^{w_0}_{\quot, i}$ is $l+k'|c|+d(j,j')+1+n-i'$ for some $i'\in n,$ $l,k'\in\omega,$ and $j,j'\in m$ such that $a_{j'i'}\neq 0$ and $|p|=n-i+1(\mymod n).$  Thus, $C(i)$ holds for such values $l,k', j, j'$ and $i'$ if and only if $|\Pa_l^{v_j}|=\omega.$

Let $P^{w_0'}_{\not d', i}$ be analogously defined for $F$ and let $P^{w_0'}_{\quot, i},$ and $P^{w_0'}_{\tails, i}$ be analogously defined for $P^{w_0'}_{\not d', i}$. The existence of the restriction $f_1$ of $f$ on the $\Gamma$-monoids of the porcupine graphs ensures that there is a bijection
$\rho_{i}: P_{\not d, i}^{w_0}\to P_{\not d', i}^{w_0'}$
and the existence of $\iota$ ensures that such $\rho_{i}$ can be found so that it maps the set of paths in $P_{\quot, i}^{w_0}$ which have exits as their first edges onto the set of such paths $P_{\quot, i}^{w_0'}$. Let us choose one such $\rho_i$ maps as few $\quot$-paths to $\tails$-paths as possible. Since the graphs are I-plus reduced and the cardinalities $|\Pa_{l_p}^{v_{j_p}}|$ and $|\Pa_{l_p}^{v_{j_p}'}|$ are equal by the existence of $\iota$, such interchanges are possible only for paths $p$ with $|\Pa_{l_p}^{v_{j_p}}|=|\Pa_{l_p}^{v_{j_p}'}|$ equal to $\omega$. For such a path $p,$ let $l_p,k'_p, j_p, j'_p$ and $i'_p$ be as above considering $C(i)$.
In this case, the condition $C(i)$ holds in $E$ and so it holds in $F$ and this makes the $i$-tails cuttable in both graphs. As both graphs are in their cut forms, the cardinality $i$-tails is zero. So, the relation $\rho_{i}(p)\in P_{\tails, i}^{w_0'}$ is not possible.

This shows that $\rho_i$ maps $P_{\quot, i}^{w_0}$ onto $P_{\quot, i}^{w_0'}$ and $P_{\tails, i}^{w_0}$ onto $P_{\tails, i}^{w_0'}.$
In particular, for every $i\in n$ for $n>0$ (for every $i\leq k_t$ if $n=0$), we have that $\rho_i$ induces a bijection of $i$-tails. If $\iota_i$ is a bijection of $i$-tails for $i\in n$ ($i\leq k_t$) let us extend $\iota$ to the tails by mapping an $i$-tail $e$ to $\iota_i(e).$ Such extension, let us keep calling it $\iota,$ is an isomorphism $E\cong F$ such that $\ol\iota=f$.
\end{proof}

Having the Quotient Proposition and the Cut Lemma, we start the proof of Theorem \ref{theorem_n=2}.

{\bf \underline{The $\mathbf{m>0}$ case.}}
If $n>0,$ then $E$ and $F$ are single exit-emitters.
In this case, we can choose $c'\in [c']$ so that $f_{2/1}([v_0])=[v_0']$ but it may not be the case that $v_0'$ is the single exit-emitter. If $v_j'$ is the single exit-emitter in $F$, then we consider the canonical form $F_0$ of $(F_{\red\cut})$ in which $v_0'$ is the single exit-emitter and the map $\sigma_0: F\to F_0$ from Lemma \ref{lemma_maps_between_canonical_forms} so that $\ol\sigma_0([v_0'])=[v_0'].$ As $v_0'$ is the single exit-emitter of $F_0,$ we can consider $F_0$ instead of $F$ and $\ol\sigma_0f$ instead of $f$ and assume that $f_{2/1}([v_0])=[v_0']$ and that both $v_0$ and $v_0'$ are single exit-emitters.

As $f_{2/1}([v_0])=[v_0'],$ we have that
\[f([v_0])=t^{lm}[v_0']+b(t)[w_0']\]
for some $l\in \Zset$ and $b(t)\in \Zset^+[t,t^{-1}].$

If $a_E$ and $a_F$ are the connecting polynomials of $E$ and $F$,  we have that $f([v_0])=t^{lm}[v_0']+b[w_0']=t^{(l+1)m}[v_0']+t^{lm}a_F[w_0']+b[w_0']$ by Lemma \ref{lemma_monoid_of_canonical_form}. Thus, while $l$ and $b$ are not uniquely determined, changing $l$ changes $b$ in a specific way: increasing $l$ by one changes $b$ to $t^{lm}a_F+b.$ This shows that we can assume that $l\geq 0.$
We claim that we can choose $b$ to be in $\Zset^+[t].$ If $n>0,$ then $[w_0']=t^n[w_0'],$ so any $t^k[w_0']$ with $k<0$ is equal to $t^{k+k'n}[w_0']$ for $k'\in\Zset^+$ large enough so that $k+k'n\geq 0.$ If $n=0,$ $[w_0']$ is incomparable. The elements of the form $t^k[w_0']$ for $k\in \Zset$ are minimal elements and $D_F$ does not contain an element of the form $t^k[w_0']$ for $k<0$ (see \cite[Proposition 3.4]{Roozbeh_Lia_Comparable}). So, $b\in \Zset^+[t]$.

By Lemma \ref{lemma_monoid_of_canonical_form}, we have that $f([v_0])=f(t^m[v_0]+a_E[w_0])=t^{(l+1)m}[v_0']+t^mb[w_0']+a_E[w_0'].$ As we also have that $f([v_0])=t^{(l+1)m}[v_0']+t^{lm}a_F[w_0']+b[w_0']$ and, as $M_F^\Gamma$ is cancellative, we have that
\begin{equation}t^{lm}a_F[w_0']+b[w_0']=t^mb[w_0']+a_E[w_0'].
\label{equation_the_main_one}
\end{equation}

If $f^{-1}([v_0'])=t^{l'm}[v_0]+b'[w_0],$ the requirement that $f^{-1}(f([v_0]))=[v_0]$ produces
\[t^{(l+l')m}[v_0]+\sum_{j=0}^{l+l'-1}t^{jm}a_E[v_0]=[v_0]=t^{(l+l')m}[v_0]+t^{lm}b'[w_0]+b[w_0].\]
Canceling the term with $[v_0]$ produces the first relation below and using
$f(f^{-1}([v_0']))=[v_0']$ similarly produces the second relation below.
\begin{equation}
(t^{lm}b'+b)[w_0]=\sum_{j=0}^{l+l'-1}t^{jm}a_E[w_0]\hskip2cm
(t^{l'm}b+b')[w_0']=\sum_{j=0}^{l+l'-1}t^{jm}a_F[w_0']
\label{equation_b_and_b_prim}
\end{equation}

Let $|b|$ denote the number of terms of the form $t^i$ for $i\in \Zset^+$ of $b$ (so the $\Zset^+[t]$ monomial of the form $kt^i$ for $k\in\Zset^+$ is considered to be the sum of $k$-terms of the form $t^i$) and let the same notation be used for any polynomial of $\Zset^+[t].$
Equation (\ref{equation_the_main_one}) implies that $|t^{lm}a_F|+|b|=|t^mb|+|a_E|.$ Since $|b|=|t^mb|$ and $|t^{lm}a_F|=|a_F|,$ we have that $|a_F|+|b|=|b|+|a_E|.$ After canceling $|b|,$ this produces the relation $|a_E|=|a_F|.$  So, $c$ and $c'$ have the same number of exits.

Next, we show that the connecting polynomials of $E$ and $F$ are the same in the cycle-to-cycle case and that they are the same up to a multiple of $t^{lm}$ for some $l$ in the cycle-to-sink case.

{\bf The $\mathbf{m>0, n>0}$ case.} Let $L$ be the least common multiple of $m$ and $n$ and $G$ be the greatest common divisor of $m$ and $n.$ Let $m'$ be such that $m=Gm'$ and $n'$ be such that $n=Gn'$.

Let us choose $d\in [d]$ so that  the exits receivers in $E$ are among $w_{n-G+1}$ to $w_{0}.$ Let $i_0\in n$ be such that $w'_{i_0-_n(G-1))}$ to $w_{i_0}'$ receive exits in $F.$
By moving all exits of $F$ by a multiple of $m$ and then $L$-reducing and cutting, if needed, we can assume that $i_0\in \{n-G+1, \ldots, n-1, 0\}.$ Hence $n-_ni_0\in\{0,\ldots, G-1\}.$

With this set up, we have that $a_{E}=a_{00}t+a_{0(0-_n1)}t^2+\ldots+a_{0(0-_n(G-1))}t^G$ holds, so that $a_{E}[w_0]=\sum_{j=1}^Ga_{0(0-_n(j-1))}t^{j}[w_0]$ and $a_F=a_{0i_0}t^{n-_ni_0+1}+a_{0(i_0-_n1)}t^{n-_ni_0+2}+\ldots+a_{0(i_0-_n(G-1))}t^{n-_ni_0+G}.$ Thus, we have that $a_F[w_0']=\sum_{j=1}^{G} a'_{0(i_0-(j-1))}t^{n-_ni_0+j}[w_0'].$

Let $A_1=a_{0(0-_n(G-1))},$ $A_0=a_E-A_1,$ and let us consider $a_E$ as the sum of $n'$ blocks $A_0+A_1t^G+0t^{2G}+\ldots+0t^{(n'-1)G}.$ We consider these blocks because we have that $t^Ga_E[w_0]=(0t^0+A_0t^G+A_1t^{(1+_{n'}1)G}+0+\ldots+0t^{(n'-_{n'}1)G})[w_0].$ Using these blocks, the multiplication of $a_E$ by a term of the form $t^{lm}$ for $l\in \Zset^+$ is well-represented. Let us represent $a_F$ using similar blocks. If $G_0$ is such that $n-_ni_0+G_0=G$, let $A'_0=a_{0i_0}'t^{n-_ni_0+1}+\ldots+a_{0(i_0-_n(G_0-2))}'t^{n-_n(i_0-G_0+2)+1}=a_{0i_0}'t^{n-_ni_0+1}+\ldots+a_{0(n-G-2)}'t^{G-1}$ and $A_1'=a_{0(i_0-_n(G_0-1))}'+\ldots+a_{0(i_0-_n(G-1))}'t^{n-_ni_0}=a_{0(n-G-1))}'+\ldots+a_{0(i_0-_n(G-1))}'t^{n-_ni_0}$ so that we have that $a_F[w_0']=(A_0'+A_1't^G)[w_0']$ and use the same rules for computing $t^{m}a_F[w_0'].$

We represent $b$ using similar blocks. Let $b=\sum_{i\in n}b_it^{i}.$ We can assume that the degree of $b$ is at most $n-1$ since $b$ is always considered either applied to $[w_0]$ or $[w_0']$. So, we have that \[b=\sum_{k'\in n'}\sum_{j\in G}b_{j+k'G}t^{j+k'G}=B_0+B_1t^G+\ldots+B_{n'-_{n'}1}t^{(n'-_{n'}1)G}\mbox{  for }B_{k'}=\sum_{j\in G}b_{j+k'G}t^{j}.\]

Let $m_0<n'$ be such that $m'=m_0(\mymod n')$ (thus $m_0G<n'G=n$ and $m=m_0G(\mymod n)$)
and $n_0<n'$ be such that $n_0=lm'(\mymod n')$ (thus $n_0G<n'G=n$ and $lm=n_0G(\mymod n)$). With these values,
$t^{lm}a_F[w_0']=t^{n_0G}a_{F}[w_0']=(A_0t^{n_0G}+A_1t^{(n_0+_{n'}1)G})[w_0']$ and $t^mb[w_0']=t^{m_0G}b[w_0']=(B_{n'-_{n'}m_0}+B_{(n'-m_0)+_{n'}1}t^G+\ldots +B_{(n'-m_0)-_{n'}1}t^{(n'-_{n'}1)G})[w_0'].$ So, the coefficients with $[w_0']$ on both sides of the equation (\ref{equation_the_main_one}) are
\[
(0+\ldots+0+A_0't^{n_0G}+A_1't^{(n_0+_{n'}1)G}+0+\ldots+0)+(B_0+B_1t^{G}+\ldots+B_{n'-_{n'}1}t^{(n'-_{n'}1)G})\mbox{ and }\]\[(B_{n'-_{n'}m_0}+B_{(n'-m_0)+_{n'}1}t^G+\ldots +B_{(n'-m_0)-_{n'}1}t^{(n'-_{n'}1)G})+(A_0+A_1t^G+0t^{(1+_{n'}1)G}+\ldots+0t^{(n'-_{n'}1)G}).
\]
Equating the terms with $jG$ for  $j\in n'$ produces $n'$ equations. Adding them all up produces an equation which has $\sum_{j\in n'} B_j$ on both sides. Canceling this sum produces
$A_0'+A_1'=A_0+A_1.$ This implies that $a_{0(n-_nj)}=a_{0(n-_nj)}'$ for $j=i_0,\ldots, G-1$ and
$a_{0(n-_nj)}=a_{0(n-_nj-_nG)}'$ for $j=0,\ldots, n-i_0+1.$
By Lemma \ref{lemma_feasibility_of_an_exit_move}, we can move $0(n-_nj-_nG)$-exits of $F$ to $0(n-_nj)$-exits for every $j=0,\ldots, n-i_0+1$. If $F'$ is an $L$-reduced and cut form of the resulting graph and if $\psi: F\to F'$ is the exit move operation followed by an $L$-reduction and the cut map, we have that the connecting matrices of $E$ and $F'$ are equal and that $\ol\psi f([w_0])=[w_0']$. We can consider $F'$ instead of $F$ and $\ol\psi f$ instead of $f$ and we will continue to refer to these new elements as $F$ and $f.$ Thus, we have that $a_E=a_F$ (i.e  $A_0=A_0'$ and $A_1=A_1'$).

Note that if $n_0>1,$ then $n'>2$ and the $(n_0+_{n'}1)$-th equation is $A_0'+B_{n_0}=B_{{n'}-_{n'}m_0+n_0},$ so $A_0'$ is a summand of $B_{n'-_{n'}m_0+n_0}.$
The $(n_0+_{n'}2)$-th equation is $A_1'+B_{n_0+_{n'}1}=B_{n'-_{n'}m_0+n_0+1}$
so $A_1'$ is a summand of $B_{n'-_{n'}m_0+n_0+1}.$ Thus, $t^{(l-1)m}a_F[w_0']=(A_0'+A_1't^G)t^{lm-m}[w_0']=(A_0'+A_1't^G)t^{(n_0-m_0)G}[w_0']$ is a summand of $b[w_0'].$ Thus, we have that
\[f([v_0])=t^{lm}[v_0']+b[w_0']=t^{(l-1)m}(t^m[v_0']+a_F[w_0'])+b_{l-1}[w_0']=t^{(l-1)m}[v_0']+b_{l-1}[w_0']\]
for a summand $b_{l-1}$ of $b.$ Since all the subscripts of the $n'$ equations are considered modulo $n',$ one shows that the same simplification is possible whenever $n_0>0$: if $n_0=1,$ we can simplify $f([v_0])$ in the same way. This simplification enables us to consider $l-1$ instead of $l.$ By renaming the new $l$-value and the new $b$-value, we can continue this process until $l=0$ or $n_0=0.$

Since $l=0$ implies $n_0=0,$ it is sufficient to consider the case $n_0=0.$ Returning the values $n_0=0, A_0=A_0',$ and $A_1=A_1'$ to the $n'$ equations coming out of the equation (\ref{equation_the_main_one}), we have that $B_j=B_{j+_{n'}(n'-m_0)}$ for every $j\in n'.$ If $n$ does not divide $m$, then $m_0\neq 0$ so all blocks of $b$ are mutually equal to each other: $B_j=B_k$ for all $j,k\in n'.$ If $n$ divides $m,$ there is only one block $B_0,$ so this trivially holds. In any case, since $lm'=0(\mymod n')$ and $m'$ and $n'$ are mutually prime, we necessarily have that $l=0(\mymod n').$

If $l'$ and $b'$ are analogous to $l$ and $b$ for $f^{-1},$ and $n_0'$ is analogous to $n_0,$ the same argument shows that $n_0'=0$ so that $n'$ divides $l'.$ So, there are nonnegative integers $k$ and $k'$ such that
$l=kn'$ and $l'=k'n'.$ Hence, $n$ divides both $lm$ and $l'm$ and we have that
$t^{lm}b=b,$ $t^{l'm}b'=b',$ $t^{l'm}b=b,$ and
$t^{lm}b'=b'$ hold in $\Zset[t]/(t^n-1).$ The equations (\ref{equation_b_and_b_prim}) become $(b'+b)[w_0]=\sum_{j=0}^{l+l'-1}t^{jm}a_E[w_0]$ and
$(b'+b)[w_0']=\sum_{j=0}^{l+l'-1}t^{jm}a_E[w_0']$ so we have that
\[(B_0+B_0')\sum_{j\in n'}t^{jG}=b+b'=\sum_{j=0}^{l+l'-1}t^{jm}a_E\] holds in $\Zset[t]/(t^n-1)$ and this  implies that $B_0+B_0'=(k+k')(A_0+A_1)$ holds.
This shows that $B_0$ is equal to $k_0a_{0(0-_n(G-1))}+d_0+\sum_{j\in G, j\neq 0}(k_ja_{0(0-_n(j-1))}+d_j)t^j$ for some nonnegative integers $k_j$ and $d_j<a_{0(0-_n(j-1))}$ for $j\in G, j\neq 0$ and $d_0<a_{0(0-_n(G-1))}.$

Let $E'$ be the cut form of the graph obtained from $E$ by moving each $0(n-_n(j-1))$-exit $Lk_j$ times and moving $d_j$-many of the $0(n-j+1)$-exits $L$ times.
Let $\phi_{+b}: E\to E'$ be this move operation followed  by the cut map. We denote this operation by $\phi_{+b}$ since  $\ol\phi_{+b}([v_0])=[v_0]+b[w_0].$ Thus, the composition $\ol\phi_{+b} f^{-1}$ maps $[w_0']$ to $[w_0]$ and $[v_0']$ to \[t^{l'm}[v_0]+(b'+b)[w_0]=t^{l'm}[v_0]+\sum_{j=0}^{l+l'-1}t^{jm}a_E[w_0]=t^{l'm}[v_0]+\sum_{j=0}^{l'-1}t^{jm}a_E[w_0]+\sum_{j=l'}^{l+l'-1}t^{jm}a_E[w_0]=\]\[[v_0]+\sum_{j=l'}^{l+l'-1}t^{jm}a_E[w_0]=[v_0]+t^{l'm}\sum_{j=0}^{l-1}t^{jm}a_E[w_0]=[v_0]+\sum_{j=0}^{l-1}t^{jm}a_E[w_0].\]
The $L$-reduction of $E'$ followed by the cut map is $\phi_{+b}^{-1}$ and its image is $E$ since $E$ is $L$-reduced and cut.
We have that $\ol\phi_{+b}^{-1}([v_0])= t^{(l+l')m}[v_0]+(\sum_{j=0}^{l+l'}t^{jm}a_E-b)[w_0']=t^{(l+l')m}[v_0]+b'[w_0']$ holds. Thus, $f\ol\phi_{+b}^{-1}([v_0])= t^{(2l+l')m}[v_0']+(b+b')[w_0']=t^{lm}[v_0'].$

Let $F'$ be the graph obtained by moving all exits of $F$ $kL$ times and then cutting the resulting graph (recall that $k=\frac{n}{l}$). Let $\phi_{+kL}:F\to F'$ be this operation. The composition $\ol\phi_{+kL}f\ol\phi_{+b}^{-1}$ is defined and it maps $M_{E'}^\Gamma$ onto $M_{F'}^\Gamma.$
We have that
$\ol\phi_{+kL}f\ol\phi_{+b}^{-1}([v_0])=\ol\phi_{+kL}(t^{lm}[v_0'])=[v_0'].$
By the Quotient Proposition (Proposition \ref{proposition_E_without_H}), we have that $(\ol\phi_{+kL}f\ol\phi_{+b}^{-1})_{2/1}$ is induced by an isomorphism of the quotients. The maps $\phi_{+kL}$ and $\phi_{+b}$ do not impact the quotients, so the quotients of $E$ and $F$ are also isomorphic and $f_{2/1}=(\ol\phi_{+kL}f\ol\phi_{+b}^{-1})_{2/1}=\ol\iota.$ As both $E'$ and $F'$ are in their cut forms, we can use the Cut Lemma (Lemma \ref{lemma_f_identity_implies_cut_graphs_isomorphic}). By this lemma, there is a graph isomorphism $\iota': E'\cong F'$ such that $\iota'(v_j)=v_j'$ and $\iota'(w_i)=w_i'$ for every $j\in m$ and $i\in n$ and that $\ol\phi_{+kL}f\ol\phi_{+b}^{-1}=\ol{\iota'}$ so that
\[f=\ol{\phi_{+kL}^{-1}\iota'\phi_{+b}}\]
holds. This shows that we can realize $f$. So, it remains to show that $E\cong F.$

The graphs $E, F,$ and $E'\cong F'$ have isomorphic quotients and the same $c$-to-$d$ part. So, only the number of tails possibly distinguishes $E, F$ and $E'$. If $t_i$ is the number of tails added to $l_i$ in the process $\phi_{+b}: E\to E'$ and $t_i'$ in the process $\phi_{+kL}: F\to F',$ we have that $l_i+t_i=l_i'+t_i'$ since $E'\cong F'.$ For $i$ such that $t_i=\omega,$ we have that  $t_i'=\omega$ also since the quotients are isomorphic and the connecting matrices equal. The converse $t_i'=\omega \Rightarrow t_i=\omega$ also holds, so the $i$-tails are cuttable in $E$ if and only if they are cuttable in $F$ and $l_i=l_i'=0$ for such $i.$ For $i$ such that $t_i<\omega$ (thus $t_i'<\omega$), we have that $l_i=\omega$ if and only if $l_i=\omega,$ so $l_i=l_i'$ again in this case. Thus, if $l_i=\omega$ for all $i\in n,$ then $E\cong E'\cong F'\cong F,$ so $E'$ and $F'$ are also canonical forms.

Hence, it remains to consider $i$ such that all four values in the formula $l_i+t_i=l_i'+t_i'$
are finite. If $k_ja_{0(0-_n(j-1))}+d_j<ka_{0(0-_n(j-1))}$ for some $j\in G, j\neq 0$ or $k_0a_{0(0-n(G-1))}+d_0<ka_{0(0-_n(G-1))}$, then a $0(0-_n(j-1))$-exit of $F$ for $j\in \{1,\ldots, G\}$ is moved more times by $\phi_{+kL}$ than any $0(0-_n(j-1))$-exit of $E$ by $\phi_{+b}.$ As $F$ is $L$-reduced, this cannot happen if there is $i\in n$ such that $l_i'<\omega$, so $k_ja_{0(0-_n(j-1))}+d_j\geq ka_{0(0-_n(j-1))}$ and $k_0a_{0(0-n(G-1))}+d_0\geq ka_{0(0-_n(G-1))}.$ If $k_ja_{0(0-_n(j-1))}+d_j>ka_{0(0-_n(j-1))}$ for some $j\in G, j\neq 0$ or $k_0a_{0(0-n(G-1))}+d_0> ka_{0(0-_n(G-1))}$,
then a $0(0-_n(j-1))$-exit of $E$ for $j\in \{1,\ldots, G\}$ is moved more times by $\phi_{+b}$ than any $0(0-_n(j-1))$-exit of $F$ by $\phi_{+kL}.$ This also cannot happen if there is $i\in n$ such that $l_i<\omega,$ so we have that $k_ja_{0(n-_n(j-1))}+d_j=ka_{0(n-_n(j-1))}$ for all $j\in G, j\neq 0$
and $k_0a_{0(0-n(G-1))}+d_0=ka_{0(0-_n(G-1))}.$ By the definition of $d_j$ for $j\in G,$ this implies that $d_j=0$ for all $j\in G$ and so $k_j=k$ for all $j\in G.$ This shows that $\phi_{+b}$ and $\phi_{+kL}$ are the same moves and that, in particular, $b[w_0']=
\sum_{j=0}^{l-1}t^{jm}a_E[w_0']$ so that $f([v_0])=[v_0']$. Thus, starting with isomorphic graphs $E'\cong F',$ the maps $\phi_{+b}^{-1}$ and $\phi_{+kL}^{-1}$ produce isomorphic graphs. So, $\phi_{+kL}^{-1}\iota'\phi_{+b}$ is a graph isomorphism $\iota: E\cong F$ and we have that $f=\ol\iota$ holds in this case.

Before continuing with the proof of the next case, let us present some nonidentity $\POM^D$-automorphisms of the graphs from Example \ref{example_moving_exits_invariant}. For both graphs  of that example,
the map given by $f([v_0])=[v_0]+t[w_0]$ and $f([w_0])=[w_0]$ corresponds to moving one of the exits.
For the first graph, $f$ has the inverse such that $f^{-1}([v_0])=t[v_0]+t[w_0]$ and for the second graph,   $f^{-1}([v_0])=t^2[v_0]$.

{\bf The $\mathbf{m>0, n=0}$ case.} Recall that we established that $f_{2/1}([v_0])=[v_0']$. Let $k$ and $k'$ be the lengths of the spines of $E$ and $F.$ Assume that $k\leq k'.$ The case $k\geq k'$ is analogous by considering $f^{-1}$ instead of $f$.

Since $n=0$ and a sink is incomparable, we can write equation (\ref{equation_the_main_one})
as $t^{lm}a_{F}+b=t^mb+a_{E}$ which holds in $\Zset[t].$
By the definition of a canonical form, $a_{E}$ contains no monomial of degree smaller than $k+1$ and of degree larger than $k+m$ and $t^{lm}a_{F}$ contains no monomial of degree smaller than $k'+1+lm$ and of degree larger than $k'+m+lm.$
Thus, $b$ does not contain a monomial of degree smaller than $k+1$ which implies that $t^mb$ does not contain a monomial of degree smaller than $k+m+1.$
Since $k\leq k'$, $k+m+1\leq k'+m+1$ and, if $l>0,$ then
$k'+m+1\leq k'+lm+1,$ so $t^mb$ contains $t^{lm}a_F$ as a summand. Consequently, $b$ contains $t^{(l-1)m}a_F$ as a summand so that \[f([v_0])=t^{lm}[v_0']+t^{(l-1)m}a_{F}[w_0']+b_{l-1}[w_0]=t^{(l-1)m}[v_0']+b_{l-1}[w_0']\] holds for a summand $b_{l-1}$ of $b.$ Thus, we can decrease the value of $l$ if $l>0$ and $b\neq 0.$
By renaming the new $l$-value to $l$ and the new $b$-value to $b,$ we can continue this process until we eventually get that $l=0$ or $b=0.$ If $b=0,$ then $t^{lm}a_F=a_E.$ Since $k'\geq k$, this would be possible only if $l=0$ in which case $a_E=a_F.$ Thus, it remains to consider the case $l=0$ and $b\neq 0.$

Returning $l=0$ to the equation (\ref{equation_the_main_one}) produces $a_F+b=t^mb+a_E$ which implies that the degree of $t^mb$ is at most $m+k'$. Hence, the degree of $b$ is at most $k'.$ The equation (\ref{equation_the_main_one}) for $f^{-1}$ produces $t^{l'm}a_E+b'=t^mb'+a_F$ and it implies that $b'$ does not contain a monomial with degree smaller than $k'+1.$ Hence, all monomials in $b'$ have degrees strictly larger than those in $b.$ The first equation of  (\ref{equation_b_and_b_prim}) implies that $b+b'=\sum_{j=0}^{l'-1}t^{jm}a_E$, so there is $j_0\in l'$ such that $b=\sum_{j=0}^{j_0-1}t^{jm}a_E+b_0$ and $b'=\sum_{j=j_0+1}^{l'-1}t^{jm}a_E+b_0'$ where $b_0$ and $b_0'$ are such that their sum is $t^{j_0m}a_E.$

Let $E'$ be the graph obtained by  moving all exits of $E$ $j_0m$ times and then moving those which correspond to the monomials present in $b_0$ $m$ more times. Let $\phi_{+b}: E\to E'.$ Since all exits of $E$ end in $w_k,$ $b\neq 0$ holds if and only if the spine of $E'$ is longer than $k.$ If $\phi_{+b}^{-1}$ is the inverse operation defined on $E'.$ As \[t^{l'm}[v_0]+b'[w_0]+b[w_0]=t^{l'm}[v_0]+\sum_{j=0}^{l'-1}t^{jm}a_E[w_0]=[v_0]=\phi_{+b}^{-1}\phi_{+b}([v_0])=\phi_{+b}^{-1}([v_0])+b[w_0],\] we can cancel $b[w_0]$ in the first and the last expression and have that $\phi^{-1}_{+b}([v_0])=t^{l'm}[v_0]+b'[w_0].$

The composition $f\ol\phi_{+b}^{-1}: E'\to F$ maps $[v_0]$ to $t^{l'm}[v_0']+(t^{l'm}b+b')[w_0']=t^{l'm}[v_0']+\sum_{j=0}^{l'-1}t^{jm}a_F[w_0']=[v_0'].$
Equation  (\ref{equation_the_main_one}) for this composition, implies that $a_F=a_{E'}.$

As $f\ol\phi_{+b}^{-1}([v_0])=[v_0'],$ and $a_F=a_{E'},$ we can use the Quotient
Proposition to deduce that
$(f\ol\phi_{+b}^{-1})_{2/1}$ is induced by an isomorphism of the quotients. The map $\phi_{+b}$ does not impact the quotients, so the quotients of $E$ and $F$ are also isomorphic and $f_{2/1}=(f\ol\phi_{+b}^{-1})_{2/1}=\ol\iota.$

By the Cut Lemma, we have that $E'\cong F$.
As $E'\cong F$, the spine length of $E'$ is  $k'.$ If $b\neq 0,$ then $k'>k$ and $E'$ can be spine-reduced to $E$. However, since $F\cong E'$ and $F$ {\em is} spine-reduced, this is a contradiction. Hence, $b=0$ and so $k'=k,$ $E'=E,$ and $\iota: E\cong F$ is such that
\[f=\ol\iota.\]

{\bf \underline{The $\mathbf{m=0}$ case.}}
Since $m=0,$ $v_0$ and $v'_0$ are infinite emitters. The relation $f_{2/1}([v_0])=[v'_0]$ implies that $f([v_0])$ has exactly one of the elements $[v_0']$ and $[q_{Z'}]$ for some nonempty and finite $Z'\subseteq \so^{-1}(v_0')$ present in its representation via the generators. Letting $q_\emptyset=v_0',$ we unify the treatment of these two cases so that for any finite $Z\subseteq \so^{-1}(v_0),$ we  write \[f([q_Z])=[q_{Z_F}]+b_Z[w_0']\] for some finite $Z_F\subseteq \so^{-1}(v'_0)$ and some
$b_Z\in \Zset^+[t, t^{-1}].$
Using the same argument as in the $m>0$ case, one shows that $b_Z$ can be taken to be in $\Zset^+[t].$
The map $f^{-1}$ necessarily has the same form as $f$ so, let $f^{-1}([q_{Z'}])=[q_{Z'_E}]+b'_{Z'}[w_0]$ for some $b'_{Z'}\in \Zset^+[t]$ and some finite $Z'_E\subseteq \so^{-1}(v_0).$
As $[q_{(\emptyset_E)_F}]+a_{(\emptyset_E)_F}[w_0']=[v_0']=f(f^{-1}([v_0']))=f([q_{\emptyset_E}]+b'_{\emptyset}[w_0])=[q_{(\emptyset_E)_F}]+b_{\emptyset_E}[w_0']+b'_{\emptyset}[w_0']$ \[(b_{\emptyset_E}+b'_{\emptyset})[w_0']=a_{(\emptyset_E)_F}[w_0']\mbox { holds and, similarly, }(b'_{\emptyset_F}+b_{\emptyset})[w_0]=a_{(\emptyset_F)_E}[w_0]\mbox{ holds.}\] Thus, $b_\emptyset[w_0]=a_{W}[w_0]$ and $b'_\emptyset[w_0']=a_{W'}[w_0']$ for some finite sets $W\subseteq \so^{-1}(v_0)$ and $W'\subseteq \so^{-1}(v_0').$

Let $E'$ be the cut form of the graph obtained by moving the exits in $\emptyset_E$ and $\phi_E: E\to E'$ be the corresponding map. We have that $\ol\phi_E([v_0])=[v_0]+a_{\emptyset_E}[w_0]$ and that $\ol\phi_E^{-1}([v_0])=[q_{\emptyset_E}].$
Let $F'$ be the cut form of the graph obtained by moving the exits in $W'$ and $\phi_F$ the corresponding map so that $\ol\phi_F([v_0'])=[v_0']+a_{W'}[w_0']=[v_0']+b_\emptyset'[w_0']$. Thus, $\ol\phi_Ff\ol\phi_E^{-1}([v_0])=\ol\phi_Ff([q_{\emptyset_E}])=\ol\phi_F([q_{(\emptyset_E)_F}]+b_{\emptyset_E}[w_0'])=[q_{(\emptyset_E)_F}]+(b_{\emptyset_E}+b_\emptyset')[w_0']=[q_{(\emptyset_E)_F}]+a_{(\emptyset_E)_F}[w_0']=[v_0'].$

By Lemma \ref{lemma_monoid_of_canonical_form_m=0}, the relation
$g([v_0])=[v_0']$ for  $g=\ol\phi_Ff\ol\phi_E^{-1},$ implies that for any finite $Z\subseteq \so^{-1}(v_0)$ there is $Z'\subseteq \so^{-1}(v_0')$ such that $a_Z[w_0']=a_{Z'}[w_0']$ and that $g([q_Z])=[q_{Z'}].$ This implies that the connecting matrices of $E'$ and $F'$ are equal.

Thus, the assumptions of the Quotient Proposition are satisfied for $E', F',$ and $g$, and, by this proposition, there is an isomorphism $\iota$ of the quotients such that $g_{2/1}=\ol\iota.$ As $\phi_E$ and $\phi_F$ do not change the quotients, we have that $f_{2/1}=\ol\iota$. So, the quotients of $E$ and $F$ are isomorphic.

The assumptions of the Cut Lemma are satisfied for $E', F'$ and $g$ so $\iota$ can be extended to $\iota: E'\cong F'$ such that $g=\ol\iota.$ Hence,
\[f=\ol{\phi_F^{-1}\iota\phi_E}.\]

Next, we show that $a_{0i}=\omega$ implies $a_{0i}'=\omega$ for all  $i\in n$ if $n>0$ and for all $i\leq k$ if $n=0$. If $a_{0i}=\omega$ and $k\in \Zset^+$ is arbitrarily large, then  $kt^{n-_ni+1}[w_0]$ is a  summand of $[v_0]$ if $n>0$ and $kt^{i+1}[w_0]$ is a  summand of $[v_0]$ if $n=0.$ Since the power of $t$ is the only difference between the two cases, we prove the claim for $n>0$ and the proof for $n=0$ is analogous. So, $n>0$ and we have that $kt^{n-_ni+1}[w_0']=f(kt^{n-_ni+1}[w_0])$ is a summand of $f([v_0])=[q_{\emptyset_F}]+b_\emptyset[w_0'].$ By taking $k>|b_\emptyset|,$ we have that
$(k-|b_\emptyset|)t^{n-_ni+1}[w_0']$ is a summand of $[q_{\emptyset_F}].$ Since $\emptyset_F$ is finite and $k$ can be taken arbitrarily large, this implies that $v_0'$ emits at least $k-|b_\emptyset|-|\emptyset_F|$ paths of length $n-_ni+1$ to $[w_0'].$ Since this holds for any $k,$ $a_{0i}'$ is infinite.
Repeating the same argument for $f^{-1}$ ensures that the converse holds, so $a_{0i}=\omega$ if and only if $a_{0i}'=\omega.$

{\bf The $\mathbf{m=0, n>0}$ case.} In this case, the equivalence $a_{0i}=\omega$ if and only if $a_{0i}'=\omega$ implies that the connecting matrices of $E$ and $F$ are equal and they are equal of those of $E'$ and $F'.$ Since $E$ and $F$ are $\omega$-reduced, $E$ is the $\omega$-reduction of $E'$ and $F$ is the $\omega$-reduction of $F'.$ As isomorphic graph have isomorphic $\omega$-reductions, we have that $E\cong F.$

{\bf The $\mathbf{m=0, n=0}$ case.}
Let $k$ be the spine length of $E$ and $k'$ the spine length of $F.$

If $k=\omega,$ let $l$ be larger than $i$ such that an exit ending in $w_i$ is in $ W.$ Since $b_\emptyset[w_0]=a_W[w_0],$  the degree of $b_\emptyset$  is smaller than $l.$ As $k=\omega,$ for every $j$ arbitrarily larger than $l,$ there is $i\geq j$ such that $v_0$ emits an exit to $w_{i}.$ Then
$t^{i+1}[w_0']=f(t^{i+1}[w_0])$ is a summand of $f([v_0])=[q_{\emptyset_F}]+b_\emptyset[w_0'].$ Since $t^{i+1}[w_0']$ is not a summand of $b_\emptyset[w_0'],$ it is a summand of $[q_{\emptyset_F}].$ As we can do this for arbitrarily large $i$-values, we necessarily have that for some $i$, $w_i'$ is different that the range of edges in $\emptyset_F.$ This shows that $v_0'$ emits an exit to $w_i'.$ Since we can take $j$ to be arbitrarily large, $k'=\omega.$ By considering $f^{-1}$ and repeating this argument, if $k'=\omega,$ then $k=\omega.$ This shows that $k=\omega$ if and only if $k'=\omega.$

If $k<\omega,$ then $v_0$ emits either zero or $\omega$ many exits to $w_i$ for every $i\leq k$ and  $a_{0k}=\omega$ by the definition of the canonical form. Thus,  $a_{0k}'=\omega$ so $k'\geq k.$ Repeating the same argument for $f^{-1},$ we have that  $k'<\omega$ implies $k'\leq k.$
So, if $k<\omega$ then $k'<\omega $ and both $k\leq k'$ and $k'\leq k$ hold. Hence, $k=k'$ in this case.

If $k=k'<\omega$ the proof of the case $m=0, n>0$ applies: $E$ and $F$ have equal connecting matrices by the definition of the canonical form, $E$ is the $\omega$-reduction of $E'$ and $F$ is the $\omega$-reduction of $F'$, so $E'\cong F'$ implies $E\cong F.$

If $k=k'=\omega,$ the relation $E'\cong F'$ implies that the reduced forms of $E'$ and $F'$ are also isomorphic. Such reduced forms are isomorphic canonical forms of $E$ and $F$.

This completes the proof of Theorem \ref{theorem_n=2}.

\begin{remark}
The proof of Theorem \ref{theorem_n=2} implies that if $m=0$ and $E$ and $F$ are cut graphs with isomorphic quotients and same connecting matrices, their $\Gamma$-monoids are isomorphic if and only if their tail graphs are isomorphic. Indeed, assuming that $f$ is an isomorphism of their $\Gamma$-monoids, it is necessary of the form  $f([v_0])=[q_{Z_{\emptyset}}]+b_{\emptyset}[w_0]$ where $Z_{\emptyset}$ are finitely many exits and its moving does not impact the isomorphism class of the graphs. In particular, they end in the vertices with infinitely many tails. So, if they end in $w_i,$ then  $w_i$ has infinitely many tails. If such $w_i$ has no tails, this would force $f([v_0])$ to be $[v_0']$ and then we have that $E\cong F$ by Lemma \ref{lemma_f_identity_implies_cut_graphs_isomorphic}.

Basically, this means that we cannot be moving infinitely many exits. For example, let us consider the graphs $E$ and $F$ below where $l$ is a  countable cardinality.
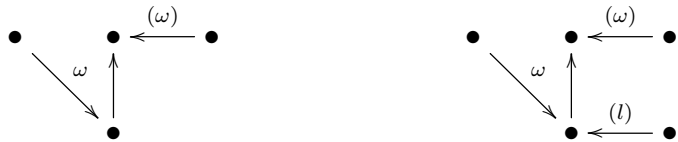
\[\xymatrix{\bullet\ar[dr]^{\omega}&\bullet&\bullet\ar[l]_{(\omega)}\\
&\bullet\ar[u]&
}\hskip3cm
\xymatrix{\bullet\ar[dr]^{\omega}&\bullet&\bullet\ar[l]_{(\omega)}\\
&\bullet\ar[u]&\bullet\ar[l]_{(l)}
}\hskip3cm\]
If $l$ is finite, then it is zero if $F$ is reduced. If $l=\omega,$ the graphs have equal connecting matrices and isomorphic both the quotients and the porcupine graphs. However, their $\Gamma$-monoids are not $\POM^D$-isomorphic: the existence of a $\POM^D$-isomorphism $f$ would imply that moving finitely many exits would make the graphs isomorphic.

This argument also shows that if $f(D_E)$ were to be  matched with $D_F,$ the terms of the form $t^2[w_0']$ which appear as summands of some $[q_Z]$ are distinguished from the terms of the form $t^2[w_0']$ which correspond to the sources of the 1-tails in $F$. As there are no such terms in $E$, no such matching is possible.
\label{remark_infinite_outsplit}
\end{remark}

\section{Composition S-NE graphs}
\label{section_n>2}

\subsection{Breaking-vertices-free forms of composition S-NE graphs and cycle schemes}\label{subsection_breaking_vertices-free}

Let $E=E_{\tot}$ be a countable $n$-S-NE graph for any $n\geq 1$ with a composition series as below.
\[(\emptyset, \emptyset)\subsetneq (H_1,S_1)\subsetneq (H_2, S_2)\subsetneq \ldots \subsetneq (H_{n-1}, S_{n-1})\subsetneq (H_n, \emptyset)=(E^0, \emptyset)\]

We claim that we can consider a graph with a composition series such that the sets of breaking vertices $B_{H_j}$ are empty for all $j=1,\ldots, n-1$ ($H_n=E^0$ so $B_{H_n}=\emptyset$).
If $B_{H_1}\neq \emptyset,$
and $v\in B_{H_1},$ then $v$ emits infinitely many edges to $H_1$ and nonzero and finitely many edges to any of $H_{j+1}-H_j, j=1,\ldots, n-1.$ Let us consider an out-split with respect to the partition of $\so^{-1}(v)$ which has each of those finitely many edges in different partition sets and $\so^{-1}(v)\cap \ra^{-1}(H_1)$ in the remaining partition set. Since $E$ has finitely many infinite emitters, $B_{H_1}$ is finite, so we can repeat this for any element of $B_{H_1}.$ In the end, we obtain a graph with $B_{H_1}=\emptyset.$ By considering  $E/H_1$ and using induction, we can have $B_{H_j}=\emptyset$ for all $j.$ So, we can assume that $E$ has a composition series as below.
\[(\emptyset, \emptyset)\subsetneq (H_1,\emptyset)\subsetneq (H_2,\emptyset)\subsetneq \ldots \subsetneq (H_{n-1}, \emptyset)\subsetneq (H_n, \emptyset)=(E^0, \emptyset).\]

The result of this process for an $n$-S-NE graph $E$ is a graph which we say is the {\em breaking-vertices-free form of $E.$} For example, the second 3-S-NE graph is
the breaking-vertices-free form of the first graph and the fourth is such form of the third one.
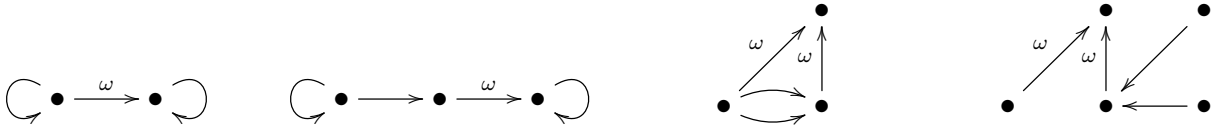
\[\xymatrix{\\\bullet\ar@(lu,ld)\ar [r]^{\omega}&\bullet\ar@(ru,rd)}\hskip2cm
\xymatrix{\\\bullet\ar@(lu,ld)\ar[r]&\bullet\ar [r]^{\omega}&\bullet\ar@(ru,rd)}\hskip2cm\xymatrix{&\bullet\\\bullet\ar[ur]^{\omega}\ar@/_/ [r]\ar@/^/ [r]&\bullet\ar[u]^{\omega}}\hskip2cm\xymatrix{&\bullet&\bullet\ar[dl]\\\bullet\ar[ur]^{\omega}&\bullet\ar[u]^{\omega}&\bullet\ar[l]}\]

Let $c_{j+1}$ be a terminal cycle of $H_{j+1}/H_j$ for $j=0,\ldots, n-1.$
If $c_l$ emits an edge to $H_j$ for $1\leq j<l\leq n$ such that its range is not in $H_{j'}$ for any $j'<l,$ then there is a path originating in $c_l$ terminating at $c_j$ and such that no edge is on $c_l$ or $c_j$. We say that such path is a {\em $c_l$-to-$c_j$ path}  or, for brevity an $lj$-path. The first edge of such path is an {\em $lj$-exit.} Also, for brevity, we let $m_j$ stands for $|c_j|,$ and $v_{ji},$ $j=1,\ldots,n$ and $i\in m_j$ be vertices of $c_j.$

We determine the order of consideration of the elements of $E$. To simplify the approach, let us introduce a finite acyclic graph $E_c$ (``c'' for cycles) which represents how the cycles of $E$ are connected with each other and which we call the {\em cycle scheme} of $E$. The vertices of $E_c$ are $v_{c_l}$ where $c_l$ is the terminal cluster of a composition factor of $E$ and $ v_{c_l}$ emits an edge $e_{lj}$ to $v_{c_j}$ if there is a path from $c_l$ to $c_j$ in $E.$

For example, $\xymatrix{\bullet\ar[r]&\bullet}$ is the cycle scheme for every 2-S-NE graph with one terminal cluster and the graphs below are all possible cycle schemes for a 3-S-NE graph (up to graph isomorphisms).
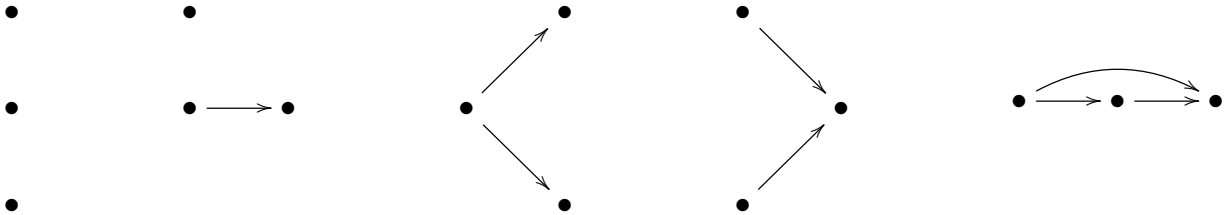
\[\xymatrix{\bullet\\\bullet\\\bullet}\hskip1.9cm\xymatrix{\bullet&\\\bullet\ar[r]&\bullet}\hskip1.9cm\xymatrix{&\bullet\\\bullet\ar[ur]\ar[dr]&\\&\bullet}\hskip1.9cm\xymatrix{\bullet\ar[dr]&\\&\bullet\\\bullet\ar[ur]&}\hskip1.9cm\xymatrix{&&\\\bullet\ar[r]\ar@/^1pc/[rr]&\bullet\ar[r]&\bullet}\]

The consideration of 3-S-NE graphs with any but the last cycle scheme above reduces to the consideration of 2-S-NE graphs. For example, if a graph has the third cycle scheme above, then changing one 2-S-NE quotient to its 2-S-NE canonical form does not impact the other 2-S-NE quotient and vice versa.  Analogously, if $E$ is an $n$-S-NE graph, its consideration boils down to that of only $(n-1)$-S-NE graphs in all cases except when
the cycle scheme contains a path of length $n$, i.e. when $c_{j+1}$ emits an edge to $c_j$ for all $j=1,\ldots, n-1.$ Because of this, we focus primarily this most demanding scenario.

Let us also note that if the cycle scheme of an $n$-S-NE graph $E$ is such that the length of a path from any source to any sink is one,
this case involves consideration of only 2-S-NE graphs and requires the use of only Theorem \ref{theorem_n=2}. The object of the consideration in \cite{Do_et_al_Williams} are such graphs without infinite emitters and sinks.

\subsection{Direct-exit forms and the sink-graphs of a 3-S-NE graph}
\label{subsection_direct_exit_n=3}
We start with a 3-S-NE graph $E$ with the cycle scheme $
\xymatrix{\bullet\ar[r]\ar@/^/[rr]&\bullet\ar[r]&\bullet}.$ If $H_1\subseteq H_2$ are the two nontrivial and proper terms of the composition series of $E$. We refer to $E/H_1$ as the {\em 2-quotient} and $E/H_2$ as the {\em 1-quotient} of $E$ (because the labels ``1'' and ``2'' reflect the length of composition series of these quotients).  We aim to transform $E$ so that its 32-connecting paths have a certain direct-exit form.

The first step of this process is to make the 1-quotient $E/H_2$ canonical quotient of the 2-quotient $E/H_1$ and we say that $E/H_2$ is {\em 1-canonical} in $E$. Recall that this is achieved by moving the 32-exits so that only one vertex of $c_3$ emits them, rearranging the paths of the 1-S-NE sink-graph so they are canonical, and making the 1-quotient I-plus-reduced.
After this we reduce the graph as much as possible with respect to the moves of the 32-exits in the sense that 2-S-NE graph $E/H_1$ is reduced. Note that such reduction does not impact the 1-quotient any more and that the 32-exits may not be emitted from the same vertex any more.

For each $i\in m_2$ (recall that this means $i=0$ if $m_2=0$), we use the 21-blow-ups to make $v_{20}$ be the single exit-emitter. Note that this step is needed only if $m_2>1$ and that it depends on the choice of $c_2\in [c_2]$. After this, we consider the sink-graph of the 1-S-NE graph $H_2/H_1$ and rearrange the 32-connecting paths using Proposition \ref{proposition_canonical_tails} with $V$ being the set of vertices in $R(c_3^0)$. Thus, only impacted vertices are in $H_2-H_1-\{c_2^0\}.$ In particular, the exit-emitter $v_{20}$ is not impacted. Using the last sentence of
Proposition \ref{proposition_canonical_tails}, we can have the part of $c_2$ from $v_{21}$ (where ``1'' is considered modulo $m_2$) to $v_{20}$ be on the spine of the connecting paths.
After such rearranging, we consider the graph together with its 1-quotient again and this brings us to the sink-graph $S_{10}$: it is the graph obtained by considering the 2-quotient but without the edge $v_{20}$ emits in $c_2$. This consideration parallels the case $n=2$.

The spine of $S_{10}$ is the {\em 2-spine} and we use $k_2$ for its length. The part of the 2-spine which received exits from $c_3$ is {\em the 32-spine} and we use $k_{32}$ for its length. Hence, $k_{32}\leq k_2$. If $m_3>0,$ the 32-spine length is necessarily finite and if $m_3=0$ it can be finite or infinite.

For example, if $E$ is the first graph and $v_{20}$ is the range of $e$, the second graph below is the sink-graph $S_{10}$ and $e$ is its 32-spine.
The 2-spine is the left-infinite path ending in $c_2.$ So, in this example, $k_{32}=1< k_2=\omega.$
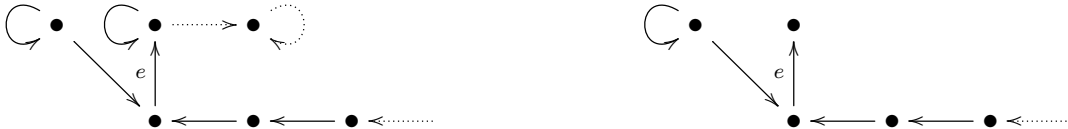
\[
\xymatrix{
\bullet\ar@(ul, dl)
\ar[dr]&\bullet\ar@{.>}[r]\ar@(ul, dl)&\bullet\ar@{.>}@(ur, dr)&&
\\
&\bullet\ar[u]^e&
\bullet\ar[l]&\bullet\ar[l]&\ar@{.>}[l]}
\hskip3cm
\xymatrix{
\bullet\ar@(ul, dl)
\ar[dr]&\bullet&&&
\\
&\bullet\ar[u]^e&
\bullet\ar[l]&\bullet\ar[l]&\ar@{.>}[l]}\]

The {\em 32-connecting matrix} is the connecting matrix of the sink-graph (note it depends on the element of $[c_2]$ we consider). For example, the 32-connecting matrix of the graph above is $[0\; 1]$.

When paths are rearranged so that the sink-graph is $S_{10},$ the graph is in its
{\em 32-direct-exit form} $E_{32-\direct}.$  For example, let $E$ be the first graph below and let $v_{20}$ be the single exit-emitter of 21-exits. The second graph is a direct-exit form with $S_{20}$ as its sink-graph. The changes resulting from $E\to E_{32-\direct}$ operation in the part of the graph outside of the 2-quotient are not shown in the diagram. The dashed arrow indicates the edge which is removed when the sink-graph is considered. The 32-spine has length two.
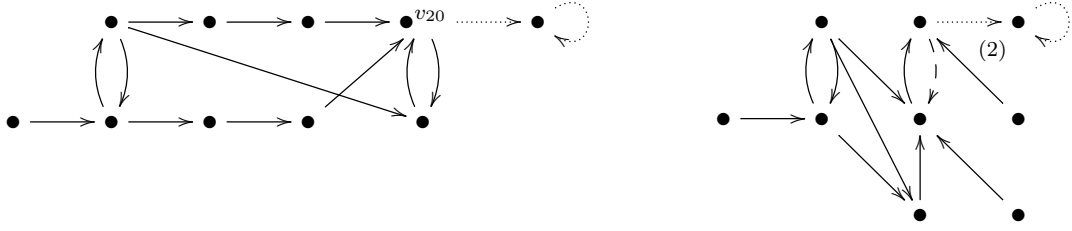
\[
\xymatrix{&\bullet\ar@/^/[d]\ar[r]\ar[drrr]&\bullet\ar[r]&\bullet\ar[r]&\bullet^{v_{20}}\ar@/^/[d]\ar@{.>}[r]&\bullet\ar@{.>}@(ur, dr)
\\
\bullet\ar[r]&\bullet\ar[r]\ar@/^/[u]&\bullet\ar[r]&\bullet\ar[ur]&\bullet\ar@/^/[u]&}
\hskip2cm
\xymatrix{&\bullet\ar@/^/[d]\ar[dr]\ar[ddr]&\bullet\ar@{-->}@/^/[d]\ar@{.>}[r]&\bullet\ar@{.>}@(ur, dr)\\
\bullet\ar[r]&\bullet\ar[dr]\ar@/^/[u]&\bullet\ar@/^/[u]&\bullet\ar[ul]_{(2)}\\&&\bullet\ar[u]&
\bullet\ar[ul]}
\]

Let us also consider an example with $m_3=0$. Let $E$ be the first graph below where $v_{30}$ emits one line of length $n$ for every $n\in \omega$ to $v_{20},$ the labels L1, L2 and L3 shorten ``length 1'', ``length 2'', and ``length 3'' and there are $\omega$ of each of those. The second graph is a direct-exit form $E_{32-\direct}$ of $E$ with $S_{20}$ as the sink-graph (compare $E$ and $E_{32-\direct}$ with the example before Proposition  \ref{proposition_canonical_tails}).
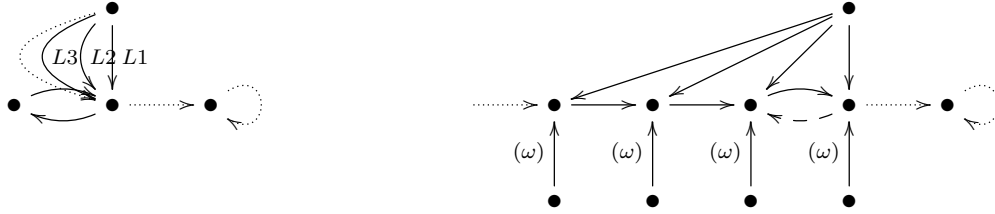
\[
\xymatrix{&\bullet\ar[d]^{L1}\ar@/_1pc/[d]^{L2}\ar@/_2.2pc/[d]^{L3}\ar@{.>}@/_2.9pc/[d]\\
\bullet\ar@/^/[r]&\bullet\ar@/^/[l]\ar@{.>}[r]&\bullet\ar@{.>}@(ur, dr)}
\hskip3cm
\xymatrix{&&&&\bullet\ar[d]\ar[dl]\ar[dll]\ar[dlll]&\\
\ar@{.>}[r]&\bullet\ar[r]&\bullet\ar[r]&\bullet\ar@/^/[r]&\bullet\ar@{-->}@/^/[l]\ar@{.>}[r]&\bullet\ar@{.>}@(ur, dr)\\
&\bullet\ar[u]^{(\omega)}&\bullet\ar[u]^{(\omega)}&\bullet\ar[u]^{(\omega)}&\bullet\ar[u]^{(\omega)}&}\]

\subsection{I-plus-reduction}
Next, we ensure that the 32-spine is as short as possible by considering in-split plus moves analogously to the case $n=2$.

We introduce some notation to shorten the arguments. The notation matches that of the $n=2$ case. We let $u_i, i\leq k_2$ be vertices on the 2-spine and $n_i$ be $|\ra^{-1}(u_i)|.$ If $a_i$ is the $i$-th entry of the 32-connecting matrix, then $a_i$ corresponds to the cardinality $|\ra^{-1}(u_i)\cap \so^{-1}(E^0-H_2)|.$ We denote the cardinality of  the difference
\[\ra^{-1}(u_i)-\left(\ra^{-1}(u_i)\cap \so^{-1}(E^0-H_2)\right)=\ra^{-1}(u_i)\cap \so^{-1}(H_2)\]
by $n_i-a_i$.
If both $n_i$ and $a_i$ are infinite, there is a slight abuse of cardinal arithmetic here because $\omega+l=\omega$ for any finite or countably infinite cardinality $l$ so we cannot
write that $\omega-\omega=l.$ Because of this, the label  $n_i-a_i$ is used only to denote {\em the cardinality of the set} $\ra^{-1}(u_i)\cap \so^{-1}(H_2).$
If $i=k_2,$ partitioning the set $\ra^{-1}(u_i)$ as the disjoint union of $\ra^{-1}(u_i)\cap \so^{-1}(E^0-H_2)$ and $\ra^{-1}(u_i)\cap \so^{-1}(H_2)$ we represent by writing $n_i$ as $a_i+(n_i-a_i)$. If $i<k_2,$ then $u_i$ receives the edge on the spine from $u_{i+1},$ so we partition
the set $\ra^{-1}(u_i)$ as
\[\ra^{-1}(u_i)=\left(\ra^{-1}(u_i)\cap\so^{-1}(u_{i+1})\right)\cup \left(\ra^{-1}(u_i)\cap\so^{-1}(E^0-H_2)\right)\cup \left(\ra^{-1}(u_i)\cap\so^{-1}(H_2-\{u_{i+1}\})\right).\]

The first partition set contains one edge $u_i$ receives on the spine and it has cardinality one. The second partition set is the set of exits $u_i$ receives and it has cardinality $a_i.$ The third partition set is the set of tails $u_i$ has in $H_2/H_1$ which remain tails in $E$. As the three partition sets are disjoint, we can consider their cardinalities and write
\[n_i=1+a_i+(n_i-a_i-1)\]
where we use convention above when writing the last term -- it only denotes the cardinality of the set $\ra^{-1}(u_i)\cap\so^{-1}(H_2-\{u_{i+1}\}).$

Conditions (1) and (2) which are sufficient for the spine extending in the $n=2$ case are analogous to the two conditions below. In (1), $i\in m_1$ if $m_1>0$ and $i\leq k_1$ if $m_1=0$.
\begin{enumerate}
\item For every $21$-exit $e$ ending at $v_{1i},$ there is a tail of $v_{1i}.$

\item Every vertex of $c_2$ has at least one tail.
Thus, for $i\in m_2$, $n_i-a_i-1>0$ and, if $k_{32}>m_2-1,$ then $n_{m_2-1}-a_{m_2-1}-1>1$.
\end{enumerate}

If $E$ is a graph which satisfies conditions (1) and (2), we elaborate on how the 32-spine is extended and what the resulting graph is. First, let us move the entry point of the spine from $v_{21}$ to $v_{20}.$ This is possible by condition (2) since $v_{2i}$ has a tail for every $i\neq 1$ and it is done in the same way as for 1-quotient of a 2-S-NE graph. The diagram below illustrates the move of the entry point from $v_{2(m_2-1)}$ to $v_{20}$. If we consider the in-split minus moves of the first graph with respect to the partitions $\{\ra^{-1}(\ra(e)), \emptyset\}$ and $\{\ra^{-1}(\ra(e))-\{e\}, \{e\}\}$ followed by the total out-splits, then we obtain the second and the third graph below.
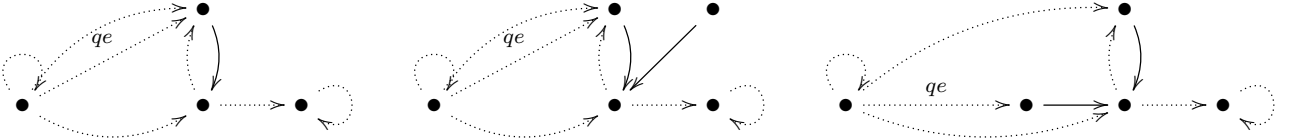
\[
\xymatrix{&&\bullet\ar@/^/[d]&\\
\bullet\ar@{.>}@/^1pc/[urr]\ar@{.>}@/_1pc/[rr]\ar@{.>}^{qe}[urr]\ar@{.>}@(ul, ur)&&\bullet\ar@{.>}[r]\ar@{.>}@/^/[u]&\bullet\ar@{.>}@(ur, dr)}
\hskip1.3cm
\xymatrix{&&\bullet\ar@/^/[d]&\bullet\ar[dl]\\
\bullet\ar@{.>}@/^1pc/[urr]\ar@{.>}@/_1pc/[rr]\ar@{.>}[urr]^{qe}\ar@{.>}@(ul, ur)&&\bullet\ar@{.>}[r]\ar@{.>}@/^/[u]&\bullet\ar@{.>}@(ur, dr)}\hskip1.3cm
\xymatrix{&&&\bullet\ar@/^/[d]&\\
\bullet\ar@{.>}@/^1pc/[urrr]\ar@{.>}@/_1pc/[rrr]\ar@{.>}[rr]^{qe}\ar@{.>}@(ul, ur)&&\bullet\ar[r]&\bullet\ar@{.>}[r]\ar@{.>}@/^/[u]&\bullet\ar@{.>}@(ur, dr)}\]
In the figure above, the terminal and the source cycle can be proper or improper, the length of the path $q$ can be zero, and the diagrams do not show non-relevant parts outside of the 2-quotient.

So, if $E$ is a graph which satisfies conditions (1) and (2), we can move the entry point of the spine to $v_{00}$ using condition (2). Let $E'$ be the resulting graph. Then, a tail at $v_{21}$ and condition (1) are needed for the in-split plus move. Let $M$ be the mother graph which is the same as $E'$ but without a tail in $v_{21}$ and without the tails from condition (1).
In this case, $E'$ is  the total out-split of the in-split minus of $M$ with respect to $\{\ra^{-1}(v_{20}), \emptyset\}.$ The first two graphs below are $M$ and $E'.$
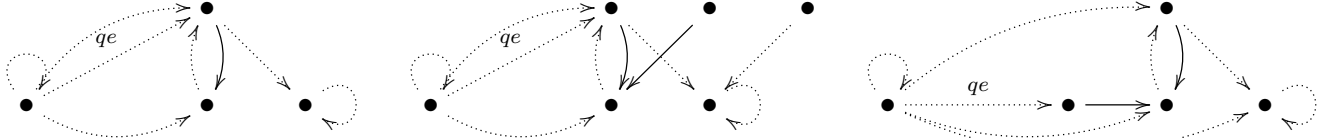
\[
\xymatrix{&&\bullet\ar@/^/[d]\ar@{.>}[dr]&\\
\bullet\ar@{.>}@/^1pc/[urr]\ar@{.>}@/_1pc/[rr]\ar@{.>}^{qe}[urr]\ar@{.>}@(ul, ur)&&\bullet\ar@{.>}@/^/[u]&\bullet\ar@{.>}@(ur, dr)}\hskip1.2cm
\xymatrix{&&\bullet\ar@/^/[d]\ar@{.>}[dr]&\bullet\ar[dl]&\bullet\ar@{.>}[dl]\\
\bullet\ar@{.>}@/^1pc/[urr]\ar@{.>}@/_1pc/[rr]\ar@{.>}[urr]^{qe}\ar@{.>}@(ul, ur)&&\bullet\ar@{.>}@/^/[u]&\bullet\ar@{.>}@(ur, dr)&}
\hskip.6cm
\xymatrix{&&&\bullet\ar@/^/[d]\ar@{.>}[dr]&\\
\bullet\ar@{.>}@/_1.5pc/[rrrr]\ar@{.>}@/^1pc/[urrr]\ar@{.>}@/_1pc/[rrr]\ar@{.>}[rr]^{qe}\ar@{.>}@(ul, ur)&&\bullet\ar[r]&\bullet\ar@{.>}@/^/[u]&\bullet\ar@{.>}@(ur, dr)}\]

If $e$ is the last edge of the spine not in $c_2$ in $E'$, we consider the total out-split of the in-split minus move on $M$ with respect to $\{\ra^{-1}(v_{20})-\{e\}, \{e\}\}$ and obtain a graph $F'$ whose total out-split is $F$, represented as the third graph above. We say that $F$ is the {\em 32-spine extending} of $E$.

If $k_{32}<m,$ then all exits are being emitted to $c_2$. In this case, if conditions (1) and (2) hold, the $2$-spine can be extended without $32$-spine being impacted. The extension of 2-spine matches the extension in the $n=2$ case. In this case, no new 31-paths are created. For example, if $E$ is the first graph and there are enough tails in the dotted part that which permit 2-spine extension, then the second graph is the result of such an extension.

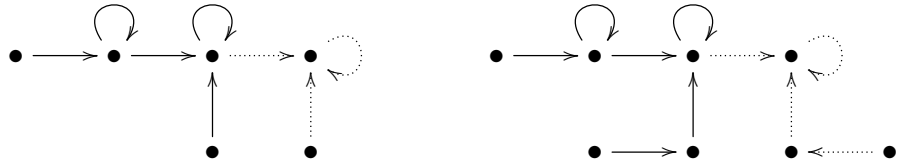
\[\xymatrix{\bullet\ar[r]&\bullet\ar@(ul, ur)\ar[r]&\bullet\ar@{.>}[r]\ar@(ul,ur)&\bullet\ar@{.>}@(ur, dr)\\&
&\bullet\ar[u]&\bullet\ar@{.>}[u]}\hskip2cm
\xymatrix{\bullet\ar[r]&\bullet\ar[r]\ar@(ul, ur)&\bullet\ar@{.>}[r]\ar@(ul, ur)&\bullet\ar@{.>}@(ur, dr)&\\&\bullet\ar[r]&\bullet\ar[u]&\bullet\ar@{.>}[u]&\bullet\ar@{.>}[l]}\]

With this set-up, the difference in conditions (2) for the 32-spine and for the 2-spine extension can possibly come from the case when $n_i=a_i+1$ for $i\in m_2$. For example, if $m_3=0$ and  $a_i=n_i=\omega$ for all $i\in m_2$, then  $n_i=a_i+1.$

If the 32-spine can be extended and if there is an exit with the range $u_i$ outside of $c_2$ (so $i>m-1$) then a 32-spine extension makes this exit end in the vertex of the spine with distance $i+m.$ So, if $E$ permits arbitrary 32-spine extension, then it also permits 32-spine reduction up to when the exit with smallest such $i$ value is such that $m\leq i< 2m$.
In this case, we say that the 32-spine is {\em as short as possible}.
If the 32-spine cannot be extended and 2-spine can be arbitrarily extended, then we can still require the part without any tails (the line of length $lm_2$ which extends the spine if $l$ extensions are made) to be absent for the 2-spine to be {\em as short as possible}. This brings us to the following two conditions and we say that $E$ is {\em 32-I-plus-reduced} if either one of them holds. Note that (a) implies that 32-spine cannot be extended and that (b) implies that 32-spine is as short as possible.
\begin{enumerate}[\upshape(a)]
\item The 2-spine cannot be extended.\hskip.6cm (b) The 2-spine can be extended and it is as short as possible.
\end{enumerate}

\subsection{O-reduction and canonical quotients}
Besides the I-plus-reduction, we also consider the out-split reduction which parallels the reduction of a 2-S-NE graph.

If $E$ is the direct-exit form of a blow-up of an $32$-$kl$-exit of $F=F_{32-\direct}$ for some $k\in m_3$ and $l\in m_2$, we write $F\to_1E$ and say that $E$ is {\em $32$-$kl$-reducible}. For example, if $E$ is the first graph below, the second one is a blow-up of $E$ and the third one is obtained by rearranging the paths. So, the first graph is a 32-reduction of the last one.

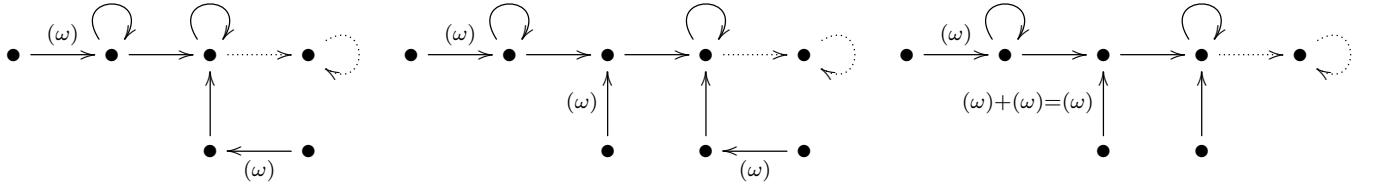
\[
\xymatrix{
\bullet\ar[r]^{(\omega)}&\bullet\ar@(ul, ur)
\ar[r]&\bullet\ar@{.>}[r]\ar@(ul, ur)&\bullet\ar@{.>}@(ur, dr)
\\
&&\bullet\ar[u]&
\bullet\ar[l]^{(\omega)}}
\hskip.9cm
\xymatrix{
\bullet\ar[r]^{(\omega)}&\bullet\ar@(ul, ur)
\ar[r]&\bullet\ar[r]&\bullet\ar@{.>}[r]\ar@(ul, ur)&\bullet\ar@{.>}@(ur, dr)
\\
&&\bullet\ar[u]^{(\omega)}&\bullet\ar[u]&
\bullet\ar[l]^{(\omega)}}
\hskip.9cm
\xymatrix{
\bullet\ar[r]^{(\omega)}&\bullet\ar@(ul, ur)
\ar[r]&\bullet\ar[r]&\bullet\ar@{.>}[r]\ar@(ul, ur)&\bullet\ar@{.>}@(ur, dr)
\\
&&\bullet\ar[u]^{(\omega)+(\omega)=(\omega)}&\bullet\ar[u]&}\]

We say that $E$ is {\em $32$-reducible} if it is $32$-$kl$-reducible for some $k$ and $l.$
If the 2-quotient changes after a 32-blow-up and path rearrangements, then $E$ is
{\em $32$-O-reduced} if it is not $32$-reducible. Otherwise, $c_3$ is necessarily not proper (because a blow-up of an exit of a proper cycle changes the 32-connecting part) and such graph is also {\em $32$-O-reduced}. For example, the 2-quotient of the graph below is invariant to a 32-blow-up followed by the 32-path rearrangements. The second graph below is a blow-up of one of the 32-exits of the first. The third graph is obtained by path rearrangement of the second and it is isomorphic to the first. Hence, the first graph is 32-O-reduced.
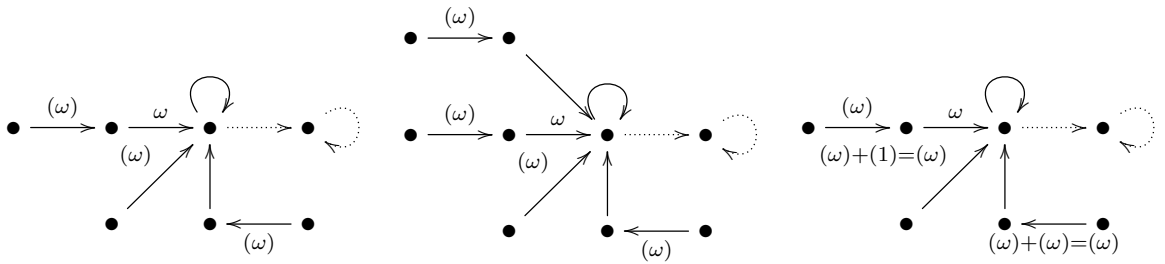
\[
\xymatrix{\\
\bullet\ar[r]^{(\omega)}&\bullet
\ar[r]^{\omega}&\bullet\ar@{.>}[r]\ar@(ul, ur)&\bullet\ar@{.>}@(ur, dr)
\\
&\bullet\ar[ur]^{(\omega)}&\bullet\ar[u]&
\bullet\ar[l]^{(\omega)}}
\hskip.9cm
\xymatrix{\bullet\ar[r]^{(\omega)}&\bullet\ar[dr]&&\\
\bullet\ar[r]^{(\omega)}&\bullet
\ar[r]^{\omega}&\bullet\ar@{.>}[r]\ar@(ul, ur)&\bullet\ar@{.>}@(ur, dr)
\\
&\bullet\ar[ur]^{(\omega)}&\bullet\ar[u]&
\bullet\ar[l]^{(\omega)}}
\hskip.9cm
\xymatrix{\\
\bullet\ar[r]^{(\omega)}&\bullet
\ar[r]^{\omega}&\bullet\ar@{.>}[r]\ar@(ul, ur)&\bullet\ar@{.>}@(ur, dr)
\\
&\bullet\ar[ur]^{(\omega)+(1)=(\omega)}&\bullet\ar[u]&
\bullet\ar[l]^{(\omega)+(\omega)=(\omega)}}\]

To unify the terminology, let us say that a 2-S-NE graph is 21-I-plus-reduced if it is I-plus-reduced in the sense used in section \ref{section_n=2}.
A 3-S-NE graph $E=E_{32-\direct}$ is {\em I-plus-reduced} if it is  32-I-plus-reduced and its 2-quotient is 21-I-plus reduced.
A 3-S-NE graph $E=E_{32-\direct}$ is {\em O-reduced} if it is 32-O-reduced. If $E$ is both I-plus and O-reduced and if $v_{2i}$ is the single emitter of 21-exits, $E$ has the {\em $2$-$i$-canonical quotient} and we write $E=E_{\can\quot, i}$ in this case.
Just like 1-quotients of 2-S-NE graphs, the 2-quotients of the graphs $E_{\can\quot, i}$ and $E_{\can\quot, i'}$ are equivalent but not necessarily isomorphic.

\subsection{Canonical quotients of composition S-NE graphs}
\label{subsection_quotients_n>3}
We generalize the construction of a canonical quotient in $n\leq 3$ to the general case. Let $E$ be an $n$-S-NE graph with a composition series with $\emptyset\subsetneq H_1\subsetneq\ldots \subsetneq H_{n-1}\subsetneq H_n=E^0.$ Let $c_k, k=1,\ldots, n$ be the terminal cycles of composition factors, let $m_k=|c_k|,$ and let us assume that $c_{k+1}$ emits exits to $c_k$ for all $k=1,\ldots, n-1.$

We can ensure that the 1-S-NE graph $E/H_{n-1}$ is a 1-canonical quotient of 2-S-NE graph $E/H_{n-2}$ by the same argument used for $n=2.$ Then, we consider the 3-S-NE graph $E/H_{n-3}$ and make its 2-quotient $E/H_{n-2}$ be canonical. Since the $(n-1)(n-2)$-exits were moved to a single exit emitter to achieve this, we $(n-1)(n-2)$-O-reduce the graph again. This does not impact the 2-quotient any more and we proceed to the next step paralleling the previous one:
we move the $(n-2)(n-3)$-exits to start at a single vertex and we consider the 3-S-NE graph $E/H_{n-3}$ as the quotient of the 4-S-NE graph
$E/H_{n-4}.$ In this graph, we rearrange
the $(n-1)(n-2)$-paths and $n(n-2)$-paths by using Proposition \ref{proposition_canonical_tails} with $V$ being the set of vertices not in $H_{n-3}$ (i.e. the vertices in $R(c_{n-2}^0)$) and we  rearrange the connecting paths so that all the exits are emitted directly to the spine of the sink-graph $S_{10}$ of the 3-quotient $E/H_{n-3}$. In this way, we have the {\em $n(n-2)$-spine} and its length $k_{n(n-2)},$ we have the {\em $(n-1)(n-2)$-spine} and its length  $k_{n(n-2)},$ and we have the spine of the sink-graph of the quotient $H_{n-2}/H_{n-3}$ and its length $k_{n-2}.$
The first two spines are suffices of the last one. So,
$k_{n(n-2)}\leq k_{n-2}$ and $k_{(n-1)(n-2)}\leq k_{n-2}.$
The lengths $k_{n(n-2)}$ and $k_{(n-1)(n-2)}$ can be in any of the relations $k_{n(n-2)}\leq k_{(n-1)(n-2)}$ and $k_{n(n-2)}\geq k_{(n-1)(n-2)}$. With such path rearrangement, we say that the $(n-2)$-quotient is {\em direct-exit}.

For example, if $E$ is the first graph below, its 3-quotient is not direct-exit because the 32-spine is not a suffix of the 42-spine. The second graph below is the form of $E$ with 3-quotient being a direct-exit graph. In this case, we have that $k_{32}=1<k_{42}=k_2=2.$
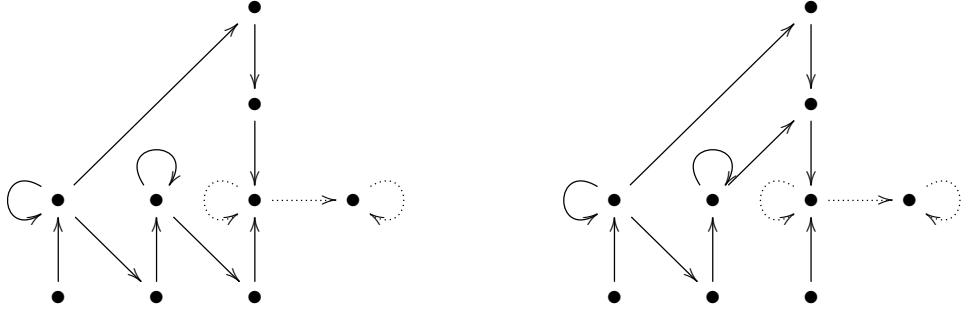
\[
\xymatrix{&&\bullet\ar[d]&\\
&&\bullet\ar[d]&\\
\bullet\ar[rruu]\ar@(ul, dl)
\ar[dr]&\bullet\ar@(ul, ur)
\ar[dr]&\bullet\ar@{.>}[r]\ar@{.>}@(ul, dl)&\bullet\ar@{.>}@(ur, dr)
\\\bullet\ar[u]&\bullet\ar[u]&\bullet\ar[u]&}\hskip3cm
\xymatrix{&&\bullet\ar[d]&\\
&&\bullet\ar[d]&\\
\bullet\ar[rruu]\ar@(ul, dl)
\ar[dr]&\bullet\ar@(ul, ur)
\ar[ur]&\bullet\ar@{.>}[r]\ar@{.>}@(ul, dl)&\bullet\ar@{.>}@(ur, dr)
\\\bullet\ar[u]&\bullet\ar[u]&\bullet\ar[u]&}\]

Next, we ensure that all
$j(n-2)$-spines are as short as possible for $j=3,\ldots, n.$

Let us introduce the notation which parallels $n_i, a_i$ and $n_i-a_i-1$ of the $n=3$ case. We let $u_i, i\leq k_2$ be vertices on the 2-spine and $n_i=|\ra^{-1}(u_i)|.$ If $a_{ji}$ is the $i$-th entry of the $j2$-connecting matrix, then $a_{ji}$ corresponds to the cardinality $|\ra^{-1}(u_i)\cap \so^{-1}(E^0-H_j)|$ and
$\sum_{j=3}^{n}a_{ji}$ corresponds to the cardinality $|\ra^{-1}(u_i)\cap \bigcup_{j=3}^n\so^{-1}(H_j-H_2)|.$ Let $a_i$ denotes this cardinality.
Note that
\[\ra^{-1}(u_i)\cap\so^{-1}(E^0-H_2)=\ra^{-1}(u_i)\cap \bigcup_{j=3}^n\so^{-1}(H_j-H_2)\]
because every edge $u_i$ receives outside of $H_2$ is necessarily a $j2$-exit for some $j=3,\ldots, n.$

If $i<k_2,$ then $u_i$ receives the edge on the spine from $u_{i+1},$ so we partition
the set $\ra^{-1}(u_i)$ as
\[\ra^{-1}(u_i)=\left(\ra^{-1}(u_i)\cap\so^{-1}(u_{i+1})\right)\cup \left(\ra^{-1}(u_i)\cap\so^{-1}(E^0-H_2)\right)\cup \left(\ra^{-1}(u_i)\cap\so^{-1}(H_2-\{u_{i+1}\})\right).\]
Just as before, a reference to the difference $n_i-a_i-1$ is only to denote {\em the cardinality of the last set above}
and we write the partition above as
\[n_i=1+a_i+(n_i-a_i-1).\]
If $i=k_2,$
the set $\ra^{-1}(u_i)$ is
$\left(\ra^{-1}(u_i)\cap\so^{-1}(E^0-H_2)\right)\cup \left(\ra^{-1}(u_i)\cap\so^{-1}(H_2)\right)$ and we write this as $n_i=a_i+(n_i-a_i)$

If $j2$-spine is to be extended for some $j=3,\ldots, n,$ then all vertices of $c_2$ need to have at least one tail {\em not counting the exits they receive from other $c_{l},$} $l=3,\ldots, n.$ So, the two conditions for extending the $j2$-spine are below.
\begin{enumerate}
\item For every $21$-exit $e$ ending at $v_{1i},$ there is a tail of $v_{1i}.$ Here $i\in m_1$ if $m_1>0$ and $i\leq k_1$ if $m_1=0.$

\item Every vertex of $c_{2}$ has at least one tail.
Thus, for $i\in m_{2}$, $n_i-a_i-1>0$ and, if $k_{j2}>m_{2}-1,$ then $n_{m_{2}-1}-a_{m_{2}-1}-1>1$.
\end{enumerate}

Note that these conditions match the case $n=3$ exactly. As these two conditions are the same for every $j$, we have that one of the $j2$-spines can be extended for some $j=3,\ldots, n$ if and only if {\em any} of the $j2$-spines can be extended for all  $j=3,\ldots, n$.
If this holds, then 2-spine can also be extended in the same sense as in the case $n=3$.

If the $j2$-spine can be extended for some $j$ and if there is an exit with the range $u_i$ outside of $c_2$ (so $i>m-1$) then a $j2$-spine extension makes this exit end in the vertex of the spine with distance $i+m.$ So, if $E$ permits arbitrary $j2$-spine extension, then it also permits $j2$-spine reduction up to when the exit with smallest such $i$ value is such that $m\leq i< 2m$.
In this case, we say that the $j2$-spine is {\em as short as possible}.
If the $j2$-spine cannot be extended and 2-spine can be arbitrarily extended, then we can still require the part without any tails (the line of length $lm_2$ which extends the spine if $l$ extensions are made) to be absent for the 2-spine to be {\em as short as possible}. This brings us to the following conditions and we say that $E$ is {\em $2$-I-plus-reduced} if either of them holds.
\begin{enumerate}[\upshape(a)]
\item The 2-spine cannot be extended.

\item The 2-spine can be extended and it is as short as possible.
\end{enumerate}
The conditions match those from  the $n=3$ case.
Condition (a) implies that none of the $j2$-spines can be extended for $j=3,\ldots, n$. With this established, we say that $E$ is I-plus-reduced if it is $j$-I-plus reduced for every $j=2,\ldots, n.$

Next, we turn to O-reduction.
A blow-up of an $j2$-exit can create a $(j+1)2$-connecting path. So, it is not only the $j2$-connecting part and the number of some tails that an exit move can change but also the other connecting paths. However, we defined the 32-O-reduction using {\em changes in the graph} not only the changes in the tails, so the definition from the $n=3$ is applicable to the general case.
If $E$ is the direct-exit form of a blow-up of an $ji$-$kl$-exit
of a direct-exit graph $F$ for some $j>i\geq 2$ and $k\in m_j$, $l\in m_i,$ we say that $E$ is {\em $ji$-$kl$-reducible}. Such graph $E$ is {\em $ji$-reducible} for some $j>i\geq 2$ if it is $ji$-$kl$-reducible for some $k$ and $l.$
If the 2-quotient changes after a $ji$-blow-up and path rearrangements, then $E$ is
{\em $ji$-O-reduced} if it is not $ji$-reducible. Otherwise, $E$ is also $ji$-O-reduced.
We say that $E$ is {\em O-reduced} if it is $ji$-O-reduced for every $j>i\geq 2.$

If a direct-exit $n$-S-NE graph is I-plus-reduced and O-reduced and if $v_{2i}$ is the single exit emitter, we write
$E=E_{\can\quot, i}$ and say that its {\em $(n-1)$-quotient is canonical}.

\subsection{Connecting paths ending at a terminal cycle and canonical forms}
\label{subsection_j1_transformations} Let $E$ be a $n$-S-NE graph with its $(n-1)$-quotient in its $(n-1)$-quotient canonical form.
We define the tail graph  analogously as in $n=2$: it is the subgraph generated by the vertices in $H_1$.
To make the $j1$-connecting paths for $j>1$ and the paths in the tail graph
canonical, we turn to  Proposition \ref{proposition_canonical_tails} and consider it with $V$ being $P_{H_1}^0-H_1.$ This produces a {\em direct-exit} form of $E$.

If $m_1>0,$ this process makes the $c_j$-to-$c_1$ paths have length one and the tail graph consists exclusively of $c_1$ and tails to it.
If $m_1=0$, this creates the spine of the tail graph which we call the {\em 1-spine} and denote its length by $k_1$. If $m_1>0$ and $c_1\in [c_1]$ is fixed, we consider the path from $v_{11}$ (where the second ``1'' is considered modulo $m_1$) to $v_{10}$ be the spine of the tail graph so that we can unify the consideration and we can write $k_1=m_1-_{m_1}1$ in this case.

In both cases, the $j1$-exits are being emitted
directly into the 1-spine.
We let $u_{i}$ be the vertices on the 1-spine for $i\leq k_1$ (hence $u_i=v_{1(m_1-_{m_1}i}$ if $m_1>0$ and $i\in m_1$). With this notation, a {\em $j1$-$ki$-exit}
is an edge which originates in $v_{jk}, k\in m_j,$ and terminates in $v_{1i}=u_{m_1-_{m_1}i}$ for $i\in m_1$ if $m_1>0$ and in $u_{1i}$ for $i\leq k$ if $m_1=0.$ With this notation, the $j1$-connecting matrix is defined and, just as before, it depends on the choice of $c_j\in [c_j]$ and $c_1\in [c_1].$

If $k_{j1}$ is the largest of the integers such that the vertex $u_{k_{j1}}$ receives $j1$-exits, then $k_{j1}$ is the length of the $j1$-spine.  If no such largest integer exits, then $m_1=0$ and $k_1=\omega$ necessarily so we have that  the $j1$-spine is the 1-spine and that $k_{j1}=k_1$.
We may have $k_{j1}\neq k_{l1}$ for some $j\neq l$. For example, the first graph below has $k_{21}=0$ and $k_{31}=k_1=1.$
The second graph has $k_{21}=k_1=1$ and $k_{31}=0.$

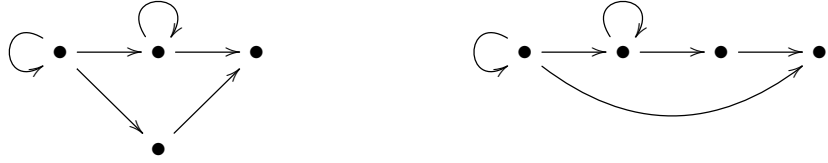
\[\xymatrix{
\bullet\ar@(lu,ld)
\ar[r]
\ar[rd]&
\bullet\ar@(lu,ru)
\ar[r]&\bullet\\&\bullet\ar[ur]&&}\hskip2cm\xymatrix{
\bullet\ar@(lu,ld)
\ar[r]
\ar@/_2pc/[rrr]&
\bullet\ar@(lu,ru)
\ar[r]&\bullet\ar[r]&\bullet}\]

A {\em 1-$i$-tail} of $E$
is an edge which originates in a source and terminates in $v_{1i}=u_{m_1-_{m_1}i}$ for $i\in m_1$ if $m_1>0$ and in $u_{1i}$ for $i\leq k$ if $m_1=0.$

A {\em blow-up} of an $j1$-exit $e$ is an out-split with respect to the partition $\{e\}$ and $\so^{-1}(\so(e))-\{e\}$ followed by the out-splits making the resulting graph be in its total out-split form. An {\em exit move} is a blow-up followed by operations making the resulting graph be a direct-exit graph. None of the operations involved in an exit move impact the $(n-1)$-quotient. If the source $v$ of a moved $j1$-exit does not receive any edges from $H_l$ for any $l>j,$ then only some tails can be created by the move, just as in the $n=2$ case. If $v$ receives an edge from $H_l$ for some $l>j,$ then some  $c_l$-to-$c_1$ paths are  created in the process. For example, the second graph is the result of a move of any the $21$-exits of the first graph.
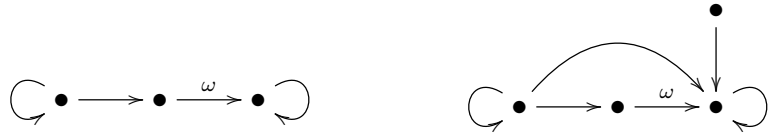
\[\xymatrix{&&\\\bullet\ar@(lu,ld)\ar[r]&\bullet\ar [r]^{\omega}&\bullet\ar@(ru,rd)}\hskip3cm
\xymatrix{&&\bullet\ar[d]\\\bullet\ar@/^2pc/[rr]\ar@(lu,ld)\ar[r]&\bullet\ar [r]^{\omega} &\bullet\ar@(ru,rd)}\]

If $E$ is obtained by a move of an $j1$-$kl$-exit of $F$ for some $j,$ $k,$ and $l,$ we write $F\to_1E$ and say that $E$ is {\em $j1$-$kl$-reducible}.
We define $E_{\red, j1, kl}$ as the graph with the same $(n-1)$-quotient and $j1$-connecting matrices as $F$ and with the number of $1$-$i$-tails equal to that number for $E$ if the cardinality of $1$-$i$-tails is not impacted by the exit move $F\to_1 E$. If the number of $1$-$i$-tails of $E$ is strictly larger than in $F$ and their number is finite in $E$, then $E_{\red, j1, kl}$ and $F$ have the same number of $1$-$i$-tails. Otherwise, $E_{\red, j1, kl}$ has infinitely many 1-$i$-tails.
We say that $E$ is {\em $j1$-reducible} if it is $j1$-$kl$-reducible for some $k$ and $l$.

If the cardinality of the 1-tails and $l1$-exits for $l>j$ changes by a $j1$-exit move, then $E$ is {\em $j1$-reduced} if it is not $j1$-reducible. If the cardinality of the 1-tails and $l1$-exits does not change by a $j1$-exit move, then $E$ is {\em $j1$-reduced}. If a graph is $j1$-reducible and $j1$-reduced and if $c_j$ both receives and emits exits, a move of a $j1$-exit which originates in the vertex which receives a path from $c_l$ for some $l>j$ produces a new $l1$-exit. So, in this case, $c_l$ is an infinite emitter (otherwise the number of $l1$ exits of $E$ and $F$ would differ). For example, a move of any of the 21-exits of the first graph below produces a graph isomorphic to the original graph. If $k$ is a finite and nonzero cardinality, the second graph below is not 31-reduced. The third graph is a 31-reduction of the second graph.
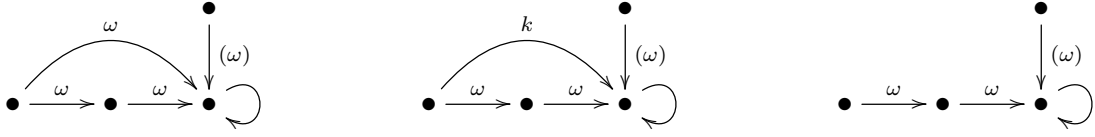
\[\xymatrix{&&\bullet\ar[d]^{(\omega)}\\\bullet\ar@/^2pc/[rr]^{\omega}\ar[r]^{\omega}&\bullet\ar[r]^{\omega} &\bullet\ar@(ru,rd)}\hskip2cm \xymatrix{&&\bullet\ar[d]^{(\omega)}\\\bullet\ar@/^2pc/[rr]^{k}\ar[r]^\omega&\bullet\ar[r]^{\omega} &\bullet\ar@(ru,rd)}\hskip2cm\xymatrix{&&\bullet\ar[d]^{(\omega)}\\\bullet\ar[r]^\omega&\bullet\ar[r]^{\omega} &\bullet\ar@(ru,rd)}\]

A graph $E$ is {\em 1-reduced} if it is $j1$-reduced for all $j=2,\ldots, n.$ In this case, we write $E=E_{1,\red}.$

For $j>1,$ if $c_j$ and $c_1$ are proper cycles and $L$ is the least common multiple of $m_j$ and $m_1,$ $E$ is {\em $L$-$j1$-$kl$ reducible} if there is $F$ such that $F\to_LE$ where $\to_1$ is a move of a $j1$-$kl$-exit and any subsequent move is the move of the resulting $j1$-exit. In this case, we define $E_{\red, L, j1, kl}$  analogously as when $n=2.$ We have that $E$ is {\em $L$-$j1$-reduced} if for any $k$ and $l,$ when $E_{\red, L, j1, kl}$ is defined, then  $E\cong E_{\red, L, j1, kl}.$

We define a {\em spine-$j1$-reduced graph} if $c_1$ is a sink and $c_j$ a proper cycle and a {\em $\omega$-$j1$-reduced graph} if $c_j$ is an infinite emitter
analogously as when $n=2.$ The $n=2$ case definitions do not involve any reference to the length of a composition series, so we can generalize them directly. We unify the four cases saying that $E$ is {\em $L_j$-$j1$-reduced} if it is $L$-$j1$-reduced if $m_j>0$ and $m_1>0$, spine-$j1$-reduced if $m_j>1$ and $m_1=0$, $\omega$-$j1$-reduced if $m_j=0, m_1>0$ or $m_j=m_1=0$ and $k_{j1}<\omega$, and $j1$-reduced if $m_j=m_1=0$ and $k_{j1}=\omega.$
If $\mathbf L$ is the $(n-1)$-tuple $(L_2,\ldots, L_{n}),$ we say that $E$ is {\em $\mathbf L$-reduced} if $E$ is $L_j$-$j1$-reduced for every $j>1.$ In this case, we write $E=E_{\mathbf L\text{-}\red}.$

Let $i\in m_1$ if $c_1$ is a proper cycle of length $m_1$ and let $i\in \omega$ if $c_1$ is a sink. For such $i,$ we consider the following condition.

\begin{itemize}
\item[$C(1i)$] There are paths $p,q,d$ and edges $g, e$ such that $g$ is a tail of $H_{j}/H_{j-1}$ for some $j=2,\ldots, n,$ $p$ is a path from $\ra(g)$ to a vertex in $c_{j}$ which contains no edges of $c_j,$ $q$ is a path from $\ra(p)$ to a vertex in $c_l$ for some $1<l\leq j$, $e$ is an $l1$-exit and $d$ is a path from $\ra(e)$ to $\so(c_1)$ which does not contain $c_1$. There are infinitely many tails with the same range as $g$ in $H_j/H_{j-1}.$
If $c_1$ is a proper cycle, then $1+|p|+|q|+|d|=m_1-i+1(\mymod m_1)$ and,
if $c_1$ is a sink, then $1+|p|+|q|+|d|=i+1.$
\end{itemize}

If this condition holds, we say that 1-$i$-tails are {\em cuttable}. For example, let $E$ be the graph below, let $\kappa$ be any countable cardinal, and let $H_1$ be the hereditary and saturated closure of the sink, $H_2$ of the source of the loop, and $H_3$ of the vertices of a cycle of length two.
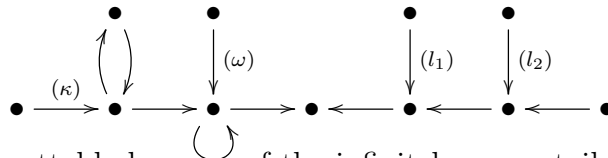
\[
\xymatrix{&\bullet\ar@/^/[d]&\bullet\ar[d]^{(\omega)}&&\bullet\ar[d]^{(l_1)}&\bullet\ar[d]^{(l_2)}&\\
\bullet\ar[r]^{(\kappa)}&\bullet\ar[r]\ar@/^/[u]&\bullet\ar@(ld,rd)\ar[r]&\bullet&\bullet\ar[l]&\bullet\ar[l]&\bullet\ar[l]}\]

The 1-1-tails of $H_1$ are cuttable because of the infinitely many tails which $c_2$ is receiving. The 1-2-tails of $H_1$ are cuttable if and only if $\kappa=\omega.$

If $l_i$ is the number of 1-$i$-tails of a graph $E,$ we have that $E_{\cut}$ has the same $(n-1)$-S-NE quotient, the same $j1$-connecting matrices as $E$ and the number of the 1-$i$-tails is specified as follows. Let
$k_{\max}$ be the maximum of all spine lengths $k_{j1}$ for $j>1$. For $i<k_{\max}$, $E_{\cut}$ has $l_i$ 1-$i$-tails if 1-$i$-tails are not cuttable and it has zero $1$-$i$-tails otherwise. If $k_{\max}=k_1,$ then the above specifies the number of tails in $E_{\cut}$ completely. If $k_1>k_{\max}$, the number of tails of $E_{\cut}$ is specified by the following two cases.
\begin{enumerate}
\item There is no $k_0\in k_1, k_0\geq k_{\max}$ such that $1$-$i$-tails are cuttable for all $i\in k_1, i\geq k_0.$ Then, the number of 1-$i$-tails is $l_i$ if 1-$i$-tails are not cuttable and zero otherwise. The spine length of the tail graph of $E_{\cut}$ remains $k_1$.

\item There is $k_0\in k_1, k_0\geq k_{\max}$ such that $1$-$i$-tails are cuttable for all $i\in k_1, i\geq k_0.$ The length of the spine graph of $E_{\cut}$ is $k_0$ in this case. For $ i<k_0,$ the number of 1-$i$-tails is $l_i$ if 1-$i$-tails are cuttable and zero otherwise. The spine length of the tail graph of $E_{\cut}$ is $k_0$.
\end{enumerate}

For the graph in the previous example, $E_{\cut}$ has $l_1=0.$ If $\kappa<\omega,$ $k_0=2$ and, if $\kappa=\omega,$ then $E_{\cut}$ has $l_2=0$ and $k_0=1.$

\begin{definition} Let $n>1$ and let $E$ be a breaking-vertex-free, direct-exit, countable $n$-S-NE such that the $(n-1)$-quotient is canonical. For every $j=2,3,
\ldots, n,$ considered in this order, we move some $j1$-exits, if needed, of the reduced and cut form of $E$ as specified below.
Then, define
$E_{\can}$ as the $\mathbf L$-reduced and cut form of the resulting graph.
\begin{itemize}
\item If $m_j>0$ and $m_1>0,$ then the required form is obtained by moving the $j1$-exits so that the resulting graph has the $j1$-part with a single exit-emitter
and that the ranges of the exits are among $G=$GCD$(m_j,m_1)$ consecutive vertices of $c_1.$

\item If $m_j>0$ and $m_1=0,$ then the required form is obtained by moving the $j1$-exits which do not end in the $j1$-spine source to end in that source. Hence, $c_j$ emits exits only to the source of the $j1$-spine.

\item If $m_j=0$ and $m_1>0$ or if $m_j=m_1=0$ and $k_{j1}<\omega$,
the required form is obtained by moving all the $j1$-$0i$-exits if  $v_{1i}$ receives nonzero and finitely many of such exits.

\item If $m_j=m_1=0$ and $k_{j1}=\omega$, no exit moves are needed.
\end{itemize}
If $E$ is any $n$-S-NE graph, we let $E_{\can}$ be any canonical form of the direct-exit form of $E$ with $(n-1)$-canonical quotient.
We say that $E$ is {\em canonical} or that it is {\em in a canonical form} if $E\cong F_{\can}$ for some $n$-S-NE graph $F$.
\label{definition_canonical_forms_n>2}
\end{definition}

If $\phi_1$ and $\phi_2$ are the maps transforming $E$ to the two canonical forms, any $\POM^D$-isomorphism $f$ of the monoids of $E_1$ and $E_2$ could be composed with $\ol{\phi_1\phi_2^{-1}}$ and $\ol{\phi_1\phi_2^{-1}}f$ could be considered instead of $f$ to ensure that the {\em same} canonical form is considered.
The consideration of the maps like $\phi_1$ and $\phi_2$ above ensures that two canonical forms obtained from the same cut and reduced graph can be transformed one to the other. The same argument was  used in Lemma \ref{lemma_maps_between_canonical_forms}) so this lemma  continues to hold.

We define the relation $\approx$ for countable $n$-S-NE graphs by the same condition as if $n\leq 2$:
\[E\approx F\mbox{ if there are canonical forms $E_{\can}$ and $F_{\can}$ such that }E_{\can}\cong F_{\can}.\]
This relation is reflexive and symmetric and transitivity holds by Lemma \ref{lemma_approx_is_transitive} stated for $n$-S-NE graphs instead of 2-S-NE graphs. Such lemma holds by the same proof as Lemma \ref{lemma_approx_is_transitive}.

\subsection{The main result and its corollaries}\label{subsection_main}

In this section, we prove the general version of the Quotient Proposition, the Cut Lemma and the main result of the paper, Theorem \ref{theorem_GCC_disjoint_cycles}. Then, we prove some of its corollaries, Corollaries \ref{corollary_unital},  \ref{corollary_C_star} and \ref{corollary_iso_conjecture}.

Let $E$ and $F$ be $n$-S-NE graphs with composition series $H_i$ and $G_i$ for $i=0,\ldots, n$ and let $f$ be a $\POM^D$-isomorphism of the graph $\Gamma$-monoids of $E$ and $F.$
If $f_{j/i}$ denotes the $\POM^D$-isomorphism of the monoid of $H_j/H_i$ for $i<j$  induced by $f$ and if $f_{i}$ denotes the restriction of $f$ to the monoid of $P_{H_i}$, we can pick the composition series of $F$ so that the image of
$f_{j/i}$ is the monoid of
$G_j/G_i$ and the image of
$f_{i}$ is the monoid of $P_{G_i}.$ Proposition \ref{proposition_on_canonical_quotients_n} below will enable us to choose the canonical forms with isomorphic $(n-1)$-quotients by an isomorphism which induces $f_{n/1}$
if we assume that the GCC for the $(n-1)$-S-NE graphs.

\begin{proposition} {\bf The General Quotient Proposition.}
Assume that Lemma \ref{lemma_glavna_in_general_case} (the general Cut Lemma) and Theorem \ref{theorem_GCC_disjoint_cycles} holds for the $(n-1)$-S-NE graphs. With $E$, $F$, $f$, $H_1,$ and $G_1$ as in the above paragraph and such that the $(n-1)$-quotients are canonical, assume that the following conditions hold.
\begin{enumerate}[\upshape(i)]
\item $f([v_{20}])=[v_{20}'],$ $v_{20}$ and $v_{20}'$ are single exit-emitters, and
$f([v_{10}])=[v_{10}'].$

\item The 21-connecting matrices, computed using $v_{20}, v_{10}, v_{20}'$ and $v_{10}'$ are equal.

\item There is an isomorphism $\iota_{n/2}: E/H_2\cong F/G_2$ such that $f_{n/2}=\ol\iota_{n/2}.$
\end{enumerate}
Under these assumptions,
there is an isomorphism $\iota_{n/1}: E/H_1\to F/G_2$ such that $f_{n/1}=\ol\iota_{n/1}.$
\label{proposition_on_canonical_quotients_n}
\end{proposition}
\begin{proof}
By the hypothesis that Theorem \ref{theorem_GCC_disjoint_cycles} holds for $(n-1)$-S-NE graphs, the existence of $f_{n/1},$ implies the existence of canonical forms $(E/H_1)_{\can}$ and $(F/G_1)_{\can}$ of $E/H_1$ and $F/G_1$, $\psi_E:E/H_1\to (E/H_1)_{\can}$ and $\psi_F: F/G_1\to (F/G_1)_{\can},$ an isomorphism $\iota_{n/1, \can}: (E/H_1)_{\can}\to (F/G_1)_{\can}$ such that $\ol{\psi_F^{-1}\iota_{n/1, \can}\psi_E}=f_{n/1}$.
The maps $\psi_E$ and $\psi_F$ are some $j2$-exit moves, $\mathbf L$-reductions, and cuts for $j>2$. We can apply the inverses of the $j2$-exit moves of $\psi_E$ to the rest of the graph which has $(F/G_1)_{\can}$ as its $(n-1)$-quotient because $(E/G_1)_{\can}\cong (F/G_1)_{\can}.$ These operations do no impact the $(n-2)$-quotient, so it remains canonical. We obtain a graph $F'$ with the same $(n-2)$-quotient,  the $j2$-connecting matrices equal to those of $E,$ and with $v_{20}$ the single emitter of $j1$-exits.

For a moment, we can consider $F'$ instead of $F$. Instead of $f$, we can consider $f$ composed with the map which the correspondence $\phi': F\to F'$  induces on the $\Gamma$-monoid level. Once we obtain the statement for the changed graphs and changed $f$, O-reduction of the $(n-1)$-quotient of both graphs would produce graphs with isomorphic $(n-2)$-quotients and equal $j2$-connecting matrices. This is because applying the same $j2$-exit moves to the graphs with the same $j2$-connecting part produces graphs with the same $j2$-connecting parts and intact $(n-2)$-quotients. Thus, returning to the original $E$ and $F$, this shows that we can assume that the $j2$-connecting matrices are equal for all $j>2$.

By the proof of Proposition \ref{proposition_E_without_H} and Remark \ref{remark_on_sink_graph}, the existence of $f_2$ on the monoids of $P_{H_2}$ and $P_{G_2}$ implies the existence of an isomorphism $\iota_{2/1}$ of the sink-graphs of $H_2/H_1$ and $G_2/G_1$ such that $f_{2/1}=\ol\iota_{2/1}$. So, besides the three assumptions above, we have the following two conditions holding.
\begin{itemize}
\item[(iv)] There is an isomorphism $\iota_{2/1}$ of the sink-graphs of $H_2/H_1$ and $G_2/G_1$ such that $f_{2/1}=\ol\iota_{2/1}$.

\item[(v)] The $j2$-connecting matrices $A_{j2}=A_{j2}'$ are equal for $j=3,\ldots, n$.
\end{itemize}

We recall that $u_i, i\leq k_2$ are the vertices of the 2-spine of $E$, $u_i', i\leq k_2'$ are the vertices of the 2-spine of $F$, $n_i=|\ra^{-1}(u_i)|,$ $n_i'=|\ra^{-1}(u_i')|$ and we let $A_{j2}=A_{j2}'=[a_{ji}],$ $i\leq k_{j2}$ and $a_i=\sum_{j=3}^n a_{ji}=|\ra^{-1}(u_i)\cap\so^{-1}(E^0-H_2)|.$ For $i<k_2,$ we also recall the partition
\[\ra^{-1}(u_i)=\left(\ra^{-1}(u_i)\cap\so^{-1}(u_{i+1})\right)\cup \left(\ra^{-1}(u_i)\cap\so^{-1}(E^0-H_2)\right)\cup \left(\ra^{-1}(u_i)\cap\so^{-1}(H_2-\{u_{i+1}\})\right)\]
which we write as $n_i=1+a_i+(n_i-a_i-1).$ For $i=k_2,$ we have
$\ra^{-1}(u_i)= \left(\ra^{-1}(u_i)\cap\so^{-1}(E^0-H_2)\right)\cup \left(\ra^{-1}(u_i)\cap\so^{-1}(H_2)\right)$ which we write as $n_i=a_i+(n_i-a_i).$

Conditions (iii) and (v) imply that when the paths in the porcupine quotients are rearranged from their 1-quotient canonical position (every tail on the 2-spine receives only tails) to their porcupine-quotient form (the cycles $c_j$ and $c_j'$ correspond to the left-infinite paths receiving the paths of the $l$-quotients for $l>j$ if $m_j>0$ and the edges infinite emitters emit being split into separate ``strands'', see Remark \ref{remark_on_replacing_tails}), these 1-S-NE graphs remain isomorphic. Indeed -- conditions (iii) and (v) ensure that the rearrangements are done in the same way for the same paths in both graphs.
Because of this, we have that $n_i=n_i'$ for all $i $ and than, in particular, $k_2=k_2'.$ Condition (v) implies that $k_{j2}=k_{j2}'$ for $j>2$

As $n_i=n_i'$ for all $i\leq k_2,$ we have that
$a_i+(n_i-a_i-1)=a_i'+(n_i'-a_i'-1)$ for $i<k_2$ and that, if $k_{32}=k_2,$ then
$a_{k_2}+(n_{k_2}-a_{k_2})=a_{k_2}'+(n_{k_2}'-a_{k_2}').$
If $n_i$ is finite, then these relations imply that
$n_i-a_i-1=n_i'-a_i'-1$ for $i<k_2$ and that, if $k_{32}=k_2$ that
$n_{k_2}-a_{k_2}=n_{k_2}'-a_{k_2}'$. If $n_i$ is infinite and $a_i$ is finite, then the same holds since
$n_i-a_i-1=n_i'-a_i'-1=\omega$ for $i<k_2$ and that
$n_{k_2}-a_{k_2}=n_{k_2}'-a_{k_2}'=\omega$ if $k_{32}=k_2.$ Hence, we only need to consider the case when both $a_i$ and $n_i$ are infinite, in which case $m_j=0$ for some $j>2.$

If $m_2=0,$ any cut maps done in $(E/H_1)_{\can}$ give rise to the analogous cut maps of $E/H_1$ for two reasons. First, if $m_2=0$  the identification $c_2c_2^*$ with $\so(c_2)$ is valid. Second, the length of the connecting path does not play a role in condition $C(2i)$ only its existence does. Because of this, we can use the cut map for the tails of the $2$-spines, so we can assume that $E/H_1$ and $F/G_1$ are cut. Hence, these $(n-1)$-quotients satisfy the assumptions of the General Cut
Lemma. By this lemma, there is $\iota_{n/1}: E/H_1\cong F/G_1$ such that $f_{n/1}=\ol\iota_{n/1}$.

Let $m_2>0$ and $m_j=0$ for some $j>2$ (we note that the arguments below are valid also if $m_j>0$ for all $j>2$).
Assume there is $i_0$ is the smallest such that $n_{i_0}-a_{i_0}-1>n_{i_0'}-a_{i_0}'-1$ if $i_0<k_{2}$ or, if $i_0=k_{32}=k_2,$ that $n_{i_0}-a_{i_0}>n_{i_0}'-a_{i_0}'.$ The excess of tails in $E$ make us be able to repeat the argument of the proof of the Quotient Proposition for the $n=2$ case (Proposition \ref{proposition_E_without_H}) to show that both graphs satisfy condition (1) for a 2-spine extending. This is because the existence of a 0-route from a tail in the quotient does not depend on whether there is only a 1-quotient or a larger quotient below the root of the range of that tail (so it works for $n>2$ also). So, the bijection $P^{v_{10}}_{\not c_1}\to P^{v_{10}'}_{\not c_1'}$, let us again call it $\rho$ as in the proof of the Quotient Proposition for $n=2$, which $f_1$ induces maps a 0-route to the tail graph.
Hence, the tail graph of $F$ contains enough tails to accommodate such an image. In particular, for every $21$-exit $e$ which ends at $v_{1i}$ for some $i\leq k_{21}$, $F$ contains a tail at $v_{1i}'$. Thus, condition (1) for the 2-spine extending holds.
Since $F$ contains at least two edges ending at $v_{1i}'$ (the exit and the tail), $v_{1i}$ has to receive at least that many edges. As one of them is $e$, $v_{1i}$ has to have at least one more tail. So, $E$ also satisfies condition (1) for spine extending.

As $\rho$ interchanges a quotient route to a tail route, at least one tail can be found so that its source $w$ is such that the image $f([w])$
cannot be matched to a single monoid generator in the correspondence $f$ creates on the generating intervals $D_E$ and $D_F$ so that
\[f([w])=f(t^{i_0+1}[v_{20}])=t^{i_0+1}\left(t^{lm_2}[v_{20}']+\sum_{j=0}^{l-1}t^{jm_2}a_{21}[v_{10}']\right)\]
is not identified with $[w']$ for no vertex $w'$ in the $(n-1)$-quotient. By considering the distance from $v_{20}'$, the first term $t^{i_0+1+lm_2}[v_{20}']$ has to correspond to a tail to $u_{i_0+lm_2}'$. The arguments here match exactly those in the proof of the Quotient Proposition for $n=2.$ Using the same arguments, one shows that $n'_{i+lm_2}\geq n_i$  for all $i\leq k_2$. By conditions (ii), (iii), and (v) the bijection $\rho$ can be chosen to map the paths starting in the sources of $j2$-exits of $E$ to the sources of $j2$-exits of $F$. So, any interchanges between the quotient and the tail paths can happen only for the quotient paths starting in the tails of the emitting cycles. Because of this, the existence of the portion of the 2-spine of $F$ between $u_{m_2}'$ and  $u_{lm_2-1}'$ implies that $f^{-1}([u_{jm_2+i_0}'])$
has to be of the form
$[pc_2^{j'}c_2^{-j'}p^*]$ for a path $p$ starting in a tail and for some $j'$. Given that $u_{m_2+i}',$ for $i=1,\ldots, m_2$ has the distance $m_2+i$ from $v_{20}',$ the value of $j'$ is one in this case and the length of $p$ has to be $i.$ Hence, such a path $p$ has to start in a tail of $u_{i-1}$. So, we have that $n_i-a_i-1>0$ and $n_{m_2-_{m_2}1}-a_{m_2-_{m_2}1}-1>1.$
This shows that $E$ satisfies condition (2) for the 2-spine extending. Hence, $E$ permits arbitrary 2-spine extending.

Extending the 2-spine of $E$ and comparing the resulting graph with $F$ enables us to use the same arguments as in the proof of Proposition \ref{proposition_E_without_H} again and have that the 2-spine of $F$ can be shortened. So, we arrive to a contradiction. This shows that any deviation from
$n_i-a_i-1=n_i'-a_i'-1$ (or $n_i-a_i=n_i'-a_i'$ if $i=k_{32}=k_2$), leads to a contradiction.

These relations and the  conditions (iii), (iv), and (v) enable us to define an isomorphism $\iota_{n/1}$ so it maps the 2-spine of $E$ and its tails onto the 2-spine of $F$ and its tails, $j2$-exits of $E$ onto the $j2$-exits of $F,$ and the rest of the $(n-2)$-quotient onto the rest of the $(n-2)$-quotient using $\iota_{n/2}$. By this construction, we have that $f_{n/1}=\ol\iota_{n/1}$ for such  $\iota_{n/1}.$
\end{proof}

We the generalization of the Cut  Lemma. The arguments in the proof are analogous to the $n=2$ case and we include them for completeness.

\begin{lemma} {\bf The General Cut Lemma.}
Let $E=E_{\cut}$ and $F=F_{\cut}$ be direct-exit $n$-S-NE graphs with canonical $(n-1)$-quotients  such that
that there is $\iota_{n/1}: E/H_1\cong F/G_1$ which maps $v_{j0}=\so(c_j)$ onto $v_{j0}'=\so(c_j')$ for $j=2,\ldots, n.$

If there is a $\POM^D$-isomorphism $f$ of the $\Gamma$-monoids such that $f_{n/1}=\ol\iota_{n/1},$ that $f([v_{10}])=[v_{10}']$ for $v_{10}=\so(c_1)$ onto $v_{10}'=\so(c_1'),$ and that the $j1$-connecting matrices (computed for $c_j,$ and $c_j', j=1,\ldots, n$) are equal and such that $\iota_{n-1}(\so(c_j))=\so(c_j')$ for $j=2,\ldots, n$, then $\iota_{n/1}$ can be extended to a graph isomorphism $\iota: E\cong F$ such that $f=\ol\iota.$
\label{lemma_glavna_in_general_case}
\end{lemma}
\begin{proof}
The assumptions of the lemma enable us to extend $\iota$ to exits and their ranges, to the vertices and edges of $c_1$ if $m_1>0$, and to the vertices and edges of the suffix of the $1$-spine of length $k_{\max}=\max\{k_{j1}\mid j=2,\ldots, n\}$ if $m_1=0.$

We show that we can extend $\iota_{n/1}$ to the rest of the 1-spine (if any) and to the 1-tails. For any $i,$ condition $C(1i)$ for both $E$ and $F$ depends only on the subgraph on which $\iota_{n/1}$ is currently defined.
Thus, for every $i\leq k_{\max}$, the 1-$i$-tails are cuttable in $E$ if and only if they are cuttable in $F.$

If $m_1>0$ and $i\in m_1,$ let $P^{v_{10}}_{\not c_1, i}$ be the set of paths of $P^{v_{10}}_{\not c_1}$ of length $m_1-i+1$ modulo $m_1$.
If $m_1=0$ and $i\leq k_1,$ let $P^{v_{10}}_{\not c_1, i}$ be the set of paths of $P^{v_{10}}$ of length $i+1$. We partition $P^{v_{10}}_{\not c_1, i}$ into
\begin{center}
$P^{v_{10}}_{\quot, i}=\{p\in P^{v_{10}}_{\not c_1, i}\mid \so(p)\in E^0-H_1)\}\;\;$ and
$\;\;P_{\tails, i}^{v_{10}}=\{p\in P^{v_{10}}_{\not c_1, i}\mid \so(p)\in H_1\},$
\end{center}
and let such sets be analogously defined for $F$.  The condition $C(1i)$ holds for some paths $p,q,d,g,e$ as in the statement of $C(1i)$ if and only if there are infinitely many tails with the same range as $g$ in $H_j/H_{j-1}$ which are not $lj$-exits in $E$ for any $l>j>1.$

The existence of the restriction $f_1$ of $f$ on the $\Gamma$-monoids of the porcupine graphs $P_{H_1}$ and $P_{G_1}$ ensures that there is a bijection
$\rho_{i}: P_{\not c_1, i}^{v_{10}}\to P_{\not c_1', i}^{v_{10}'}$
and the existence of $\iota_{n/1}$ ensures that such $\rho_{i}$ can be found so that it maps the set of paths in $P_{\quot, i}^{v_{10}}$ which have exits as their first edges onto the set of such paths $P_{\quot, i}^{v_{10}'}$. Let us choose one such $\rho_i$ which maps as few $\quot$-paths to $\tails$-paths as possible. By the existence of $\iota_{n/1}$, the cardinalities $|\Pa_{l_p}^{v_{j_p}}|$ and $|\Pa_{l_p}^{v_{j_p}'}|$ are equal.
The $(n-1)$-quotients are canonical, so  both $E$ and $F$ are I-plus reduced. By the proof of the General Quotient Proposition, the existence of an interchange is possible only if
$|\Pa_{l_p}^{v_{j_p}}|=|\Pa_{l_p}^{v_{j_p}'}|$ is equal to $\omega$. For such a path $p,$ let $l_p,k'_p, j_p, j'_p$ and $i'_p$ be as above considering $C(1i)$.
In this case,
the condition $C(1i)$ holds
in $E$ and so it holds in $F$ and this makes the 1-$i$-tails cuttable in both graphs. As both graphs are in their cut forms, the cardinality 1-$i$-tails is zero.

This shows that $\rho_i$ maps $P_{\quot, i}^{v_{10}}$ onto $P_{\quot, i}^{v_{10}'}$ and $P_{\tails, i}^{v_{10}}$ onto $P_{\tails, i}^{v_{10}'}.$
In particular, for every $i\in m_1$ if $m_1>0$ (for every $i\leq k_1$ if $m_1=0$), $\rho_i$ induces a bijection of $1i$-tails. Let $\iota_i$ be one such bijection. We extended $\iota_{n/1}$ to the 21-exits and the terminal cycles and now we can further extend it to the 1-spine and the tails using $\iota_{i}.$ If $\iota$ is such an extension, then  $\iota: E\cong F$ and $f=\ol\iota$.
\end{proof}

\begin{theorem} Let $E$ and $F$ be countable composition S-NE graphs. The following conditions are equivalent.
\begin{enumerate}
\item There is a $\POM^D$-isomorphism $f:M_E^\Gamma\to M_F^\Gamma.$

\item The relation $E\approx F$ holds.

\item There is a graded $*$-isomorphism $\phi:L_K(E)\to L_K(F).$
\end{enumerate}
If (1) holds, there are  canonical forms $E_{\can}$ and $F_{\can}$ and operations $\phi_E: E\to E_{\can},$ $\iota: E_{\can}\cong F_{\can},$ and $\phi_F: F\to F_{\can},$ such that $\ol{\phi_F^{-1}\iota\phi_E}=f.$
\label{theorem_GCC_disjoint_cycles}
\end{theorem}
\begin{proof}
If $E$ is an $n$-S-NE graph and if any of the four conditions holds, then $F$ is an $n$-S-NE graph by Theorem \ref{theorem_comp_series}. So, we can assume that $E$ and $F$  have composition series of the same length.

Since the implications (2) $\Rightarrow$ (3) and (3) $\Rightarrow$ (1) are direct, we prove (1) $\Rightarrow$ (2) and the last sentence of the theorem.

Let
$\phi_E: E\to E_{\can}$ and $\phi_F: F\to F_{\can}$
be operations transforming the graphs to their canonical forms. We can consider $\ol\phi_Ff\ol{\phi_E^{-1}}$ instead of $f$ and so we can assume that $E=E_{\can}$ and $F=F_{\can}.$

Assume that (1) holds and that $H_i, i=0,\ldots, n$ are the terms of a composition series of $E$. Let $G_i,$ $i=0,\ldots n$ be the admissible pairs of $F$ such that $f(J^\Gamma(H_i))=J^\Gamma(G_i)$ which implies that the sequence $G_i$ constitutes the terms of a composition series of $F.$ Let $f_i$ denotes the restriction of $f$ to $J^\Gamma(H_i).$ For $i=1,\ldots, n-1,$ $j=2,\ldots, n$ and $i<j,$ let $f_{j/i}$ be the induced $\POM^D$-isomorphism $M^\Gamma_{H_{j}/H_{i}}\to M^\Gamma_{G_{j}/G_{i}}.$  Let $m_i$ be the length of the terminal cycle of $H_i/H_{i-1}$ (and $G_i/G_{i-1}$) and $A_{ji}=[a_{ji, lk}], i=1\ldots, n-1, j=2, \ldots n, i<j$ be the $c_j$-to-$c_i$ connecting matrix  and let $A'_{ji}=[a_{ji, lk}]$ be $c_j'$-to-$c_i'$ connecting matrix of $F.$
Let $v_{ji}$ be the vertices of $c_j$ for $i\in m_j$ and $v_{ji}'$ are the vertices of $c_j'$ for $i\in m_j=|c_j'|.$

If there are no edges emitted from $E/H_j$ to $H_j$ for some $j=1,\ldots, n,$ then $E$ is a disjoint union of $E/H_j$ and $H_j$ and $M_E^\Gamma$ is a direct sum of $J^\Gamma(H_j)$ and $M^\Gamma_{E/H_j}.$ The existence of $f$ implies that $M_F^\Gamma$ splits into a direct sum of two terms each of which is $\POM^D$-isomorphic with $J^\Gamma(H_j)$ and $M^\Gamma_{E/H_j}.$ Since the generators of $J^\Gamma(G_j)$ and $M^\Gamma_{F/G_j}$ are independent, we have that there are no paths from $F/G_j$ to $G_j$. So, $F$ is a disjoint union of $F/G_j$ and $G_j.$ Since the composition length of $P_{H_j}$ and $E/H_j$ is smaller than $n,$ we can use induction hypothesis and establish that the needed claim holds.

Thus, we can consider only connected graphs. The above consideration shows that there are edges from $E/H_j$ to $H_j$ if and only if there are exits from $F/G_j$ to $G_j$.

Let us start the induction now. Our induction base is the $n=2$ case in which case
the General Quotient Proposition, the General Cut Lemma, and Theorem \ref{theorem_GCC_disjoint_cycles} hold by Quotient Proposition, the Cut Lemma, and Theorem \ref{theorem_n=2}. The validity of the General Quotient Proposition and the General Cut Lemma for $n$-S-NE graphs are shown assuming that Theorem \ref{theorem_GCC_disjoint_cycles} hold for $(n-1)$-S-NE graphs. So, we are proving the theorem for $n$-S-NE graphs assuming that
the General Quotient Proposition and the General Cut Lemma hold for $n$-S-NE graphs.

First, we show that the assumptions of the General Quotient Proposition can be made to hold for graphs equivalent with $E, F$ and a $\POM^D$-isomorphism  which is realizable if and only if $f$ is realizable.

The graphs $E/H_1$ and $F/G_1$ are equivalent so they have isomorphic canonical forms and such forms can be founds so that the maps $\psi_E: E/H_1\to (E/H_1)_{\can},$ $\psi_F: F/G_1\to (F/G_1)_{\can},$ and $\iota_{n/1,\can}:  (E/H_1)_{\can}\cong  (F/G_1)_{\can}$ are such that
$f_{n/1}=\ol{\psi_F^{-1}\iota_{n/1, \can}\psi_E}.$ The $(n-2)$-quotients of $E/H_1, (E/H_1)_{\can}\cong (F/G_1)_{\can},$ and $F/G_1$ are isomorphic by an isomorphism $\iota_{n/2}$ such that $f_{n/1}=\left(\ol{\psi_F^{-1}\iota_{n/1, \can}\psi_E}\right)_{n/2}=f_{n/1}=\left(\ol{\iota_{n/1, \can}}\right)_{n/2}=\iota_{n/2}$ because the exit moves, reductions, and the cuts which maps $\psi_E$ and $\psi_F$ may involve do not change the $(n-2)$-quotients. So, we have that the $(n-2)$-quotients are isomorphic by an isomorphism which induces $f_{n/2}.$

If $c_2$ does not emit exits to $H_1,$ then $c_2$ is terminal for $E$ and we have that $f([v_{20}])=f_{n/1}([v_{20}])=[v_{20}'].$ If $c_2$ emits edges to $H_1,$ we consider the 2-S-NE graph $P_{H_2}$ and the restriction $f_2$ of $f$ to the $\Gamma$-monoid of $P_{H_2}.$ By the proof of Theorem \ref{theorem_n=2}, we have that canonical forms $E$ and $F$ can be found so that $P_{H_2}\cong P_{G_2}.$ Thus the 21-connecting matrices are equal and they are computed using the cycles such that $f_{2/1}([v_{20}])=[v_{20}']$ and $f_2([v_{10}])=[v_{10}']$.

If $c_2$ is proper and $c_1$ is not, then
$f_2$ can be realized by a graph isomorphism $\iota_2: P_{H_2}\cong P_{G_2}$ such that $\iota_2(v_{20})=v_{20}'$ by the proof of Theorem \ref{theorem_n=2}. Thus, we have that $f([v_{20}])=[v_{20}']=[\iota_2(v_{20})].$ In all other cases when $f_2$ is not realized by a graph isomorphism, there are 2-canonical forms $P_{H_2}'$ and $P_{G_2}'$, exit moves and cuts $\phi_{E,2}:P_{H_2}\to P_{H_2}'$ and $\phi_{F,2}:P_{G_2}\to P_{G_2}'$ and
a graph isomorphism $\iota_2: P_{H_2}'\cong P_{G_2}'$ such that $f_2=\ol{\phi_{F,2}^{-1}\iota_2\phi_{E, 2}}$ by
the proof of Theorem \ref{theorem_n=2}. These exit moves and cuts do not impact the $(n-2)$-quotients.

Let $E_2$ be the graph obtained by the moves and cuts as in
$\phi_{E,2}$ and $F_2$ the graph obtained by the moves and cuts as in $\phi_{F,2}.$
We still use $\phi_{E,2}$ and $\phi_{F,2}$ for the operations $E\to E_2$ and $ F\to F_2.$ Let $g$ be $\ol\phi_{F,2}f\ol\phi_{E, 2}^{-1}.$ Since the moves
$\phi_{E,2}$ and $\phi_{F,2}$ do not impact $E/H_1$ and $F/G_1,$ $g_{n/1}=f_{n/1}.$ As $E$ and $F$ are $\mathbf L$-reduced, $E_2$ and $F_2$ are still canonical forms of $E$ and $F$, so we can consider $E_2$ and $F_2$ instead of $E$ and $F$ and $g$ instead of $f.$ Hence, we can assume that $f([v_{20}])=[v_{20}'].$

So, we have that all assumptions of the General Quotient Proposition hold.
This enables us to deduce that there is a graph isomorphism
$\iota_{n/1}:E/H_1\cong F/G_1$
such that $f_{n/1}=\ol\iota_{n/1}.$ The existence of $\iota_{n/1}$ also implies that $A_{lj}=A_{lj}'$ for every $l>j>1.$

At this point, we start the consideration of the $j1$-parts for $j=3,\ldots, n$ in that order. First, $j=3.$
Let us write $f([v_{30}])$ in terms of $[v_{20}'],$  $[v_{10}']$ and $[v_{30}']$ if $m_3>0$ and $[v_{20}'],$  $[v_{10}'],$ and $[q_{Z}]$ for finite $Z\subseteq \so^{-1}(v_{30})$ if $m_3=0.$ Just as in the proof of Theorem \ref{theorem_n=2}, we consider the cases $m_3>0$ and $m_3=0$ separately.

If $m_3>0,$ let $a_{3l}$ be the connecting polynomial of the $c_3$-to-$c_l$ part for $l=1,2,$ defined analogously as before, and let $a_{3l}'$ be analogous such polynomial for $F$. Let
$f([v_{30}])=t^{lm_3}[v_{30}']+q+b_1[v_{10}']$ for some $l\geq 0$ some $b_1\in \Zset^+[t]$ and some  element $q$ of $M^\Gamma_{P_{G_2}}.$ Since $[v_{30}']=f_{n/1}([v_{30}])=t^{lm_3}[v_{30}']+q,$ $q$ is equal to $\sum_{j=0}^{l-1}t^{m_3j}a_{32}[v_{20}'].$ Let us shorten this last term to $b_2[v_{20}'].$
The relation
\begin{equation}
[v_{30}]=t^{m_3}[v_{30}]+a_{32}[v_{20}]+a_{31}[v_{10}]
\label{equation_relation_in_n_case}
\end{equation}
holds in $M_{E}^\Gamma$ and an analogous such relation holds in $M_{F}^\Gamma.$ Since $A_{32}=A'_{32},$ we have that $a_{32}=a_{32}'.$
Using relation (\ref{equation_relation_in_n_case}), we have that \[t^{lm_3}[v_{30}']+b_2[v_{20}']+b_1[v_{10}']=f([v_{30}])=f(t^{m_3}[v_{30}]+a_{32}[v_{20}]+a_{31}[v_{10}])=\]\[t^{(l+1)m_3}[v_{30}']+t^{m_3}b_2[v_{20}']+t^{m_3}b_1[v_{10}']+a_{32}[v_{20}']+a_{31}[v_{10}'].\]
By using the analogue of relation (\ref{equation_relation_in_n_case}) in $M_F^\Gamma,$ we have that  $t^{lm_3}[v_{30}']+b_2[v_{20}']+b_1[v_{10}']=t^{(l+1)m_3}[v_{30}']+t^{lm_3}a_{32}'[v_{20}']+t^{lm_3}a_{31}'[v_{10}']+b_2[v_{20}']+b_1[v_{10}'].$
Thus, we have that
\[t^{(l+1)m_3}[v_{30}']+t^{lm_3}a_{32}'[v_{20}']+t^{lm_3}a_{31}'[v_{10}']+b_2[v_{20}']+b_1[v_{10}']=\]\[t^{(l+1)m_3}[v_{30}']+t^{m_3}b_2[v_{20}']+t^{m_3}b_1[v_{10}']+a_{32}[v_{20}']+a_{31}[v_{10}']\]
and the terms with $[v_{30}']$ cancel out.
By using that $b_2[v_{20}']=\sum_{j=0}^{l-1}t^{m_3j}a_{32}[v_{20}'],$ we can cancel all terms with $[v_{20}']$ and we arrive to
\begin{equation}
t^{lm_3}a_{31}'[v_{10}']+b_1[v_{10}']=t^{m_3}b_1[v_{10}']+a_{31}[v'_{10}]
\label{equation_the_main_one_in_n_case}
\end{equation}
which generalizes  (\ref{equation_the_main_one}) of the $n=2$ case. If $l', b_1',$ and $b_2'$ are defined analogously as $l, b_1,$ and $b_2$ for $f^{-1},$ we  consider the relation $f^{-1}(f([v_{30}]))=[v_{30}]$ and have that
\[t^{(l+l')m_3}[v_{30}]+\sum_{j=0}^{l+l'-1}t^{jm_3}a_{32}[v_{20}]+\sum_{j=0}^{l+l'-1}t^{jm_3}a_{31}[v_{10}]=[v_{30}]=f^{-1}(f([v_{30}]))=\]\[f^{-1}(t^{lm_3}[v_{30}']+b_2[v_{20}']+b_1[v_{10}'])=t^{(l+l')m_3}[v_{30}]
+t^{lm_3}b_2'[v_{20}]+t^{lm_3}b_1'[v_{10}]
+b_2[v_{20}]+b_1[v_{10}]=\]\[t^{(l+l')m_3}[v_{30}]
+\sum_{j=0}^{l+l'-1}t^{lm_3}a_{32}[v_{20}]+t^{lm_3}b_1'[v_{10}]
+b_1[v_{10}]\]
which implies that $(t^{lm_3}b_1'+b_1)[v_{10}]=\sum_{j=0}^{l+l'-1}t^{jm_3}a_{31}[v_{10}].$  Using $f(f^{-1}([v_{30}']))=[v_{30}']$ similarly, we arrive to
the analogues of relations (\ref{equation_b_and_b_prim}) of the $n=2$ case:
\begin{equation}
(t^{lm_3}b_1'+b_1)[v_{10}]=\sum_{j=0}^{l+l'-1}t^{jm_3}a_{31}[v_{10}]\;\;\mbox{ and }\;\;(t^{l'm_3}b_1+b'_1)[v_{10}']=\sum_{j=0}^{l+l'-1}t^{jm_3}a_{31}'[v_{10}'].
\label{equation_b_and_b_prim_in_n_case}
\end{equation}
Thus, $f([v_{30}])$ depends only on the $c_3$-to-$c_1$ part. Having the formulas  (\ref{equation_the_main_one_in_n_case}) and
(\ref{equation_b_and_b_prim_in_n_case}) which are analogous to (\ref{equation_the_main_one}) and (\ref{equation_b_and_b_prim}), enables us to deduce that $f([v_{30}])=t^{lm_3}[v_{30}']+b_1[w_{10}']$ for $l$ and $b_1$ having the following form. If
$L_1$ is the least common multiple of $m_1$ and $m_3$, the map $\phi_{31, E}: [v_{30}]\mapsto [v_{30}']+b_1[w_{10}']$ corresponds to moving  moving all 31-exits by a nonnegative integer multiple of $L_1$ and some of the 31-exits $L_1$ times. By the same arguments as in the proof of Theorem \ref{theorem_n=2}, $l$ divides $m_1$ and if $k=\frac{m_1}{l},$ then the operation $\phi_{31, F}^{-1}:[v_{30}]\mapsto t^{lm_3}[v_{30}]$
corresponds to the inverse of $\phi_{31, F}$ which is moving all 31-exits $kL_1$ times. If $E_{31}'$ and $F_{31}'$ are graphs resulting from applying $\phi_{31, E}$ to $E$ and applying $\phi_{31, F} $ to $F$, then the composition $g_{31}=\ol\phi_{31, F}f \ol\phi_{31, E}^{-1}$ is such that $g_{31}([v_{30}])=[v_{30}'].$ Thus, the values $b_1$ and $l$ for $g_{31}$ are zeros. With these values, the formula (\ref{equation_the_main_one_in_n_case}) implies that the 31-connecting polynomials are equal modulo $m_1$. Since the maps
$\phi_{31, E}$ and $\phi_{31, F}$ do not change the $31$-connecting polynomials, $E$ and $F$ have equal 31-connecting polynomials. By the canonical format of the 31-exits, the 31-connecting matrices of $E$ and $F$ are equal.

As the 31-exit moves do not impact the $(n-1)$-quotients nor the 21-connecting part, the $(n-1)$-quotients are isomorphic and the 21-connecting matrices are equal for all four graphs $E$, $F,$$E_{31}',$ and $F_{31}'$. We also have that
$(g_{31})_{n/1}=f_{n/1}=\ol\iota_{n/1}$ and that
$g_{31}([v_{i0}])=[v_{i0}']$ for $i=1,2.$

If $m_1=0,$ the formulas  (\ref{equation_the_main_one_in_n_case}) and (\ref{equation_b_and_b_prim_in_n_case}) and the requirement that the spine of the 31-part is as short as possible imply that
$f([v_{30}])=[v_{30}].$ In this case, we let $E=E_{31}',$  $F=F_{31}',$ and we let
$\phi_{31, E}$ and $\phi_{31, F}$ be the identity maps, and $g_{31}$ be $f.$

If $m_3=0,$ Let $f([v_{30}])=[q_{Z_1'\cup Z_2'}]+b_2[v_{02}']+b_1[v_{10}']$ for some finite $Z'_1\subseteq \so^{-1}(v_{30}')\cap \ra^{-1}(G_1),$ finite $Z'_2\subseteq \so^{-1}(v_{30}')\cap \ra^{-1}(G_2-G_1),$ and some $b_1, b_2\in \Zset^+[t].$ Since $f_{n/1}([v_{30}])=[v_{30}'],$
$b_2[v_{20}']$ is $a_{Z'_2}[v_{20}'],$ so
\[f([v_{30}])=[q_{Z_1'\cup Z_2'}]+a_{Z_2'}[v_{20}']+b_1[v_{10}']=[q_{Z_1'}]+b_1[v_{10}']\]
holds showing that $f([v_{30}])$ depends only on the $c_3$-to-$c_1$ part.
To follow the notation from the proof of Theorem \ref{theorem_n=2}, let $Z_1'=\emptyset_F$ and $b_1=b_{\emptyset}$. Considering $f^{-1}$ and letting
$f^{-1}([v_{30}])=[q_{\emptyset_E}]+b_{\emptyset}'[v_{10}],$ we can repeat the same arguments used in the $n=2$ case. The conditions $f^{-1}(f([v_{30}]))=[v_{30}]$
and  $f(f^{-1}([v_{30}']))=[v_{30}']$ imply that
$(b_{\emptyset_E} + b'_{\emptyset})[v_{10}'] = a_{(\emptyset_E)_F}[v_{10}']$ and that $(b_{\emptyset_F}' + b_{\emptyset})[v_{10}] = a_{(\emptyset_F)_E}[v_{10}].$
This implies the existence of
finite sets $W\subseteq \so^{-1}(v_{30})$ and $W'\subseteq \so^{-1}(v_{30}')$
such that $b_{\emptyset}[v_{10}]=a_W[v_{10}]$ and
$b_{\emptyset}'[v_{10}']=a_{W'}[v_{10}']$ The map $\phi_{31, E}: [v_{30}]\to [v_{30}]+a_{\emptyset_E}[v_{10}']$ is moving exits in $\emptyset_E$. Let $E_{31}'$ be the result of making these moves in $E$. Let $\phi_{31, F}: [v_{30}']\to [v_{30}']+a_{W'}[v_{10}']=[v_{30}']+b_{\emptyset}[v_{10}']$ be the operation of moving exits in $W'$ and let $F_{31}'$ be the resulting graph. If $g_{31}=\ol\phi_{31, F}f \ol\phi_{31, E}^{-1},$ then $g_{31}([v_{30}])=\ol\phi_{31, F}f ([q_{\emptyset_E}])=\ol\phi_{31, F}([q_{(\emptyset_E)_F}]+b_{\emptyset_E}[v_{10}'])
=[q_{(\emptyset_E)_F}]+b_{\emptyset_E}[v_{10}']+b_{\emptyset}[v_{10}']=[q_{(\emptyset_E)_F}]+a_{(\emptyset_E)_F}[v_{10}']
=[v_{30}'].$

This relation implies that the 31-connecting matrices of $E_{31}'$ and $F_{31}'$ are equal. If $m_1>0$ or if $k_{31}$ is finite, this implies that $E$ and $F$ have the same 31-connecting matrix by the definition of the canonical forms. If $m_1=0$ and $k_{31}=\omega,$ we may not have that $E$ and $F$ have the 31-connecting matrices equal but that holds for
$E_{31}'$ and $F_{31}'$. As the 31-exit moves do not impact the $(n-1)$-quotients nor the 21-connecting part, we have that the $(n-1)$-quotient and the 21-connecting part of $E_{31}'$ is the same as those elements of $E$ and that the same holds for $F_{31}'$ and $F$. We also have that $(g_{31})_{n/1}=f_{n/1}=\ol\iota_{n/1}$ and that
$g_{31}([v_{i0}])=[v_{i0}']$ for $i=1,2,3.$

If $n=3,$ we let $E_{31}$ and $F_{31}$ be the cut forms of $E_{31}'$ and $F_{31}'$, $\iota_{\cut, E}: E_{31}'\to E_{31}$ and $\iota_{\cut, F}: F_{31}'\to F_{31}$ be the cut maps, and $g=\ol\iota_{\cut, F}g_{31} \ol\iota_{\cut, E}^{-1}$.
The assumptions of the General Cut  Lemma hold for $E_{31}, F_{31}$ and $g$ and we have that there is an isomorphism $\iota: E_{31}\cong F_{31}$ such that $g=\ol\iota.$ If we $\mathbf L$-reduce the graphs $E_{31}\cong F_{31}$ (in most cases this is not needed, only in the $m_1=m_3=0,$ $k_{31}=\omega$ it may be needed), then we arrive to isomorphic canonical forms $(E_{31})_{\mathbf L\text{-}\red}$ and $(F_{31})_{\mathbf L\text{-}\red}$
such that $f=\ol{\phi_F\iota\phi_E}$ for
$\phi_E=\iota_{\mathbf L\text{-}\red, E}\iota_{\cut, E}
\phi_{31, E}$ and $\phi_F=\iota_{\mathbf L\text{-}\red, F}\iota_{\cut, F}
\phi_{31, F}.$

If $n>3,$ we continue the process by considering $E_{31}'$ instead of $E=E_{21}'$, $F_{31}'$ instead of $F=F_{21}'$, and $g_{31}$ instead of $f$. Repeating the argument results in graphs $E_{41}',$ $F_{41}',$ and exit moves $\phi_{41, E}: E_{31}'\to E_{41}',$ $\phi_{41, F}: F_{31}'\to F_{41}'$ such that the $(n-1)$-quotient and the 31- and the 21-connecting parts of $E_{41}'$ are the same as those elements of $E_{31}'$ and that the same holds for $F_{41}',$ $F_{31}',$ and $F$. If $g_{41}=\ol\phi_{41, F}g_{31} \ol\phi_{41, E}^{-1},$ then
$g_{41}([v_{40}])=[v_{40}'].$ This implies that the 41-connecting matrices of $E_{41}'$ and $F_{41}'$ are equal.  We also have that
$(g_{41})_{n/1}=f_{n/1}=\ol\iota_{n/1}$ and that
$g_{41}([v_{i0}])=[v_{i0}']$ for $i=1,2,3,4.$

At the last step of the process, we obtain the graphs $E_{n1}'$ and $F_{n1}'$ and
exit moves $\phi_{n1, E}: E_{(n-1)1}'\to E_{n1}'$ and $\phi_{n1, F}: F_{(n-1)1}'\to F_{n}'$
such that
$g_{n1}([v_{j0}])=[v_{j0}']$ for all $j=1\ldots, n$ where
$g_{n1}=\ol\phi_{n1, F}g_{(n-1)1} \ol\phi_{n1, E}^{-1}$ and such that $E_{n1}'$ and $F_{n1}'$ have $j1$-connecting matrices equal for all $j<n.$
The relation
$g_{n1}([v_{n0}])=[v_{n0}']$ implies that the $n1$-connecting matrices of $E_{n1}'$ and $F_{n1}'$ are also equal.

Let $E_{n1}$ and $F_{n1}$ be the cut forms of $E_{n1}'$ and $F_{n1}'$, let $\iota_{\cut, E}$ and
$\iota_{\cut, F}$ be the cut maps, and let
\begin{center}
$\psi_E=\iota_{\cut, E}
\phi_{n1, E}\ldots\phi_{41, E}\phi_{31, E}$ and $\psi_F=\iota_{\cut, F}
\phi_{n1, F}\ldots\phi_{41, F}\phi_{31, F}.$
\end{center} We have that
$E_{n1},$ $F_{n1},$ and $g=\ol\iota_{\cut, F}g_{n1} \ol\iota_{\cut, E}^{-1}$ satisfy the assumptions of the General Cut Lemma, so there is an isomorphism $\iota_{n1}: E_{n1}\cong F_{n1}$ such that $g=\ol\iota'$. So,
$f=\ol{\psi_f^{-1}\iota'\psi_E}$. Finally, $\mathbf L$-reducing isomorphic graphs their isomorphism does not change, so if $\iota_{\mathbf L\text{-}\red, E}: E_{n1}\to (E_{n1})_{\mathbf L\text{-}\red}$ and $\iota_{\mathbf L\text{-}\red, F}: F_{n1}\to (F_{n1})_{\mathbf L\text{-}\red}$ and
$\iota:  (E_{n1})_{\mathbf L\text{-}\red}\cong (F_{n1})_{\mathbf L\text{-}\red},$ then $\iota'=\iota_{\mathbf L\text{-}\red, F}^{-1}\iota\iota_{\mathbf L\text{-}\red, E}.$ So, for
$\phi_E=\iota_{\mathbf L\text{-}\red, E}\psi_E$ and $\phi_F=\iota_{\mathbf L\text{-}\red, F}\psi_F,$ $f$ can be realized as \[f=\ol{\phi_F^{-1}\iota\phi_E}.\]
The isomorphic graphs $(E_{n1})_{\mathbf L\text{-}\red}$ and $(E_{n1})_{\mathbf L\text{-}\red}$ are canonical since their quotients are $(n-1)$-canonical, their exits have the prescribed format and they are cut and $\mathbf L$-reduced.
\end{proof}

Theorem \ref{theorem_GCC_disjoint_cycles} implies the following corollary.

\begin{corollary} {\bf The GCC holds for graphs with disjoint cycles, finitely many vertices and countably many edges.}
If $E$ and $F$ are graphs with disjoint cycles, finitely many vertices and countably many edges the statement of Theorem \ref{theorem_GCC_disjoint_cycles} holds.
\label{corollary_unital}
\end{corollary}
\begin{proof}
If $E$ has disjoint cycles and finitely many vertices, then
there can be only finitely many cycles, sinks and infinite emitters and $E$ is a countable $n$-S-NE graph for some $n.$
\end{proof}

\subsection{Graph \texorpdfstring{$\mathbf{C^*}$}{}-algebra version of the main result. The Graded Strong Isomorphism Conjecture}

If $E$ is a graph, the {\em graph $C^*$-algebra $C^*(E)$ of $E$} is the universal $C^*$-algebra generated by mutually orthogonal projections $\{p_v\mid v\in E^0\}$ and partial isometries with mutually orthogonal ranges $\{s_e\mid e\in E^1\}$ satisfying the analogues of the (CK1) and (CK2) axioms and the axiom (CK3) stating that $s_es_e^*\leq p_{\so(e)}$ for every $e\in E^1$ (where $\leq$ is the order on the set of projections given by $p\leq q$ if $p=pq=qp$). \cite[Definition 5.2.5]{LPA_book} has more details. By letting $s_{e_1\ldots e_n}$ be $s_{e_1}\ldots s_{e_n}$ and $s_v=p_v$ for $e_1,\ldots ,e_n\in E^1$ and $v\in E^0,$ $s_p$ is defined for every path $p.$

The set $\{p_v, s_e\mid v\in E^0, e\in E^1\}$ is referred to as a {\em Cuntz-Krieger $E$-family}.
For such an $E$-family and an element $z$ of the unit
circle $\mathbb T$, one defines a map $\gamma^E_z$ by $\gamma^E_z(p_v)=p_v$ and $\gamma^E_z(s_e)=zs_e$ and then uniquely extends this map to an automorphism of $C^*(E)$ (we assume a homomorphism of a $C^*$-algebra to be bounded and $*$-invariant). The {\em gauge action} $\gamma^E$ on $\mathbb T$ is given by $\gamma^E(z)=\gamma^E_z.$ Note that $\gamma^E_z(s_ps_q^*)=z^{|p|-|q|}s_ps_q^*$ for $z\in \mathbb T$ and paths $p$ and $q.$ The presence of the degree $|p|-|q|$ of $pq^*$ in the previous formula explains the connection of this action and the $\mathbb Z$-grading of $L_{\mathbb C}(E).$ In particular, the gauge action induces a $\Zset$-grading of $C^*(E)$ (see \cite[Section 3.1]{Lia_porcupine} for more details).
If $R$ is a $C^*$-algebra with an action $\beta: \mathbb T\to$ Aut$(R),$ we say that a homomorphism $\phi: C^*(E)\to R$ is {\em gauge-invariant} if $\beta_z\phi=\phi\gamma_z^E$ for every $z\in\mathbb T.$ In particular, $\phi: C^*(E)\to C^*(F)$ is gauge-invariant if  $\gamma^F_z \phi=\phi\gamma_z^E$ for every $z\in \mathbb T.$

In the case when the field $K$ is the field of complex numbers $\mathbb C$ considered with the complex-conjugate involution, the existence of a natural injective map $L_{\mathbb C}(E)\to C^*(E)$ (see \cite[Definition 5.2.1 and Theorem 5.2.9]{LPA_book}) implies that if there is a $*$-isomorphism of $L_{\mathbb C}(E)$ and $L_{\mathbb C}(F)$ for  graphs $E$ and $F,$ then there is an isomorphism of $C^*(E)$ and $C^*(F).$ In particular, if the algebra $*$-isomorphism is graded, then it induces a gauge-invariant isomorphism of the graph $C^*$-algebras. This  holds basically because for $\phi: C^*(E)\to C^*(F)$ induced by a graded algebra map, both $\gamma_z^F\phi(s_ps_q^*)$ and $\phi\gamma_z^E(s_ps_q^*)$ are equal to $z^{|p|-|q|}\phi(s_ps_q^*)$ for any paths $p,q$ of $E$.

The natural injection $L_{\mathbb C}(E)\to C^*(E)$ induces a natural monoid isomorphism $\V(L_{\mathbb C}(E))\to \V(C^*(E))$ (see \cite[Theorem 5.3.5]{LPA_book}) and, in fact, a natural $\Gamma$-monoid isomorphism $\V^\Gamma(L_{\mathbb C}(E))\to \V^\Gamma(C^*(E)).$ This isomorphism maps the elements $[v]$ and $[q_Z]$ onto themselves, so we have a natural isomorphism of both $\V^\Gamma(L_{\mathbb C}(E))$ and $ \V^\Gamma(C^*(E))$ and $M_E^\Gamma.$ Because of this, the graph $C^*$-algebra version of the main result holds for graph $C^*$-algebras as the following corollary shows.

\begin{corollary} {\bf The GCC holds for the graph $\mathbf{C^*}$-algebras of composition S-NE graphs.}
If $E$ and $F$ are countable composition S-NE graphs,
the conditions (1) to (3) from Theorem \ref{theorem_GCC_disjoint_cycles} are equivalent to the condition that there is a gauge-invariant isomorphism $\phi: C^*(E)\to C^*(F)$ such that $\ol\phi=f$ where $f$ is a map from condition (1) of Theorem \ref{theorem_GCC_disjoint_cycles}.
\label{corollary_C_star}
\end{corollary}
\begin{proof}
If condition (1) holds and $f$ is a map as in this condition, then there is a graded $*$-algebra isomorphism $\phi: L_{\mathbb C}(E)\to L_{\mathbb C}(F)$ such that $\ol\phi=f$ by Theorem \ref{theorem_GCC_disjoint_cycles}. The map $\phi$ extends to a gauge-invariant isomorphism $C^*(E)\to C^*(F),$ which we also call $\phi,$ such that $\ol\phi=f.$
\end{proof}

\begin{corollary} {\bf The Graded Strong Isomorphism Conjecture holds for composition S-NE graphs.}
If $E$ and $F$ are countable composition S-NE graphs, the conditions below are equivalent to conditions (1) to (3) of Theorem \ref{theorem_GCC_disjoint_cycles}.
\begin{enumerate}
\item[{\em (4)}] The algebras $L_K(E)$ and $L_K(F)$ are graded isomorphic as algebras.

\item[{\em (5)}] The algebras $L_K(E)$ and $L_K(F)$ are graded isomorphic as rings.

\item[{\em (6)}] The algebras $L_K(E)$ and $L_K(F)$ are graded isomorphic as $*$-rings.

\item[{\em (7)}] The algebras $C^*(E)$ and $C^*(F)$ are graded isomorphic.
\end{enumerate}
\label{corollary_iso_conjecture}
\end{corollary}
\begin{proof}
Since the implications (3) $\Rightarrow$ (4) $\Rightarrow$ (5), (3) $\Rightarrow$ (6) $\Rightarrow$ (5), and (5) $\Rightarrow$ (1) are direct and (1) $\Rightarrow$ (3) holds by Theorem \ref{theorem_GCC_disjoint_cycles},
we have that (1) to (6) are equivalent. The equivalence of (1) and (7) follows by Corollary \ref{corollary_C_star}.
\end{proof}

\subsection{Diagonal preservation and relation to Williams' Conjecture}\label{subsection_diagonal}

The {\em diagonal} of a Leavitt path algebra of $E$ is the $K$-linear span of the elements of the form $pp^*$ where $p$ is a path of $E.$
We use $D_E$ to denote this algebra. The diagonal-preserving maps of Leavitt path and graph $C^*$-algebras are particularly relevant due to the results from \cite{Arklint_et_al_diagonal} and \cite{Carlsen_diagonal}. Because of this, we show that conditions (1) to (7) are equivalent with conditions obtained by adding the requirement that the isomorphisms in conditions (3) to (7) are diagonal preserving. Using the terminology of \cite{Eilers_Ruiz_3-bit}, this shows that 111=110 for the class of graphs we consider.

We start with a short lemma.

\begin{lemma}
If $E$ and $F$ are any graphs and $\phi: L_K(E)\to L_K(F)$ is a graded $*$-homomorphism such that $\phi(v),\phi(ee^*)\in D_F$ and $\phi(e)qq^*\phi(e^*)\subseteq D_F$ for every $v\in E^0,$ $e\in E^1,$ and a path $q$ of $F,$ then $\phi(D_E)\subseteq D_F.$
\label{lemma_on_diagonal}
\end{lemma}
\begin{proof}
If $p$ is a path of $E$ of length zero or one, the assumed properties imply that $\phi(pp^*)$ is in $D_F.$ If $p=er$ where $e\in E^1$ and $r$ is a path of $E,$ we can assume that $\phi(rr^*)\in D_F$ by induction and have that
$\phi(pp^*)=\phi(e)\phi(rr^*)\phi(e^*)$ is in $ D_F$ by induction and the assumed properties.
\end{proof}

By the lemma above, the operations on graphs considered in this work preserve the diagonal.

\begin{lemma}
The graph operations realizing the relation $\approx$ on the class of composition S-NE graphs induce graded $*$-isomorphisms which preserve the diagonal.
\label{lemma_approx_preserves_the_diagonal}
\end{lemma}
\begin{proof}
If $E$ and $F$ are composition S-NE graphs such that $E\approx F,$ then one can be transformed to the other by using three types of operations and their inverses.
\begin{enumerate}
\item The operations making the exits direct and the tails and quotients canonical are of the same nature -- they map the vertices on the generators of the diagonal algebra (elements of the form $pp^*$ for paths $p$) and edges to the elements of the form $pq^*$ for paths $p$ and $q.$ We refer to such an operation as an {\em 1-S-NE move}.

\item Converting the graphs to their total out-split forms and considering blow-ups boils down to a series of out-splits and path rearranging. If the path rearranging involves only finitely many vertices it is equivalent to in-split plus moves. In particular, an exit move is a blow-up composed with path rearranging.

\item If the graphs have infinitely many vertices, the {\em tail-cuttings} of section \ref{subsection_tail_cutting} may be applicable.
\end{enumerate}

Because of this, in order to prove the lemma, it is sufficient to show that a path rearranging, an out-split, a tail cut, and the inverses of such operations induce diagonal-preserving isomorphisms.

Let us first consider a path rearranging $\phi: E\to F$ The graded $*$-isomorphism extension
$\phi: L_K(E)\to L_K(F)$ is such that for any $v\in E^0$ and $e\in E^1,$ $
\phi(v)=pp^*\in D_F$ for some path $p$ of $F$ and $\phi(e)=pq^*$ for some paths $p$ and $q$ of $F$ so that $\phi(ee^*)=pq^*qp^*=pp^*\in D_F.$ It is direct to check that $\phi(e)rr^*\phi(e)^*=pq^*rr^*qp^*$ is also a diagonal element: the products $q^*r$ and $r^*q$ are nonzero when either $q=rs$ or $r=qs$ for some path $s$. In the first case,
$pq^*rr^*qp^*=ps^*r^*rr^*rsp^*=ps^*sp^*=pp^*$ and, in the second case, $pq^*rr^*qp^*=pq^*qss^*q^*qp^*=pss^*p^*.$ So, in both cases, $\phi(e)rr^*\phi(e)^*$ is a diagonal element. The inverse of $\phi$ also maps edges to the elements of the form $pq^*$ for some paths $p$ and $q,$ so the argument that $\phi^{-1}$ preserves the diagonal is the same.

The proofs of
\cite[Theorems 2.1.2 and 2.2.3]{Eilers_Ruiz_3-bit}
contain arguments that an out-split and an in-split plus moves preserve the diagonal. This can be checked also using Lemma \ref{lemma_on_diagonal}. For an out-split, for example,  if
$E$ is a graph, $v$ a vertex which emits edges, and $\mathcal E_1,\ldots \mathcal E_n$ is a partition $\mathcal P$ of $\so^{-1}(v),$ then let $f_1, \ldots, f_n$ be the new edges which replace $f\in E^1$ such that $\ra(f)=v.$ Let $\phi$ be the graded $*$-isomorphism of the algebras of $E$ and $E_{v,\mathcal P}$ given by $\phi(v)=\sum_{i=1}^n v_i,$ $\phi(w)=w$ for $w\in E^0-\{v\},$ $\phi(f)=\sum_{i=1}^n f_i $ if $\ra(f)=v$ and $\phi(f)=f$ if $\ra(f)\neq v.$ This isomorphism preserves the diagonal since
$\phi(ff^*)=\sum_{i=1}^n f_if_i^*$ where we use that the edges $f_i$ have different ranges for different $i$ values. If $q$ is any path of $E_{v, \mathcal P},$ then the product
$\sum_{i=1}^n f_iqq^*\sum_{i=1}^nf_i^*$ is zero in the case that $\so(q)$ is different than $\ra(f_i)$ for any $i$ or the product is $f_iqq^*f_i^*$ for $i$ such that $\so(q)=\ra(f_i).$ In either case, this product is in $D_F.$  The argument for checking that the properties of the lemma hold for the inverse of $\phi$ is similar.

It remains to show that
the tail cutting and its inverse preserve the diagonal. To do that, let $\iota_{\cut}: E\to E_{\cut}$ be the tail cutting map from the proof of Proposition \ref{proposition_cut_form}. By the proof of Proposition \ref{proposition_cut_form}, the images of all vertices and edges except the edges $g_{j,l, n'}$ and their sources are mapped to elements of the form $pq^*$ for some paths $p$ and $q.$ So, the argument that the assumptions of Lemma \ref{lemma_on_diagonal} hold for $\iota_{\cut}$ and such vertices and edges is the same as  when 1-S-NE moves were considered. Hence, it remains to check that $\iota_{\cut}(g_{j,l, n'}g_{j,l, n'}^*)$ and $\iota_{\cut}(g_{j,l, n'})qq^*\iota_{\cut}(g_{j,l, n'}^*)$ are in the diagonal for any path $q$ of $E_{\cut}.$

We use the notation from the proof of Proposition \ref{proposition_cut_form} and we recall that
\[\iota_{\cut}(g_{j,l, n'})=
g_{j,l, n'}-\sum_{i\in C_{j,l,\leq\lfloor\frac{n'}{2}\rfloor}}g_{j,l, n'}e_{i}
+\sum_{i\in C_{j,l, \leq \lfloor \frac{n'}{2}\rfloor}}g_{j,l, 2n'+1}e_{i}.\]
Using this definition and the fact that the idempotents $e_i$ are mutually orthogonal, one can foil the terms of the product $\iota_{\cut}(g_{j,l, n'})\iota_{\cut}(g_{j,l, n'})^*$ and note that all the non-diagonal elements cancel out. When considering the product
$\iota_{\cut}(g_{j,l, n'})qq^*\iota_{\cut}(g_{j,l, n'}^*)$ for a path $q,$ this product is nonzero if and only if $q_i$ is a prefix of $q$ or $q$ is a prefix of $q_i$ for some $i\in C_{j,l, \leq \lfloor\frac{n'}{2}\rfloor}.$ If $q=q_is,$ then   $e_iqq^*=e_iq_iss^*q_i^*=q_iss^*q_i=qq^*.$  Thus, $\iota_{\cut}(g_{j,l,n'})qq^*=g_{j,l,n'}gg^*-g_{j,l,n'}gg^*+g_{j,l,2n'+1}gg^*=g_{j,l,2n'+1}gg^*$ and $\iota_{\cut}(g_{j,l,n'})qq^*\iota_{\cut}(g_{j,l,n'}^*)=g_{j,l,2n'+1}qq^*g_{j,l,2n'+1}^*$ is a diagonal element. If $q_i=qs,$
then  $e_iqq^*=qss^*q^*qq^*=qss^*q^*=e_i.$ This implies that  $qq^*e_i=e_i$ and these relations ensure that all non-diagonal terms of the product
$\iota_{\cut}(g_{j,l, n'})qq^*\iota_{\cut}(g_{j,l, n'}^*)$ cancel out.

Turning to $\iota_{\cut}^{-1},$ we recall that $\iota_{\cut}^{-1}(g_{j,l, 2n'+1}e_{i})=g_{j,l, n'}e_{i},$ so that the element
$\iota_{\cut}^{-1}(g_{(j,l, n')})$ has the same format as $\iota_{\cut}(g_{(j,l, n')})$: it is a sum of $g_{j,l, n'}-\sum_{i\in C_{j,l,\leq\lfloor\frac{n'}{2}\rfloor}}g_{j,l, n'}e_{i}$ and an element of the form
$\sum_{i\in C_{j,l, \leq \lfloor \frac{n'}{2}\rfloor}}x_iq_{i}^*$ where $x_i$ are some paths.  Because of this form, checking  that $\iota_{\cut}^{-1}$ satisfies the assumption of Lemma \ref{lemma_on_diagonal} is completely analogous to checking the same for $\iota_{\cut}.$
\end{proof}

\begin{proposition}
Let (8), (9) and (10) be the conditions below and let (8*) and (9*) be the conditions obtained by requiring that the isomorphisms in (8) and (9) are $*$-isomorphisms.
If $E$ and $F$ are composition S-NE graphs, the conditions (1) to (7) are equivalent to any of (8) to (10), (8*) and (9*).
\begin{enumerate}[\upshape(1)]
\item[{\em (8)}] There is a diagonal-preserving graded algebra isomorphism $L_K(E)\to L_K(F)$.

\item[{\em (9)}] There is a diagonal-preserving graded ring isomorphism $L_K(E)\to L_K(F)$.

\item[{\em (10)}] There is a diagonal-preserving graded (equivalently, equivariant) isomorphism $C^*(E)\to C^*(F)$.
\end{enumerate}
\label{proposition_on_diagonal}
\end{proposition}
\begin{proof}
The ``equivariant'' is equivalent with ``graded'' in condition (10) since the only group action we consider on the graph $C^*$-algebras is the gauge action, so ``equivariant'' and ``gauge-invariant'' are synonymous. The grading of a graph $C^*$-algebra is such that ``gauge-invariant'' and ``graded'' are equivalent terms. The implications (8*) $\Rightarrow$ (8) $\Rightarrow$ (9) $\Rightarrow$ (5), (8*) $\Rightarrow$ (9*) $\Rightarrow$ (6), and (8*) $\Rightarrow$ (10) $\Rightarrow$ (7)
are direct. Since we have that (1) to (7) are equivalent, it is sufficient to show that (2) $\Rightarrow$ (8*). This implication holds by Lemma \ref{lemma_approx_preserves_the_diagonal}.
\end{proof}

We note that the equivalent conditions (1) to (10) are also equivalent to the groupoids of the graphs being isomorphic by \cite[Theorem 6.1]{Ben_Steinberg}.

We conclude the paper by discussing the relation of the GCC with Williams' Problem. Williams' Problem, formulated for graphs instead of matrices, is asking whether the shift equivalence of the incidence matrices of two graphs implies the strong shift equivalence of those two matrices. This problem, formulated in 1973, is known to have a negative answer since 1997. Since then, a significant volume of work has been devoted to study the difference between shift equivalence and strong shift equivalence over nonnegative integers. Recently, consideration of various classes of finite graphs with disjoint cycles and their graph algebras resulted in affirmative answers to Williams' problem for some classes of graphs considered in \cite{Cordiero_et_al},  \cite{Do_et_al_Williams}, and \cite{Eilers_Ruiz_3-bit}.

The term ``Williams' Conjecture'' has been used to refer to the statement that Williams' Problem has an affirmative answer. Although it is not clear whether Williams' Conjecture implies GCC or vice versa, the two conjectures are related via Krieger's dimension group of a graph $E$ which is isomorphic to $G_E^\Gamma$ considered as an object of $\POG$ (see \cite{Roozbeh_Dynamics} or \cite{Eilers_Ruiz_3-bit}). By \cite[Corollary 9]{Roozbeh_Dynamics}, the incidence matrices of two finite graphs with no sinks are shift equivalent if and only if the $\Gamma$-groups of the two graphs are $\POG$-isomorphic. If such an isomorphism were a $\POG^u$-isomorphism and if these two graphs were S-NE graphs, then we would have that the conditions (1) to (10) hold. If the graphs have no sources then, by \cite[Theorem 7.3]{Carlsen_et_al}, the two incidence matrices would be strongly shift equivalent. So, for the implication GCC $\Rightarrow$ Williams' Conjecture for finite S-NE graphs without sources or sinks, the ``missing step'' is going from $\POG$ to $\POG^u$ isomorphism.

For the converse implication (Williams' Conjecture $\Rightarrow$ GCC) for finite S-NE graphs with no sinks and sources, the ``missing step'' is going from a stable isomorphism to an isomorphism. Indeed, assuming that the $\Gamma$-groups of two finite graphs are $\POG^u$-isomorphic, then the incidence matrices are shift equivalent and, by Williams' Conjecture, strongly shift equivalent. If these graphs have no sinks and sources, then their graph $C^*$-algebras are equivariantly stably isomorphic in a way
that respects the diagonals by \cite[Theorem 7.3]{Carlsen_et_al}. If one could obtain one such isomorphism without ``stable'' and if the graphs were S-NE graphs, then we would have that the Leavitt path algebras of such graphs are graded $*$-isomorphic.

\end{document}